%% file: padicHMF_v6_arXiv.tex
\documentclass[english,10pt]{amsart}
\usepackage{geometry}
\usepackage[displaymath, mathlines]{lineno}
\newcommand*\patchAmsMathEnvironmentForLineno[1]{%
  \expandafter\let\csname old#1\expandafter\endcsname\csname #1\endcsname
  \expandafter\let\csname oldend#1\expandafter\endcsname\csname end#1\endcsname
  \renewenvironment{#1}%
     {\linenomath\csname old#1\endcsname}%
     {\csname oldend#1\endcsname\endlinenomath}}%
\newcommand*\patchBothAmsMathEnvironmentsForLineno[1]{%
  \patchAmsMathEnvironmentForLineno{#1}%
  \patchAmsMathEnvironmentForLineno{#1*}}%
\AtBeginDocument{%
\patchBothAmsMathEnvironmentsForLineno{equation}%
\patchBothAmsMathEnvironmentsForLineno{align}%
\patchBothAmsMathEnvironmentsForLineno{flalign}%
\patchBothAmsMathEnvironmentsForLineno{alignat}%
\patchBothAmsMathEnvironmentsForLineno{gather}%
\patchBothAmsMathEnvironmentsForLineno{multline}%
}

\usepackage{amsthm}
\usepackage{amsmath}
\usepackage[normalem]{ulem}
\usepackage{amssymb}
\usepackage{mathrsfs}
\usepackage{hyperref}
\usepackage{verbatim}
\usepackage{cancel}
\usepackage[all,cmtip]{xy}
\usepackage{enumerate}
\hypersetup{colorlinks=true,citecolor=blue}

\numberwithin{equation}{subsection}

\theoremstyle{plain}
\newtheorem{thm}{Theorem}[subsection]
  \theoremstyle{plain}
  \newtheorem{prop}[thm]{Proposition}
    \newtheorem{defn-prop}[thm]{Definition/Proposition}
  \theoremstyle{definition}
  \newtheorem{defn}[thm]{Definition}
  \newtheorem{exam}[thm]{Example}
  
  \newtheorem{rmk}[thm]{Remark}

  \theoremstyle{plain}
  \newtheorem{lem}[thm]{Lemma}
    \newtheorem{cor}[thm]{Corollary}
\theoremstyle{remark}
\newtheorem*{claim}{Claim}

\newcommand{\sce}{\mathscr{E}}

\newcommand{\scw}{\mathscr{W}}
\newcommand{\scm}{\mathscr{M}}

\newcommand{\sco}{\mathscr{O}}

\newcommand{\scd}{\mathscr{D}}

\newcommand{\GL}{\operatorname{GL}}
\newcommand{\leftmod}{\backslash}
\newcommand{\alg}{\operatorname{alg}}
\newcommand{\rec}{\operatorname{rec}}
\newcommand{\Art}{\operatorname{Art}}
\newcommand{\Gal}{\operatorname{Gal}}
\newcommand{\Cl}{\operatorname{Cl}}
\newcommand{\Spf}{\operatorname{Spf}}
\newcommand{\Sp}{\operatorname{Sp}}
\newcommand{\rig}{\operatorname{rig}}

\newcommand{\ad}{\operatorname{ad}}
\newcommand{\Hom}{\operatorname{Hom}}

\newenvironment{smallpmatrix}
  {\left(\begin{smallmatrix}}
  {\end{smallmatrix}\right)}
\newcommand{\pd}{\operatorname{pd}}
\newcommand{\depth}{\operatorname{depth}}
\newcommand{\hol}{\operatorname{hol}}
\newcommand{\dR}{\operatorname{dR}}
\newcommand{\Frob}{\operatorname{Frob}}
\newcommand{\sgn}{\operatorname{sgn}}
\newcommand{\BM}{\operatorname{BM}}
\newcommand{\PD}{\operatorname{PD}}
\newcommand{\cl}{\operatorname{cl}}
\newcommand{\tw}{\operatorname{tw}}
\newcommand{\ES}{\operatorname{ES}}
\newcommand{\im}{\operatorname{im}}
\newcommand{\dirlim}{\varinjlim}
\newcommand{\invlim}{\varprojlim}
\newcommand{\pr}{\operatorname{pr}}
\newcommand{\tr}{\operatorname{tr}}
\newcommand{\Res}{\operatorname{Res}}
\newcommand{\sstable}{\operatorname{st}}
\newcommand{\crys}{\operatorname{crys}}
\newcommand{\Ext}{\operatorname{Ext}}
\newcommand{\LT}{\operatorname{LT}}
\newcommand{\HT}{\operatorname{HT}}
\providecommand{\Ref}{\operatorname{Ref}}
\renewcommand{\Ref}{\operatorname{Ref}}
\newcommand{\End}{\operatorname{End}}
\newcommand{\rmmid}{\operatorname{mid}}
\newcommand{\WD}{\operatorname{WD}}
\newcommand{\nr}{\operatorname{nr}}

\newcommand{\ab}{\operatorname{ab}}
\newcommand{\soc}{\operatorname{soc}}
\newcommand{\new}{\operatorname{new}}
\newcommand{\red}{\operatorname{red}}
\newcommand{\supp}{\operatorname{supp}}
\newcommand{\Tor}{\operatorname{Tor}}
\newcommand{\Ann}{\operatorname{Ann}}

\newcommand{\SO}{\operatorname{SO}}

\newcommand{\cycl}{\operatorname{cycl}}
\newcommand{\ev}{\operatorname{ev}}
\newcommand{\Spec}{\operatorname{Spec}}

\usepackage{babel}

\usepackage[dvipsnames]{xcolor}

\newcommand\redsout{\bgroup\markoverwith{\textcolor{red}{\rule[0.5ex]{2pt}{1pt}}}\ULon}
\newcommand\bluesout{\bgroup\markoverwith{\textcolor{magenta}{\rule[0.5ex]{2pt}{1pt}}}\ULon}


\begin{document}

\title{On $p$-adic $L$-functions for Hilbert modular forms}

\author{John Bergdall}
\address{John Bergdall\\Department of Mathematics\\Bryn Mawr College\\101 North Merion Avenue \\Bryn Mawr, PA 19096\\USA}
\email{jbergdall@brynmawr.edu}
\urladdr{https://jbergdall.digital.brynmawr.edu/}

\author{David Hansen} 
\address{David Hansen\\Max Planck Institute for Mathematics\\Vivatsgasse 7\\53111 Bonn\\Germany}
\email{dhansen@mpim-bonn.mpg.de}
\urladdr{http://www.davidrenshawhansen.com}

\begin{abstract} We construct $p$-adic $L$-functions associated with $p$-refined cohomological cuspidal Hilbert modular forms over any totally real field under a mild hypothesis.  Our construction is canonical, varies naturally in $p$-adic families, and does not require any small slope or non-criticality assumptions on the $p$-refinement.  The main new ingredients are an adelic definition of a canonical map from overconvergent cohomology to a space of locally analytic distributions on the relevant Galois group, and a smoothness theorem for certain eigenvarieties at critically refined points.
\end{abstract}

\maketitle
\setcounter{tocdepth}{2}
\tableofcontents

\section{Introduction}\label{sec:introduction}

\input{altintro.tex}

\section{Cohomology and local systems}\label{sec:cls}
\input{topology.tex}

\section{Hilbert modular forms}\label{sec:hmf}
\input{hmf_basics.tex}

\section{Algebraicity of special values}\label{sec:shimuras-theorem}
\input{shimuras-theorem.tex}

\section{Locally analytic distributions and $p$-adic weights}\label{sec:loc-anal-dist}
\input{distributions.tex}

\section{The eigenvariety}\label{sec:mid-eigen}
\input{eigen.tex}

\section{Period maps}\label{sec:period-maps}
\input{period.tex}

\section{$p$-adic $L$-functions}\label{sec:padicL-functions}
\input{final_section.tex}

\begin{appendix}

\section{A deformation calculation}\label{app:semistable}

\input{semistable-appendix.tex}

\section{Decency}\label{app:decent}

\input{decent-appendix.tex}

\end{appendix}

\bibliography{bibpadichmf}
\bibliographystyle{abbrv}

\end{document}

%% file: altintro.tex
The goal of this article to define canonical $p$-adic $L$-functions associated with $p$-refined cohomological cuspidal automorphic representations of ${\GL_2}$ over totally real number fields. We make no assumptions on the so-called {\em slope} (other than finiteness), and our construction varies naturally in $p$-adic families.

\subsection{The main result}
To state our results we begin by setting notation. Let $F$ be a totally real number field of degree $d$ and write $\Sigma_F$ for the set of embeddings $F \hookrightarrow \mathbf R$. The completion of $F$ at a place $v$ will be written $F_v$; the ramification index will be written $e_v$; the residue field will have $q_v$-many elements. We write $\pi$ for a cohomological cuspidal automorphic representation of $\GL_2(\mathbf A_F)$ and $\lambda$ for its weight. Throughout the introduction we will omit `cohomological cuspidal' and simply refer to $\pi$ as an automorphic representation, except when more precision is helpful. In our normalization, the cohomological condition means the weight $\lambda$ is a pair $(\kappa,w)$ such that $\kappa = (\kappa_{\sigma})_{\sigma \in \Sigma_F}$ is a $\Sigma_F$-tuple of non-negative integers, $w \in \mathbf Z$, and $\kappa_{\sigma} \equiv w \bmod 2$.  An integer $m$ is called (Deligne-)critical with respect to $\lambda$ if 
\begin{equation*}
{w-\kappa_{\sigma} \over 2} \leq  m \leq { w+ \kappa_{\sigma}\over 2} \;\;\;\;\; (\forall \sigma \in \Sigma_F).
\end{equation*}
For precise explanations of the basic definitions and normalizations, see Sections \ref{sec:cls} and \ref{sec:hmf}.

The starting point of our work is a famous algebraicity result of Shimura for special values of the $L$-functions associated with such $\pi$. More precisely, for any finite order Hecke character $\theta$ we may consider the completed $L$-function $\Lambda(\pi\otimes \theta,s)$ associated with the twist of $\pi$ by $\theta$. It is entire in the variable $s$, and it satisfies a functional equation under $s \mapsto w+2-s$. Shimura proved (\cite{Shimura-HMF}) that there is a collection of periods $\Omega_\pi^{\epsilon} \in \mathbf C^\times$ indexed by signs $\epsilon = (\epsilon_\sigma) \in \{\pm 1\}^{\Sigma_F}$ with the property that for any integer $m$ critical with respect to $\lambda$ and any finite order $\theta$, the number
\begin{equation*}
\Lambda^{\alg}(\pi\otimes \theta,m+1) := {\left(\prod_{\sigma\in\Sigma_F} \theta_\sigma(-1) i^{1+m+{\kappa_\sigma-w\over 2}}\right) \Delta_{F/\mathbf Q}^{m+1} \Lambda(\pi\otimes \theta,m+1)\over \Omega_\pi^{\epsilon} G(\theta)}
\end{equation*}
lies in the field $\mathbf Q(\pi,\theta)$ generated by the Hecke eigenvalues of $\pi$ together with the values of $\theta$. Here $G(\theta)$ is a certain Gauss sum and the sign $\epsilon$ is determined by  $\epsilon_\sigma = (-1)^m \theta_\sigma(-1)$ for all $\sigma\in\Sigma_F$ ($\theta_\sigma$ being the $\sigma$-th component of $\theta$). Technically, Shimura assumes the weights $\kappa_\sigma$ are 3 or larger. See Theorem 4.3 of {\em loc.\ cit.} We will give a complete exposition of this result in Section \ref{sec:shimuras-theorem}, roughly following Hida (\cite{Hida-CriticalValues}).

Now let $p$ be a prime number. We will fix an isomorphism $\iota: \mathbf C \overset{\sim}{\longrightarrow} \overline{\mathbf Q}_p$ where $\overline{\mathbf Q}_p$ is a fixed algebraic closure of the field of $p$-adic numbers $\mathbf Q_p$. It then makes sense to try to $p$-adically interpolate the algebraic special values $\iota\left(\Lambda^{\alg}(\pi\otimes \theta,m+1)\right)$ as $m$ and $\theta$ vary. 

In order to do this, we introduce a certain $p$-adic analytic space of characters.  Let $\Gamma_F$ be the Galois group of the maximal abelian extension of $F$ unramified away from $p$ and $\infty$. This is a compact and abelian topological group. It also contains an open (so finite index) subgroup topologically isomorphic to finitely many copies of the $p$-adic integers $\mathbf Z_p$. Given any such group, there is a canonically associated rigid analytic character variety $\mathscr X(\Gamma_F)$ whose $\mathbf C_p$-points correspond to continuous characters $\Gamma_F \rightarrow \mathbf C_p^\times$. In particular, if $\theta$ is finite order Hecke character with $p$-power conductor, then $\theta^{\iota} := \iota\circ \theta$ defines a point in $\mathscr X(\Gamma_F)$. By global class field theory, each character $\chi \in \mathscr X(\Gamma_F)$ can be seen as a $p$-adic Hecke character and so in particular has signs, at infinity, as above. The group $\Gamma_F$ and its character variety play a key role in this article: our $p$-adic $L$-functions will be elements in the ring $\mathscr O(\mathscr X(\Gamma_F))$ of rigid analytic functions on $\mathscr X(\Gamma_F)$. 

We also need the notion of a $p$-refinement. For simplicity, we assume for the remainder of the introduction that $\pi$ is an unramified principal series at each $v \mid p$. In the body of the text we will also allow $\pi$ to be an unramified special representation. Let $\chi_\pi$ be the nebentype character of $\pi$. If $v \mid p$, then write $a_\pi(v)$ for the $v$-th eigenvalue in the Hecke eigensystem associated to $\pi$ and $\varpi_v$ for a uniformizing parameter.

\begin{defn}\label{defn:intro-refinement}
A $p$-refinement for $\pi$ is a tuple $(\alpha_v)_{v\mid p}$ where $\alpha_v$ is a root of the $v$-th Hecke polynomial $X^2 - a_\pi(v) X  + \chi_\pi(\varpi_v)q_v^{w+1}$.
\end{defn}
If $\alpha$ is a $p$-refinement, we write $(\beta_v)_{v\mid p}$ for the list of `other' roots determined by the factorizations
\begin{equation*}
X^2 - a_\pi(v) X  + \chi_\pi(\varpi_v)q_v^{w+1} = (X-\alpha_v)(X-\beta_v).
\end{equation*}
We often refer to the pair $(\pi,\alpha)$ as a $p$-refined automorphic representation (or some minor variant thereof). When $F = \mathbf Q$ and $\pi$ corresponds to a holomorphic eigenform $f(z)$ of level $N$ that is prime to $p$, a $p$-refinement $\alpha$ is often instantiated through the eigenform
\begin{equation*}
f_\alpha(z) = f(z) - \beta f(pz)
\end{equation*}
which now has level $Np$. See Section \ref{subsec:refinements} for more details.

In Section \ref{subsec:intro-control}, we will define what it means for a $p$-refined $(\pi,\alpha)$ to be {\em non-critical} and, more generally, {\em decent}. We call $\alpha$ critical if it is not non-critical.\footnote{There are two completely unrelated uses of the word `critical' in this article, an unfortunate collision. We will stress the context by always referring to an integer as being critical with respect to a weight and a refinement being a (non-)critical refinement.}  We note immediately that non-critical is implied by a `small slope' condition on $\alpha$, but it is certainly not equivalent, and that non-critical implies decent. The condition of being decent is very mild in our estimation. Conjecturally, being critical but decent reduces to the condition that $\alpha_v$ and $\beta_v$ as above are distinct for all $v\mid p$, which is expected to always hold when $p$ is totally split in $F$. In Section \ref{subsec:intro-decent-hyps} we discuss the hypothesis of decency in detail. 

Absent the definition of decent we can state our main theorem. We re-iterate that we have assumed $\pi$ cohomological cuspidal and, for simplicity only, that $\pi$ is an unramified principal series at each $v \mid p$.

\begin{thm}[Section \ref{subsec:padicLfunctions}]\label{thm:intro-main-theorem}
Let $(\pi,\alpha)$ be a decently $p$-refined cohomological cuspidal automorphic representation of weight $\lambda$. Let $E = \mathbf Q(\pi,\alpha)$ be the subfield of $\mathbf C$ generated by $\mathbf Q(\pi)$ and the refinement $\alpha$, and let $L \subset \overline{\mathbf Q}_p$ be the subfield generated by $\iota(E)$. 

Then, for each $\epsilon \in \{\pm 1\}^{\Sigma_F}$ there exists an element $L_{p}^{\epsilon}(\pi,\alpha) \in \mathscr O(\mathscr X(\Gamma_F)) \otimes_{\mathbf Q_p} L$ satisfying the following properties.\\

\noindent \textbf{\emph{a. Canonicity:}} The construction of $L_p^{\epsilon}(\pi,\alpha)$ is canonically specified up to $L^\times$-multiple in general and up to $\iota(E^\times)$-multiple if $\alpha$ is non-critical.\\

\noindent \textbf{\emph{b. Support:}} $L_p^{\epsilon}(\pi,\alpha)(\chi) = 0$ unless $\sgn(\chi_\sigma) = \epsilon_\sigma$ for all $\sigma \in \Sigma_F$.\\

\noindent \textbf{\emph{c. Growth:}} $L_p^{\epsilon}(\pi,\alpha)$ has growth bounded by $\sum_{v \mid p} e_vv_p(\iota(\alpha_v)) + \sum_{\sigma \in \Sigma_F} {\kappa_\sigma - w\over 2}$.\footnote{Growth is defined in Section \ref{subsec:growth-properties}.}\\

\noindent \textbf{\emph{d. Interpolation:}} Let $m$ be an integer that is critical with respect to $\lambda$, and assume that $\theta$ is a finite order Hecke character of $p$-power conductor with $\epsilon_{\sigma} = \sgn(\theta_\sigma)(-1)^m$ for each $\sigma \in \Sigma_F$. Then,
\begin{equation*}
L_p^{\epsilon}(\pi,\alpha)(\theta^{\iota}\chi_{\cycl}^m) = e_p(\alpha,m)\cdot \iota\left(\Lambda^{\alg}(\pi\otimes \theta,m+1)\right)
\end{equation*}
where the interpolation factor $e_p(\alpha,m)=\prod_{v\mid p} e_v(\alpha,m)$ is defined as follows:
\begin{enumerate}[(i)]
\item If $\alpha$ is non-critical, then 
\begin{equation*}
\iota^{-1}(e_v(\alpha,m)) = \begin{cases}
 \biggl(1-\theta(\varpi_v)\alpha_v^{-1}q_v^m\biggr)\biggl(1-\theta(\varpi_v)\beta_vq_v^{-(m+1)}\biggr) & \text{(if $\theta_v$ is unramified);}\\
\left(\displaystyle {q_v^{m+1} \over \alpha_v}\right)^{f_v} & \text{(if $\theta_v$ is ramified of conductor $\varpi_v^{f_v}$).}
\end{cases}
\end{equation*}
\item If $\alpha$ is critical then $e_v(\alpha,m) = 0$ for all $v \mid p$.\\
\end{enumerate}

\noindent \textbf{\emph{e. Variation:}} Suppose the eigenvariety $\mathscr E(\mathfrak n)_{\rmmid}$ is smooth at the classical point $x_{\pi,\alpha}$ associated with $(\pi,\alpha)$.\footnote{This is almost always satisfied for decent $(\pi,\alpha)$. See Theorem \ref{thm:intro-smoothness} and the discussion following that result.} Then for any sufficiently small open neighborhood $U$ of $x$ in $\mathscr E(\mathfrak n)_{\rmmid}$ there exists an element $\mathbf L^{\epsilon}_p \in \mathscr O(U)\widehat{\otimes}_{\mathbf Q_p} \mathscr O(\mathscr X(\Gamma_F))$ canonically specified up to $\mathscr O(U)^\times$-multiple and such that for each decent point $x' \in U$ associated with a $p$-refined cohomological cuspidal automorphic representation $(\pi',\alpha')$ we have
\begin{equation*}
\mathbf L^{\epsilon}_p|_{x'} = c_{x'}L^{\epsilon}_p(\pi',\alpha')
\end{equation*}
for some non-zero constant $c_{x'}$ in the residue field $k_{x'}$ of $\mathscr E(\mathfrak n)_{\rmmid}$ at $x'$.\\

\noindent \textbf{\emph{f. Uniqueness:}} If the Leopoldt defect of $F$ at $p$ is zero, then (up to $L^\times$ ambiguity) the assignment $(\pi,\alpha) \rightsquigarrow L_p^{\epsilon}(\pi,\alpha)$ is uniquely determined by conditions b.\ through e.

\end{thm}

This article is not the first place a result like Theorem \ref{thm:intro-main-theorem} has been proven, and we owe a great deal to previous work. On the other hand, our results apply in many situations that were previously inaccessible. For instance, recall that if $F$ is a (real) quadratic extension of $\mathbf Q$, then any elliptic curve over $F$ is modular by \cite{FreitasLeHungSiksek-Modularity}.

\begin{cor} Let $F$ be a real quadratic field, and let $E_{/F}$ be a non-CM elliptic curve with associated automorphic represenatation $\pi_E$. Then, for all sufficiently large rational primes $p$ that are split in $F$, all four possible $p$-refinements $(\pi_E,\alpha)$ are decent, and consequently each determines a canonical p-adic L-function $L_p^{\epsilon}(\pi_E,\alpha)$ as in Theorem \ref{thm:intro-main-theorem}.
\end{cor}

By contrast, at 100\% of the prime split in $F$ all previous constructions were only able to unconditionally access \emph{one} of these four $p$-adic $L$-functions, namely the one associated with the ``ordinary-ordinary'' refinement.

We will further compare our results with the literature in Section \ref{subsec:intro-comparison}. In order to put these comparisons in the proper context, however, we first expand on the definition of decency and the method of our construction. We hope this delay is not taken as a slight.

\subsection{The story when $F = \mathbf Q$}\label{subsec:intro-F=Q}
Our strategy is modeled on the case $F = \mathbf Q$ which is more or less understood.  To motivate our constructions, we outline the necessary ingredients in that case.

\subsubsection{Archimedean considerations}
Let $f = \sum a_n(f) q^n$ be a normalized cuspidal Hecke newform of weight $k\geq 2$ and level $\Gamma_1(N)$ with $N$ prime to $p$. The construction of Eichler and Shimura associates with $f$ a canonical cohomology class $\omega_f \in H^1_c(Y_1(N), \mathscr L_{k-2})$ where $\mathscr L_{k-2}$ is a local system on the modular curve $Y_1(N)$ defined by a `weight $k-2$' action on the space of complex polynomials of degree at most $k-2$ in a single variable.  It turns out that when $m=0,1,\dotsc,k-2$ (i.e., when $m$ is critical with respect to $k$), the special value $\Lambda(f,m+1)$ can be realized as $\Lambda(f,m+1) = \ev_m(\omega_f)$ where
\begin{equation*}
\ev_m : H^1_c(Y_1(N),\mathscr L_{k-2}) \rightarrow \mathbf C
\end{equation*}
is a certain canonical linear functional.  The functional $\ev_m$ is actually defined over $\mathbf Q$, and after renormalizing the Eichler--Shimura construction by a period, everything is defined over a number field. Putting these observations together, one obtains Shimura's result. (One also considers variants of these constructions taking finite-order twists into account, cf. below.) To summarize, this argument for Shimura's result makes use of two essentially distinct ingredients:
\begin{enumerate}
\item Canonical cohomology classes $\omega_f$ associated with each $f$.
\item Natural functionals $\ev_m$ on cohomology that record $L$-values.
\end{enumerate}

\subsubsection{$p$-adic considerations}
In the authors' view, the construction of $p$-adic $L$-functions should closely mirror the steps (1) and (2) above. The emphasis on a dichotomy like this is largely due to Stevens in the case $F = \mathbf Q$. Let us explain the two steps in reverse.

The local systems $\mathscr L_{k-2}$ are algebraic, so they can be taken to have $p$-adic coefficients, and they exist on modular curves of any level. On modular curves of level $Np$ (with $p \nmid N$) there is a second local system $\mathscr D_{k-2}$ of locally analytic distributions on $\mathbf Z_p$ equipped with a `weight $k-2$' action of a certain monoid containing $\Gamma:= \Gamma_1(N)\cap \Gamma_0(p)$. If $\Phi \in H^1_c(Y(\Gamma),\mathscr D_{k-2})$ is any cohomology class, then it makes sense to evaluate $\Phi$ on the cycle `$\{\infty\}-\{0\}$' on $Y(\Gamma)$. The output of this evaluation is thus a distribution on $\mathbf Z_p$ that can be restricted to $\mathbf Z_p^\times$. So, each $\Phi$ defines natural elements in the space $\mathscr D(\mathbf Z_p^\times)$ of locally analytic distributions on $\mathbf Z_p^\times$. Now note that $\Gamma_{\mathbf Q} \simeq \mathbf Z_p^\times$, and so a theorem of Amice from the 1970's (\cite{Amice-Interpolation}) implies that $\mathscr D(\mathbf Z_p^\times)$ is canonically isomorphic to $\mathscr O(\mathscr X(\Gamma_{\mathbf Q}))$, which is exactly where our $p$-adic $L$-functions are meant to live. This suggests the following  (2') as an analog of (2) above:
\begin{enumerate}
\item[(2')] Consider the linear map 
\begin{equation*}
\mathscr P_{k-2}: H^1_c(Y(\Gamma),\mathscr D_{k-2}) \rightarrow \mathscr O(\mathscr X(\Gamma_{\mathbf Q}))
\end{equation*}
that associates to each $\Phi \in H^1_c(Y(\Gamma),\mathscr D_{k-2})$ the element $\Phi(\{\infty\}-\{0\})|_{\mathbf Z_p^\times}$.
\end{enumerate}
To further illuminate the connection with the maps $\ev_m$, note that there is a canonical map $I_{k-2}: \mathscr D_{k-2} \rightarrow \mathscr L_{k-2}$ of local systems  over $Y(\Gamma)$ given by recording the first $k-2$ moments of a distribution. It is then not difficult to establish a direct relationship between the map $\mathscr P_{k-2}$, the map induced by $I_{k-2}$ on cohomology, the evaluation maps $\ev_m$ defined above, and the Hecke operators at $p$. (More glibly:\ the cycle `$\{\infty\}-\{0\}$' is `clearly' related to $L$-values by the integral representation of $L$-series as a Mellin transform on the upper-half plane.)

One important point to stress is that the local system $\mathscr D_{k-2}$ can only be defined over modular curves with $\Gamma_0(p)$-structure. Thus to an eigenform $f$ of level $N$ with $p \nmid N$, we are naturally led to consider the $p$-refined eigenform $f_{\alpha}$ of level $\Gamma$, corresponding to some choice of refinement $\alpha$. An ambitious choice for the $p$-adic analog to the archimedean step (1) would then be:
\begin{enumerate}
\item[(1')] `Canonically' associate with each $p$-refined eigenform $f_{\alpha}$ a class $\Phi_{f_\alpha} \in H^1_c(Y(\Gamma),\mathscr D_{k-2})$.
\end{enumerate}
If (1') can be carried out, then one may combine (1') and (2') to produce a $p$-adic $L$-function as in Theorem \ref{thm:intro-main-theorem}.

To what extent is (1') possible? For any $\alpha$, the class $\omega_{f_\alpha} \in H^1_c(Y(\Gamma),\mathscr L_{k-2})$ is in the image of the map $I_{k-2}$, but the kernel of $I_{k-2}$ is infinite-dimensional. One might then try to produce a Hecke eigenclass $\Phi_{f_\alpha}$ that maps to $\omega_{f_\alpha}$ under $I_{k-2}$, and one might hope that it is unique; this would certainly pin down a 	`canonical' $\Phi_{f_\alpha}$. However, this is only possible some of the time. Specifically, $\omega_{f_\alpha}$ can be uniquely lifted to a Hecke eigenclass exactly when the refinement $\alpha$ is non-critical in our sense. In the case $F = \mathbf Q$ this combines the two cases commonly referred to as being `non-critical slope' or `critical slope but not $\theta$-critical'. These cases were handled by Pollack and Stevens (\cite{PollackStevens-OverconvergentModSymb,PollackStevens-CriticalSlope}).

When $\alpha$ is critical, but still decent, Bella\"iche (\cite{Bellaiche-CriticalpadicLfunctions}) observed that it is {\em never} possible to lift $\omega_{f_\alpha}$ to a Hecke eigenclass via $I_{k-2}$. He did this by showing, in an indirect way, that there is still a unique (up to scalar) Hecke eigenclass $\Phi_{f_\alpha} \in H^1_c(Y(\Gamma),\mathscr D_{k-2})$ with the same Hecke eigensystem as $f_\alpha$; it just happens to lie in the kernel of $I_{k-2}$. This is precisely why one sees `funny' behavior in the interpolation properties of $p$-adic $L$-functions for critical $\alpha$ (Theorem \ref{thm:intro-main-theorem}).  We will explain Bella\"iche's method in more detail below; the argument uses $p$-adic families in a crucial way.  

In any case, we can safely say that when $F = \mathbf Q$ the ingredient (1') is available (under the decency hypothesis). The aim of the present paper is to generalize both steps (1') and (2') to any totally real base field $F$, while maintaining a view towards unrestrictive hypotheses.

\subsection{Basic objects}

Having stated our result and outlined the known methods when $F = \mathbf Q$, we now unload the requisite terminology and notations for the general case.

Write $\mathbf A_F$ for the adeles of $F$, $\mathbf A_{F,f}$ for the finite adeles. The $p$-th component of $\mathbf A_F$ is $F_p = F\otimes_{\mathbf Q} \mathbf Q_p \simeq \prod_{v \mid p} F_v$, and we also write $\mathcal O_F \otimes_{\mathbf Z} \mathbf Z_p = \mathcal O_p \subset F_p$ for the corresponding product of rings of integers. The tuple of uniformizers $\varpi_v$ at $v \mid p$ thus defines an element $\varpi_p \in \mathcal O_p$. Suppose that $\mathfrak n \subset \mathcal O_F$ is an integral ideal that is prime to $p$. We will assume from now on that $\pi$ has conductor exactly $\mathfrak n$. We will write $K=\prod_{v} K_v$ for the compact open subgroup of $\GL_2(\widehat{\mathcal{O}_F})$ consisting of matrices $\begin{smallpmatrix} a & b \\ c & d \end{smallpmatrix}$ whose entries satisfy $c \equiv 0 \bmod \varpi_p \mathfrak n \widehat{\mathcal{O}_F}$ and $d \equiv 1 \bmod \mathfrak n \widehat{\mathcal{O}_F}$. We write $Y_K$ for the open Hilbert modular variety of level $K$ (it is the analog of the modular curve $Y(\Gamma)$ above).

For a fixed cohomological weight $\lambda=(\kappa,w)$, we will consider a finite-dimensional local system $\mathscr L_\lambda$ on $Y_K$ of $L$-vector spaces, where $L \subset \overline{\mathbf{Q}}_p$ is the field generated over $\mathbf{Q}_p$ by all embeddings $\iota(\sigma(F))$.  More precisely, $\mathscr L_\lambda$ is defined as the finite-dimensional vector space $\mathscr L_\lambda \subset L[\{X_\sigma\}_{\sigma \in \Sigma_F}]$ spanned by polynomials whose $X_\sigma$-degree is at most $\kappa_\sigma$, and the group $\GL_2(F_p)$ acts by a natural weight $\lambda$ left  action (see Section \ref{subsec:weights} for the precise definition of the action). The cohomology $H^{\ast}_c(Y_K,\mathscr L_\lambda)$ is naturally acted upon by the Hecke algebra $\mathbf T$ generated by the `standard' Hecke operators $T_v$ ($v \nmid \mathfrak n p$), $U_v$ ($v \mid p$), and $S_v$ ($v \nmid \mathfrak n$), cf. Definition \ref{defn:hecke-operators}. If $(\pi,\alpha)$ is a $p$-refined automorphic representation, then it has (via $\iota$) an associated $\overline{\mathbf Q}_p$-valued $\mathbf T$-eigensystem (in particular, the eigenvalue of $U_v$ is $\iota(\alpha_v)$). This defines a maximal ideal $\mathfrak m_{\pi,\alpha} \subset \mathbf T$ and the Eichler--Shimura construction implies that $H^{\ast}_c(Y_K,\mathscr L_\lambda)_{\mathfrak m_{\pi,\alpha}}$ is non-zero and concentrated in middle degree. More precisely, the cohomology $H^{\ast}_c(Y_K,\mathscr L_\lambda)$ decomposes into $2^d$-many direct summands $H^{\ast}_c(Y_K,\mathscr L_\lambda)^{\epsilon}$ indexed by signs $\epsilon \in \{\pm 1\}^{\Sigma_F}$, which correspond to choosing eigenvalues for each of the $d$ `archimedean Hecke operators' induced by the partial complex conjugations on $Y_K$ (cf. Section \ref{subsec:archimedean} for a precise discussion). For each $\epsilon$ the eigenspace $$ \left( H^{\ast}_c(Y_K,\mathscr L_\lambda)\otimes_{L} \overline{\mathbf{Q}}_p \right)^{\epsilon}[\mathfrak m_{\pi,\alpha}]$$ is one-dimensional and concentrated in middle degree. 

To introduce $p$-adic automorphic forms we first consider $p$-adic weights. For us, this is a pair $\lambda=(\lambda_1,\lambda_2)$ of continuous characters $\lambda_i : \mathcal O_p^\times \rightarrow \mathbf C_p^\times$. If $\lambda=(\kappa,w)$ is cohomological then it defines a $p$-adic weight $(\lambda_1,\lambda_2)$ by the recipe
\begin{equation*}
\lambda_1(x) = \prod_{\sigma \in \Sigma_F} (\iota\circ\sigma)(x)^{w + \kappa_\sigma \over 2}, \;\;\;\;\;\; \lambda_2(x) =  \prod_{\sigma \in \Sigma_F} (\iota\circ\sigma)(x)^{w - \kappa_\sigma \over 2}.
\end{equation*}Note that if $\lambda$ is a cohomological weight, then the values of the characters $\lambda_i$ generate a field $k_{\lambda}$ that is a subfield of $L$.

For each $p$-adic weight we then define a $k_{\lambda}$-Frechet space $\mathscr D_{\lambda}$ whose underlying module is the locally analytic distributions $\mathscr D(\mathcal O_p)$ on $\mathcal O_p$. The subscripted $\lambda$ indicates that we equip it with a specific left action of the monoid 
\begin{equation*}
\Delta = \{\begin{smallpmatrix} a & b \\ c & d \end{smallpmatrix} \in M_2(\mathcal O_p) \cap \GL_2(F_p) \mid c \in \varpi_p\mathcal O_p\ \text{ and } d \in \mathcal O_p^\times\}.
\end{equation*}
We omit the definition of the action here (see Section \ref{subsec:monoid-action}). Now, since $\Delta \supset K_p$, we can also consider the cohomology $H^{\ast}_c(Y_K,\mathscr D_\lambda)$ for each $p$-adic weight $\lambda$, and the Hecke algebra $\mathbf T$ still acts on this cohomology by endomorphisms. Moreover, in the special case that $\lambda$ is a cohomological weight, there is a natural map
\begin{equation*}
I_\lambda: H^{\ast}_c(Y_K,\mathscr D_\lambda \otimes_{k_\lambda} L) \rightarrow H^{\ast}_c(Y_K,\mathscr L_\lambda)
\end{equation*}
induced by a $\Delta$-equivariant map on the underlying local systems. In particular, $I_\lambda$ commutes with the $\mathbf T$-action, and it commutes with the archimedean Hecke operators.\footnote{Strictly speaking, the map $I_{\lambda}$ only commutes with the $U_v$-operators for $v|p$ up to a scaling; we elide this point in the introduction.\label{footnote:scaling-intro}}

All of these objects are designed as analogs of the objects we considered when discussing the case $F = \mathbf Q$ earlier. Let us now turn towards our ingredients for $p$-adic $L$-functions.

\subsection{The period maps}
The portion of this article that requires no hypotheses is the construction of a certain $\mathscr O(\mathscr X(\Gamma_F))$-valued functional $\mathscr P_\lambda$ on the middle-degree distribution-valued cohomology $H^{d}_c(Y_K,\mathscr D_\lambda)$. We call $\mathscr P_\lambda$ a period map because of its interaction with the Hecke integrals that compute the completed $L$-series of automorphic representations in the case where $\lambda$ is a cohomological weight. We remark ahead of time that is absolutely crucial to the generality of Theorem \ref{thm:intro-main-theorem} that the definition of $\mathscr{P}_\lambda$ works for more general $p$-adic weights, as well as for affinoid families of weights.

To state a precise result here, we need a little more notation. Let $\lambda$ be a cohomological weight. Then we can consider the local system $\mathscr L_\lambda^\vee$ on $Y_K$ dual to $\mathscr L_\lambda$, and then we can take its middle degree Borel--Moore homology $H_d^{\BM}(Y_K,\mathscr L_\lambda^\vee)$ (homology defined by locally finite chains). There is a natural pairing
\begin{equation*}
\langle - , - \rangle : H^d_c(Y_K,\mathscr L_\lambda) \otimes_{L} H_d^{\BM}(Y_K,\mathscr L_\lambda^\vee) \rightarrow L \subset \overline{\mathbf Q}_p.
\end{equation*}
In Section \ref{subsec:padic-eval-class} we will define, for each integer $m$ critical with respect to $\lambda$, a certain evaluation class $\cl_p(m) \in H_d^{\BM}(Y_K,\mathscr L_\lambda^\vee)$. Its purpose is that if $\psi_{\pi,\alpha} \in H^d_c(Y_K,\mathscr L_\lambda)$ is the Hecke eigenclass associated with a $p$-refined cohomological cuspidal automorphic representation $(\pi,\alpha)$ of weight $\lambda$ (via Eichler--Shimura), then $\langle \psi_{\pi,\alpha} , \cl_p(m) \rangle$ is a natural scaling (depending on $\alpha$) of  the special value $\Lambda(\pi,m+1)$. In fact, $\psi \mapsto \langle \psi, \cl_p(m) \rangle$ is a $p$-adic analog of the evaluation maps $\ev_m$.

\begin{thm}\label{thm:intro-period-map}
For each $p$-adic weight $\lambda$, there exists a canonical linear morphism
\begin{equation*}
\mathscr P_\lambda: H^{d}_c(Y_K,\mathscr D_\lambda) \rightarrow \mathscr O(\mathscr X(\Gamma_F)) \otimes k_\lambda
\end{equation*}
that, among other things, satisfies the following formal interpolation property:\ 

If $\lambda$ is a cohomological weight, $m$ is an integer that is critical with respect to $\lambda$, and $\Psi \in H^d_c(Y_K,\mathscr D_{\lambda})$ is a $U_v$-eigenvector with eigenvalue $\alpha_v^{\sharp}$, for each $v \mid p$, then
\begin{equation*}
\mathscr P_\lambda(\Psi)(\chi_{\cycl}^m) = \prod_{v \mid p} (1 - (\alpha_v^{\sharp}\varpi_v^{w-\kappa\over 2})^{-1}q_v^m) \cdot \langle I_\lambda(\Psi), \cl_p(m)\rangle.
\end{equation*}
\end{thm}
One should compare the formal interpolation in Theorem \ref{thm:intro-period-map} with the interpolation property in Theorem \ref{thm:intro-main-theorem}. (The scalar factor $\varpi_v^{w-\kappa\over 2}$, whose meaning can be found in Section \ref{subsec:notation}, appears because of the implicit scaling mentioned in Footnote \ref{footnote:scaling-intro}.) The formal interpolation of course generalizes to also allow twists by finite order Hecke characters of $p$-power conductor; see Theorem \ref{thm:abstract-interpolation} and Corollary \ref{cor:abstract-interp-eigen} for these more complicated statements. In addition, the period maps enjoy certain growth properties (Section \ref{subsec:growth-properties}) and natural interaction with the signs $\epsilon$ (Section \ref{subsec:compatibilities}). Finally, they also vary naturally in the $p$-adic weight variable $\lambda$ (in fact, we define period maps functorially for any affinoid weight). The map described in Theorem \ref{thm:intro-period-map} is thus a natural analog of `evaluating at $\{\infty\}-\{0\}$' in the setting of $F = \mathbf Q$. (It is also a short exercise to check that our definition truly generalizes that construction.)

In fact, the definition of $\mathscr P_\lambda$ is quite brief once the groundwork is laid. It involves first constructing a natural $k_\lambda$-linear map $\mathscr P_\lambda: H^d_c(Y_K,\mathscr D_\lambda) \rightarrow \Hom_{k_\lambda}(\mathscr A(\Gamma_F) \otimes k_\lambda, k_\lambda)$ where $\mathscr A(\Gamma_F)$ is the ring of locally analytic functions on $\Gamma_F$. We then manage to check that the image of $\mathscr P_\lambda$ actually lands in the subspace of locally analytic distributions $\mathscr D(\Gamma_F)$, which is the continuous (as opposed to abstract) $k_\lambda$-linear dual of $\mathscr A(\Gamma_F) \otimes k_\lambda$. Once this is proven (Theorem \ref{theorem:period-image}), it is easy to obtain the map described in Theorem \ref{thm:intro-period-map} using the theorem of Amice we previously mentioned. The proof of the continuity condition in the definition of $\mathscr P_\lambda$ amounts to constructing it canonically enough that it naturally preserves various integral structures on both sides. We refer to Section \ref{subsec:period-definition} for further details.

\subsection{Control of Hecke eigenclasses}\label{subsec:intro-control}
With Theorem \ref{thm:intro-period-map} in hand, we also need a means of canonically associating distribution-valued Hecke eigenclasses with $p$-refined automorphic representations $(\pi,\alpha)$. Recall that there is a natural integration map $I_\lambda: H^{\ast}_c(Y_K,\mathscr D_\lambda \otimes_{k_\lambda} L) \rightarrow H^{\ast}_c(Y_K,\mathscr L_\lambda)$, and that to a pair $(\pi,\alpha)$ we have a maximal ideal $\mathfrak m_{\pi,\alpha} \subset \mathbf T$.

\begin{defn}[Non-critical]\label{defn:intro-noncritical}
A $p$-refined automorphic representation  $(\pi,\alpha)$ is called non-critical if $I_\lambda: H^{\ast}_c(Y_K,\mathscr D_\lambda \otimes_{k_\lambda} L)_{\mathfrak m_{\pi,\alpha}} \rightarrow H^{\ast}_c(Y_K,\mathscr L_\lambda)_{\mathfrak m_{\pi,\alpha}}$ is an isomorphism.
\end{defn}
A well-known argument shows that non-critical slope implies non-critical, but the two conditions are not equivalent (see Section \ref{subsec:special-points}). In the case $F=\mathbf Q$, non-critical is equivalent to what is sometimes known as being `not $\theta$-critical' as in \cite{PollackStevens-CriticalSlope}. Reasoning with classical facts about automorphic representations, it is easy to prove that if $(\pi,\alpha)$ is non-critical, then the Hecke eigenspace $\left( H^d_c(Y_K,\mathscr D_\lambda) \otimes_{k_{\lambda}} \overline{\mathbf{Q}}_p \right)^{\epsilon}[\mathfrak m_{\pi,\alpha}]$ is one-dimensional (for any $\epsilon$) and so Theorem \ref{thm:intro-period-map} can be used to associate $p$-adic $L$-functions $L_p^{\epsilon}(\pi,\alpha)$ with non-critically refined forms $(\pi,\alpha)$. More precisely the Eichler--Shimura construction gives us, after scaling by a period, a canonical class in $H^d_c(Y_K,\mathscr L_\lambda)^{\epsilon}[\mathfrak m_{\pi,\alpha}]$. We lift this class via the isomorphism $I_{\lambda}$ (in the non-critical case) and thus define the $p$-adic $L$-function $L_p^{\epsilon}(\pi,\alpha)$ as the output of $\mathscr P_\lambda$ applied to this lift.

In general, and already when $F = \mathbf Q$, there definitely exist critically refined $(\pi,\alpha)$. To handle these cases, our methods demand some input from the theory of Galois representations. Given any $\pi$, write $\rho_\pi$ for the natural two-dimensional irreducible  representation of the absolute Galois group $G_F = \Gal(\overline F/ F)$ associated with $\pi$. Recall also that if $\alpha = (\alpha_v)_{v\mid p}$ is a refinement then there is an evident tuple of `other roots' $\beta = (\beta_v)_{v \mid p}$ (Definition \ref{defn:intro-refinement}).

\begin{defn}\label{defn:intro-decent}
A $p$-refined automorphic representation  $(\pi,\alpha)$ is called decent if at least one of the following two conditions is true.
\begin{enumerate}
\item $(\pi,\alpha)$ is non-critical.
\item The following three conditions hold.
\begin{enumerate}
\item $H^{j}_c(Y_K,\mathscr D_\lambda)_{\mathfrak m_{\pi,\alpha}}$ is non-zero if and only if $j=d$ (the middle degree).
\item The adjoint Bloch--Kato Selmer group $H^1_f(G_F,\ad \rho_\pi)$ is trivial.
\item $\alpha_v \neq \beta_v$ for each $v\mid p$.
\end{enumerate}
\end{enumerate}
\end{defn}
Before discussing the three conditions in part (2) of this definition, we state our main result on the Hecke eigenspaces in distribution-valued cohomology associated with a decently $p$-refined $(\pi,\alpha)$.
\begin{thm}\label{thm:intro-one-dimensional}
If $(\pi,\alpha)$ is a decently $p$-refined automorphic representation of weight $\lambda$, then 
\begin{equation*}
\dim_{\overline{\mathbf Q}_p} H^d_c(Y_K,\mathscr D_\lambda \otimes_{k_\lambda} \overline{\mathbf Q}_p)^{\epsilon}[\mathfrak m_{\pi,\alpha}] = 1
\end{equation*}
for each $\epsilon \in \{\pm 1\}^{\Sigma_F}$.
\end{thm}
We already mentioned why Theorem \ref{thm:intro-one-dimensional} is true when $(\pi,\alpha)$ is non-critical, but the fact that it extends to all decently refined $(\pi,\alpha)$ is rather more difficult. In any case, if we apply the period map of Theorem \ref{thm:intro-period-map} to the unique-up-to-scalar Hecke eigenclass provided by Theorem \ref{thm:intro-one-dimensional}, we get the $p$-adic $L$-functions $L_p^{\epsilon}(\pi,\alpha)$ claimed in Theorem \ref{thm:intro-main-theorem}. We make no further claim on how to canonically choose a non-zero vector in the above one-dimensional vector space, so we are ambiguous up to scalars in a $p$-adic field rather than a number field. 

The proof of Theorem \ref{thm:intro-one-dimensional} relies in a crucial way on $p$-adic families of $p$-refined automorphic representations and their finer geometric properties. Before discussing this further, let us explain what is known about the decency hypothesis.

\subsection{The decency hypothesis}\label{subsec:intro-decent-hyps}
It is worth detailing what is known about part (2) of the `decent' hypothesis.\footnote{The terminology is borrowed directly from Bella\"iche (\cite{Bellaiche-CriticalpadicLfunctions}).} In order to orient the discussion from least technical to most technical, let us discuss the conditions in reverse from (c) to (a).

The simplest condition is the condition that $\alpha_v \neq \beta_v$ for each $v \mid p$. Unfortunately, this is also the only condition we do not conjecture always holds. For instance, if $E_{/\mathbf Q}$ is an elliptic curve with good supersingular reduction at $p$, $F$ is a real quadratic field in which $p$ is inert, and $\pi$ is the parallel weight two automorphic representation associated with the base change $E_{/F}$, then the Hecke polynomial of $\pi$ at the unique $p$-adic place is $(X-p)^2$. We do not know if all such examples are non-critical, but we have no strong feeling either way. We do note, however, that when $p$ is totally split in $F$, it would follow from the Tate conjecture that $\alpha_v \neq \beta_v$ for each $v \mid p$ (cf. \cite{ColemanEdixhoven}). Moreover, if $\pi$ is associated with a modular elliptic curve and $v|p$ has degree one, then $\alpha_v \neq \beta_v$; in particular, if $p$ splits completely, it is easy to see that condition 2(c) holds. In any case, for a fixed $\pi$ and $p$, the condition that $\alpha_v \neq \beta_v$ is surely easy to check (depending on the data you are given to represent the fixed $\pi$, of course).

The next condition we consider is the vanishing of the Selmer group in part (b). This is a well-established consequence of a conjecture of Bloch and Kato (\cite{BlochKato-TamagawaNumbersOfMotives}) extending the Birch--Swinnerton-Dyer conjecture. In fact, condition 2(b) is now known to hold for all non-CM $\pi$ by very recent work of Newton-Thorne \cite{NewtonThorneSelmer}, building on earlier work of Kisin (for $F = \mathbf Q$, see \cite{Kisin-GeometricDeformations}) and Allen (for general totally real $F$, see \cite{Allen-SelmerGroups}). Note as well that hypothesis (b) does not involve the refinement $\alpha$ in any way.

Finally we come to the thorniest of the three hypotheses:\ the assumption that the distribution-valued eigensystem associated to $(\pi,\alpha)$ occurs only in the middle degree. This is a classically known fact for the finite-dimensional classical cohomology $H^\ast_c(Y_K,\mathscr L_\lambda)$. So, in particular the non-critical hypothesis overlaps with the middle-degree support hypothesis. Further, when $F = \mathbf Q$ the condition 2(a) is also true by a direct analysis:\ the relevant $H^2_c$'s only contain Eisenstein Hecke eigensystems. With current technology, we can verify condition 2(a) unconditionally under a mild assumption on the mod $p$ representation $\overline \rho_\pi$, building on recent work of Caraiani-Tamiozzo. We refer the reader to Appendix \ref{app:decent} for a precise statement and proof.  Based on these evidences, we conjecture that condition 2(a) \emph{always} holds (remember that $\pi$ is cuspidal). 

Synthesizing these observations, we deduce that decent refinements are ubiquitous. Here is one precise result in this direction; the proof is given in the final lines of Appendix \ref{app:decent}.

\begin{thm}\label{thm:decency-ubiquitous} Let $E/F$ be a non-CM modular elliptic curve over a totally real field $F$ of degree $d$, with associated automorphic representation $\pi$. Let $p$ be any sufficiently large prime that splits completely in $F$. Then all of the $2^d$ distinct $p$-refinements of $\pi$ are decent.
\end{thm}

\subsection{The eigenvariety (proving Theorem \ref{thm:intro-one-dimensional})}\label{subsec:intro-eigenvariety}
The method we use to prove Theorem \ref{thm:intro-one-dimensional} in the decent, but possibly critical cases, is closely modeled on the method used by Bella\"iche in \cite{Bellaiche-CriticalpadicLfunctions}. However, there are a number of new complications that arise in our more general setting. We would like to discuss this in some detail since we expect it will also help explain the role of the hypothesis 2(a) for the reader whose experience with $p$-adic families is limited to the eigencurve and to other simple situations like groups that are compact modulo their center at infinity.

The first point is the Hecke eigenvarieties parameterizing eigensystems corresponding to (finite slope) automorphic representations for ${\GL_2}_{/F}$ come in different flavors. For instance, there is the parallel weight eigencurve of Kisin and Lai (\cite{KisinLai-OHMF}) and one modeled on overconvergent $p$-adic Hilbert cusp forms by Andreatta, Iovita and Pilloni (\cite{AndreattaIovitaPilloni-HMF}). But history (and Theorem \ref{thm:intro-period-map}) teaches us that the models for eigenvarieties that are closest to seeing $p$-adic $L$-functions are those built using distribution-valued cohomology. Beyond the case of $F = \mathbf Q$, these appear in the work of Urban (\cite{Urban-Eigenvarieties}) and the more general construction of the second author (\cite{Hansen-Overconvergent}). (They are exposed for $F = \mathbf Q$ in \cite{Bellaiche-CriticalpadicLfunctions} following ideas of Stevens).

More precisely, in \cite{Hansen-Overconvergent} the second author constructed a rigid analytic space $\mathscr E(\mathfrak n)$ parametrizing the finite slope $\mathbf T$-eigensystems appearing in the total cohomology $H^{\ast}_c(Y_K,\mathscr D_{\lambda})$ as $\lambda$ runs over the space of $p$-adic weights $\mathscr W(1) \subset \mathscr W$ that are trivial on the image of the global units (these are the only weights where the cohomology is non-trivial; see Section \ref{subsec:wt-space}). For notation, if $\psi$ is a finite slope $\mathbf T$-eigensystem appearing in the total cohomology, then write $x_\psi \in \mathscr E(\mathfrak n)$ for the corresponding point. For instance, if $(\pi,\alpha)$ is a $p$-refined automorphic representation as above then its eigensystem appears in the cohomology, in some degree, and thus we get classical points $x_{\pi,\alpha}$ on $\mathscr E(\mathfrak n)$.

The first difficulties are that $\mathscr E(\mathfrak n)$ is certainly not equidimensional if $F \neq \mathbf Q$, and it is possibly not reduced. Both the equidimensionality and reducedness of the Coleman-Mazur eigencurve are crucial in the proof of Theorem \ref{thm:intro-one-dimensional} given by  Bella\"iche in \cite{Bellaiche-CriticalpadicLfunctions} for $F = \mathbf Q$. One of the theorems we prove is the following.

\begin{thm}[Section \ref{subsec:middle-degree}]\label{thm:intro-eigenvariety-existence}
There exists a Zariski-open subspace $\mathscr E(\mathfrak n)_{\rmmid}$ inside $\mathscr E(\mathfrak n)$ uniquely characterized by the following property:\ a point $x_\psi$, of weight $\lambda$, is in $\mathscr E(\mathfrak n)_{\rmmid}$ if and only if the eigensystem $\psi$ appears only in the middle degree $H^d_c(Y_K,\mathscr D_\lambda)$.

Moreover, $\mathscr E(\mathfrak n)_{\rmmid}$ is reduced, equidimensional of the same dimension as its weight space $\mathscr W(1)$, and the classical points (up to twist) are Zariski-dense and accumulating.
\end{thm}

The space $\mathscr E(\mathfrak n)_{\rmmid}$ is defined as the complement of a finite union of closed subspaces in $\mathscr E(\mathfrak n)$, each of which has dimension strictly smaller than the dimension of weight space. The characterization of $\mathscr E(\mathfrak n)_{\rmmid}$ in Theorem \ref{thm:intro-eigenvariety-existence} follows from two spectral sequences developed by the second author in \cite{Hansen-Overconvergent}. The density of classical points and the reduced-ness follow standard lines of argument. Finally, the equidimensionality uses a theorem of Newton proved in an appendix to \cite{Hansen-Overconvergent}.

Now the role of the hypothesis 2(a) comes into view:\ assuming that $(\pi,\alpha)$ is decent tells us that the corresponding classical point $x_{\pi,\alpha}$ on $\mathscr E(\mathfrak n)$ in facts lies on the much better behaved sub-eigenvariety $\mathscr E(\mathfrak n)_{\rmmid}$. We then prove the following statement:

\begin{thm}\label{thm:intro-smoothness}
If $(\pi,\alpha)$ satisfies condition (2) in Definition \ref{defn:intro-decent}, then $x_{\pi,\alpha}$ is a smooth point on $\mathscr E(\mathfrak n)_{\rmmid}$.
\end{thm}

The proof is an argument using deformations of Galois representations; this is where conditions 2(b) and 2(c) come in. The local deformation-theoretic calculations that are needed were carried out by the first author in \cite{Bergdall-Smoothness} (see also \cite{BreuilHellmannSchraen-Classicality}). We should emphasize that the properties in Theorem \ref{thm:intro-eigenvariety-existence}, thus condition 2(a), are absolutely crucial to getting the strategy off the ground:\ they are used not just to guarantee the variation of Galois representations over $\mathscr E(\mathfrak n)_{\rmmid}$ but also that the key generalizations of Kisin's theorem on crystalline periods (\cite{Kisin-OverconvergentModularForms,Liu-Triangulations}) hold as well.

Theorem \ref{thm:intro-smoothness} (Theorem \ref{thm:smoothness} in the text) is also true when $(\pi,\alpha)$ is non-critical, if it is further assumed that condition 2(c) in Definition \ref{defn:intro-decent} holds. The argument (due to Chenevier) is somewhat different and proves the stronger statement that the weight map is \'etale. While we expect that \'etaleness of the weight map definitely fails whenever 2(c) fails, it is open whether or not Theorem \ref{thm:intro-smoothness} as stated holds without 2(c).

Finally we deduce the one-dimensionality result in Theorem \ref{thm:intro-one-dimensional} as a consequence of Theorem \ref{thm:intro-smoothness} (again, it was already known in the non-critical case). The strategy is to prove that the image $T_{\pi,\alpha}$ of the Hecke algebra $\mathbf{T}$ in the endomorphism ring of $M_{\pi,\alpha}=H^{d}_c(Y_K,\mathscr D_\lambda)_{\mathfrak m_{\pi,\alpha}}$ is Gorenstein (of dimension zero), and that each sign eigenspace $M_{\pi,\alpha}^{\epsilon}$ is free of rank one over $T_{\pi,\alpha}$.  The idea to prove this is that Theorem \ref{thm:intro-smoothness} implies the statement for $T_{\pi,\alpha}$ replaced by the weight fiber $\mathscr{O}_{\mathscr{E}(\mathfrak n), x_{\pi,\alpha}} \otimes_{\mathscr{O}_{\mathscr{W}(1),\lambda}} k_{\lambda}$. In general, the base change map $\mathscr{O}_{\mathscr{E}(\mathfrak n), x_{\pi,\alpha}} \otimes_{\mathscr{O}_{\mathscr{W}(1),\lambda}} k_{\lambda} \to T_{\pi,\alpha}$ is surjective with nilpotent kernel, but we use classical theorems in commutative algebra (the Auslander--Buchsbaum formula and properties of depth) to show it is an isomorphism at $x_{\pi,\alpha}$. We refer to the text (Section \ref{subsec:consequences}) for more details.

\subsection{Comparison to other results}\label{subsec:intro-comparison}

As we have already indicated, when $F=\mathbf{Q}$ the results we prove can be found in Bella\"iche's article{\color{magenta}.}  The first paragraph of that article provides more than ample references to the relevant history.  

We note, however, that there is something a bit special about $F = \mathbf Q$. Precisely, the truth of Leopoldt's conjecture implies that the group $\Gamma_F$ is a $1$-dimensional $p$-adic Lie group, so a theorem of Amice and V\'elu (\cite{AmiceVelu-pAdicLfunction}) implies in turn that the $p$-adic $L$-functions described in Theorem \ref{thm:intro-main-theorem} are uniquely determined by their growth and interpolation properties when the growth is sufficiently small. This has the notable advantage that constructions by different methods (for instance, modular symbols vs. Rankin--Selberg methods) necessarily give the same $p$-adic $L$-functions in non-critical slope cases, and so only $p$-adic $L$-functions beyond non-critical slope have any ambiguity. In the critical slope case, there are constructions by Pollack--Stevens (\cite{PollackStevens-CriticalSlope}) and Bella\"iche (\cite{Bellaiche-CriticalpadicLfunctions}). These obviously agree on their overlap. There is also a construction, which applies in the critical slope case, using Kato's Euler systems the dual exponential map of Perrin-Riou (cf.\ the introduction to \cite{LeiLoeffersZerbes-CMforms}). This construction agrees with the previous references in the non-theta-critical case (see\ \cite{Wang-EulerFamily} for instance).

Now let us move to a general totally real field $F$. We would first like to mention the articles of Ash--Ginzburg (\cite{AshGinzburg-pAdicLfns}), Januszewski (\cite{Januszewski-Symbols,Januszewski-pAdicIMRN,Januszewski-FamiliesGLn}), Manin (\cite{Manin-pAdic}),  and Haran (\cite{Haran-pAdicLfunctions}), which all give constructions of $p$-adic $L$-functions associated with Hilbert modular forms in varying degrees of generality. However, the main goals of these articles are somewhat orthogonal to ours. On the one hand they are more general in some ways. For instance, they actually do not assume the base field is totally real and \cite{AshGinzburg-pAdicLfns} and \cite{Januszewski-Symbols,Januszewski-FamiliesGLn} construct $p$-adic $L$-functions for $\mathrm{GL}_{2n}$ and $\mathrm{GL}_{n+1} \times \mathrm{GL}_n$, respectively.  On the other hand, of these only the very recent \cite{Januszewski-FamiliesGLn} considers variation in families (ordinary, in this case), and none of them go beyond small slope cases. And without input from Leopoldt's conjecture, we can not say for certain that their methods produce the same objects as ours in the overlapping cases.

More closely related to the present article are the recent works of Dimitrov (\cite{Dmitrov-OrdinarypAdic}), Barrera (\cite{Barrera-Thesis}), Barrera and Williams (\cite{BarreraWilliams-pAdicLfunctions}), and a very recent article of Dimitrov, Barrera, and Jorza (\cite{BarreraDimitrovJorza-pAdicGL2}). Dimitrov's article, in particular, gives a clean and definitive construction of $p$-adic $L$-functions associated with ordinary $p$-refined Hilbert modular forms and with Hida families thereof.  In \cite{Barrera-Thesis}, Barrera combined the formalism of overconvergent cohomology with the modular cycles introduced in \cite{Dmitrov-OrdinarypAdic}, obtaining a construction of $p$-adic $L$-functions in the non-critical case with the correct growth and interpolation properties. This method was generalized in \cite{BarreraWilliams-pAdicLfunctions} to allow for any number field. (The statements in \cite{Barrera-Thesis, BarreraWilliams-pAdicLfunctions} assume non-critical slope, but it is clear from reading these works that non-criticality is a sufficient hypothesis.) In the course of all these works, and in \cite{BarreraDimitrovJorza-pAdicGL2} in particular, one finds a map from eigenclasses in overconvergent cohomology to distributions on a Galois group that bears a resemblance to the period map we have defined and which presumably can be verified to be the \emph{same} map. In particular, even without Leopoldt one might hope that our constructions and those of \cite{Barrera-Thesis,BarreraDimitrovJorza-pAdicGL2,BarreraWilliams-pAdicLfunctions} coincide in the overlapping cases.

The difference between our period map and that of the above works is best illustrated by examining the proofs of the interpolation property. For instance, in \cite{BarreraWilliams-pAdicLfunctions}, the authors check the interpolation property by making use of  modular cycles and ``hands-on'' calculations with group cohomology.  These modular cycles do not appear explicitly in our calculations (although they are implicit in some way in what we do). Rather than introduce auxiliary cycles, we instead calculate directly using the adelic chains and cochains introduced by Ash and Stevens (see Section \ref{sec:cls}).  At first glance, this may seem more complicated. However, we believe our approach is quite natural, for at least two reasons.

First, modular cycles were originally introduced in the context of Hida theory, and in particular in a framework where $p$-adic families can be constructed by considering cohomology with constant coefficients of a $Y_1(\mathfrak{n} p^\infty)$-tower. In this context, it is natural (and in some sense, necessary) to introduce fairly complicated cycles when defining $p$-adic $L$-functions and checking their interpolation property. In Stevens's setup, by contrast, there is no tower, but the cohomology has extremely complicated coefficients.  Our perspective then is that the difficulty should be shifted from defining the correct modular cycles to defining the correct period map. Second, the details of our construction are consistent with the adelic philosophy we have adopted. For instance, our definition eliminates the need to choose representatives for various objects, thereby avoiding the ambiguities such choices can engender. This is in contrast to several points in the arguments of the referenced works where one has to check somewhat non-trivial independence of choices.  Our approach avoids this kind of issue.

\subsection{Organization}
The body of this article is divided into seven main sections. The first three (Sections \ref{sec:cls}, \ref{sec:hmf}, and \ref{sec:shimuras-theorem}) are comprised of a verbose discussion of adelic (co)chains on locally symmetric spaces, Hilbert modular forms, and Shimura's algebraicity theorem. Here we have adopted a maximalist approach to the exposition, so that our notations are as precise as possible and to ensure this work is reasonably self-contained.

Starting in Section \ref{sec:loc-anal-dist} we turn towards $p$-adic matters. First we discuss generalities on certain $p$-adic Lie groups and define various modules of locally analytic functions and distributions. 

Section \ref{sec:mid-eigen} is devoted to an exposition of the middle-degree Hilbert modular eigenvariety mentioned above. We include here (and in the previous section) a lengthy discussion, most of which is moot if we were to assume Leopoldt's conjecture, of twisting classical points by $p$-adic Hecke characters.

In Section \ref{sec:period-maps} we define and analyze the period maps. The heart of this section is the proof of the abstract interpolation theorem, which is the key ingredient in proving the correct interpolation formula for our $p$-adic $L$-functions.

The final section, Section \ref{sec:padicL-functions}, contains the definition of $p$-adic $L$-functions and the proofs of Theorem \ref{thm:intro-one-dimensional} and Theorem \ref{thm:intro-main-theorem}.

\subsection{Notations}\label{subsec:notation}
For convenience, we list here notations that will remain in force throughout the paper.

We will always write $\GL_2$ for the general linear group over $\mathbf Z$ (and ${\GL_2}_{/R}$ for its base change to a ring $R$ if needed). We write $Z \subset T \subset \GL_2$ for the center, resp. the diagonal torus. If $H$ is a real Lie group we generally write $H^{\circ}$ for the connected component of $H$ containing the identity.

Unless noted, $F$ is a totally real number field of degree $d$. Its ring of integers is written $\mathcal O_F$. We write $\Sigma_F = \Hom(F,\mathbf C)$. The adeles of $F$ are written $\mathbf A_F$. We write $F_{\infty} = F\otimes_{\mathbf Q} \mathbf R$ for the infinite component of $\mathbf A_F$. We write $\mathbf A_{F,f}$ for the finite component of $\mathbf A_F$.

The map $F \rightarrow \mathbf R^{\Sigma_F}$ given by $\xi \mapsto (\sigma(\xi))$ for $\xi \in F$ extends $\mathbf R$-linearly to an isomorphism $F_{\infty} \simeq \mathbf R^{\Sigma_F}$ of $\mathbf R$-algebras. If $x \in F_\infty$ we write $x =(x_\sigma)$ for its coordinates in $\mathbf R^{\Sigma_F}$. We say $x \in F_\infty$ is totally positive if $x_\sigma > 0$ for all $\sigma \in \Sigma_F$; the set of totally positive elements is written $F_{\infty,+}$. Or, the invertible totally positive elements of $F_\infty$ is equal to $(F_\infty^\times)^{\circ}$ (our preferred notation in many places).

We fix a prime number $p$. We write $\overline{\mathbf Q}_p$ for an algebraic closure of the $p$-adic numbers. We also fix an isomorphism $\iota: \mathbf C \overset{\sim}{\rightarrow}\overline{\mathbf Q}_p$. Using $\iota$ we have a decomposition
\begin{equation}\label{eqn:union}
\Sigma_F = \bigsqcup_{v \mid p} \Sigma_v
\end{equation}
where an element $\sigma \in \Sigma_F$ lies in $\Sigma_v$ if and only if the composition $\iota\circ \sigma$ induces the $p$-adic place $v$ on $F$. Write $F_p = F\otimes_{\mathbf Q} \mathbf Q_p \simeq \prod_{v \mid p} F_v$ where $F_v$ is the completion of $F$ with respect to $v \mid p$. If $\sigma \in \Sigma_v$ then $\sigma$ extends to a $\mathbf Q_p$-linear embedding $\sigma: F_v \hookrightarrow \overline{\mathbf Q}_p$ for which we use the same symbol. In this way $\Sigma_v$ is identified with $\Hom_{\mathbf Q_p}(F_v,\overline{\mathbf Q}_p)$.

If $K/\mathbf Q_\ell$ is a finite extension, $\ell$ a prime, we write $\Art_K : K^\times \rightarrow G_K^{\ab}$ for the local Artin map, normalized so uniformizers map to geometric Frobenius elements. If $\pi$ is a smooth, irreducible representation of $\GL_n(K)$ we denote $\rec_K(\pi)$ the Weil--Deligne representation corresponding to $\pi$ by the local Langlands correspondence as constructed by Harris and Taylor (\cite{HarrisTaylor-LocalLanglands}). We further specify $r(\pi) = \rec_K(\pi\otimes |{\det}|^{1-n\over 2})$ for the arithmetically normalized local Langlands correspondence. Finally, we write $r^{\iota}(\pi)$ for the corresponding representation over $\overline{\mathbf Q}_p$ obtained via $\iota$.

We also use two shorthand notations for tuple-based operations. First, suppose that $S$ is a set and we are given collections $\{X_s\}_{s\in S}$, $\{Y_s\}_{s \in S}$,  and $\{Z_s\}_{s \in S}$ with a binary operations $X_s \times Y_s \overset{\bullet_s}{\longrightarrow} Z_s$. If $X = \prod_{s \in S} X_s$, $Y =\prod_{s \in S} Y_s$ and $Z = \prod_{s \in S} Z_s$ we then define a binary operation $X\times Y \overset{\bullet}{\longrightarrow} Z$ by $(x_s) \bullet (y_s) := (x_s \bullet_s y_s)$. A typical situation where we might use this is when, for each $s \in S$, $X_s$ is a group acting on a set $Y_s$ (so $Y_s = Z_s$). The second situation we will find ourselves in is we are given a collection $x=(x_s)_{s \in S}$ of elements of a common ring $R$, and we are given a collection $n=(n_s)_{s \in S}$ of integers. In that case we define $x^n = \prod_{s \in S} x_s^{n_s}$. This notation satisfies the obvious compatibilities with usual multiplication in a ring.

If $v$ is a place of $F$ then we write $\mathfrak p_v$ for the corresponding prime ideal. If $p$ is a prime then we will use the bold letter $\mathbf p := \prod_{v \mid p} \mathfrak p_v$ for the product of the primes above $p$.

\subsection{Acknowledgments}  

This project began in May 2012 when D.H. attended William Stein's plenary lecture on elliptic curves over $\mathbf{Q}(\sqrt{5})$ at the Atkin Memorial Conference, and he would like to heartily thank Stein for this crucial initial inspiration.  Some of these results (in particular the non-critical case of Theorem \ref{thm:intro-main-theorem}) were first announced by D.H. in a conference at UCLA in May 2013.  Decisive progress beyond the non-critical case occurred in early 2016, and the authors gave some lectures on these results beginning in Spring 2016.  In any case, the authors would like to apologize for the very long delay between the first announcement(s) of these results and the appearance of this manuscript. 

The first author's research was partially supported by NSF grant DMS-1402005. J.B.\ would also like to thank the Institut des Hautes \'Etudes Scientifiques (Bures-sur-Yvette), and the Max-Planck-Institut f\"ur Mathematik (Bonn) for hospitality during visits in the spring of 2017. The majority of this work was carried out while J.B.\ was a postdoctoral researcher at Boston University, and we thank them for their stimulating atmosphere and for providing material support for D.H. to make multiple visits during this collaboration. D.H. would like to thank Boston College, l'Institut de Math\'ematiques de Jussieu, and Columbia University for providing congenial working conditions during the various stages of this project. 

We would finally like to thank Avner Ash, Michael Harris, Keenan Kidwell, Barry Mazur, and Glenn Stevens for useful and inspiring conversations at various times throughout this project.

%% file: topology.tex
This section concerns the cohomology of local systems on symmetric spaces which arise in the context of Hilbert modular forms. Almost nothing is original in our treatment. However, a number of calculations later in the paper rely on the precise formulas we present and so we found it prudent to expose them in detail. The reader is encouraged to skim the results as needed.

\subsection{Topology}\label{subsec:topology}
Throughout this subsection, we write $X$ and $Y$ for topological spaces which are locally compact and Hausdorff. We let $R$ be a fixed principal ideal domain. Unless otherwise noted, ``sheaves'' are sheaves of $R$-modules.

If $\mathscr L$ is a sheaf on $X$ we consider the cohomology $H^{\ast}(X,\mathscr L)$, homology $H_{\ast}(X,\mathscr L)$, compactly supported cohomology $H^{\ast}_c(X,\mathscr L)$ or Borel--Moore homology $H_{\ast}^{\BM}(X,\mathscr L)$. These are all $R$-modules. Primary sources for $H^{\ast}_c$ and $H_{\ast}^{\BM}$ are \cite{Steenrod-LocalCoefficients,BorelMoore-Homology}. We refer to \cite{Bredon-SheafTheory}  for what follows.\footnote{We warn the reader that our homology notation is in conflict with \cite{Bredon-SheafTheory}. Namely, $H_{\ast}^{\BM}$ here is written $H_{\ast}$ there and $H_{\ast}$ here is written $H_{\ast}^{c}$ there (cf.\ the caution at the start of \cite[Section V.3]{Bredon-SheafTheory}).} Along with the usual functorialities in algebraic topology (pushforward in homology, pullback in cohomology, and so forth) we summarize important properties of compactly supported cohomology and Borel--Moore homology.

If $\mathscr L$ and $\mathscr M$ are two sheaves on $X$, there is a functorial cup product (\cite[Sections II.7]{Bredon-SheafTheory})\begin{equation}\label{eqn:cup-products}
\cup: H^p_{?}(X,\mathscr L)\otimes_R H^q(X,\mathscr M) \rightarrow H^{p+q}_{?}(X,\mathscr L\otimes_R \mathscr M)
\end{equation}
for ? either $c$ or the empty symbol.  Further, there are two separate cap products (\cite[Section V.10]{Bredon-SheafTheory})
\begin{align}
H^{p}_c(X,\mathscr L) \otimes_R H_q^{\BM}(X,\mathscr M) &\overset{\cap}{\longrightarrow} H_{q-p}(X, \mathscr L\otimes_R \mathscr M);\label{eqn:cap-products}\\
H^{p}(X,\mathscr L) \otimes_R H_q^{\BM}(X,\mathscr M) &\overset{\cap}{\longrightarrow} H_{q-p}^{\BM}(X, \mathscr L\otimes_R \mathscr M).\nonumber
\end{align}
The cup and cap products commute in the sense that 
\begin{equation}\label{eqn:cap-cup-commute}
(\Psi \cup \Psi') \cap \Phi = \Psi \cap (\Psi' \cap \Phi),
\end{equation}
under apparent qualifications on where these elements are defined.

If $\mathscr L$ is a sheaf on $Y$ and $f: X \rightarrow Y$ is a  proper morphism, then there are functorial
 pushforward and pullback maps 
 \begin{align}
 H_\ast^{\BM}(X,f^{\ast}\mathscr L) &\overset{f_{\ast}}{\longrightarrow} H_{\ast}^{\BM}(Y,\mathscr L);\label{eqn:pushfoward/pullback}\\
H^{\ast}_c(Y,\mathscr L) &\overset{ f^{\ast}}{\longrightarrow} H^{\ast}_c(X,f^{\ast}\mathscr L).\nonumber
 \end{align}
The cup product commutes with pullbacks. The cap products are compatible with pushfowards and pullbacks along proper morphisms $f: X \rightarrow Y$ in that
\begin{equation}\label{eqn:push-pull-caps-formula}
f_{\ast}(f^{\ast}\Psi \cap \Phi) = \Psi \cap f_{\ast}\Phi
\end{equation}
for all $\Psi \in H^p_{?}(Y,\mathscr L)$ and $\Phi \in H_q^{\BM}(X, f^{\ast}\mathscr M)$.

Now suppose that $p=q$ in \eqref{eqn:cap-products} and that we have a pairing $\mathscr L\otimes_R \mathscr M \rightarrow R$. Taking the natural composition
\begin{equation*}
H_{0}^{?}(X,\mathscr L\otimes_R \mathscr M)  \rightarrow H_0^{?}(X,R)  \overset{\tr}{\longrightarrow}R
\end{equation*}
and combining it with the cap product, $\langle \Psi , \Phi \rangle := \tr(\Psi \cap \Phi)$  defines a functorial $R$-bilinear pairing
\begin{equation*}
\langle - , - \rangle : H^p_{c}(X,\mathscr L) \otimes_R H_p^{\BM}(X,\mathscr M) \rightarrow R
\end{equation*}
under which $f^{\ast}$ and $f_{\ast}$ are adjoint (by \eqref{eqn:push-pull-caps-formula} and because trace commutes with pushforwards). Thus, our convention is that cap products $\Phi \cap \Psi$ are homology classes and values of pairings $\langle \Phi, \Psi \rangle$  are elements of $R$.

Suppose now that $X$ is an oriented real manifold of dimension $n$. Then there is a Borel--Moore fundamental class $[X] \in H^{\BM}_n(X,R)$ with the property that $\PD(\Psi) := \Psi \cap [X]$ defines a functorial isomorphism
\begin{equation}
\PD: H^q(X,\mathscr L) {\rightarrow} H^{\BM}_{n-q}(X,\mathscr L)\label{eqn:poincare-duality}
\end{equation}
for each $0 \leq q \leq n$. See \cite[Theorem V.10.1 and Corollary V.10.2]{Bredon-SheafTheory}. We refer to $\PD$ as ``Poincar\'e duality.'' It satisfies the following properties. First, if $f: X \rightarrow X$ is an orientation preserving homeomorphism, then $f_{\ast}[X] = [X]$ and so 
\eqref{eqn:push-pull-caps-formula} implies that
\begin{equation}\label{eqn:naturality-PD}
f_{\ast} \PD f^{\ast} = \PD.
\end{equation}
Second, if $f: X \rightarrow Y$ is a proper morphism, $\mathscr L$ is a sheaf on $X$, $\mathscr M$ is a sheaf on $Y$ and we have a pairing $\mathscr L \otimes_R \mathscr M \rightarrow R$, then  from \eqref{eqn:cap-cup-commute}, \eqref{eqn:push-pull-caps-formula}, and \eqref{eqn:poincare-duality} we obtain
\begin{equation}\label{eqn:push-pull-duality-formula}
\langle \Phi, f_{\ast}\PD(\Psi) \rangle = \langle f^{\ast}\Phi \cup \Psi, [X] \rangle
\end{equation}
for all $\Phi \in H^p_c(Y,\mathscr L)$ and $\Psi \in H^{n-p}(X,f^{\ast}\mathscr M)$ (the cup product $f^{\ast}\Phi \cup \Psi$ is implicitly viewed in $H^n_c(X,R)$ for the purposes of this formula).  Finally, when $R$ is a subring of $\mathbf C$ there is an integration map $\int_X : H^n_c(X,R) \rightarrow R$ which is natural with respect to Poincar\'e duality in that $\int_X = \tr \circ \PD$ on $H^n_c(X,R)$.

\subsection{Adelic cochains on symmetric spaces}\label{subsec:adelic-cochains}
In this subsection, we review the adelic (co)chains introduced by Ash and Stevens (see \cite[Section 2]{Hansen-Overconvergent} and the references there).

Write $G$ for a connected reductive group over $\mathbf Q$, $\mathbf A$ for the adeles of $\mathbf Q$, and $\mathbf A_{f}$ for the finite adeles (rather than $\mathbf A_{\mathbf Q}$ and $\mathbf A_{\mathbf Q,f}$ as Section \ref{subsec:notation} would suggest). Let $G(\mathbf R)^{\circ}$ be the connected component of the identity in $G(\mathbf R)$ and let $K_\infty^{\circ} \subset G(\mathbf R)^{\circ}$ be a subgroup which is either maximal compact or maximal compact mod-center. Let $Z$ be the center of $G$.

Write $D_\infty = G(\mathbf R)^{\circ}/K_\infty^{\circ}$ and $D_{\mathbf A} = D_\infty \times G(\mathbf A_f)$, which we view as topological spaces where $D_\infty$ gets its structure as a real manifold and $G(\mathbf A_f)$ gets the discrete topology.  Then, we write $C_{\bullet}(D_{\mathbf A})$ for the chain complex of singular chains in $D_{\mathbf A}$. The discrete topology is totally disconnected, so any singular chain in $G(\mathbf A_f)$ is a single point, meaning $C_{\bullet}(D_{\mathbf A}) = C_{\bullet}(D_\infty)\otimes_{\mathbf Z} \mathbf Z[G(\mathbf A_f)]$ with $\partial \otimes 1$ as the boundary map (and we could have also given $G(\mathbf A_f)$ its natural topology).

Fix a compact open subgroup $K \subset G(\mathbf A_f)$. Then $G(\mathbf Q)^{\circ}$ acts diagonally on $D_{\mathbf A}=D_\infty \times G(\mathbf A_f)$ from the left, and $K$ acts on the right via the second coordinate. We write $Y_K$ for the double quotient
\begin{equation}\label{eqn:symmetric-space}
Y_K := G(\mathbf Q)\backslash G(\mathbf A) /K_\infty^{\circ} K = G(\mathbf Q)^{\circ} \backslash D_{\mathbf A} / K.
\end{equation}
In general, although this may not be a manifold, it is a disjoint union of orbifolds. Specifically, if $\{g_i\}$ is a finite collection of elements $g_i \in G(\mathbf A_f)$ such that $G(\mathbf A) = \bigsqcup_i  G(\mathbf Q)^{\circ}G(\mathbf R)g_iK$, then
\begin{equation}\label{eqn:union-decomposition}
Y_K = \bigsqcup_i \Gamma(g_i) \leftmod D_\infty,
\end{equation}
where $\Gamma({g}) := gKg^{-1} \cap G(\mathbf Q)^{\circ} \subset G(\mathbf Q)^{\circ}$ for $g \in G(\mathbf A_f)$. When the $\Gamma(g_i)/(Z \cap \Gamma(g_i))$ are without torsion, $Y_{K}$ is a real manifold of dimension $2d$ (compare with Proposition \ref{prop:neatness} below).

Now suppose that $N$ is a $(G(\mathbf Q)^{\circ},K)$-bimodule, meaning:
\begin{enumerate}
\item $N$ is a right $K$-module whose action we write $n|k$ for $n \in N$ and $k \in K$, and
\item $N$ is a left $G(\mathbf Q)^{\circ}$-module whose actions we write $\gamma \cdot n$ for $n \in N$ or $\gamma \in G(\mathbf Q)^{\circ}$.
\end{enumerate}
For instance, the left action of $G(\mathbf Q)^{\circ}$, and the right action of $K$, on $D_{\mathbf A}$ equips $C_{\bullet}(D_{\mathbf A})$ with a natural structure of complex of $(G(\mathbf Q)^{\circ},D_{\mathbf A})$-bimodules. We consider any $N$ with the discrete topology and write $\underline N$ (in the text we will remove underlines for readability) for the local system defined by the sheaf of locally constant sections of the natural projection map
\begin{equation*}
G(\mathbf Q)^{\circ} \backslash (D_{\mathbf A} \times N)/K \rightarrow Y_K.
\end{equation*}
We also use the standard abuse of notation to write $\underline N$ for the double quotient itself. 

The adelic cochain complex associated with $N$ is
\begin{equation*}
C_{\ad}^{\bullet}(K,N) := \Hom_{(G(\mathbf Q)^{\circ}, K)}(C_{\bullet}(D_{\mathbf A}), N).
\end{equation*}
The subscript ``$\ad$'' in this context refers to the word {\em ad}elic.

Let $g_f \in G(\mathbf A_f)$. Then, for each singular chain $\sigma_\infty \in C_{\bullet}(D_\infty)$ there is a singular chain $\sigma_\infty \otimes [g_f] \in C_{\bullet}(D_{\mathbf A})$. This allows us to define a morphism of abelian groups
\begin{align}
\Hom(C_{\bullet}(D_{\mathbf A}), N) &\rightarrow \Hom(C_{\bullet}(D_{\infty}), N);\label{eqn:augment}\\
\phi &\mapsto \left[\phi_{g_f}: \sigma_\infty \mapsto \phi(\sigma_\infty\otimes [g_f])\right].\nonumber
\end{align}
We note that the chain complex $C_{\bullet}(D_\infty)$ is naturally a chain complex of left $\Gamma(g_f)$-modules, where $\Gamma(g_f)$ acts on $D_\infty$ through the inclusion $\Gamma(g_f) \subset G(\mathbf Q)^{\circ}$. On the other hand, we write $N(g_f)$ for the left $\Gamma(g_f)$-module whose underlying abelian group is still $N$ but equipped with a left $\Gamma(g_f)$-action
\begin{equation*}
\gamma \cdot_{g_f} n = \gamma\cdot n|{(g_f^{-1}\gamma^{-1}g_f)}.
\end{equation*}
These definitions given, it is straightforward to see that the map \eqref{eqn:augment} descends to a morphism
\begin{equation*}
C_{\ad}^{\bullet}(K,N) \rightarrow \Hom_{\Gamma(g_f)}(C_{\bullet}(D_\infty), N(g_f)).
\end{equation*}
Finally, let $C^{\bullet}(D_\infty;N) = \Hom(C_{\bullet}(D_\infty), N)$ and write $C^{\bullet}_c(D_\infty;N) \subset C^{\bullet}(D_\infty;N)$ for the cochains on $D_\infty$ with compact support. We define the compactly supported adelic cochains by
\begin{equation*}
C_{\ad,c}^{\bullet}(K,N) := \{ \phi \in C_{\ad}^{\bullet}(K,N) \mid \phi_{g_f} \in C^{\bullet}_c(D_\infty;N) \text{ for all $g_f \in G(\mathbf A_f)$}\}.
\end{equation*}
\begin{prop}\label{prop:canonical}
There are canonical isomorphisms
\begin{equation*}
\xymatrix{
H^{\ast}(C^{\bullet}_{\ad,c}(K,N)) \ar[r]^-{\simeq} \ar[d] & H^{\ast}_c(Y_K,\underline N) \ar[d] \\
H^{\ast}(C^{\bullet}_{\ad}(K,N)) \ar[r]^-{\simeq} & H^{\ast}(Y_K,\underline N)
}
\end{equation*}
\end{prop}
\begin{proof}
This follows from the same argument as in \cite[Proposition 2.1.1]{Hansen-Overconvergent}.
\end{proof}
``Canonical'' in Proposition \ref{prop:canonical} refers to at least the following functorialities:
\begin{enumerate}[(i)]
\item \label{list:adelic-Nfunctoriality} If $f: N \rightarrow N'$ is a $(G(\mathbf Q)^{\circ},K)$-equivariant morphism, then the natural map $H^{\ast}_{?}(Y_K,\underline N) \overset{f}{\rightarrow} H^{\ast}_{?}(Y_K,\underline{N}')$ is induced by the morphism of cochain complexes $f\circ-: C_{\ad,?}^{\bullet}(K,N) \rightarrow C_{\ad,?}^{\bullet}(K,N')$.
\item \label{list:adelic-subgroupfunctoriality} If $K' \subset K$ is a subgroup then the inclusion $C_{\ad,?}^{\bullet}(K,N) \subset C_{\ad,?}^{\bullet}(K',N)$ induces the pullback $\pr^{\ast}:H^{\ast}_{?}(Y_K,\underline N) \rightarrow H^{\ast}_{?}(Y_K,\underline N')$ on cohomology.
\item \label{list:adelic-finindexfunctoriality}
Suppose that $K' \subset K$ is a subgroup of finite index. Then, $\pr: Y_{K'}\rightarrow Y_K$ is proper, so it induces a pushfoward map $\pr_{\ast}: H^{\ast}_{?}(Y_{K'},\underline N) \rightarrow H^{\ast}_{?}(Y_K,\underline N)$. On the other hand, if $K = \coprod x_i K'$ then $\tr(\phi)(\sigma) = \sum \phi(\sigma x_i)|x_i^{-1}$ induces a natural map of cochain complexes $\tr : C^{\bullet}_{\ad,?}(K',N) \rightarrow C^{\bullet}_{\ad,?}(K,N)$, whose induced map on cohomology is $\pr_{\ast}$.
\item \label{list:adelic-multiplicationfunctoriality} 
Finally, let $g \in G(\mathbf A_f)$. Write $N(g^{-1})$ for the $(G(\mathbf Q)^{\circ},g^{-1}Kg)$-module whose right $g^{-1}Kg$-action is given by $n|_{g^{-1}}x = n|gxg^{-1}$. Then, the map $r_g: Y_{K} \rightarrow Y_{g^{-1}Kg}$ given by $x \mapsto xg$ induces a map on cohomology $r_g^{\ast} : H^{\ast}_?(Y_{g^{-1}Kg}, \underline{N(g^{-1})}) \rightarrow H^{\ast}_?(Y_K,\underline N)$. On the other hand, if we set $r_g(\phi)(\sigma) = \phi(\sigma g)$ then $r_g: C^{\bullet}_{\ad,?}(g^{-1}Kg,N(g^{-1})) \rightarrow C^{\bullet}_{\ad,?}(K,N)$ is a map of cochain complexes which induces $r_g^{\ast}$ on cohomology.
\end{enumerate}

We next recall how to canonically lift Hecke operators to endomorphisms of adelic cochains. Let $K \subset G(\mathbf A_f)$ be an compact open subgroup, and let $\Delta \subset G(\mathbf A_f)$ be a monoid containing $K$ such that $K$ and $\delta^{-1}K\delta$ are commensurable for all $\delta \in \Delta$. We assume that $N$ is equipped with a {\em left} $\Delta$-module structure $\delta \cdot n$ which commutes with a given left $G(\mathbf Q)^{\circ}$-module structure. We give $N$ the structure of a right $K$-module by $n|k = k^{-1}\cdot n$ under which we now have a $(G(\mathbf Q)^{\circ},K)$-bimodule again.  We equip $\Hom_{G(\mathbf Q)^{\circ}}(C_{\bullet}(D_{\mathbf A}), N)$ with the left action of $\Delta$ given by $(\delta \cdot\phi)(\sigma) = \delta\cdot \phi(\sigma\delta)$ under which we have $C^{\bullet}_{\ad}(K,N) = \Hom_{G(\mathbf Q)^{\circ}}(C_{\bullet}(D_{\mathbf A}), N)^{K}$ (and an obvious analog for $C_{\ad,c}^{\bullet}(K,N)$). If $\delta \in \Delta$ is a given element and $K\delta K = \coprod  \delta_i K$ is a decomposition into right cosets, then for any $\phi \in C^{\bullet}_{\ad,?}(K,N)$ the sum
\begin{equation}\label{eqn:hecke-action-adelic-cochains}
[K\delta K](\phi) = \sum_i \delta_i\cdot \phi
\end{equation}
is independent of the choice of $\delta_i$'s and defines another element of $C^{\bullet}_{\ad,?}(K,N)$. The resulting endomorphism of $C^{\bullet}_{\ad,?}(K,N)$ canonically lifts the usual Hecke operator on cohomology, i.e. the operator defined by the composition
\begin{multline}\label{eqn:adelic-cochains}
H^{\ast}_{?}(Y_K,\underline N) \overset{\pr^{\ast}}{\longrightarrow} H^{\ast}_{?}(Y_{K\cap \delta^{-1}K\delta},\underline N) \rightarrow H^{\ast}_{?}(Y_{K\cap \delta^{-1}K\delta},\underline{N(\delta^{-1})})\\ \overset{r_{\delta}^{\ast}}{\longrightarrow} H^{\ast}_{?}(Y_{K\cap \delta K \delta^{-1}}, \underline N)
\overset{\pr_{\ast}}{\longrightarrow} H^{\ast}_{?}(Y_K,\underline N).
\end{multline}
Here, for $\delta \in \Delta$ the morphism $n \mapsto \delta \cdot n$ defines a morphism $N \rightarrow N(\delta^{-1})$ which is equivariant for the action of $K \cap \delta^{-1}K\delta$ on either side, giving the unlabeled arrow.

We end our discussion with an algebraic situation. Suppose that $F$ is any number field, and write $\mathscr N$ for an $F$-algebraic representation of $G$, i.e.\ an $F$-vector space $\mathscr N$ and a representation $G \rightarrow \Res_{F/\mathbf Q} \GL(\mathscr N)$. Recall that we fixed an isomorphism $\iota: \mathbf C \simeq \overline{\mathbf Q}_p$. Suppose that $E \subset \mathbf C$ is a field and $L := \mathbf Q_p(\iota(E))$. Then, we deduce linear representations $G(L) \rightarrow \GL_L(N_p)$, and $G(E) \rightarrow \GL_E(N_\infty)$ where $N_p := \mathscr N\otimes_{\mathbf Q} L$ and $N_\infty := \mathscr N\otimes_{\mathbf Q} E$. By construction, $\iota$ induces a morphism of $\mathbf Q$-vector spaces $\iota: N_\infty \rightarrow N_p$, which becomes an isomorphism $\iota: N_\infty \otimes_{E,\iota} L \simeq N_p$. Let $K$ be a compact open subgroup of $G(\mathbf A_f)$, and write $K_p \subset G(\mathbf Q_p)$ for its $p$-th component. Using the inclusion $G(\mathbf Q)^{\circ} \subset G(E)$ we thus get a local system $\underline N_\infty$ on $Y_K$; or we can use the inclusion $K_p \subset G(\mathbf Q_p) \subset G(L)$ to get a local system $\underline N_p$. Note that $k_p \in K_p$ acts on the right of $N_p$ via $n|k_p = k_p^{-1}\cdot n$.

\begin{prop}\label{prop:archi-shift}
\leavevmode
\begin{enumerate}
\item If $\gamma \in G(\mathbf Q)$ then $\gamma_p \iota(n) = \iota(\gamma_\infty n)$ for all $n \in N_\infty$.
\item The map $\iota((g,n)) = (g,g_p^{-1}\iota(n))$ defines a morphism $\iota:\underline N_\infty \longrightarrow \underline N_p$ of local systems on $Y_K$.
\item The map $\iota(\phi)(\sigma_\infty\otimes[g_f]) = g_p^{-1}\iota(\phi(\sigma_\infty\otimes[g_f]))$ defines a morphism $\iota:C^{\bullet}_{\ad,?}(K,N_\infty) \rightarrow C^{\bullet}_{\ad,?}(K,N_p)$ of cochain complexes.
\item The maps in parts (2) and (3) induce a canonical commuting diagram
\begin{equation*}
\xymatrix{
H^{\ast}_{?}(C^{\bullet}_{\ad}(K,N_\infty)) \ar[d]_-{\iota}\ar[r]^-{\simeq} & H^{\ast}_{?}(Y_K,\underline N_\infty) \ar[d]^-{\iota}\\
H^{\ast}_{?}(C^{\bullet}_{\ad}(K,N_p)) \ar[r]^-{\simeq} & H^{\ast}_{?}(Y_K,\underline N_p).
}
\end{equation*}
\end{enumerate}
\end{prop}
\begin{proof}
Everything is straightforward to check.
%
\end{proof}

\subsection{Symmetric spaces for $F$}\label{subsec:sym-spaces}
Here we specialize the above discussion to the setting of this article. In particular, $F$ now denotes a totally real number field.

First, let $G = \Res_{F/\mathbf Q}{\GL_1}$. Write $\widehat{\mathcal O}_F$ for the profinite completion of ${\mathcal O}_F$ and $K_\infty^{\circ} = \{1\} \subset (F_\infty^\times)^{\circ}$ (maximal compact) and $K =  \widehat{\mathcal O}_F^\times \subset \GL_1(\mathbf A_{F,f})$ . The corresponding symmetric space is written $\mathrm C_\infty := F^\times \backslash \mathbf A_F^\times / \widehat{\mathcal O}_F^\times$. 

Write $\mathbf A_{F,+}^\times := (F_\infty^\times)^{\circ} \times \mathbf A_{F,f}^\times$ and $F_+^\times = F^\times \cap (F_\infty^\times)^{\circ}$. By weak approximation, $F^\times \backslash \mathbf A_{F}^\times \simeq F_+^\times \backslash \mathbf A_{F,+}^\times$ and so we may also write
\begin{equation}\label{eqn:shin-cone}
\mathrm C_\infty = F^\times \backslash \mathbf A_F^\times / \widehat{\mathcal O}_F^\times \simeq F_+^\times \backslash \mathbf A_{F,+}^\times/\widehat{\mathcal O}_F^\times.
\end{equation}
This is a real Lie group that sits inside an exact sequence
\begin{equation}
1 \rightarrow (F_\infty^\times)^\circ/{\mathcal O_{F,+}^\times} \rightarrow \mathrm C_\infty \rightarrow \Cl_F^+ \rightarrow 1
\end{equation}
where $\Cl_F^+$ is the narrow class group and $\mathcal O_{F,+}^\times$ are the totally positive units in $F$. 

We will write ${dx_\infty \over x_\infty}$ for the canonical choice of Haar measure on $(F_\infty^\times)^{\circ}$, which then induces a translation-invariant orientation on $\mathrm C_\infty$. This fixes a Borel--Moore fundamental class $[\mathrm C_\infty] \in H_d^{\BM}(\mathrm C_\infty, \mathbf Z)$. We record this discussion as a proposition.

\begin{prop}\label{prop:oritentation-preserving}
If $x \in \mathbf A_F^\times$ then right multiplication $r_x: \mathrm C_\infty \rightarrow \mathrm C_\infty$ is orientation preserving. In particular, $(r_x)_{\ast}[\mathrm C_\infty] = [\mathrm C_\infty]$.
\end{prop}

Now let $G = \Res_{F/\mathbf Q} \GL_2$. Here, we take $K_\infty^{\circ} = \mathrm{SO}_2(F_\infty)Z(F_\infty) \subset \GL_2(F_\infty)^{\circ}$ (maximal compact mod-center). For $K \subset \GL_2(\mathbf A_{F,f})$ we write $Y_K$ for the symmetric space as in \eqref{eqn:symmetric-space}. We will be a bit more concrete regarding $Y_K$. Let $\mathfrak h$ denote the complex upper half plane. Then, $\GL_2(F_\infty)^{\circ}$ acts on $\mathfrak h^{\Sigma_F}$ via fractional linear transformations
\begin{equation}\label{eqn:mobius-action}
g \cdot z := {az + b\over cz+d}
\end{equation}
for $g = \begin{smallpmatrix}
a & b \\ c & d 
\end{smallpmatrix} \in \GL_2(F_\infty)^\circ$ and $z \in \mathfrak h^{\Sigma_F}$. If $i \in \mathfrak h^{\Sigma_F}$ means the complex number $i$ diagonally embedded then $K_\infty^{\circ}$ is the stabilizer of $i$ so that $D_\infty =  \GL_2(F_\infty)^\circ/K_\infty^{\circ} \simeq \mathfrak h^{\Sigma_F}$. Thus
\begin{equation}\label{eqn:double-coset}
Y_{K} = \GL_2(F)\leftmod \GL_2(\mathbf A_F) / K_\infty^{\circ} K \simeq \GL_2^+(F) \leftmod D_\infty \times \GL_2(\mathbf A_{F,f})/K,
\end{equation}
and $Y_K$ is a $2d$-dimensional real orbifold, decomposing into a finite disjoint union of quotients $\Gamma(g)\leftmod D_\infty$ where $\Gamma(g) = gKg^{-1} \cap \GL_2^+(F)$ (see \eqref{eqn:union-decomposition}). We make the following definition.

\begin{defn}\label{defn:good-level}
Let $K \subset \GL_2(\mathbf A_{F,f})$ be a compact open subgroup.
\begin{enumerate}
\item $K$ is neat if $\Gamma(g)/Z(\Gamma(g))$ is torsion-free for all $g \in \GL_2(\mathbf A_{F,f})$.
\item $K$ is $\mathrm t$-good if $\begin{smallpmatrix} \widehat{\mathcal O}_F^\times \\ & 1 \end{smallpmatrix} \subset K$.
\end{enumerate}
\end{defn}

As mentioned above, if $K$ is a neat level then $Y_K$ is a manifold. The purpose of the $\mathrm t$-good definition is that for $\mathrm t$-good levels $K$,  the map $\mathbf A_{F}^\times \rightarrow \GL_2(\mathbf A_F)$ given by $x \mapsto \begin{smallpmatrix} x \\ & 1 \end{smallpmatrix}$ descends to a closed (thus, proper) immersion 
\begin{equation}\label{eqn:shintani-torus}
\mathrm t: \mathrm C_\infty \hookrightarrow Y_K.
\end{equation}
In particular, for such $K$ one gets pullbacks (resp.\ pushforwards) along $\mathrm t$ on compactly supported cohomology (resp.\ Borel--Moore homology).

Beginning in Section \ref{subsec:hecke-operators} we will mostly be concerned with level subgroups of the form 
\begin{equation}\label{eqn:K1(n)}
K_1(\mathfrak n) = \biggl\{g = \begin{pmatrix} a & b \\ c & d\end{pmatrix} \in \GL_2(\widehat{\mathcal O}_{F}) \mid c \equiv 0 \bmod \mathfrak n \widehat{\mathcal O}_F, d \equiv 1 \bmod \mathfrak n \widehat{\mathcal O}_F\biggr\}
\end{equation}
with $\mathfrak n$ an integral ideal.
\begin{prop}\label{prop:neatness}
Let $\mathfrak n \subset \mathcal O_F$ be an integral ideal.
\begin{enumerate}
\item There exists $\mathfrak n' \subset \mathfrak n$ such that $K_1(\mathfrak n')$ is neat.
\item $K_1(\mathfrak n)$ is $\mathrm t$-good.
\end{enumerate}
\end{prop}
\begin{proof}
(1) follows from \cite[Lemma 2.1]{Dimitrov-Ihara}. (2) is clear.
\end{proof}

\subsection{Weights and algebraic local systems}
\label{subsec:weights}
Here we specify a collection algebraic local systems.
\begin{defn}\label{defn:cohomological-wt}
A cohomological weight $\lambda=(\lambda_1,\lambda_2)$ is a pair of characters $\lambda_i : F^\times \rightarrow \mathbf C^\times$ of the form
\begin{equation*}
\lambda_i(\xi) = \prod_{\sigma \in \Sigma_F} \sigma(\xi)^{e_i(\sigma)}
\end{equation*}
for $e_i(\sigma) \in \mathbf Z$ such that:
\begin{enumerate}
\item If $\omega_{\lambda} = \lambda_1\lambda_2 : F^\times \rightarrow \mathbf C^\times$ then $\omega_{\lambda}$ is trivial on a finite index subgroup of $\mathcal O_F^\times$, and
\item $e_1(\sigma) \geq e_2(\sigma)$ for all $\sigma \in \Sigma_F$.
\end{enumerate}
\end{defn}
Let $\lambda$ be a cohomological weight. An argument of Weil implies that $w(\sigma) = e_1(\sigma) + e_2(\sigma)$ is independent of $\sigma \in \Sigma_F$. Set $\kappa_\sigma = e_1(\sigma) - e_2(\sigma)$; this is a non-negative integer. Thus a cohomological weight $\lambda$ is the same data as a pair $(\kappa,w) \in \mathbf Z_{\geq 0}^{\Sigma_F} \times \mathbf Z$ with $\kappa_\sigma \equiv w \bmod 2$ for each $\sigma \in \Sigma_F$. We will almost always write $\lambda = (\kappa,w)$ to indicate a cohomological weight in this way.

If $n$ is a non-negative integer, write $\mathscr L_n$ for the space of polynomials over $\mathbf Z$ with degree at most $n$. If $R$ is a ring, write $\mathscr L_n(R) = \mathscr L_n \otimes_{\mathbf Z} R$. We equip $\mathscr L_n$ with an algebraic left-action of ${\GL_2}$ via 
\begin{equation}\label{eqn:polynomial-action}
(g \cdot P)(X) = (a+cX)^nP\left({b + dX \over a + cX}\right)
\end{equation}
for $g = \begin{smallpmatrix} a & b \\ c & d \end{smallpmatrix} \in \GL_2(R)$ and $P \in \mathscr L_n(R)$. Given a cohomological weight $\lambda = (\kappa,w)$ we write
\begin{equation}\label{eqn:ell-lambda}
\mathscr L_{\lambda} := \bigotimes_{\sigma \in \Sigma_F}\left( \mathscr L_{\kappa_{\sigma}}(F) \otimes {\det}^{w - \kappa_{\sigma}\over 2}\right)
\end{equation}
(where $\det: \GL_2 \rightarrow \mathbf G_m$ is the determinant character).  Thus $\mathscr L_\lambda$ is an $F$-vector space equipped with an algebraic representation of the $F$-algebraic group $(\Res_{F/\mathbf Q} \GL_2) \times_{\mathbf Q} F$, and so we can apply the discussion at the end of Section \ref{subsec:adelic-cochains} to $G = \Res_{F/\mathbf Q} \GL_2$ and $\mathscr N = \mathscr L_{\lambda}$. 

 Specifically, suppose that $E \subset \mathbf C$ contains $\sigma(F)$ for all $\sigma \in \Sigma_F$, and let $L = \mathbf Q_p(\iota(E))$. Then, $G(E)=\GL_2(F\otimes_{\mathbf Q} E) \simeq \GL_2(E)^{\Sigma_F}$ and the action of $\GL_2(E)^{\Sigma_F}$ on 
\begin{equation*}
\mathscr L_\lambda(E) := \bigotimes_{\sigma\in \Sigma_F} \mathscr L_{\kappa_\sigma}(E) \otimes {\det}^{w-\kappa_\sigma\over 2}
\end{equation*}
is the one where the $\sigma$-th factor acts on the $\sigma$-th term in the tensor product, as in \eqref{eqn:polynomial-action}. On the other hand,
\begin{equation*}
G(L) = \GL_2(F\otimes_{\mathbf Q} L) \simeq \GL_2(F_p\otimes_{\mathbf Q_p} L) \simeq \prod_{v\mid p} \GL_2(F_v\otimes_{\mathbf Q_p} L) \simeq \prod_{v \mid p} \GL_2(L)^{\Sigma_v}
\end{equation*}
and $G(L)$ acts on the $L$-vector space
\begin{equation}\label{eqn:padic-Llambda}
\mathscr L_{\lambda}(L) := \bigotimes_{v \mid p} \bigotimes_{\sigma \in \Sigma_v} \mathscr L_{\kappa_{\sigma}}(L) \otimes {\det}^{w - \kappa_\sigma\over 2}
\end{equation}
in the analogous way.

\begin{rmk}
 For any compact open subgroup $K \subset \GL_2(\mathbf A_{F,f})$, the above representations define local systems $\mathscr L_\lambda(E)$ and $\mathscr L_\lambda(L)$ on $Y_K$, and $\iota$ induces a $\mathbf Q$-linear morphism of local systems $\iota: \mathscr L_{\lambda}(E) \rightarrow \mathscr L_{\lambda}(L)$ by Proposition \ref{prop:archi-shift}.  However, we note that the $\iota$-transfer from $\mathscr L_\lambda(E)$ to $\mathscr L_\lambda(L)$ has a non-trivial effect on certain formulas (cf.\ Section \ref{subsec:padic-twisting}).
 
 For instance, suppose that $g \in \GL_2(\mathbf A_{F,f})$, $K \subset \GL_2(\mathbf A_{F,f})$ is a compact open subgroup and $K' \subset K$ is another compact open subgroup so that $g^{-1}K'g \subset K$.  Write $\mathscr L_\lambda(L)(g)$ for the left $G(L)$-representation whose action is given by $h\cdot_{g} P := g_p^{-1} h g_p\cdot  P$ for $P \in \mathscr L_\lambda(L)$ and $h \in G(L)$. Then $P \mapsto g_p^{-1}\cdot P $ defines a $G(L)$-equivariant isomorphism $\mathscr L_\lambda(L) \simeq \mathscr L_\lambda(L)(g)$ (compare with \eqref{eqn:adelic-cochains}) that fits into a commutative diagram of local systems whose bases are as indicated:
\begin{equation}\label{eqn:iota-transfer}
\xymatrix{
{{\mathscr L}_\lambda(E) \ar[r]^-{\iota}}_{/Y_{K'}} \ar[d]_-{r_g} & {{\mathscr L}_\lambda(L)}_{/Y_{K'}}  \ar[dr]^-{P \mapsto g_p^{-1}\cdot P}\\
{{\mathscr L}_{\lambda}(E)}_{/Y_{g^{-1}K'g}} \ar[d]_-{\pr} & {{\mathscr L}_{\lambda}(L)}_{/Y_{g^{-1}K'g}} \ar[d]^-{\pr} & {{\mathscr L}_{\lambda}(L)(g)}_{/Y_{K'}} \ar[l]^-{r_g}\\
{{\mathscr L}_{\lambda}(E)}_{/Y_{K}} 
\ar[r]_-{\iota} & {{\mathscr L}_{\lambda}(L)}_{/Y_{K}}.
}
\end{equation}
\end{rmk}

%% file: hmf_basics.tex
\subsection{Recollection of definitions}\label{subsec:hmf-3ways}
The goal of this subsection is to describe the three points of view that we need to adopt regarding Hilbert modular forms. General references for automorphic representation theory are \cite{BorelJacquet-Corvallis, Bump-AutomorphicForms}. Specific to Hilbert modular forms, one might refer to \cite[Section 2]{Hida-padicHeckeTotallyReal} or \cite[Section 3]{Hida-CriticalValues}. For us, the most frequently useful reference is \cite{GetzGoresky}, which itself follows Hida's papers in my places. To help the reader, after our definitions we will explain how to translate between our notations (or, if you like, conventions) and those of \cite{GetzGoresky}. See Remark \ref{rmk:normalization}.

Let $t$ be a real number. We write $\omega_t$ for the character of $F_\infty^\times$ given by $\omega_t(x_\infty) = \prod_{\sigma} x_\sigma^t$ for $x_\infty = (x_\sigma) \in F_\infty^\times$. When $t = w$ is an integer, the restriction to $F^\times \subset F_\infty^\times$ is what we called $\omega_{\lambda}$ in Definition \ref{defn:cohomological-wt}.  Suppose that $\omega: F_\infty^\times \rightarrow \mathbf C^\times$ is a continuous character such that $\omega|_{(F_\infty^\times)^{\circ}} = {\omega_t}|_{(F_\infty^\times)^{\circ}}$. We write $L^2(\GL_2(F)\leftmod \GL_2(\mathbf A_F),\omega)$ for the space of functions $f : \GL_2(F)\leftmod \GL_2(\mathbf A_F) \rightarrow \mathbf C$ that satisfy the following two properties:
\begin{enumerate}
\item $f(x_\infty g) = \omega^{-1}(x_\infty)f(g)$ for all $g \in \GL_2(\mathbf A_F)$ and $x_\infty \in F_\infty^\times$.
\item $|\!\det g|^{t/2}|f(g)|$ is square-integrable on $(F_\infty^\times)^\circ\GL_2(F)\leftmod \GL_2(\mathbf A_F)$.
\end{enumerate}
The condition in (2) is well-defined by the condition (1) and the assumption on $\omega$.  We further write $L^2_0(\GL_2(F)\leftmod \GL_2(\mathbf A_F),\omega)$ for those $f \in L^2(\GL_2(F)\leftmod \GL_2(\mathbf A_F),\omega)$ which are cuspidal, meaning that
\begin{equation}\label{eqn:cuspidal-condition}
\int_{F\leftmod \mathbf A_F} f\left(\begin{smallpmatrix} 1 & u \\ & 1 \end{smallpmatrix} g \right) du = 0 \;\;\;\;\;\;\; \text{(for all $g \in \GL_2(\mathbf A_F)$).}
\end{equation}
Note that the group $\GL_2(\mathbf A_{F})$ acts on these $L^2$-spaces by right translation in the domain.

\begin{defn}
A  cuspidal automorphic representation $\pi$ for $\GL_2(\mathbf A_F)$ is an irreducible admissible $\GL_2(\mathbf A_F)$-subrepresentation of $L^2_0(\GL_2(F)\leftmod \GL_2(\mathbf A_F),\omega)$ for some $\omega$.
\end{defn}

By admissible here, we mean the induced $(\mathfrak{gl}_2(F_\infty),K_\infty^{\circ}) \times \GL_2(\mathbf A_{F,f})$-module on the $K_\infty^{\circ}$-finite vectors of $\pi$ are admissible in the usual sense (\cite[Section 3.3]{Bump-AutomorphicForms}). For a cuspidal automorphic representation $\pi$, we write $\pi = \bigotimes_{v}' \pi_v$ for its factorization as a restricted tensor product (\cite{Flath-TensorProduct}). We further specify the notation $\pi_{\infty} := \bigotimes_{\sigma \in \Sigma_F} \pi_{\sigma}$, and $\pi_f:= \bigotimes_{v}' \pi_v$ where $v$ runs over finite places of $F$, so $\pi = \pi_\infty \otimes \pi_f$.

For the rest of this subsection, fix a cohomological weight $\lambda = (\kappa,w)$. We need two representations associated to $\lambda$. First, $\mathbf C_\lambda$ is the 1-dimensional $\mathbf C$-vector space $\mathbf C_\lambda = \mathbf C \cdot v$ on which we let $K_\infty^{\circ}$ act by
\begin{equation}\label{eqn:Clambda-action}
v|_{k_\infty} := \omega_w^{-1}(x_\infty)e^{i\theta_\infty(\kappa+2)}\cdot v.
\end{equation}
Here, $k_\infty \in K_\infty^{\circ}$ is written $k_\infty = x_\infty r_\infty$ with $x_\infty \in F_\infty^\times$ and $r_\infty = \begin{smallpmatrix} \cos \theta_\infty & \sin \theta_\infty \\ -\sin \theta_\infty & \cos \theta_\infty\end{smallpmatrix} \in \SO_2(F_\infty)$.  Second, for $\sigma \in \Sigma_F$ we write $D_{\kappa_\sigma+2,w}$ for the weight $\kappa_\sigma + 2$ discrete series representation of $\GL_2(\mathbf R)$ with central character $x \mapsto x^{-w}$ (see \cite[Section 11]{KnightlyLi-TracesHecke} for example). Then, we define $D_\lambda := \bigotimes_{\sigma \in \Sigma_F} D_{\kappa_\sigma+2,w}$ (a representation of $\GL_2(F_\infty)$).
\begin{defn}
A cuspidal automorphic representation $\pi$ is cohomological of weight $\lambda$ if $\pi_\infty \simeq D_{\lambda}$.
\end{defn}
We recall that there is a unique $K_\infty^{\circ}$-equivariant embedding $\mathbf C_\lambda \subset D_\lambda$, the image of which generates $D_\lambda$ as a $\GL_2(F_\infty)$-representation. Given $\pi$, cohomological of weight $\lambda$, we write $\pi_\infty^+ \subset \pi_\infty$ for the corresponding line. We also note that the irreducibility and admissibility of such a $\pi$ implies that $\mathbf A_F^\times$ acts on $\pi$ through a Hecke character $\omega_\pi$ (the central character). Of course, $\omega_{\pi,\infty} := \omega_\pi|_{F_\infty^\times} = \omega_w^{-1}$ and thus $\pi \subset L^2_0(\GL_2(F)\leftmod \GL_2(\mathbf A_F),\omega_w)$.

We now turn towards automorphic forms.

\begin{defn}\label{defn:automorphic-form}
Let $K \subset \GL_2(\mathbf A_{F,f})$ be a compact open subgroup. The space of cohomological cuspidal automorphic forms of weight $\lambda$ and level $K$ is the set $S_{\lambda}(K)$  of all functions $\phi: \GL_2(\mathbf A_F) \rightarrow \mathbf C_{\lambda}$ satisfying the following conditions.
\begin{enumerate}
\item If $g_f \in \GL_2(\mathbf A_{F,f})$, then the function $g_\infty \mapsto \phi(g_\infty g_f)$ is a smooth function on $\GL_2(F_\infty)$.
\item If $\sigma \in \Sigma_F$, then $C_{\sigma}(\phi) = \left(\kappa_{\sigma} + {1\over 2}\kappa_\sigma^2\right)\phi$, where $C_{\sigma}$ denotes the Casimir operator.\footnote{The Casimir operator is the element $XY + YX + {1\over 2}H^2$ in the center of $\mathrm U(\mathfrak{sl}_2(\mathbf R)\otimes_{\mathbf R} \mathbf C)$ where $X = {1\over 2}\begin{smallpmatrix} 1 & i \\ i & - 1\end{smallpmatrix}$, $Y = {1\over 2}\begin{smallpmatrix} 1 & -i \\ -i & - 1\end{smallpmatrix}$ and $H = \begin{smallpmatrix} 0 & i \\ -i & 0 \end{smallpmatrix}$. It acts as a differential operator on smooth functions $\GL_2(\mathbf R) \rightarrow \mathbf C$. What we mean by $C_{\sigma}$ is the Casimir operator acting on the $\sigma$-th component of functions $\GL_2(F_\infty) \rightarrow \mathbf C$.}

\item If $\gamma \in \GL_2(F)$, $g \in \GL_2(\mathbf A_F)$, $k_\infty \in K_\infty^{\circ}$ and $k \in K$, then $\phi(\gamma g k_\infty k) = \phi(g)|_{k_\infty}$.
\item $\phi$ is cuspidal in the sense that \eqref{eqn:cuspidal-condition} holds for $f = \phi$ and all $g \in \GL_2(\mathbf A_F)$.
\end{enumerate}
\end{defn}
The $\mathbf C$-vector space $S_\lambda(K)$ is finite-dimensional, but it is not a representation of $\GL_2(\mathbf A_F)$. Instead, if $g \in \GL_2(\mathbf A_F)$ and $\phi \in S_\lambda(K)$ then $(g\cdot \phi)(g') := \phi(g'g)$ defines a natural $\mathbf C$-linear map $S_\lambda(K) \rightarrow S_\lambda(gKg^{-1})$. Note as well that $S_\lambda(K) \subset L^2_0(\GL_2(F)\leftmod \GL_2(\mathbf A_F),\omega_w)$. Indeed, this is true by \cite[Section 4.4]{BorelJacquet-Corvallis} when $\phi \in S_\lambda(K)$ has a central character (i.e.\ there exists a Hecke character $\omega_\phi$ such that $\phi(z g) = \omega(z) \phi(g)$ for all $z \in \mathbf A_F^\times$) and it is not difficult to see that any $\phi$ is a finite sum of $\phi$'s with central character (because $S_\lambda(K)$ is finite-dimensional). Moreover, the discussion in \cite{BorelJacquet-Corvallis} implies:

\begin{prop}\label{prop:decomposition-cusp-auto-forms}
Let $\mathcal A_\lambda^0$ be the set of all cohomological cuspidal automorphic representations of weight $\lambda$. Then, for each compact open subgroup $K \subset \GL_2(\mathbf A_{F,f})$ there is a canonical isomorphism
\begin{equation}\label{eqn:auto-rep-v-form}
S_{\lambda}(K) \simeq \bigoplus_{\pi \in \mathcal A_\lambda^0} \pi_\infty^+ \otimes_{\mathbf C} \pi_f^{K}
\end{equation}
as subspaces of $L_0^2(\GL_2(F)\leftmod \GL_2(\mathbf A_F), \omega_{w})$.
\end{prop}

In order to describe the Eichler--Shimura construction (Section \ref{subsec:eichler-shimura}), we also need a holomorphic version of the previous notion. Recall from Section \ref{subsec:sym-spaces} that we write $D_\infty := \mathfrak h^{\Sigma_F}$. If $g = \begin{smallpmatrix} a_\sigma & b_\sigma \\ c_\sigma & d_\sigma\end{smallpmatrix} \in \GL_2(F_\infty)$ and $z = (z_\sigma) \in D_\infty$, then we define an automorphy factor
\begin{equation*}
j(g,z) = (c_\sigma z_\sigma + d_{\sigma})_{\sigma \in \Sigma_F} \in \mathbf C^{\Sigma_F}.
\end{equation*}
In particular, one can take $g = \gamma \in \GL_2(F)$ embedded diagonally into $\GL_2(F_\infty)$. Recall also that $\gamma \in \GL_2^+(F)$ acts on $z \in D_\infty$ by fractional linear transformation $z \mapsto \gamma \cdot z$.
\begin{defn}\label{defn:holomorphic-HMF}
Let $K \subset \GL_2(\mathbf A_{F,f})$ be a compact open subgroup. A holomorphic Hilbert cuspform $\mathbf f$ of weight $(\kappa+2,w)$ and level $K$ is a function 
\begin{equation*}
\mathbf f: D_\infty \times \GL_2(\mathbf A_{F,f}) \rightarrow \mathbf C
\end{equation*}
satisfying the following conditions.
\begin{enumerate}
\item If $g_f \in \GL_2(\mathbf A_{F,f})$, then the function $z \mapsto \mathbf f(z,g_f)$ is holomorphic in $z$.
\item If $\gamma = \begin{smallpmatrix} a & b \\ c & d \end{smallpmatrix} \in \GL_2^+(F)$, $k \in K$ and $g_f \in \GL_2(\mathbf A_{F,f})$ then 
\begin{equation}\label{eqn:holo-transform}
\mathbf f(\gamma\cdot z, \gamma g_f k) = \det(\gamma)^{{w - \kappa \over 2} - 1}j(\gamma,z)^{\kappa+2}\mathbf f(z,g_f).
\end{equation}
\item $\mathbf f$ is cuspidal in the sense that $\phi_{\mathbf f}$ as defined below satisfies \eqref{eqn:cuspidal-condition}. 
\end{enumerate}
\end{defn}
We write $S^{\hol}_\lambda(K)$ for the space of holomorphic Hilbert cuspforms $\mathbf f$ of weight $(\kappa + 2, w)$. It is entirely straightforward to compare the spaces $S_{\lambda}^{\hol}(K)$ and $S_\lambda(K)$. Namely, given $\phi \in S_\lambda(K)$ we define
\begin{equation*}
\mathbf f_{\phi}(g_\infty,g_f) := \det(g_\infty)^{{w-\kappa\over 2} - 1}\cdot j(g_\infty,i)^{\kappa+2}\phi(g_\infty g_f).
\end{equation*}
Here $g_\infty \in \GL_2(F_\infty)^{\circ}$ and $g_f \in \GL_2(\mathbf A_{F,f})$. It is straightforward to see that $g_\infty \mapsto \mathbf f_{\phi}(g_\infty,g_f)$ is invariant under right-multiplication by $K_\infty^{\circ}$ and thus descends to a function on $D_\infty \times \GL_2(\mathbf A_{F,f})$.\footnote{To be clear:\ to compute $\mathbf f_{\phi}(z,g_f)$ one finds a $g_\infty \in \GL_2(F_\infty)^{\circ}$ such that $g_\infty \cdot i = z$ and then computes $\mathbf f_{\phi}(g_\infty, g_f)$ by the formula we just gave.} It is also readily verified that $\mathbf f_{\phi} \in S^{\hol}(K)$, i.e. that $\mathbf f_{\phi}$ is actually holomorphic, cf. \cite[p. 460]{Hida-CriticalValues}. To go backwards, given $\mathbf f \in S_\lambda^{\hol}(K)$, view it as a function on $\GL_2(F_\infty)^{\circ} \times \GL_2(\mathbf A_{F,f})$. Then define $\phi_{\mathbf f}$ on the same domain by
\begin{equation*}
\phi_{\mathbf f}(g) := \det(g_\infty)^{1 - {w-\kappa\over 2}}j(g_\infty,i)^{-\kappa - 2}\mathbf f(g_\infty,g_f)
\end{equation*}
for $g = g_\infty g_f \in \GL_2(F_\infty)^{\circ} \times \GL_2(\mathbf A_{F,f})$. Finally, extend $\phi_{\mathbf f}$ to all of $\GL_2(\mathbf A_{F})$ by \eqref{eqn:double-coset}. We finally remark that $\phi \leftrightarrow \mathbf f_{\phi}$ and $\mathbf f \leftrightarrow \phi_{\mathbf f}$ are clearly compatible with right translation by $g_f \in \GL_2(\mathbf A_{F,f})$.

\begin{rmk}\label{rmk:normalization}
It remains an open question how many notations and normalizations for Hilbert modular forms possibly exist. We pause here to align our own notations with just one of our major references, the text of Getz--Goreskey \cite{GetzGoresky}. 

The most fundamental normalization is the weight. For us, a weight is a certain pair $(\kappa,w)$. Getz and Gorskey define weights in \cite[Section 5.3]{GetzGoresky} as a certain pair $(k,m)$, each of $k,m$ being a $\Sigma_F$-tuple. Our pair $(\kappa,w)$ defines a pair $(k,m)$ via the dictionary $k_\sigma = \kappa_\sigma$ and $m_\sigma = \frac{w - \kappa_\sigma}{2}$. 

Under this correspondence the space $S_\lambda(K)$ we defined in Definition \ref{defn:automorphic-form} is, up to fixing central character, the spaces denoted $S_{(k,m)}(K,\chi)$ in \cite[Section 5.4]{GetzGoresky}. (Notice the assumption (5.4.2) in {\em loc.\ cit.}, which explains in part how we came to choose to insert an inverse into the $L^2(-,\omega)$ notation at the start of this subsection.)

Getz and Goresky also helpfully discuss the issue of normalizations for Hilbert modular forms. One discussion focuses on what they call the cohomogical normalization. See \cite[Section 5.5]{GetzGoresky}. The cohomological normalization will be needed for readers of this paper who want to compare our claims in Section \ref{subsec:eichler-shimura} below, on the Eichler--Shimura constructions, with those in \cite[Chapter 6]{GetzGoresky}. The second discussion, which can be found in see \cite[Section 5.13]{GetzGoresky}, compares their notations (and thus ours) with those of Hida's papers \cite{Hida-padicHeckeTotallyReal,Hida-CriticalValues}. .
\end{rmk}

\subsection{Hecke operators, Fourier expansions and newforms}\label{subsec:hecke-operators}
The main goal of this subsection to make precise the notion of the newform associated to a cohomological cuspidal automorphic representation $\pi$. We will also record information about Hecke operators and Fourier expansions. We leave  transcription of the discussion to $S_{\lambda}^{\hol}(K)$ to the reader.

Let $K$ be a compact open subgroup in $\GL_2(\mathbf A_{F,f})$ and $g \in \GL_2(\mathbf A_{F,f})$. The double coset $KgK$ can be decomposed into a finite disjoint union $KgK = \bigcup_i g_i K$ of right $K$-cosets. Then for any cohomological weight $\lambda$, we get a Hecke operator $[KgK]$ acting on the space $S_{\lambda}(K)$ by
\begin{equation}\label{eqn:hecke-operator}
([KgK]\phi)(g) = \sum \phi(g g_i) \;\;\;\;\; (\phi \in S_{\lambda}(K)).
\end{equation}
The operator $[KgK]$ is independent of the choice of the $g_i$'s.

For the rest of the subsection we are interested in $K$ of the form $K_1(\mathfrak n)$ (see \eqref{eqn:K1(n)}) for $\mathfrak n \subset \mathcal O_F$ an integral ideal.
\begin{defn}\label{defn:hecke-operators}
Let $\mathfrak m \subset \mathcal O_{F}$ be an integral ideal, written $\mathfrak m = \prod_{v} \mathfrak p_v^{m_v}$, and $\varpi_{\mathfrak m} = \prod_{v} \varpi_v^{m_v} \in \mathbf A_{F,f}^\times$.
\begin{enumerate}
\item $T_{\mathfrak m} := [K_1(\mathfrak n)\begin{smallpmatrix} \varpi_{\mathfrak m} \\ & 1\end{smallpmatrix}K_1(\mathfrak n)]$. 
\item If $(\mathfrak m, \mathfrak n) = 1$, $S_{\mathfrak m} := [K_1(\mathfrak n)\begin{smallpmatrix} \varpi_{\mathfrak m} \\ & \varpi_{\mathfrak m} \end{smallpmatrix}K_1(\mathfrak n)]$.
\item When $\mathfrak m = \mathfrak p_v$ is a prime ideal we write $T_v := T_{\mathfrak p_v}$ and $S_v = S_{\mathfrak p_v}$ (when $(\mathfrak p_v,\mathfrak n) = 1$).
\end{enumerate}
\end{defn}

We denote $\mathbf T_{\mathbf Z}(K_1(\mathfrak n))$ the $\mathbf Z$-algebra abstractly generated by the Hecke operators. So, for each cohomological weight $\lambda$ we have a natural morphism of $\mathbf C$-algebras
\begin{equation*}
\mathbf T_{\mathbf C}(K_1(\mathfrak n)):=\mathbf T_{\mathbf Z}(K_1(\mathfrak n)) \otimes_{\mathbf Z} \mathbf C \rightarrow \End_{\mathbf C}(S_{\lambda}(K_1(\mathfrak n))).
\end{equation*}

\begin{rmk}\label{rmk:well-posed-locally-calculated}
We will assume the reader is familiar with basic properties of the $T_{\mathfrak m}$ (see \cite[Section 5.6]{GetzGoresky} for example). For instance, $T_{\mathfrak m}$ and $S_{\mathfrak m}$, when defined, are independent of the choice of uniformizers and they are multiplicative over co-prime ideals $\mathfrak m$ because the double coset representatives $g_i$ as in \eqref{eqn:hecke-operator} are calculated ``locally at $\mathfrak m$'' in that they can be chosen to be $\begin{smallpmatrix} 1 \\ & 1 \end{smallpmatrix}$ at each place $v$ where $\mathfrak p_v \nmid \mathfrak n$.
\end{rmk}
\begin{rmk}\label{rmk:U-operator-notation}
If $\mathfrak m \mid \mathfrak n$, then we will sometimes use the notations $U_{\mathfrak m} := T_{\mathfrak m}$, $U_v := T_v$, etc. Let us recall an explicit formula in that case. When $\mathfrak m \mid \mathfrak n$, one may check that the representatives $K_1(\mathfrak n)\begin{smallpmatrix}\varpi_{\mathfrak m} \\ & 1 \end{smallpmatrix} K_1(\mathfrak n)/K_1(\mathfrak n)$ can be chosen to be of the form $\begin{smallpmatrix} \varpi_{\mathfrak m} & a \\ & 1 \end{smallpmatrix}$ where $a$ runs over a choice of representatives in $\prod_{v \mid \mathfrak m} \mathcal O_v$ for $\prod_{\mathfrak p_v \mid \mathfrak m} \mathcal O_{v}/\mathfrak m \mathcal O_{v}$. So, we will often write expressions like
\begin{equation}\label{eqn:Uv-summation}
(U_{\mathfrak m} \phi)(g) = \sum_{a \in \mathcal O_v/\mathfrak m \mathcal O_v} \phi\left(g \begin{smallpmatrix} \varpi_v & a \\ & 1 \end{smallpmatrix}\right),
\end{equation}
omitting the choices of lifts. This makes clear, for instance, that $U_{\mathfrak p_v^j} = U_{v}^j$ for all integers $j\geq 0$.
\end{rmk}

\begin{rmk}\label{rmk:Tv-formula}
If $\mathfrak p_v \nmid \mathfrak n$ then there is a formula similar to \eqref{eqn:Uv-summation} for $T_v$. Specifically,
\begin{equation*}
(T_v\phi)(g) = \phi(g \begin{smallpmatrix} 1 \\ & \varpi_v \end{smallpmatrix}) + \sum_{a \in \mathcal O_v/\varpi_v \mathcal O_v} \phi\left(g \begin{smallpmatrix} \varpi_v & a \\ & 1 \end{smallpmatrix}\right).
\end{equation*}
Thus $T_v$ ``is equal to'' $U_v + V_v^-$ where $V_v^-$ means translation by $\begin{smallpmatrix} 1 \\ & \varpi_v \end{smallpmatrix}$ (see Section \ref{subsec:refinements} below). The quotes refer to $T_v$ being the bona fide endomorphism of $S_{\lambda}(K_1(\mathfrak n))$ given by Definition \ref{defn:hecke-operators} whereas $U_v$ (resp.\ $V_v^-$) means the formal operator on functions $\GL_2(\mathbf A_F) \rightarrow \mathbf C$ given by \eqref{eqn:Uv-summation} (resp.\ right translation by $\begin{smallpmatrix} 1 \\ & \varpi_v \end{smallpmatrix}$). Their sum $U_v +V_v^-$ happens to be well-defined on $S_{\lambda}(K_1(\mathfrak n))$. See the calculation in Proposition \ref{prop:refined-eigenform} below.
\end{rmk}

In this article, an eigenform means an element $\phi \in S_{\lambda}(K_1(\mathfrak n))$ such that there exists a $\mathbf C$-algebra morphism $\psi : \mathbf T_{\mathbf C}(K_1(\mathfrak n)) \rightarrow \mathbf C$ such that $T\phi = \psi(T)\phi$ for all $T \in \mathbf T(K_1(\mathfrak n))$. If $\phi$ is an eigenform then we refer to $\psi = \psi_{\phi}$ as its Hecke eigensystem. 

An eigenform is only possibly unique up to scalar,  but we can normalize it in a natural way using Fourier expansions. Start by writing $e_{\mathbf Q}: \mathbf A_{\mathbf Q}\rightarrow \mathbf C^\times$ for the natural non-degenerate character
\begin{equation*}
e_{\mathbf Q}(x) = e^{2 \pi i x_{\infty}} e^{-2\pi i \{ x_f \}},
\end{equation*}
where $\{ - \}$ is the morphism on the finite adeles  given by the composition 
\begin{equation*}
\{-\}: \mathbf A_{\mathbf Q,f} \rightarrow \mathbf A_{\mathbf Q,f}/\widehat{\mathbf Z} \simeq \mathbf Q/\mathbf Z \hookrightarrow \mathbf R / \mathbf Z.
\end{equation*}
Then, define $e_F : \mathbf A_{F} \rightarrow \mathbf C^\times$ to be the composition $e_F := e_{\mathbf Q} \circ \tr_{F/\mathbf Q}$.  Next, if $\lambda = (\kappa,w)$ is a cohomological weight, then we define $W_{\lambda} : F_\infty^\times \rightarrow \mathbf C$ (an Archimedean Whittaker function) to be
\begin{equation*}
W_{\lambda}(x_\infty) := \prod_{\sigma \in \Sigma_F} |x_{\sigma}|^{\kappa_{\sigma} - w \over 2} e^{-2 \pi |x_{\sigma}|}.
\end{equation*}
Finally, we set two more notations. If $x_f \in \mathbf A_{F,f}$, then we define $[x_f]$ to be the fractional ideal $F \cap x_f \widehat{\mathcal O}_F$ and we also write $\mathcal D_{F/\mathbf Q}$ for the different ideal associated to the extension $F/\mathbf Q$.
\begin{prop}\label{prop:fourier-expansion}
For each $\phi \in S_{\lambda}(K_1(\mathfrak n))$ there exists a uniquely determined function $\widetilde a_{\phi}: \mathbf A_{F,f}^\times \rightarrow \mathbf C$ such that $\widetilde a_{\phi}(x_f)$ depends only on $[x_f]$ and
\begin{equation}\label{eqn:fourier-expansion}
\phi\left(\begin{smallpmatrix} x & y \\ & 1 \end{smallpmatrix}\right) = |x|_{\mathbf A_F} \sum_{\xi \in F_{+}^\times} \widetilde a_{\phi}(\xi x_f) W_{\lambda}(\xi x_\infty) e_F(\xi y).
\end{equation}
Moreover, $\widetilde a_{\phi}(x_f) = 0$ if $[x_f]\mathcal D_{F/\mathbf Q}$ is not integral.
\end{prop}
\begin{proof}
See \cite[Theorem 5.8]{GetzGoresky} (also, \cite[Theorem 6.1]{Hida-CriticalValues}).
\end{proof}

\begin{defn}
Let $\phi \in S_{\lambda}(K_1(\mathfrak n))$.
\begin{enumerate}
\item If $\mathfrak m \subset \mathcal O_F$ is an integral ideal, then $a_{\phi}(\mathfrak m) := \widetilde a_\phi(\xi x_f)$ for any choice of $\xi \in F_+^\times$ and $x_f \in \mathbf A_{F,f}^\times$ such that $\mathfrak m = [\xi x_f]\mathcal D_{F/\mathbf Q}$.
\item We say that $\phi$ is a normalized if $a_{\phi}(\mathcal O_F) = 1$.
\end{enumerate}
\end{defn}
\begin{rmk}\label{rmk:fourier-linearity}
For each $\mathfrak m$, the function $\phi \mapsto a_{\phi}(\mathfrak m)$ is linear. It is also helpful to note that $a_{\phi}(\mathfrak m) = a_{T_{\mathfrak m}\phi}(\mathcal O_F)$ (see \cite[Corollary 6.2]{Hida-CriticalValues} where the central character is not fixed and \cite[Chapter VI]{Weil-LNM189}). Combining these points, if $\phi$ is an eigenvector for $T_{\mathfrak m}$ and $a_{\phi}(\mathcal O_F) = 0$, then $a_{\phi}(\mathfrak m) = 0$ as well.
\end{rmk}

\begin{prop}\label{prop:hecke-eigensystems}
Let $\phi \in S_{\lambda}(K_1(\mathfrak n))$ be a normalized eigenform.
\begin{enumerate}
\item If $\mathfrak m$ is an integral ideal, then $a_{\phi}(\mathfrak m) = \psi_{\phi}(T_{\mathfrak m})$.
\item $\phi$ has a central character $\omega_{\phi}$ of conductor dividing $\mathfrak n$, and $\omega_{\phi}(\varpi_v) = \psi_\phi(S_v)$ for $\mathfrak p_v \nmid \mathfrak n$.
\end{enumerate}
\end{prop}
\begin{proof}
For (1), see the end of \cite[Section 5.9]{GetzGoresky} (and \cite[Corollary 6.2]{Hida-CriticalValues}). For part (2), we give a standard argument. If $x \in \mathbf A_{F,f}^\times$ then the translate $x\cdot \phi$ is a $T_{\mathfrak m}$-eigenvector with the same eigenvalue as $\phi$, so Remark \ref{rmk:fourier-linearity} above implies $a_{x\cdot \phi}(\mathcal O)\neq 0$. So, by multiplicity one, $x\cdot \phi = \omega_{\phi}(x)\phi$ for some non-zero constant $\omega_\phi(x)$. The assertions about $\omega_\phi$ follow immediately from Definitions \ref{defn:automorphic-form} and \ref{defn:hecke-operators}.
\end{proof}

If $\delta \in \widehat{\mathcal O}_F$ and $\mathfrak n'$ is an integral ideal with $\mathfrak n\widehat{\mathcal O}_F \subset \delta\mathfrak n'\widehat{\mathcal O}_F$, then $\phi \mapsto \phi_{\delta}(g) := \phi\left(g \begin{smallpmatrix} 1/\delta \\ & 1 \end{smallpmatrix}\right)$ gives a well-defined morphism $j_{\mathfrak n',\delta}:S_{\lambda}(K_1(\mathfrak n')) \rightarrow S_{\lambda}(K_1(\mathfrak n))$. The Hecke-stable subspace $S^{\new}_\lambda(K_1(\mathfrak n)) \subset S_{\lambda}(K_1(\mathfrak n))$ is the orthogonal complement of $\sum_{\mathfrak n \subsetneq \mathfrak n'} \im(j_{\mathfrak n',\delta})$ under the Petersson product (see \cite[Section 3]{Hida-CriticalValues} or \cite[Sections 5.7-8]{GetzGoresky}). We highlight our convention for the word ``newform'':\footnote{Note that by \cite[Theorem 5.7]{GetzGoresky}, an equivalent definition would be to require that $\phi \in S_\lambda^{\new}(K_1(\mathfrak n))$ which is normalized and an eigenform just for almost all the Hecke operators $T_v$.}
\begin{defn}
A newform $\phi$ of level $\mathfrak n$ is a normalized eigenform $\phi \in S_{\lambda}^{\new}(K_1(\mathfrak n))$.
\end{defn}
If $\pi$ is a cohomological cuspidal automorphic representation then there exists an ideal $\mathfrak n$, called the conductor of $\pi$, which is maximal among all ideals with $\pi_f^{K_1(\mathfrak n)} \neq (0)$. A famous result of Casselman (\cite{Casselman-NewVectorGL2}) implies in fact that $\dim_{\mathbf C} \pi_f^{K_1(\mathfrak n)} = 1$.

\begin{defn-prop}\label{prop:new-vector}
If $\pi$ is a cohomological cuspidal automorphic representation of conductor $\mathfrak n$, then there exists a unique newform $\phi_{\pi}$ of level $\mathfrak n$ such that $\phi_{\pi}$ generates the representation $\pi$ under the isomorphism \eqref{eqn:auto-rep-v-form}. We call $\phi_\pi$ the newform associated to $\pi$.
\end{defn-prop}
\begin{proof}
From Casselman's theorem,  we immediately get a unique normalized cuspform $\phi_{\pi} \in S_{\lambda}(K_1(\mathfrak n))$ which generates $\pi$ under \eqref{eqn:auto-rep-v-form}. Its unicity implies it is a normalized eigenform, and checking it is a newform is straightforward (see \cite[Theorem E.1]{GetzGoresky} for instance).
\end{proof}
Now let $\pi$ be a cohomological cuspidal automorphic representation. We define its Hecke eigensystem $\psi_{\pi}$ to be $\psi_{\pi} = \psi_{\phi_\pi}$ where $\phi_\pi$ is the associated newform, $a_{\pi}(\mathfrak m) = \psi_\pi(T_{\mathfrak m})$ for each integral ideal $\mathfrak m$, and the Hecke field of $\pi$ is $\mathbf Q(\pi) := \mathbf Q(\psi_\pi(T) \mid T \in \mathbf T_{\mathbf Z}(K_1(\mathfrak n)))$.
\begin{prop}
If $\pi$ is a cohomological cuspidal automorphic representation then $\mathbf Q(\pi)$ is a finite extension of $\mathbf Q$.
\end{prop}
\begin{proof}
See \cite[Proposition 2.8]{Shimura-HMF} (and replace $\phi_\pi$ by $\mathbf f_{\phi_\pi}$).
\end{proof}

\subsection{$L$-functions}\label{subsec:L-functions}
Suppose that $\phi \in S_{\lambda}(K_1(\mathfrak n))$. Its $L$-series is defined to be
\begin{equation}\label{eqn:L-series}
L(\phi,s) := \sum_{\mathfrak m \subset \mathcal O_F} a_{\phi}(\mathfrak m) N_{F/\mathbf Q}(\mathfrak m)^{-s},
\end{equation}
where the sum $\mathfrak m$ runs over integral ideals of $F$ and $N_{F/\mathbf Q}(-)$ means the absolute norm. The series \eqref{eqn:L-series} converges absolutely for the real part of $s$ sufficiently large. Further, it admits analytic continuation to all $s \in \mathbf C$ as we now recall. 

Define $\Gamma_{\mathbf C}(s) = (2\pi)^{-s}\Gamma(s)$ and then complete $L(\phi,s)$ by defining
\begin{equation*}
\Lambda(\phi, s) := \Gamma_{\mathbf C}\left(s+{\kappa-w\over 2}\right) L(\phi,s) = \left(\prod_{\sigma \in \Sigma_F} \Gamma_{\mathbf{C}}\left(s +  {\kappa_\sigma - w \over 2}\right)\right)L(\phi,s).
\end{equation*}
We can also define the Mellin transform of $\phi$
\begin{equation}\label{eqn:mellin-transform}
\mathbf M(\phi,s) := \int_{F^\times \leftmod \mathbf A_F^\times} \phi\left(\begin{smallpmatrix} x & \\  & 1\end{smallpmatrix}\right) |x|^{s} d^{\times}x.
\end{equation}
The integral \eqref{eqn:mellin-transform} is absolutely convergent for all $s \in \mathbf C$ (\cite[Section 3.5]{Bump-AutomorphicForms}). Here, $d^\times x$ is the natural Haar measure on $\mathbf A_F^\times$:\ $d^{\times}x_\infty$ is the canonical measure $\prod_{\sigma} \frac{dx_\sigma}{  |x_\sigma|}$ on $F_\infty^\times$ and $d^{\times}x_v$ is the unique multiple of $\frac{dx_v}{|x_v|_v}$ on $F_v^\times$ such that $\mathcal{O}_v^{\times}$ has measure one.

Now write $\Delta_{F/\mathbf Q}$ for the absolute discriminant $\Delta_{F/\mathbf Q} = N_{F/\mathbf Q}(\mathcal D_{F/\mathbf Q})$. The analytic continuation of $\Lambda(\phi,s)$ follows from the proposition.
\begin{prop}\label{prop:integral-representation}
If $\phi \in S_\lambda(K_1(\mathfrak n))$, then $\mathbf{M}(\phi,s) = \Delta_{F/\mathbf Q}^{s+1}\Lambda(\phi,s+1)$.
\end{prop}
We include a proof of this proposition for completeness, especially as this integral expression of the (completed) $L$-function is crucial for the algebraicity of the special values (see Section \ref{subsec:special-values}).
\begin{proof}[Proof of Proposition \ref{prop:integral-representation}]
By weak approximation, the integral \eqref{eqn:mellin-transform} is unchanged by replacing $F^\times \backslash \mathbf A_F^\times$ by $F_+^\times \backslash \mathbf A_{F,+}^\times$. Further, $x \mapsto \phi\left(\begin{smallpmatrix} x \\ & 1 \end{smallpmatrix}\right)|x|^s_{\mathbf A_F}$ is invariant under $x \mapsto \xi x$ for $\xi \in F^\times$. Thus, using the Fourier expansion (Proposition \ref{prop:fourier-expansion}) and unfolding the integral \eqref{eqn:mellin-transform}, we get
\begin{align}
\int_{F_+^\times \backslash \mathbf A_{F,+}^\times} \phi\left(\begin{pmatrix}x \\ & 1\end{pmatrix}\right)|x|_{\mathbf A_F}^s d^\times x & =\int_{F_+^\times \backslash \mathbf A_{F,+}^\times}  \left(\sum_{\xi \in F_+^\times} \tilde a_\phi(\xi x_f) |x|^{s+1}_{\mathbf A_F}W_\lambda(\xi x_\infty)\right) d^\times x \nonumber\\
&= \int_{\mathbf A_{F,+}^\times} \tilde a_\phi(x_f)|x|^{s+1}_{\mathbf A_F} W_\lambda(x_\infty) d^\times x\nonumber\\
&= \left(\int_{(F_\infty^\times)^{\circ}} x_\infty^{1+s+\kappa - w \over 2}e^{-2\pi x_\infty} {dx_\infty \over x_\infty}\right) \cdot \left(\int_{\mathbf A_{F,f}^\times} \tilde a_\phi(x_f) |x_f|^{s+1}_{\mathbf A_F} d^\times x_f\right).\label{eqn:split-product}
\end{align}
The first integral in the product \eqref{eqn:split-product} is clearly
\begin{align}
\int_{(F_\infty^\times)^{\circ}} x_\infty^{1+s+\kappa - w \over 2}e^{-2\pi x_\infty} {dx_\infty \over x_\infty} &= \prod_{\sigma \in \Sigma_F} \int_{0}^\infty \left({x_\sigma \over 2\pi}\right)^{1 + s + {\kappa_\sigma - w \over 2}}e^{-x_\sigma} {dx_\sigma\over x_\sigma}\nonumber\\
&= \Gamma_{\mathbf C}\left(1 + s + {\kappa - w \over 2}\right).\label{eqn:product1}
\end{align}
For the second integral in \eqref{eqn:split-product}, recall that $\tilde a_\phi(x_f)$ depends only on $[x_f]$ and is trivial unless $[x_f]\mathcal D_{F/\mathbf Q}$ is an integral ideal. Thus we may compute the integral
\begin{align}
\int_{\mathbf A_{F,f}^\times} \tilde a_\phi(x_f) |x_f|^{s+1} d^\times x_f &= \sum_{\mathfrak m \subset \mathcal O_F} a_\phi(\mathfrak m)\int_{\mathfrak m \mathcal D^{-1}_{F/\mathbf Q}\widehat{\mathcal O}_F^\times} |x_f|^{s+1}_{\mathbf A_F}d^\times x_f\nonumber\\
&= \Delta_{F/\mathbf Q}^{s+1}\sum_{\mathfrak m \subset \mathcal O_F} a_\phi(\mathfrak m) N_{F/\mathbf Q}(\mathfrak m)^{-(1+s)}.\label{eqn:product2}
\end{align}
For the final equality we used that $x_f \in \mathfrak m \mathcal D^{-1}_{F/\mathbf Q}\widehat{\mathcal O}_F^\times$ if and only if $|x_f|_{\mathbf A_F} = \Delta_{F/\mathbf Q}N_{F/\mathbf Q}(\mathfrak m)^{-1}$. Putting \eqref{eqn:product1} and \eqref{eqn:product2} back into \eqref{eqn:split-product}, the proof is complete.
\end{proof}

If $\phi$ is a normalized eigenform with central character $\omega_{\phi}$ (Proposition \ref{prop:hecke-eigensystems}), the Dirichlet series $L(\phi,s)$ admits an Euler product expansion $L(\phi,s) = \prod_{v} L_v(\phi,s)$, where
\begin{equation}\label{eqn:local-factors}
L_v(\phi,s)^{-1} = \begin{cases}
1 - a_{\phi}(\mathfrak p_v)q_v^{-s} + \omega_{\phi}(\varpi_v)q_v^{1-2s} & \text{(if $\mathfrak p_v \nmid \mathfrak n$);}\\
1 - a_{\phi}(\mathfrak p_v)q_v^{-s} & \text{(if $\mathfrak p_v \mid \mathfrak n$).}
\end{cases}
\end{equation}
See \cite[Section 5.12.1]{GetzGoresky}.  If, furthermore, $\phi = \phi_{\pi}$ is the newform associated to a cohomological cuspidal automorphic representation $\pi$ (Proposition \ref{prop:new-vector}) then this is the same as the Euler product expresssion
\begin{equation}\label{eqn:agreement-L-function}
L(\phi,s) = L(\pi,s) := \prod_{v} L_v(\pi_v,s)
\end{equation}
where the product runs over finite places $v$ of $F$ and the local $L$-factor $L_v(\pi_v,s)$ is defined to be
\begin{equation*}
L_v(\pi_v,s) := \det\left(1 - q_v^{-s}\Frob_v\big|_{r(\pi_v)^{I_v,N=0}}\right)^{-1}.
\end{equation*}
Here, $r(\pi_v)$ is Weil--Deligne representation associated to $\pi_v$ via the normalized local Langlands correspondence (see Section \ref{subsec:notation}), and $N$ is the monodromy operator acting on $r(\pi_v)$.

\subsection{Refinements}\label{subsec:refinements}
In this subsection we discuss the notion of ($p$-)refinements of cohomological cuspidal automorphic representations. Fix a cohomological weight $\lambda$. 

If $v$ is a finite place of $F$ and $\varpi_v$ is a choice of uniformizer then write $V_v^{-} = \begin{smallpmatrix} 1 \\ & \varpi_v \end{smallpmatrix}$. (As for why ``$-$'' is here, the opposite ordering of the diagonal will be denoted by $V_v^+$ in Section \ref{subsec:padic-eval-class}.) If $\phi \in S_\lambda(K)$, then the translate $V_v^{-} \phi$ belongs to $S_\lambda(V_v^{-} K (V_v^{-})^{-1})$ and explicitly depends on the choice of $\varpi_v$. Its independence of $\varpi_v$ can be shown if the level is prime to $v$.

\begin{lem}\label{lem:refinement-lemma}
Let $\mathfrak n$ be an integral ideal, $\phi \in S_\lambda(K_1(\mathfrak n))$, and assume that $\mathfrak p_v \nmid \mathfrak n$.
\begin{enumerate}
\item $V_v^- \phi$ belongs to $S_{\lambda}(K_1(\mathfrak n \mathfrak p_v))$ and it is independent of the choice of $\varpi_v$.
\item For any $c \in \mathbf C$, we have $a_{\phi}(\mathcal O) = a_{(1-c V_v^-)\phi}(\mathcal O)$. In particular, if $\phi$ is normalized then so is $(1-cV^-_v)\phi$.
\item $U_v V_v^{-} \phi = q_v S_v \phi$.
\item If $\mathfrak m$ is an integral ideal and $\mathfrak p_v \nmid \mathfrak m$, then $V_v^- T_{\mathfrak m} \phi = T_{\mathfrak m} V_v^- \phi$.
\end{enumerate}
\end{lem}
\begin{proof}
Since $\mathfrak p_v \nmid \mathfrak n$,  $\begin{smallpmatrix} 1 \\ & \mathcal O_v^\times \end{smallpmatrix} \subset K_1(\mathfrak n)$ and thus $V_v^-\phi$ is independent of the choice of $\varpi_v$. That it is an automorphic form of level $K_1(\mathfrak n \mathfrak p_v)$ follows from the straightforward inclusion $K_1(\mathfrak n \mathfrak p_v) \subset V_v^- K_1(\mathfrak n) (V_v^-)^{-1}$. This completes the proof of (1).

We will check (2) using Fourier expansions. As mentioned in Remark \ref{rmk:fourier-linearity}, $\phi \mapsto a_{\phi}(\mathfrak m)$ is linear. So, it suffices to show that $a_{V_v^-\phi}(\mathcal O_F) = 0$. To this end, we note the relation
\begin{equation}\label{eqn:commutation}
\begin{smallpmatrix} x & y \\ 0 & 1 \end{smallpmatrix} \begin{smallpmatrix} 1 & 0 \\ 0 & \varpi_v \end{smallpmatrix} = \begin{smallpmatrix} x\varpi_v^{-1} & y \\ 0 & 1\end{smallpmatrix}\begin{smallpmatrix} \varpi_v & 0 \\ 0 & \varpi_v \end{smallpmatrix}.
\end{equation}
By \eqref{eqn:commutation} and Proposition \ref{prop:fourier-expansion} we deduce that
\begin{equation}\label{eqn:fourier-shift}
\widetilde a_{V_v^-\phi}(\xi x_f) = |\varpi_v^{-1}|_{\mathbf A_F} \widetilde a_{S_v \phi}(\xi x_f \varpi_v^{-1}).
\end{equation}
In particular, if $\xi$ and $x_f$ are chosen so that $[\xi x_f]\mathcal D_{F/\mathbf Q} = \mathcal O_F$ then certainly $[\xi x_f \varpi_v^{-1}]\mathcal D_{F/\mathbf Q}$ is not an integral ideal. But then the quantity \eqref{eqn:fourier-shift} vanishes by Proposition \ref{prop:fourier-expansion}, completing the proof of (2).

For part (3), we have already checked in part (1) that $V_v^-\phi \in S_{\lambda}(K_1(\mathfrak n \mathfrak p_v))$. Thus by Remark \ref{rmk:U-operator-notation} and \eqref{eqn:commutation} we get 
\begin{equation}\label{eqn:UvVv-}
(U_v V_v^- \phi)(g) = \sum_{a \in \mathcal O_v / \varpi_v \mathcal O_v} \phi\left(g \begin{smallpmatrix} \varpi_v & a \\ & 1 \end{smallpmatrix}\begin{smallpmatrix} 1\\ & \varpi_v \end{smallpmatrix} \right) = \sum_{a \in \mathcal O_v /\varpi_v \mathcal O_v} \phi(g\begin{smallpmatrix} 1 & a \\ & 1 \end{smallpmatrix}\begin{smallpmatrix} \varpi_v\\ & \varpi_v \end{smallpmatrix}).
\end{equation}
The $a$-th term in the sum \eqref{eqn:UvVv-} is equal to $(S_v\phi)(g\begin{smallpmatrix}1 & a \\ & 1 \end{smallpmatrix})$ which equals $(S_v\phi)(g)$ because $\begin{smallpmatrix}1 & a \\ & 1 \end{smallpmatrix} \in K_1(\mathfrak n)$ and $S_v\phi \in S_\lambda(K_1(\mathfrak n))$. Thus from \eqref{eqn:UvVv-} we get
\begin{equation*}
(U_vV_v^-\phi)(g) = \sum_{a \in \mathcal O_v/\varpi_v\mathcal O_v} (S_v\phi)(g) = (q_v S_v \phi)(g),
\end{equation*}
as promised.

Part (4) is clear. Indeed, the matrices involved in the definition of $T_{\mathfrak m}$ are the identity at $v$ because $\mathfrak p_v \nmid \mathfrak m$ (Remark \ref{rmk:well-posed-locally-calculated}), so they obviously commute with the action of $V_v^-$.
\end{proof}

For the rest of this subsection, we fix a cohomological cuspidal automorphic representation $\pi$ and a prime $p$. We write $\mathfrak n$ for the conductor of $\pi$ (not necessarily prime to $p$) and assume that $\pi$ has weight $\lambda$.
\begin{defn}\label{defn:refinement-intext}\label{defn:p-refined}
\leavevmode
\begin{enumerate}
\item $\pi$ is called $p$-refinable if for each place $v \mid p$, $\pi_v$ is either an unramified principal series representation or an unramified twist of the Steinberg.\footnote{There is a more general notion of $\pi$ being ``finite slope'' at $p$ (we will not use it). Specifically one could say that $\pi$ is finite slope at $p$ provided the smooth $\GL_2(F_v)$-representation $\pi_v$ has non-zero Jacquet module $(\pi_v)_{N_v}$ for all $v \mid p$ (\cite[Section 3.2]{Casselman-padicReductiveGroups}). It follows from Frobenius reciprocity that a $p$-refinement as in Definition \ref{defn:refinement-intext}(2) is the equivalent to an eigenspace for the torus action on $(\pi_v)_{N_v}$.}
\item If $\pi$ is $p$-refinable, then a $p$-refinement $\alpha$ for $\pi$ is the choice of $\alpha = (\alpha_v)_{v \mid p}$ of one of the following equivalent data.
\begin{enumerate}
\item For each $v$ where $\pi_v$ is an unramified principal series,  $\alpha_v$ is  a root of $X^2 - a_\pi(\mathfrak p_v)X + \omega_\pi(\varpi_v)q_v$, and for each $v$ where $\pi_v$ is Steinberg, $\alpha_v = a_{\pi}(\mathfrak p_v)$.
\item $\alpha_v = \chi_v(\varpi_v)$ where $\chi_v$ is the choice of smooth character $\chi_v : F_v^\times \rightarrow \mathbf C^\times$ such that $\chi_v \circ \Art_{F_v}^{-1}$ is a subrepresentation of $r(\pi_v)$.
\end{enumerate}
\item If $\alpha$ is a $p$-refinement of $\pi$, then the associated refined eigenform is
\begin{equation*}
\phi_{\pi,\alpha} := \prod_{\substack{v \mid p\\ \mathfrak p_v \nmid \mathfrak n}} (1 - \alpha_v^{-1}V_v^-)\cdot \phi_{\pi},
\end{equation*}
where $\phi_\pi\in S_\lambda(K_1(\mathfrak n))$ is newform associated with $\pi$ (see Proposition \ref{prop:new-vector}).
\end{enumerate}
\end{defn}
The equivalence in parts (a) and (b) of Definition \ref{defn:refinement-intext}(2) is the same unwinding of definitions that goes into \eqref{eqn:agreement-L-function}. As a matter of course, we will often abuse language and simply say things like ``Let $\alpha$ be a $p$-refinement for $\pi$...''\; by which we mean ``Assume that $\pi$ is $p$-refineable and that $\alpha$ is a $p$-refinement for $\pi$...'' (we already did this in part (3) of Definition \ref{defn:p-refined} for instance).

\begin{rmk}\label{rmk:steinberg-situation}
We stress that if $v \mid p$ and $\mathfrak p_v \mid \mathfrak n$ then $\pi_v$ is necessarily a Steinberg representation, so $\alpha_v = a_{\pi}(\mathfrak p_v)$ already, and $\mathfrak p_v^2 \nmid \mathfrak n$.
\end{rmk}

Recall that we write $\mathbf p = \prod_{v \mid p} \mathfrak p_v$ for the product of the primes above $p$ in $F$.
\begin{prop}\label{prop:refined-eigenform}
Let $\alpha$ be a $p$-refinement for $\pi$.
\begin{enumerate}
\item $\phi_{\pi,\alpha} \in S_{\lambda}(K_1(\mathfrak n \cap \mathbf p))$
\item $\phi_{\pi,\alpha}$ is a normalized eigenform which generates the representation $\pi$ under \eqref{eqn:auto-rep-v-form} and the Fourier coefficients/Hecke eigenvalues of $\phi_{\pi,\alpha}$ are given by 
\begin{equation*}
 a_{\phi_{\pi,\alpha}}( \mathfrak{p}_{v}^j) = \begin{cases} a_{\pi}(\mathfrak{p}_{v}^j) & \text{if $v \nmid p$;}\\ \alpha_{v}^j &\text{if $v \mid p$}.
 \end{cases}
 \end{equation*}
 In particular, $U_v(\phi_{\pi,\alpha}) = \alpha_v \phi_{\pi,\alpha}$ for each $v \mid p$.
\end{enumerate}
\end{prop}
\begin{proof} 
The fact that $\phi_{\pi,\alpha}$ lies in $S_{\lambda}(K_1(\mathfrak n \cap \mathbf p))$ and is normalized (thus non-zero!) follows from repeated uses of parts (1) and (2) in Lemma \ref{lem:refinement-lemma}. Since $\phi_{\pi,\alpha}$ is a $\GL_2(\mathbf A_{F,f})$-translate of $\phi_\pi$, it lies in $\pi$ under \eqref{eqn:auto-rep-v-form} and thus generates $\pi$ since $\pi$ is irreducible and $\phi_{\pi,\alpha}$ is non-zero. This proves parts (1) and the normalized portion of part (2).

It remains to check that $\phi_{\pi,\alpha}$ is an eigenform with the prescribed Hecke eigensystem. For that, it is enough to show that $\phi_{\pi,\alpha}$ is a $U_v$-eigenvector with eigenvalue $\alpha_v$ when $v \mid p$ and $\mathfrak p_v \nmid \mathfrak n$ (by Lemma \ref{lem:refinement-lemma}(4) and the end of Remark \ref{rmk:U-operator-notation}). So, fix $v \mid p$ and $\mathfrak p_v \nmid \mathfrak n$. Then, $\alpha_v$ is a root of $X^2 - a_\pi(\mathfrak p_v) X + \omega_\pi(\varpi_v)q_v$. Write $\beta_v = a_\pi(\mathfrak p_v) - \alpha_v = \alpha_v^{-1}\omega_\pi(\varpi_v)q_v$ for the other root. Then,
\begin{equation}\label{Uv-give-beta}
U_v(1-\alpha_v^{-1}V_v^-)\phi_\pi = U_v\phi_\pi - \beta_v\phi_\pi
\end{equation}
by Lemma \ref{lem:refinement-lemma}(3). Since the operator $T_v$ on $S_{\lambda}(K_1(\mathfrak n))$ decomposes into a sum $T_v = U_v + V_v^-$ (Remark \ref{rmk:Tv-formula}) we can continue \eqref{Uv-give-beta} and get
\begin{multline*}
U_v(1-\alpha_v^{-1}V_v^-)\phi_{\pi} = U_v \phi_\pi  - \beta_v \phi_\pi = (T_v - V_v^-)\phi_\pi - \beta_v \phi_\pi
= a_\pi(\mathfrak p_v)\phi_\pi - V_v^-\phi_\pi - \beta_v \phi_\pi\\ = (\alpha_v - V_v^-)\phi_\pi.
\end{multline*}
Thus, $(1-\alpha_v^{-1}V_v)\phi_\pi$ is a $U_v$-eigenvector with eigenvalue $\alpha_v$, completing the proof.
\end{proof}

%% file: shimuras-theorem.tex
\subsection{Archimedean Hecke operators}\label{subsec:archimedean}
We denote by $K$ any compact open subgroup of $\GL_2(\mathbf A_{F,f})$ and $N$ any $(\GL_2^+(F), K)$-bimodule with a left action of a monoid $\Delta \subset \GL_2(\mathbf A_{F,f})$ as in Section \ref{subsec:adelic-cochains}. Write $\pi_0(F_\infty^\times) = F_\infty^\times/(F_{\infty}^\times)^{\circ}$ for the component group of $F_\infty^\times$. There is a natural isomorphism  $\widehat{\pi_0(F_\infty^\times)} \simeq \{\pm 1\}^{\Sigma_F}$ where $\widehat{\pi_0(F_\infty^\times)}$ is the character group of $\pi_0(F_\infty^\times)$. So, we will often confuse signs $\epsilon \in \{\pm 1\}^{\Sigma_F}$ with the corresponding character of $\pi_0(F_\infty^\times)$. On the other hand, the function $\sgn : F_\infty^\times \rightarrow \{\pm 1\}^{\Sigma_F}$ defines a section $\pi_0(F_\infty^\times) \hookrightarrow F_\infty^\times$ of the natural quotient map. We fix this identification. By doing so, we may consider the  double coset operator $T_\zeta = [K_{\infty}^\circ \begin{smallpmatrix} \zeta \\ & 1\end{smallpmatrix} K_{\infty}^{\circ}]$ acting on the cohomology $H^\ast_c(Y_K,N)$ (trivially on $N$). Since $\begin{smallpmatrix} \zeta \\ & 1\end{smallpmatrix}$ normalizes $K_\infty^{\circ}$, this operator is just pullback under right multiplication by $\begin{smallpmatrix} \zeta \\ & 1\end{smallpmatrix}$. Since $\begin{smallpmatrix} \zeta \\ & 1 \end{smallpmatrix} \subset \GL_2(F_\infty)$, $T_\zeta$ obviously commutes with any Hecke action arising from elements of $\Delta \subset \GL_2(\mathbf A_{F,f})$. Further, if $\zeta,\zeta' \in \pi_0(F_\infty^\times)$, then $T_{\zeta}T_{\zeta'} = T_{\zeta\zeta'}$. In particular $T_{\zeta}$ commutes with $T_{\zeta'}$ and $T_\zeta^2 = 1$. Thus $T_{\zeta}$ has only eigenvalues $\pm 1$. If $2$ acts invertibly on $N$, then for each $\epsilon \in \{\pm 1\}^{\Sigma_F}$ we define 
\begin{equation*}
\pr^{\epsilon} = {1\over 2^d} \sum_{\zeta \in \pi_0(F_\infty^\times)} \epsilon(\zeta)T_{\zeta}
\end{equation*}
as an endomorphism of $H^\ast_c(Y_K,N)$. It is an idempotent projector mapping onto
\begin{equation*}
H^\ast_c(Y_K,N)^{\epsilon} = \{ v \in H^\ast_c(Y_K,N) \mid T_{\zeta}(v) = \epsilon(\zeta) v \text{ for all $\zeta \in \pi_0(F_\infty^\times)$}\}.
\end{equation*}

\subsection{The Eichler--Shimura construction}\label{subsec:eichler-shimura}
We now recall a transcendental construction associating a certain differential form to a holomorphic Hilbert modular form. Throughout this subsection we fix a cohomological weight $\lambda$.

Recall that $D_\infty=\mathfrak h^{\Sigma_F}$. Denote by $\Omega^d(D_\infty)$ the space of $\mathbf C$-valued smooth differential forms on $D_\infty$ (as a real manifold). For $z = (z_{\sigma})$ the canonical coordinate on $D_\infty$, we define $dz := \wedge_{\sigma} dz_{\sigma} \in \Omega^d(D_\infty)$. Here we have to choose an ordering of $\Sigma_F$, technically, and so we do that by insisting that $dz$ restricts to $dx_\infty/x_\infty$ along $(F_\infty^\times)^{\circ} \hookrightarrow D_\infty$ (see Section \ref{subsec:sym-spaces}). Before the next lemma, we remind ourselves that $\GL_2^+(F)$ acts on both $D_\infty$ and the algebraic local system $\mathscr L_\lambda(\mathbf C)$ defined in Section \ref{subsec:weights}.

\begin{lem}\label{eqn:polynomial-invariance}
If $z \in D_\infty$ and $P_z \in \mathscr L_\lambda(\mathbf C)$ is defined by $P_z = (z+X)^{\kappa}$, then
\begin{equation*}
P_{\gamma(z)} = (\det \gamma)^{\kappa-w\over 2}j(\gamma,z)^{\kappa} (\gamma \cdot P_z)
\end{equation*}
for all $\gamma \in \GL_2^+(F)$.
\end{lem}
\begin{proof}
Clear.
\end{proof}

Now denote by $\Omega^d(D_\infty, \mathscr L_{\lambda}(\mathbf C)) = \Omega^d(D_\infty)\otimes_{\mathbf C} \mathscr L_\lambda(\mathbf C)$ the smooth $\mathscr L_\lambda(\mathbf C)$-valued differential forms on $D_\infty$. If $K$ is a neat level, so that $Y_K$ is a smooth real manifold, then we denote by $\Omega^d(Y_K,\mathscr L_{\lambda}(\mathbf C))$ the smooth $\mathscr L_{\lambda}(\mathbf C)$-valued $d$-forms on $Y_K$.

\begin{prop}\label{prop:eichler-shimura}
Let $K \subset \GL_2(\mathbf A_{F,f})$ be a neat compact open subgroup and $\mathbf f \in S_\lambda^{\hol}(K)$.
\begin{enumerate}
\item The differential form
\begin{equation*}
\omega_{\mathbf f}(z, g_f) := \mathbf f(z,g_f) (z+X)^\kappa dz  \in \Omega^d(D_\infty,\mathscr L_{\lambda}(\mathbf C)) \otimes_{\mathbf{C}} \mathcal{C}^{\infty}(\GL_2(\mathbf{A}_{F,f}),\mathbf{C})
\end{equation*}
descends to $\Omega^d(Y_K,\mathscr L_\lambda(\mathbf C))$. In fact, it defines a canonical element $\omega_{\mathbf f} \in H^d_c(Y_K,\mathscr L_\lambda(\mathbf C))$.
\item If $g \in \GL_2(\mathbf A_{F,f})$, then $gKg^{-1}$ is also neat and if $r_g: Y_{gKg^{-1}} \rightarrow Y_{K}$ is right multiplication by $g$, then $r_g^{\ast} \omega_{\mathbf f} = \omega_{g\cdot \mathbf f}$.
\item If $K' \subset K$ is an open subgroup and $\pr: Y_{K'} \rightarrow Y_K$ is the projection map, then $\pr^{\ast}\omega_{\mathbf f} = \omega_{\mathbf f}$.
\end{enumerate}
\end{prop}
\begin{proof}
For (1), the descent of $\omega_{\mathbf f}$ to $Y_K$ follows from \eqref{eqn:holo-transform}, Lemma \ref{eqn:polynomial-invariance} and the chain rule. Next, $\omega_{\mathbf f}$ naturally defines an element of $H^d(Y_K,\mathscr L_\lambda(\mathbf C)$ and, in fact, by \cite[Proposition 2.1]{Hida-CriticalValues}, it lies in the cuspidal cohomology $H^d_{\mathrm{cusp}}(Y_K,\mathscr L_\lambda(\mathbf C)) \subseteq H^d(Y_K,\mathscr L_\lambda(\mathbf C)$. (See \cite[p.\ 465]{Hida-CriticalValues} or \cite[p.\ 61]{Harder} for definitions of $H_{\mathrm{cusp}}^{\bullet}$. The reader might like to compare with \cite[Proposition 6.6]{GetzGoresky} for this article's normalizations.) As explained in \cite[Section 5]{Hida-CriticalValues}, the cuspidal cohomology is itself a subspace of the compactly supported cohomology. This completes the proof of (1). The claims (2) and (3) of the proposition are formal.
\end{proof}
Now let $K$ be any compact open subgroup. We may choose a finite index normal subgroup $K' \subset K$ so that $K'$ is neat. Then we have a natural map $S_\lambda^{\hol}(K') \rightarrow H^d_c(Y_{K'},\mathscr L_\lambda(\mathbf C))$ given by $\mathbf f \mapsto \omega_{\mathbf f}$. By part (2) of Proposition \ref{prop:eichler-shimura}, it is equivariant with respect to the action of $K/K'$ on either side, so descends to well-defined map $S_\lambda^{\hol}(K) \rightarrow H^d_c(Y_K,\mathscr L_\lambda(\mathbf C))$. By part (3) of Proposition \ref{prop:eichler-shimura}, construction is independent of the choice of $K'$.
\begin{defn}
If $K \subset \GL_2(\mathbf A_{F,f})$ is a compact open subgroup, then the Eichler--Shimura map is the map $\ES: S_{\lambda}^{\hol}(K) \rightarrow H^d_c(Y_K,\mathscr L_\lambda(\mathbf C))$ defined above.
\end{defn}
We will sometimes also write $\ES$ for the map $\ES: S_\lambda(K) \rightarrow H^d_c(Y_K,\mathscr L_\lambda(\mathbf C))$ obtained by pre-composing with $\phi \mapsto \mathbf f_{\phi}$. This should cause no confusion. Note as well that parts (2) and (3) of Proposition \ref{prop:eichler-shimura} imply that $\ES$ is Hecke-equivariant. We now state a theorem proven by Hida and its apparent applications.

\begin{thm}\label{thm:normalize-by-periods}
Suppose that $\pi$ is a cohomological cuspidal automorphic representation of weight $\lambda$ and conductor $\mathfrak n$. Assume that $E \subset \mathbf C$ is any subfield containing $\mathbf Q(\pi)$ and the Galois closure of $F$ inside $\mathbf C$. Then, for each choice of sign $\epsilon \in \{\pm 1\}^{\Sigma_F}$,
\begin{equation*}
\dim_E H^d_c(Y_1(\mathfrak n), \mathscr L_{\lambda}(E))^{\epsilon}[\psi_\pi] = 1,
\end{equation*}
where $(-)[\psi_\pi]$ denotes subspace on which the Hecke operators acts through the character $\psi_\pi$, and so there exists an element $\Omega_{\pi}^{\epsilon} \in \mathbf C^\times$ such that 
\begin{equation*}
{\pr^{\epsilon}\ES({\mathbf f_{\pi})} \over \Omega_{\pi}^{\epsilon}} \in H^d_c(Y_1(\mathfrak n), \mathscr L_{\lambda}(E))^{\epsilon}[\psi_\pi].
\end{equation*}
\end{thm}
\begin{proof}
By (4.2) in \cite[Section 4]{Hida-CriticalValues}, for any choice of sign $\epsilon \in \{\pm 1\}^{\Sigma_F}$, the composition 
\begin{equation*}
\pr^\epsilon \circ \ES: S_{\lambda}^{\hol}(Y_1(\mathfrak n)) \to H^d_{c}(Y_1(\mathfrak n), \mathscr L_\lambda(\mathbf C))^{\epsilon}
\end{equation*}
is a Hecke-equivariant injection onto the $\epsilon$-component of the cuspidal cohomology $H^{d}_{\mathrm{cusp}}(Y_1(\mathfrak n),\mathscr L_\lambda(\mathbf C))^{\epsilon}$. Moreover, the cokernel of the inclusion $H^d_{\mathrm{cusp}}(Y_1(\mathfrak n),\mathscr L_\lambda(\mathbf C)) \subseteq H^d_c(Y_1(\mathfrak n),\mathscr L_\lambda(\mathbf C))$ supports only Eisenstein Hecke eigensystems. Indeed, by \cite[Section 2]{Harder}, up to Eisenstein eigensystems there is no distinction between $H^d_c(Y_1(\mathfrak n),\mathscr L_\lambda(\mathbf C))$ and its image $H^d_!(Y_1(\mathfrak n),\mathscr L_\lambda(\mathbf C))$ (sometimes called interior cohomology) in $H^d(Y_1(\mathfrak n),\mathscr L_\lambda(\mathbf C))$. Then, the difference between the interior cohomology and the cuspidal cohomology is shown to be Eisenstein in \cite[Section 3.2]{Harder} (especially p.\ 61 of {\em loc.\ cit.}, where $H^{\bullet}_{!}$ will be written $\widetilde H^{\cdot}$). And so, in fact, $\pr^{\epsilon}\circ \ES$ induces an isomorphism
\begin{equation}\label{eqn:iso-ES}
S_\lambda^{\hol}(Y_1(\mathfrak n))[\psi_{\pi}] \cong H^d_c(Y_1(\mathfrak n),\mathscr L_\lambda(\mathbf C))^{\epsilon}[\psi_{\pi}].
\end{equation}
By the existence and the uniqueness of the newform associated with $\pi$ (see Proposition \ref{prop:new-vector}), either side of \eqref{eqn:iso-ES} is thus one-dimensional. Since $\psi_{\pi}$ takes values in $E$, for $E$ as in the theorem statement, this completes the proof.
\end{proof}

\begin{rmk}
The choice of $\Omega_{\pi}^{\epsilon}$ in Theorem \ref{thm:normalize-by-periods} is unique up to an element of $E^\times$ (for $E$ as in in the theorem statement). We do not discuss further how to possibly specify these periods.
\end{rmk}

\subsection{Twisting}\label{subsec:twisting-archimedean}
In this subsection we discuss twisting by finite order Hecke characters. We will do this carefully since we will need a less familiar $p$-adic version of these ideas in Section \ref{subsec:padic-twisting}. Our treatment here is inspired by \cite[Sections 5.10 and 9.4]{GetzGoresky}. Throughout, we fix a cohomological weight $\lambda$ and an integral ideal $\mathfrak n$. We will also let $E$ denote a variable subfield of $\mathbf C$ containing the Galois closure of $F$.

To start, if $t \in \mathbf A_{F,f}$ then write $u_t := \begin{smallpmatrix} 1 & t\\ & 1 \end{smallpmatrix}$. For an integral ideal $\mathfrak m$, we write
\begin{equation*}
K_{11}(\mathfrak m) = \{\begin{smallpmatrix} a & b \\ c& d \end{smallpmatrix} \in \GL_2(\widehat{\mathcal O}_F)\mid a,d \equiv 1 \bmod \mathfrak m \widehat{\mathcal O}_F \text{ and } c \equiv 0 \bmod \mathfrak m\widehat{\mathcal O}_F\}.
\end{equation*}
Now let $\mathfrak f$ be an integral ideal and $t \in \mathfrak f^{-1}\widehat{\mathcal O}_F$. Then, $K_{11}(\mathfrak n \mathfrak f^2)_t:=u_t^{-1}K_{11}(\mathfrak n \mathfrak f^2)u_t \subset K_1(\mathfrak n)$. In particular if $\phi \in S_\lambda(K_1(\mathfrak n))$, then $\phi_t(g) := \phi(gu_t)$ is in $S_\lambda(K_{11}(\mathfrak n \mathfrak f^2))$. We also have a diagram of Hilbert modular varieties
\begin{equation*}
\xymatrix{
& Y_{11}(\mathfrak n \mathfrak f^2) \ar[dl]_-{\pr} \ar@{-->}[drr]_-{v_t} \ar[r]^-{r_{u_t}} & Y_{K_{11}(\mathfrak n \mathfrak f^2)_t} \ar[dr]^-{\pr} \\
Y_{1}(\mathfrak n \mathfrak f^2) & & & Y_1(\mathfrak n)
}
\end{equation*}
where $v_t$ is defined to be the composition as indicated. Since $u_t \in \GL_2(\mathbf A_{F,f})$, the identity map defines an isomorphism $v_t^{\ast}{\mathscr L}_{\lambda}(E) \simeq \mathscr L_{\lambda}(E)$ of local systems on $Y_{11}(\mathfrak n \mathfrak f^2)$.

\begin{lem}\label{lem:part-twist}
For each $t \in \mathfrak f^{-1}\widehat{\mathcal O}_F$, the diagram
\begin{equation*}
\xymatrix{
S_{\lambda}(K_1(\mathfrak n)) \ar[r]^-{\ES} \ar[d]_-{\phi\mapsto\phi_t} & H^d_c(Y_1(\mathfrak n), {\mathscr L}_{\lambda}(\mathbf C)) \ar[d]^-{v_t^{\ast}}\\
S_{\lambda}(K_{11}(\mathfrak n\mathfrak f^2))  \ar[r]_-{\ES} & H^d_c(Y_{11}(\mathfrak n\mathfrak f^2), {\mathscr L}_{\lambda}(\mathbf C)).
}
\end{equation*}
is commutative.
\end{lem}
\begin{proof}
See parts (2) and (3) of Proposition \ref{prop:eichler-shimura}.
\end{proof}
Now consider a finite order Hecke character $\theta$ and let $\mathfrak f$ be an ideal dividing the conductor of $\theta$.\footnote{Recall this means that $\theta(1 + \mathfrak f \widehat{\mathcal O}_F) = \{1\}$. } Write $\Upsilon_{\mathfrak f} = \mathfrak f^{-1} \widehat{\mathcal O}_F/\widehat{\mathcal O}_F$ and $\Upsilon_{\mathfrak f}^\times$ for cosets represented by $x/f$ with $f \in \mathfrak f$ and $x \in \widehat{\mathcal O}_F^\times$. We naturally view $\theta$ as a character on $\Upsilon_{\mathfrak f}^\times$.  If $t \in \Upsilon_{\mathfrak f}^\times$ write $t_0 \in \widehat{\mathcal O}_F$ for a lift of $t$ which is zero outside of $v \mid \mathfrak f$. Then, for $\phi \in S_{\lambda}(K_1(\mathfrak n))$ then we define $\tw_{\theta}(\phi)$ by
\begin{equation*}
\tw_{\theta}(\phi)(g) = \theta(\det g)\sum_{t \in \Upsilon_{\mathfrak f}^\times} \theta(t)\phi\left(gu_{t_0}\right).
\end{equation*}
By \cite[Proposition 5.11]{GetzGoresky}, this defines a linear map $\tw_{\theta}:  S_\lambda(K_1(\mathfrak{n})) \to S_\lambda(K_1(\mathfrak{nf}^2))$. 

On the other hand, suppose $E$ contains the values of $\theta$. Then, $\theta_{\det}(g) := \theta(\det g)$ defines an element of $H^0(Y_{11}(\mathfrak n \mathfrak f^2), E)$ (compare with Remark \ref{rmk:adelic-norm-twisting} below). So, cup product with $\theta_{\det}$ defines an endomorphism of $H^{\ast}_c(Y_{11}(\mathfrak{nf}^2), \mathscr L_\lambda(E))$ and we get a natural map
\begin{equation*}
\tw_{\theta} : H^{\ast}_c(Y_1(\mathfrak n), \mathscr L_\lambda(E)) \rightarrow H^{\ast}_c(Y_{11}(\mathfrak n \mathfrak f^2), {\mathscr L}_{\lambda}(E))
\end{equation*}
given by
\begin{equation}\label{eqn:twistin-cohomology}
\tw_\theta = \theta_{\det} \cup \sum_{t \in \Upsilon_{\mathfrak f}^\times} \theta(t)v_{t_0}^{\ast}.
\end{equation}

We claim that \eqref{eqn:twistin-cohomology} descends to the cohomology at level $K_1(\mathfrak{n f}^2)$. To see that, note that $Y_{11}(\mathfrak n\mathfrak f^2) \rightarrow  Y_1(\mathfrak n \mathfrak f^2)$ is a Galois cover with Galois group $(\widehat{\mathcal O}_F/\mathfrak n \mathfrak f^2\widehat{\mathcal O}_F)^\times$. Specifically, if $a \in \widehat{\mathcal O}_F^\times$ then $\eta_a:=\begin{smallpmatrix} a \\ & 1 \end{smallpmatrix}$ normalizes $K_{11}(\mathfrak n \mathfrak f^2)$ and so right multiplication by $\eta_a$ defines an automorphism of $Y_{11}(\mathfrak n \mathfrak f^2)$ over $Y_{1}(\mathfrak n \mathfrak f^2)$ which depends only on the image of $a$ inside $(\widehat{\mathcal O}_F/\mathfrak n \mathfrak f^2\widehat{\mathcal O}_F)^\times$. Since $(\widehat{\mathcal O}_F/\mathfrak n \mathfrak f^2\widehat{\mathcal O}_F)^\times$ is a finite group, and $E$ has characteristic zero, we may identify $H^{\ast}_c(Y_{1}(\mathfrak n\mathfrak f^2),{\mathscr L}_{\lambda}(E))$ as the $(\widehat{\mathcal O}_F/\mathfrak n \mathfrak f^2\widehat{\mathcal O}_F)^\times$-invariants in $H^{\ast}_c(Y_{11}(\mathfrak n \mathfrak f^2), {\mathscr L}_{\lambda}(E))$ (with $a$ acting via pullback $\eta_a^{\ast}$).
\begin{lem}\label{lemma:twisting-compatibilities}
$\tw_{\theta}\left(H^\ast_c(Y_1(\mathfrak n)),{\mathscr L}_{\lambda}(E))\right) \subset H^{\ast}_c(Y_{1}(\mathfrak n\mathfrak f^2),{\mathscr L}_{\lambda}(E))$ and the diagram
\begin{equation*}
\xymatrix{
S_{\lambda}(K_1(\mathfrak n)) \ar[r]^-{\ES} \ar[d]_-{\tw_{\theta}} & H^d_c(Y_1(\mathfrak n), {\mathscr L}_{\lambda}(\mathbf C)) \ar[d]^-{\tw_{\theta}}\\
S_{\lambda}(K_{1}(\mathfrak n\mathfrak f^2))  \ar[r]_-{\ES} & H^d_c(Y_{1}(\mathfrak n\mathfrak f^2), {\mathscr L}_{\lambda}(\mathbf C))
}
\end{equation*}
is commutative.
\end{lem}
\begin{proof}
We need to show that $\eta_a^{\ast}\tw_{\theta} = \tw_{\theta}$ for each $a \in \widehat{\mathcal O}_F^\times$. If $t \in \mathbf A_{F,f}$ then
\begin{equation*}
\eta_a u_t = \begin{smallpmatrix} a & at \\ & 1\end{smallpmatrix}= \begin{smallpmatrix} 1 & at \\ & 1\end{smallpmatrix}\begin{smallpmatrix} a &  \\ & 1\end{smallpmatrix} \in u_{at}K_1{(\mathfrak n)},
\end{equation*}
so $\eta_a^{\ast}v_t^{\ast} = v_{at}^{\ast}$. Moreover, $\eta_a^{\ast}\theta_{\det} = \theta(a)\theta_{\det}$. So, since $at_0 = (at)_0$ for $t \in \Upsilon_{\mathfrak f}^\times$, we can finally compute:
\begin{equation*}
\eta_a^{\ast}\tw_{\theta} = \eta_a^{\ast}\theta_{\det} \cup \sum_{t \in \Upsilon_{\mathfrak f}^\times} \theta(t)\eta_a^{\ast}v_{t_0}^{\ast} = \theta(a)\theta_{\det} \cup \sum_{t \in \Upsilon_{\mathfrak f}^\times} \theta(t)v_{(at)_{0}}^{\ast} = \theta_{\det} \cup \sum_{t \in \Upsilon_{\mathfrak f}^\times} \theta(at)v_{(at)_{0}}^{\ast} = \tw_{\theta}.
\end{equation*}
The commutativity of $\tw_{\theta}$ with $\ES$ follows from Lemma \ref{lem:part-twist}.
\end{proof}
\begin{rmk}\label{rmk:adelic-norm-twisting}
One may also consider twisting by characters of the form $|\cdot|_{\mathbf A_F}^n \theta$ where $\theta$ is finite order and $n$ is an integer. Namely, there is a suitable modification of $\theta_{\det}$ (compare with Definition \ref{prop:eval-3}) so that the cup product \eqref{eqn:twistin-cohomology} induces a linear map
\begin{equation}\label{eqn:allow-adelic-twisting}
\tw_{|\cdot|_{\mathbf A_F}^n \theta} : H^{\ast}_c(Y_1(\mathfrak n), \mathscr L_{\kappa,w}(E)) \rightarrow H^{\ast}_c(Y_1(\mathfrak n \mathfrak f^2), \mathscr L_{\kappa,w-2n}(E)).
\end{equation}
We omit an explicit description, but in Section \ref{subsec:padic-twisting} we will explain the same idea.
\end{rmk}

We note for later (Lemma \ref{lem:allowed-to-project}) the interaction between twisting and Archimedean Hecke operators.
\begin{prop}\label{prop:twisting-Tzetas}
If $\zeta \in \pi_0(F_\infty^\times)$, then $T_{\zeta} \circ \tw_{\theta} = \theta(\zeta) \tw_{\theta} \circ T_{\zeta}$.\footnote{The $\tw_{\theta}$ here means the one on cohomology. It must be, since the $T_{\zeta}$ are not defined on automorphic forms.}
\end{prop}
\begin{proof}
Recall that $T_{\zeta}$ is pullback along right-multiplication by $\begin{smallpmatrix} \zeta \\ & 1 \end{smallpmatrix}$ on $Y_K$ (for any $K$). In the definition \eqref{eqn:twistin-cohomology} of $\tw_{\theta}$, the pullbacks $v_{t_0}^{\ast}$ are pullbacks along multiplication by elements of $\GL_2(\mathbf A_{F,f})$, so they obviously commute with $T_{\zeta}$. Since pullbacks commute over cup products, the result is a straightforward check after noticing that $T_{\zeta}\circ\theta_{\det} = \theta(\zeta)\theta_{\det}$.
\end{proof}

We continue to assume that $\theta$ is a finite order Hecke character as above. We define a Gauss sum
\begin{equation*}
G(\theta^{-1}) = \sum_{t \in \Upsilon_{\mathfrak f}^\times} \theta(\delta^{-1})\theta(t)e_F(\delta^{-1}t)
\end{equation*}
where $\delta_{F/\mathbf Q} \in \mathbf A_{F,f}^\times$ is any choice of finite idele with $[\delta_{F/\mathbf Q}] = \mathcal D_{F/\mathbf Q}$ (notations as in Section \ref{subsec:hecke-operators}). We note now that if $\theta$ has conductor exactly $\mathfrak f$, then
\begin{equation}\label{eqn:gauss-sum-relation}
G(\theta^{-1}) = {\sgn(\theta_\infty) N_{F/\mathbf Q}(\mathfrak f) \over G(\theta)},
\end{equation}
where $N_{F/\mathbf Q}(-)$ is the absolute norm. (This is a classical calculation.)

By \cite[Proposition 5.11]{GetzGoresky}, if $\phi \in S_{\lambda}(K_1(\mathfrak n))$ then $G(\theta^{-1})^{-1}\tw_{\theta}(\phi) =: \phi \otimes \theta$ is what one usually thinks of as the ``twist'':\ the Fourier coefficients of $\phi\otimes \theta$ are given by 
\begin{equation}\label{eqn:twist-coefficients}
a_{\phi\otimes \theta}(\mathfrak m)  = \begin{cases}
\theta(\mathfrak m) a_{\phi}(\mathfrak m) & \text{if $(\mathfrak m,\mathfrak f) = 1$;}\\
0 & \text{otherwise}.
\end{cases}
\end{equation}
Here we have descended $\theta$ to a character of the prime-to-$\mathfrak f$ part of the ideal class group. It follows from \eqref{eqn:twist-coefficients} that if $\phi$ is a normalized eigenform of level $\mathfrak n$ then $\phi\otimes \theta$ is a normalized eigenform of level $\mathfrak n \mathfrak f^2$.

We end with the following synopsis of the relationship between twisting and $p$-refinements.

\begin{prop}\label{prop:L-values-agree}
Let $p$ be a prime. Suppose that $\pi$ is a cohomological cuspidal automorphic representation of conductor $\mathfrak n$, $\alpha$ is a $p$-refinement of $\pi$, and $\theta$ is a finite order Hecke character with conductor of the form $\mathfrak f = \prod_{v \mid p} \mathfrak p_v^{f_v}$ with $f_v \geq 0$. If $v \mid p$ and $\pi_v$ is a principal series representation, then write $\beta_v = a_\pi(\mathfrak p_v) -\alpha_v$.
\begin{enumerate}
\item $\phi_{\pi}\otimes \theta$ and $\phi_{\pi,\alpha}\otimes \theta$   are normalized eigenforms of levels $\mathfrak n \mathfrak f^2$ and $(\mathfrak n \cap \mathbf p) \mathfrak f^2$, respectively.
\item If $v \nmid p$ or $f_v > 0$ or $\mathfrak p_v \mid \mathfrak n$ then $L_v(\phi_{\pi,\alpha}\otimes \theta,s) = L_v(\phi_{\pi}\otimes \theta,s)$.
\item If $v \mid p$ and $f_v = 0$ and $\mathfrak p_v \nmid \mathfrak n$ then 
\begin{equation*}
L_v(\phi_{\pi,\alpha}\otimes\theta,s) = (1 - \theta_v(\varpi_v)\beta_vq_v^{-s})L_v(\phi_{\pi}\otimes \theta,s).
\end{equation*}
\item $L_v(\phi_{\pi}\otimes\theta,s) = L_v(\pi\otimes\theta,s)$ for all $v$.\footnote{Here, $\pi\otimes \theta$ is the automorphic representation on which the action of $\GL_2(\mathbf A_F)$ on $\pi$ is twisted by $\theta(\det g)$.}\label{prop-part:factors-agree}
\item Writing $\mathbf{M}(-,s)$ for the Mellin transform as in \S 3.3, we have \begin{equation*}
\mathbf{M}(\phi_{\pi,\alpha} \otimes \theta, s) = \bigl[\prod_{\substack{v\mid p\\ \mathfrak p_v \nmid \mathfrak{nf}}} (1- \beta_v\theta_v(\varpi_v) q_v^{-(s+1)})\bigr] \Delta_{F}^{s+1} \Lambda(\pi \otimes \theta , s+1).
\end{equation*}
\end{enumerate}
\end{prop}
\begin{proof}
As mentioned above, twisting by $\theta$ preserves the property of being a normalized eigenform. Since $\phi_\pi$ is a normalized eigenform, and $\phi_{\pi,\alpha}$ is one by Proposition \ref{prop:refined-eigenform}, part (1) is proven.

We will prove (2) and (3) at the same time. First note that since $\mathfrak f$ is divisible only by primes above $p$, the level of $\phi_\pi \otimes \theta$ and the level of $\phi_{\pi,\alpha}\otimes \theta$ are the same away from $p$. Note as well that the central characters are the same:\ they are both $\omega_{\pi}\theta^2$. Thus we see that (2) is true in the case $v \nmid p$ by Proposition \ref{prop:refined-eigenform} and \eqref{eqn:twist-coefficients}.

Now we consider $v \mid p$. If $f_v > 0$ or $\mathfrak p_v \mid \mathfrak n$ then $\mathfrak p_v$ divides the level of both $\phi_{\pi,\alpha}\otimes\theta$ and $\phi_{\pi}\otimes \theta$, and the $v$-th Fourier coefficient of either eigenform is the same:\ if $f_v > 0$ then the coefficients are both zero, and if $f_v = 0$ but $\mathfrak p_v \mid \mathfrak n$ then both coefficients are $\theta(\varpi_v)\alpha_v = \theta(\varpi_v)a_{\pi}(\mathfrak p_v)$ (compare with Remark \ref{rmk:steinberg-situation}). This completes the proof of (2).

Finally suppose that $v \mid p$ and $f_v=0$ and $\mathfrak p_v \nmid \mathfrak n$. Since $\mathfrak p_v$ is then co-prime to the level of $\phi_{\pi}\otimes \theta$, we have from \eqref{eqn:twist-coefficients} that
\begin{small}
\begin{equation*}
L_v(\phi_{\pi}\otimes \theta,s)\\ = {1 \over 1 - \theta(\varpi_v)a_{{\pi}}(\mathfrak p_v)q_v^{-s} + \omega_{\pi}\theta^2(\varpi_v)q_v^{1-2s}}
 = {1 \over (1-\theta(\varpi_v)\alpha_v q_v^{-s})(1-\theta(\varpi_v)\beta_v q_v^{-s})}.
\end{equation*}
\end{small}
On the other hand, by Proposition \ref{prop:refined-eigenform} and  \eqref{eqn:twist-coefficients} we have $a_{\phi_{\pi,\alpha}\otimes \theta}(\mathfrak p_v) = \theta(\varpi_v)\alpha_v$. Since $\phi_{\pi,\alpha}\otimes \theta$ has level divisible by $\mathfrak p_v$, its local $L$-factor is 
\begin{equation*}
L_v(\phi_{\pi,\alpha}\otimes\theta,s) = {1 \over 1 - a_{\phi_{\pi,\alpha}\otimes\theta}(\mathfrak p_v)q_v^{-s}} = {1\over 1 - \theta(\varpi_v)\alpha_vq_v^{-s}}.
\end{equation*}
Comparing the previous two displayed equations completes the proof of (3).

Point (4) is obvious if $f_v = 0$. Otherwise $\theta$ is ramified at $v$ and in particular $v \mid p$. We claim that $L_v(\pi\otimes\theta,s) = 1 = L_v(\phi_{\pi}\otimes\theta,s)$. Since $\pi$ is $p$-refineable and $v \mid p$, the first equality follows because twisting an unramified principal series or an unramified twist of the Steinberg by a ramified character trivializes the local $L$-factor. For the second equality, note that if $\theta_v$ is ramified then $\mathfrak p_v$ divides the level of $\phi_{\pi}\otimes \theta$ and $a_{\phi_{\pi}\otimes \theta}(\mathfrak p_v) = 0$ by \eqref{eqn:twist-coefficients}. The second inequality now follows from \eqref{eqn:local-factors}.

Finally, (5) follows from the previous parts and  Proposition \ref{prop:integral-representation}
\end{proof}

\subsection{Evaluation classes}\label{subsect:eval-classes}
In this subsection, $E$ denotes a subfield of $\mathbf C$ that contains the Galois closure of $F$. We will also fix a cohomological weight $\lambda = (\kappa,w)$. Our goal is to define an evaluation class in homology which is used to detect $L$-values.

Recall $\mathscr L_\lambda(E)$ is equipped with a left action of $\GL_2(F)$. We write $\mathscr L_\lambda(E)^{\vee}$ for $E$-linear dual space of $\mathscr L_\lambda(E)$ with its canonical right action of $\GL_2(F)$
\begin{equation*}
\mu\big|_{g}(P) = \mu(g\cdot P)
\end{equation*}
if $\mu \in \mathscr L_\lambda(E)^{\vee}$, $g \in \GL_2(F)$ and $P \in \mathscr L_\lambda(E)$.

\begin{lem}\label{lem:eval-1}
If $x \in F^\times$ and $P \in \mathscr L_\lambda(E)$, then $\left(\begin{smallpmatrix} x & \\ & 1 \end{smallpmatrix} \cdot P\right)(X) = x^{w+\kappa\over 2} P\left({X\over x}\right)$. 
\end{lem}
\begin{proof}
See definition \eqref{eqn:ell-lambda}.
\end{proof}
We now make two definitions.
\begin{defn}
An integer $m$ is critical with respect to $\lambda$ if 
\begin{equation*}
{w -\kappa_\sigma\over 2} \leq m \leq {w + \kappa_\sigma \over 2}
\end{equation*}
for all $\sigma \in \Sigma_F$.
\end{defn}

\begin{defn}\label{defn:deltam-star}
Let $m$ be critical with respect to $\lambda$. Then, $\delta_m^{\star} \in \mathscr L_\lambda(E)^\vee$ is defined by
\begin{equation*}
\delta_m^{\star}(X^j) = \begin{cases}
{\kappa \choose j}^{-1} & \text{if $j = {\kappa + w \over 2} - m$},\\
0 & \text{otherwise.}
\end{cases}
\end{equation*}
\end{defn}

\begin{lem}\label{lem:eval-2}
If $x \in F^\times$, then $\delta_m^{\star}\big|_{\begin{smallpmatrix} x \\ & 1 \end{smallpmatrix}} = x^m \delta_m^{\star}$.
\end{lem}
\begin{proof}
By Lemma \ref{lem:eval-1}, if $ 0\leq j \leq \kappa$ then 
\begin{equation}\label{eqn:some-trivial-shit}
\delta_m^{\star}\big|_{\begin{smallpmatrix} x \\ & 1 \end{smallpmatrix}}(X^j) = x^{{\kappa+w\over 2} - j}\delta_m^{\star}(X^j).
\end{equation}
If $j \neq {\kappa + w \over 2} - m$, then both $x^m\delta_m^{\star}(X^j)$ and the right-hand side of \eqref{eqn:some-trivial-shit} vanish. And if $j = {\kappa + w \over 2} - m$ then clearly $x^m \delta_m^{\star}(X^j)$ is equal to the right-hand side of \eqref{eqn:some-trivial-shit}. The result follows.
\end{proof}
Recall the definition \eqref{eqn:shin-cone} of the Shintani cone $\mathrm C_\infty = F^\times\backslash \mathbf A_F^\times / \widehat{\mathcal O}_F^\times$.  Above we took a right action of $\GL_2(F)$ on $\mathscr L_\lambda(E)^{\vee}$ but now we restrict this to the {\em left} action of $F^\times$ where $x \in F^\times$ acts by $x\cdot \mu = \mu|_{\begin{smallpmatrix} x^{-1} & \\ & 1 \end{smallpmatrix}}$. Using this action, we define a local system
\begin{equation*}
\mathrm t^{\ast}{\mathscr L}_\lambda(E)^{\vee} = F^\times \backslash \mathbf A_F^\times \times \mathscr L_\lambda(E)^\vee / \widehat{\mathcal O}_F^\times \twoheadrightarrow \mathrm C_\infty.
\end{equation*}

\begin{defn-prop}\label{prop:eval-3}
If $m$ is critical with respect to $\lambda$, then the formula
\begin{equation*}
\delta_m(x) := \left(\sgn(x_\infty)|x_f|_{\mathbf A_F}\right)^m\delta_m^{\star},\;\; x \in \mathbf{A}_{F}^{\times}
\end{equation*}
defines an element of $H^0(\mathrm C_\infty, \mathrm t^{\ast}{\mathscr L}_\lambda(E)^\vee)$.
\end{defn-prop}
\begin{proof}
Since $\delta_m(x)$ is clearly constant on the connected component $(F_\infty^\times)^{\circ}$, what we need to show is that if $\xi \in F^\times$, $x \in \mathbf A_{F}^\times$ and $u \in \widehat{\mathcal O}_F^\times$ then
\begin{equation}\label{eqn:eval-3}
\delta_m(\xi x u) = \delta_m(x)\big|_{\begin{smallpmatrix} \xi^{-1} & \\ & 1 \end{smallpmatrix}}.
\end{equation}
Since elements of $\widehat{\mathcal O}_F^\times$ have trivial adelic norm and no infinite component, we see  that $\delta_m$ is right $\widehat{\mathcal O}_F^\times$-invariant. On the other hand, the product formula implies that
\begin{equation*}
\delta_m(\xi x) = (\sgn(\xi_\infty)|\xi_f|_{\mathbf A_F})^m \delta_m(x) = \xi_\infty^{-m} \delta_m(x),
\end{equation*}
if $\xi \in F^\times$. But this is exactly the right-hand side of \eqref{eqn:eval-3} by Lemma \ref{lem:eval-2}.
\end{proof}

Now suppose that $K \subset \GL_2(\mathbf A_{F,f})$ is a $\mathrm t$-good subgroup (Definition \ref{defn:good-level}). As in \eqref{eqn:shintani-torus} we consider the proper embedding $\mathrm t: \mathrm C_\infty \rightarrow Y_K$ given by $\mathrm t(x) = \begin{smallpmatrix} x \\ & 1 \end{smallpmatrix}$. The local system ${\mathscr L}_\lambda(E)^\vee$ on $Y_K$ defined by the left-action of $\GL_2(F)$ on $\mathscr L_\lambda(E)^\vee$ pulls back exactly to the local system $\mathrm t^{\ast}{\mathscr L}_{\lambda}(E)^\vee$ on $\mathrm C_\infty$ which we just considered.\footnote{We  consider left actions in order to define the local systems on $Y_K$ because the quotient by $\GL_2^+(F)$ is on the left.} Since $\mathrm t$ is proper, we get a pushforward map
\begin{equation*}
\mathrm t_{\ast}: H_{\ast}^{\BM}(\mathrm C_\infty, \mathrm t^{\ast} {\mathscr L}_\lambda(E)^{\vee}) \rightarrow H_{\ast}^{\BM}(Y_K, {\mathscr L}_\lambda(E)^\vee)
\end{equation*}
on the level of Borel--Moore homology. Furthermore, we also have a Poincar\'e duality map (see \eqref{eqn:poincare-duality})
\begin{equation*}
\PD: H^0(\mathrm C_\infty, \mathrm t^{\ast} {\mathscr L}_\lambda(E)^{\vee}) \rightarrow H_d^{\BM}(\mathrm C_\infty, \mathrm t^{\ast} {\mathscr L}_\lambda(E)^{\vee})
\end{equation*}
given by cap product with a Borel--Moore fundamental class $[\mathrm C_\infty]$.

\begin{defn}\label{defn:infinity-class}
If $m$ is critical with respect to $\lambda$, and $K$ is a $\mathrm t$-good subgroup, then we define
\begin{equation*}
\cl_\infty(m) = \mathrm t_\ast(\PD(\delta_m)) \in H_d^{\BM}(Y_K, {\mathscr L}_\lambda(E)^\vee).
\end{equation*}
We call $\cl_\infty(m)$ an Archimedean evaluation class.
\end{defn}
Note that strictly speaking we should write something like $\cl_\infty^K(m)$ to indicate the dependence on $K$. But, the local systems $\mathscr L_\lambda(E)^\vee$ live at all levels simultaneously and the next lemma shows we do not need this extra notation.
\begin{lem}\label{lemma:infinity-class}
If $K' \subset K$ are two compact open subgroups of $\GL_2(\mathbf A_{F,f})$ and $K'$ is $\mathrm t$-good, then $\pr^{\ast}(\cl_\infty^K(m)) = \cl_\infty^{K'}(m)$.
\end{lem}
\begin{proof}
The two possible embeddings $\mathrm t$ commute with the projection $Y_{K'} \rightarrow Y_K$.
\end{proof}
We end by recording how Archimedean Hecke operators act on the Archimedean evaluation classes.

\begin{prop}\label{prop:pushforward-eval-class}
If $\zeta \in \pi_0(F_\infty^\times)$ then $T_{\zeta} \cl_\infty(m) = \zeta^{-m}\cl_\infty(m)$.
\end{prop}
\begin{proof}
Write $\rho_\zeta: Y_K \rightarrow Y_K$ for right-multiplication by $\begin{smallpmatrix} \zeta \\ & 1 \end{smallpmatrix}$, so $T_\zeta$ acting on homology is the pushfoward $(\rho_\zeta)_{\ast}$. Also write $r_\zeta: \mathrm C_\infty \rightarrow \mathrm C_\infty$ for right multiplication by $\zeta$ so that $\rho_\zeta \circ \mathrm t = \mathrm t \circ r_\zeta$. Since $\zeta = \sgn(\zeta_\infty)$, it follows from the definition of $\delta_m$ that $r_\zeta^{\ast}\delta_m = \zeta^m \delta_m$. Using this, we get
\begin{equation*}
T_\zeta \cl_\infty(m) = (\rho_\zeta)_{\ast} \mathrm t_{\ast} \PD (\delta_m)
= \mathrm t_{\ast} (r_{\zeta})_{\ast} \PD(\zeta^{-m} r_\zeta^{\ast}\delta_m)
= \zeta^{-m} \mathrm t_{\ast} (r_\zeta)_{\ast}\PD(r_{\zeta}^{\ast}\delta_m).
\end{equation*}
The proposition now follows from Proposition \ref{prop:oritentation-preserving} and \eqref{eqn:naturality-PD}.
\end{proof}

\subsection{Special values of $L$-functions}\label{subsec:special-values}
Throughout this subsection we will use $\lambda$ to denote a cohomological weight, $m$ an integer that is critical with respect to $\lambda$, and $\mathfrak n$ an integral ideal. Further, we will use $\langle -, - \rangle$ to denote the natural pairing (see Section \ref{subsec:topology})
\begin{equation*}
\langle -, - \rangle: H^d_c(Y_K,{\mathscr L}_\lambda(E)) \otimes_{E} H_d^{\BM}(Y_K, {\mathscr L}_\lambda(E)^\vee) \rightarrow E.
\end{equation*}
We combine our previous results to compute pairing between the image of the Eichler--Shimura map and Archimedean evaluation classes.
\begin{thm}\label{thm:cap-product}
If $\phi \in S_{\lambda}(K_1(\mathfrak n))$, then $\langle \ES(\phi),\cl_\infty(m)\rangle = i^{1 + m + {\kappa - w\over 2}}\mathbf M(\phi,m)$.
\end{thm}
\begin{rmk}
Note that since $\kappa = (\kappa_\sigma)$ is a $\Sigma_F$-tuple, $i^{1+m+{\kappa-w\over 2}}$ means the product $\prod_\sigma i^{1+m+{\kappa_\sigma-w \over 2}}$.
\end{rmk}
\begin{proof}[Proof of Theorem \ref{thm:cap-product}]
By Proposition \ref{prop:eichler-shimura}(3), Lemma \ref{lemma:infinity-class}, Proposition \ref{prop:neatness}, and because $\mathbf M(\phi,s)$ only depends on the underlying automorphic form $\phi$, we can and will assume that $K_1(\mathfrak n)$ is a neat level subgroup. Then, we will write $\mathbf f = \mathbf f_{\phi} \in S_{\lambda}^{\hol}(K_1(\mathfrak n))$ for the holomorphic Hilbert modular form corresponding to $\phi$, and $\omega_{\mathbf f} = \ES(\mathbf f)$ for the bona fide differential form on $Y_{1}(\mathfrak n)$ constructed in Proposition \ref{prop:eichler-shimura}. Now we turn towards computation. By the push-pull formula \eqref{eqn:push-pull-duality-formula} we have
\begin{equation}\label{eqn:evaluation-precompute-pairing}
\langle \omega_{\mathbf f}, \cl_\infty(m) \rangle = \langle \mathrm t^{\ast}\omega_{\mathbf f} \cup \delta_m, [\mathrm C_\infty] \rangle
\end{equation}
where $\cup$ is the cup product
\begin{equation*}
\cup : H^d_c(\mathrm C_\infty, \mathrm t^{\ast}{\mathscr L}_\lambda(E)) \otimes_{E} H^0(\mathrm C_\infty, \mathrm t^{\ast}{\mathscr L}_\lambda(E)^\vee) \rightarrow H^d_c(\mathrm C_\infty, E).
\end{equation*}
Let us first compute the $E$-valued differential form $\mathrm t^{\ast}\omega_{\mathbf f} \cup \delta_m$ on $\mathrm C_\infty$.  We recall that we have fixed our coordinate $z$ at the start of Section \ref{subsec:eichler-shimura} to be compatible with the canonical coordinate $x_\infty$ on $(F_\infty^\times)^{\circ}$. Thus, $\mathrm t^{\ast}\omega_{\mathbf f}$ is the $d$-form on $\mathrm C_\infty$ given in coordinates on $\mathbf A_{F,+}^\times = (F_{\infty}^\times)^{\circ} \times \mathbf A_{F,f}^\times$ by
\begin{equation*}
\mathrm t^{\ast}\omega_{\mathbf f}(x_\infty, x_f) = \mathbf f\left(i x_\infty, \begin{smallpmatrix} x_f \\ & 1 \end{smallpmatrix}\right)(ix_\infty + X)^{\kappa} d(ix_\infty)
\end{equation*}
for $x = x_\infty x_f \in \mathbf A_{F,+}^\times$. Further, by definition, $\delta_m^{\star}\left((ix_\infty + X)^{\kappa}\right) = (ix_\infty)^{{\kappa - w \over 2}+m}$. So, in coordinates we have
\begin{align}\label{eqn:form-expansion}
(\mathrm t^{\ast}\omega_{\mathbf f} \cup \delta_m)(x_\infty, x_f) &= \delta_m(x)\left(\mathbf f\left(i x_\infty, \begin{smallpmatrix} x_f \\ & 1 \end{smallpmatrix}\right)(ix_\infty + X)^{\kappa}\right)d(ix_\infty)\\
&= i^d \mathbf f\left(i x_\infty, \begin{smallpmatrix} x_f \\ & 1 \end{smallpmatrix}\right)|x_f|^m_{\mathbf A_F} (ix_\infty)^{{\kappa - w \over 2} + m}dx_\infty\nonumber\\
&= i^{1 + m + {\kappa - w \over 2}} |x|_{\mathbf A_F}^m\phi\left(\begin{smallpmatrix} x \\ & 1 \end{smallpmatrix}\right) {dx_\infty \over x_\infty}\nonumber.
\end{align}
Now we note that the pairing \eqref{eqn:evaluation-precompute-pairing} is computed by integrating $\mathrm t^{\ast} \omega_{\mathbf f} \cup \delta_m$ over $\mathrm C_\infty$. Since $x \mapsto |x|_{\mathbf A_F}^m \phi\left(\begin{smallpmatrix} x \\ & 1 \end{smallpmatrix}\right)$ is invariant under right multiplication by $\widehat{\mathcal O}_F^\times$, we get from \eqref{eqn:form-expansion} that
\begin{align*}
\langle \mathrm t^{\ast}\omega_{\mathbf f} \cup \delta_m, [\mathrm C_\infty] \rangle &= \int_{\mathrm C_\infty} \mathrm t^{\ast}\omega_{\mathbf f} \cup \delta_m\\
 &= i^{1 + m + {\kappa - w \over 2}} \int_{F_+^\times \backslash \mathbf A_{F,+}^\times}\phi \left(\begin{smallpmatrix} x \\ & 1 \end{smallpmatrix}\right) |x|^m_{\mathbf A_F} d^\times x\\
&= i^{1+m+{\kappa-w\over 2}}\mathbf M(\phi,m).
\end{align*}
This completes the proof.
\end{proof}

\begin{cor}\label{cor:cap-L-values}
If $\phi \in S_{\lambda}(K_1(\mathfrak n))$, then
\begin{equation*}
\langle \ES(\phi), \cl_\infty(m) \rangle =  i^{1 + m + {\kappa-w\over 2}}\Delta_{F/\mathbf Q}^{m+1}\Lambda(\phi,m+1).
\end{equation*}
\end{cor}
\begin{proof}
This is immediate from Proposition \ref{prop:integral-representation} and Theorem \ref{thm:cap-product}.
\end{proof}
In the special case of a $p$-refined newform, we have the following.
\begin{cor}\label{cor:mellin-twist-pair}
Let $p$ be a prime. Suppose that $\pi$ is a cohomological cuspidal automorphic representation of conductor $\mathfrak n$, $\alpha$ is a $p$-refinement of $\pi$, and $\theta$ is a finite order Hecke character with conductor of the form $\mathfrak f = \prod_{v \mid p} \mathfrak p_v^{f_v}$ with $f_v \geq 0$. If $v \mid p$ and $\pi_v$ is a principal series representation, then write $\beta_v = a_\pi(\mathfrak p_v) -\alpha_v$. Then,
\begin{equation*}
\langle \ES(\phi_{\pi,\alpha}\otimes \theta), \cl_\infty(m) \rangle = \bigl(\prod_{\substack{v\mid p\\ \mathfrak p_v \nmid \mathfrak{nf}}} (1-\beta_v\theta_v(\varpi_v)q_v^{-(m+1)}) \bigr) i^{1+m+{\kappa-w\over 2}}\Delta_{F/\mathbf Q}^{m+1} \Lambda(\pi\otimes\theta,m+1).
\end{equation*}
\end{cor}
\begin{proof}
Apply Theorem \ref{thm:cap-product} to $\phi = \mathbf \phi_{\pi,\alpha}\otimes\theta$, and then apply Proposition \ref{prop:L-values-agree}.
\end{proof}

Prior to the final result of this section, we need one more calculation.
\begin{lem}\label{lem:allowed-to-project}
Let $\theta$ be a finite order Hecke character and $E \subset \mathbf C$ a field containing the Galois closure of $F$ and the values of $\theta$. For each $\omega \in H^d_c(Y_1(\mathfrak n),\mathscr L_{\lambda}(E))$ and  $\zeta \in \pi_0(F_\infty)$ we have
\begin{equation}\label{eqn:key-pairing}
\langle \tw_{\theta}(T_\zeta \omega), \cl_\infty(m) \rangle = \theta(\zeta)\zeta^{-m}\langle \tw_\theta(\omega), \cl_\infty(m) \rangle.
\end{equation}
In particular, if $\epsilon \in \{\pm 1\}^{\Sigma_F}$ is uniquely defined by $\epsilon(\zeta) = \theta^{-1}(\zeta)\zeta^{m}$ for all $\zeta \in \pi_0(F_\infty^\times)$, then
\begin{equation*}
\langle \tw_{\theta}(\omega), \mathrm \cl_\infty(m)\rangle = \langle \tw_{\theta}(\pr^{\epsilon} \omega), \cl_\infty(m) \rangle.
\end{equation*}
\end{lem}
\begin{proof}
Proposition \ref{prop:twisting-Tzetas} and the adjointness of pushfowards/pullbacks under $\langle -, - \rangle$ implies that
\begin{equation*}
\langle \tw_{\theta}(T_\zeta \omega), \cl_\infty(m) \rangle = \theta(\zeta) \langle T_\zeta \tw_{\theta}(\omega), \cl_\infty(m) \rangle = \theta(\zeta) \langle \tw_{\theta}(\omega), T_{\zeta}\cl_\infty(m) \rangle.
\end{equation*} 
So, \eqref{eqn:key-pairing} follows from Proposition \ref{prop:pushforward-eval-class}.
\end{proof}

\begin{rmk}
The next result is originally due to Shimura \cite{Shimura-HMF}, albeit with a minor restriction on the weight $\lambda$. The general result, and the method we have followed, is due to Hida. See \cite{Hida-CriticalValues}.
\end{rmk}

\begin{thm}\label{thm:algebraic-normalization}
Let $\pi$ be a cohomological cuspidal automorphic representation of weight $\lambda$. Write $E$ for the smallest subfield of $\mathbf C$ containing $\mathbf Q(\pi)$ and the Galois closure of $F$. Then, for each $\epsilon \in \{\pm 1\}^{\Sigma_F}$ there exists $\Omega_\pi^{\epsilon} \in \mathbf C^\times$ such that, if $\theta$ is a finite order Hecke character of conductor $\mathfrak f$, then
\begin{equation}\label{eqn:shimura-algebraicity}
{\sgn(\theta_\infty)N_{F/\mathbf Q}(\mathfrak f)i^{1+m+{\kappa-w\over 2}}\Delta_{F/\mathbf Q}^{m+1}\Lambda(\pi\otimes \theta, m+1)\over G(\theta)\Omega_\pi^{\epsilon}} \in E(\theta),
\end{equation}
where 
\begin{enumerate}
\item $E(\theta)$ is the field generated by $E$ and the values of $\theta$, and
\item $\epsilon$ is chosen so that $\epsilon(\zeta) = \theta^{-1}(\zeta)\zeta^{m}$ for all $\zeta \in \pi_0(F_\infty^\times)$.
\end{enumerate}
\end{thm}
\begin{proof}
Write $\phi_{\pi}$ for the newform associated to $\pi$.  For each $\epsilon \in \{\pm 1\}^{\Sigma_F}$ choose the period  $\Omega_{\pi}^{\epsilon}$ as in Theorem \ref{thm:normalize-by-periods}. We claim that, given $\theta$, \eqref{eqn:shimura-algebraicity} now holds for the specific $\epsilon$ as in (3). 

To see the claim, let $\omega = \ES(\phi_{\pi})/\Omega_\pi^{\epsilon} \in H^d_c(Y_1(\mathfrak n),\mathscr L_{\lambda}(\mathbf C))$. The choice of period $\Omega_\pi^{\epsilon}$ means that $\pr^{\epsilon} \omega$ is actually defined over $E$ and so Lemma \ref{lem:allowed-to-project} implies that
\begin{equation}\label{eqn:key-rationality}
\langle \tw_{\theta}(\omega), \cl_\infty(m) \rangle \in E(\theta).
\end{equation}
On the other hand,
\begin{equation*}
\tw_{\theta}(\omega) = {1\over \Omega_\pi^{\epsilon}} \tw_\theta(\ES(\phi_\pi)) = {1\over \Omega_\pi^{\epsilon}}\ES(\tw_{\theta} \phi_\pi) = {G(\theta^{-1})\over \Omega_\pi^{\varepsilon}}\ES(\phi_\pi \otimes \theta).
\end{equation*}
Here we used Lemma \ref{lemma:twisting-compatibilities} for the second equality. Combining Corollary \ref{cor:cap-L-values} and \eqref{eqn:key-rationality}, we conclude
\begin{equation*}
{G(\theta^{-1})   i^{1 + m + {\kappa-w\over 2}}\Delta_{F/\mathbf Q}^{m+1}\Lambda(\phi_\pi\otimes \theta,m+1)\over\Omega^{\epsilon}_\pi} \in E(\theta).
\end{equation*}
The translation between this and \eqref{eqn:shimura-algebraicity} follows from \eqref{eqn:gauss-sum-relation}. Finally, $\phi_\pi$ and $\pi$ have the same $L$-function up to an element of $E^\times$ so we can replace $\Lambda(\phi_\pi\otimes \theta,m+1)$ with $\Lambda(\pi\otimes\theta ,m+1)$ as well.
\end{proof}

%% file: distributions.tex
\subsection{Compact abelian $p$-adic Lie groups}\label{subsec:cpa-groups}

\begin{defn}\label{defn:CPA-group}
\leavevmode
\begin{enumerate}
\item A compact abelian $p$-adic Lie group $G$ (CPA group for short) is an abelian topological group $G$ which is compact and which contains an open subgroup topologically isomorphic to $\mathbf Z_p^n$ for some $0 \leq n < \infty$. 
\item The dimension of a CPA $G$ is the integer $\dim G := n$.
\item A chart of a CPA group $G$ is an injective and open group morphism $\mathbf Z_p^{\dim G} \hookrightarrow G$.
\end{enumerate}
\end{defn}

We note that CPA groups are exactly the $p$-adic Lie groups which are compact and abelian (\cite{Serre-CompactVarieties}) and the dimension is the dimension of the underlying $p$-adic manifold.\footnote{As in ``na\"ive'' $p$-adic manifolds, as opposed to rigid analytic spaces, etc.} The salient facts are contained in the next lemma. The proofs are left to the reader.

\begin{lem}\label{lem:CPA-groups}
\leavevmode
\begin{enumerate}
\item If $G$ and $H$ are CPA groups then $G \times H$ is a CPA group.
\item If $G$ is a CPA group and $H$ is a closed subgroup then $H$ and $G/H$ are CPA groups.
\item If $f : G \rightarrow H$ is a group morphism between CPA groups then $f$ is continuous, $\ker(f) \subset G$ and $\im(f) \subset H$ are closed subgroups and the group isomorphism $G/\ker(f) \simeq \im(f)$ is a homeomorphism.
\item Let $0 \to G \to H \to J \to 0$ be any short exact sequence of abelian groups. If any two of the groups are CPA, then all three are CPA and the morphisms in the sequence are continuous.  In particular, any abelian group which is an extension of one CPA group by another is automatically CPA.
\end{enumerate}
\end{lem}

For the rest of this subsection we fix a CPA group $G$ and write $n = \dim G$. We also fix a $\mathbf Q_p$-Banach algebra $R$.

For each integer $s \geq 0$ and each chart $\nu: \mathbf Z_p^n \hookrightarrow G$, we write $\mathbf A^s(G,\nu,R)$ for the functions $f: G \rightarrow R$ with the following property:\ for each $g \in G$, the function $z \mapsto f( g \nu\left(p^{s}z\right))$ is an $R$-valued rigid analytic function in the variable  $z=(z_1,\dotsc,z_n) \in \mathbf Z_p^n$. If $f \in \mathbf A^s(G,\nu,R)$ then $f(g\nu(p^s z))$ is defined by an element in the Tate-algebra $R\langle z_1,\dotsc,z_n\rangle$ (for each $g$) and so  $\mathbf A^s(G,\nu,R)$ is naturally an $R$-Banach algebra by considering the largest of the pullback norms from $R\langle z_1,\dotsc,z_n\rangle$ for any finite choice of coset representatives of $G/\nu(p^s\mathbf Z_p^n)$. Further, for $s' \geq s$ the canonical map $\mathbf A^s(G,\nu,R) \rightarrow \mathbf A^{s'}(G,\nu,R)$ is injective with dense image and compact if $s' > s$. We define the $R$-valued locally analytic functions on $G$ as the compact type space (see \cite[Section 1.1]{Emerton-LocalAnalMemoir})
\begin{equation*}
\mathscr A(G,R) := \dirlim_{s \rightarrow \infty} \mathbf A^s(G,\nu,R).
\end{equation*}
This is independent of the chart $\nu$.

 Next, we define $\mathbf D^s(G,\nu,R) := \mathbf A^s(G,\nu,R)'$ as the $R$-Banach module dual (equipped with the operator topology). This is also an $R$-Banach algebra under the convolution product $(\mu_1,\mu_2)\mapsto \mu_1\ast \mu_2$. If $s' \geq s$ then the canonical map $\mathbf D^{s'}(G,\nu,R) \rightarrow \mathbf D^s(G,\nu,R)$ is still injective (because the transpose has dense image) and compact when $s' > s$. We define the $R$-valued locally analytic distributions on $G$ as the projective limit 
\begin{equation*}
\mathscr D(G,R) := \invlim_{s \rightarrow \infty} \mathbf D^s(G,\nu,R).
\end{equation*}
Notice that there is a natural $R$-bilinear pairing
\begin{equation*}
\mathscr D(G,R) \otimes_R \mathscr A(G,R) \rightarrow R
\end{equation*}
which we write $(\mu,f) \mapsto \mu(f)$. 

\begin{rmk}\label{rmk:commuting-tensor-product-functions}
Each of $R\mapsto \mathbf A^s(G,\nu,R)$, $\mathscr A(G,R)$, and $\mathscr D(G,R)$ commute with completed tensor products; the distributions with a fixed value of $s$ do not. Compare with \cite[Remark 3.1]{Bellaiche-CriticalpadicLfunctions}.
\end{rmk}

We now define the space of $p$-adic characters on $G$.
\begin{defn}\label{defn:character-space}
$\mathscr X(G) = \Spf(\mathbf Z_p[[G]])^{\rig}$.
\end{defn}
Thus $\mathscr X(G)$ is a rigid analytic space over $\mathbf Q_p$ whose $R$-valued points are nothing but continuous characters $\chi: G \rightarrow R^\times$. It is well-known (see \cite[Proposition 3.6.10]{Emerton-LocalAnalMemoir} for example) that if $\chi \in \mathscr X(G)(\mathbf Q_p)$ then $g\mapsto \chi(g)$ defines an element of $\mathscr A(G,\mathbf Q_p)$. Further, if $\mu \in \mathscr D(G,\mathbf Q_p)$
\begin{equation*}
\mathcal A_{\mu}(\chi) := \mu(g\mapsto \chi(g))
\end{equation*}
extends to a rigid analytic function on $\mathscr X(G)$. For instance, if $g \in G$ and $\delta_g \in \mathscr D(G,\mathbf Q_p)$ is the Dirac distribution then $\mathcal A_{\delta_g}$ is the rigid function $\operatorname{ev}_g$ on $\mathscr X(G)$ given by ``evaluation at $g$''. Further, $\mathcal A_{\mu_1\ast \mu_2} = \mathcal A_{\mu_1}\mathcal A_{\mu_2}$. See \cite[Sections 1-2]{SchneiderTitelbaum-Fourier} for more details.

\begin{defn}\label{defn:amice-transform}
The Amice transform is the natural map 
\begin{align*}
\mathscr D(G,R) &\overset{\mathcal A}{\longrightarrow} \mathscr O(\mathscr X(G))\widehat{\otimes}_{\mathbf Q_p} R\\
\mu &\mapsto \mathcal A_\mu.
\end{align*}
\end{defn}

\begin{prop}
\label{prop:amice}
The Amice transform is a topological isomorphism.
\end{prop}
\begin{proof}
By Remark \ref{rmk:commuting-tensor-product-functions}, we can assume that $R =\mathbf Q_p$. Let $H$ be an open (thus finite index) subgroup of $G$. Then, $\mathscr D(G,\mathbf Q_p)$ is finite free over $\mathscr D(H,\mathbf Q_p)$ with basis given by $\{\delta_g\}$ with $g$ running over coset representatives of $G/H$ and $\mathscr O(\mathscr X(G))$ is finite and free over $\mathscr O(\mathscr X(H))$ with basis given by $\{\operatorname{ev}_g\}$. Since $\mathcal A_{\delta_g} = \operatorname{ev}_g$, the result for $G$ follows from the result for such an $H$. Since $G$ is a CPA group, there exists an $H$ topologically isomorphic to $\mathbf Z_p^n$, in which case the theorem is known by a multi-variable version of Amice's theorem \cite{Amice-Interpolation} (see \cite{SchneiderTitelbaum-Fourier}).
\end{proof}

\subsection{Locally analytic distributions on $\mathcal O_p$}\label{subsec:Op-dist}
In this section we consider the CPA group $\mathcal O_p = \mathcal O_F \otimes_{\mathbf Z} \mathbf Z_p = \prod_{v \mid p} \mathcal O_v$. For $v \mid p$, we fix a uniformizer $\varpi_v \in \mathcal O_v$ and we write $\varpi_p \in \mathcal O_p$ for the corresponding tuple. Let $e_v$ be the ramification index at $v \mid p$, and $\mathbf e = (e_v)_{v\mid p} \in \mathbf Z_{\geq 1}^{\{v\mid p\}}$.

Start by choosing a $\mathbf Z_p$-linear isomorphism $\nu: \mathbf Z_p^d \simeq \mathcal O_p$ which we use as a chart. Using this we write $\mathbf A^{\circ}(\mathcal O_p,\mathbf Q_p)$ for the ring of functions $f: \mathcal O_p \rightarrow \mathbf Q_p$ such that $f \circ \nu$ is defined by an element of the Tate algebra $\mathbf Z_p\langle z_1,\dotsc,z_d \rangle$. The ring $\mathbf A(\mathcal O_p, \mathbf Q_p) := \mathbf A^{\circ}(\mathcal O_p, \mathbf Q_p)[1/p]$ is the ring we denoted  $\mathbf A^0(\mathcal O_p, \nu, \mathbf Q_p)$ in Section \ref{subsec:cpa-groups}, so $f \mapsto f \circ \nu$ defines an isomorphism $\mathbf A(\mathcal O_p, \mathbf Q_p) \simeq \mathbf Q_p\langle z_1,\dotsc,z_d\rangle$. The $\mathbf Q_p$-Banach structure on with the norm $\|f\|_{\mathbf 0}$ on $\mathbf A(\mathcal O_p,\mathbf Q_p)$ defined by pulling back the supremum norm on $\mathbf Q_p\langle z_1,\dotsc,z_d\rangle$. It is independent of $\nu$.

For $\mathbf s = (s_v)_{v \mid p} \in \mathbf Z_{\geq 0}^{\{v\mid p\}}$ we now define
\begin{align*}
\mathbf A^{\mathbf s, \circ}(\mathcal O_p, \mathbf Q_p) &:= \{f : \mathcal O_p \rightarrow \mathbf Q_p \mid z \mapsto f(a + \varpi_p^{\mathbf s} z) \text{ lies in $\mathbf A^{\circ}(\mathcal O_p,\mathbf Q_p)$ for all $a \in \mathcal O_p$}\};\\
\mathbf A^{\mathbf s}(\mathcal O_p, \mathbf Q_p) &= \mathbf A^{\mathbf s, \circ}(\mathcal O_p,\mathbf Q_p)[1/p].
\end{align*}
If $f \in \mathbf A^{\mathbf s}(\mathcal O_p,\mathbf Q_p)$, then $f(a+\varpi_p^{\mathbf s}z)$ depends on $a \bmod \varpi_p^{\mathbf s}\mathcal O_p$ only up to translation in the $z$-variable. Thus we equip $\mathbf A^{\mathbf s}(\mathcal O_p,\mathbf Q_p)$ with a $\mathbf Q_p$-Banach norm by
\begin{equation*}
\|f\|_{\mathbf s} := \max_{a \in \mathcal O_p/\varpi_p^{\mathbf s}\mathcal O_p} \|f(a + \varpi_p^{\mathbf s}z)\|_{\mathbf 0}.
\end{equation*}
If $\mathbf s' \geq \mathbf s$ (i.e.\ $s_v' \geq s_v$ for all $v \mid p$) then the natural map $\mathbf A^{\mathbf s}(\mathcal O_p,\mathbf Q_p) \rightarrow \mathbf A^{\mathbf s'}(\mathcal O_p,\mathbf Q_p)$ is continuous with dense image. If $\mathbf s' \geq \mathbf s+\mathbf e$ (i.e. $s_v' \geq s_v + e_v$ for each $v \mid p$) then it is compact. Furthermore, the $\mathbf Q_p$-Banach algebras $\mathbf A^{\mathbf s}(\mathcal O_p,\mathbf Q_p)$ are a co-final defining sequence for $\mathscr A(\mathcal O_p,\mathbf Q_p)$, as in Section \ref{subsec:cpa-groups}, because if $s \in \mathbf Z_{\geq 0}$ and $\mathbf s := (se_v)_{v\mid p}$ then we have an obvious (topological) equality
\begin{equation*}
\mathbf A^{\mathbf s}(\mathcal O_p, \mathbf Q_p) = \mathbf A^{s}(\mathcal O_p, m_{u_p}\circ \nu, \mathbf Q_p)
\end{equation*}
where $m_{u_p}$ is multiplication by $\varpi_p^{\mathbf e}p^{-1}$ on $\mathcal O_p$. Thus we also have a topological isomorphism
\begin{equation}\label{eqn:defn-locally-anal}
\mathscr A(\mathcal O_p, \mathbf Q_p) = \dirlim_{|\mathbf s| \rightarrow +\infty} \mathbf A^{\mathbf s}(\mathcal O_p, \mathbf Q_p)
\end{equation}
where $|\mathbf s|=\min(s_v : v \mid p)$.

If $R$ is a $\mathbf Q_p$-Banach algebra and $\mathbf s \in \mathbf Z_{\geq 0}^{\{v\mid p\}}$, we define $\mathbf A^{\mathbf s}(\mathcal O_p, R) := \mathbf A^{\mathbf s}(\mathcal O_p,\mathbf Q_p) \widehat{\otimes}_{\mathbf Q_p} R$ with its inductive tensor product topology. Any $\mathbf Q_p$-Banach space is potentially orthonormalizable (\cite[Proposition 1]{Serre-CompactOperators}), so the $R$-Banach modules $\mathbf A^{\mathbf s}(\mathcal O_p,R)$ are potentially orthonormalizable as well (\cite[Lemma 2.8]{Buzzard-Eigenvarieties}). If $\mathbf s' \geq \mathbf s$ then the natural map $\mathbf A^{\mathbf s}(\mathcal O_p, R)\rightarrow \mathbf A^{\mathbf s}(\mathcal O_p,R)$ is injective with dense image (\cite[Corollary 1.1.27]{Emerton-LocalAnalMemoir}) and if $\mathbf s' \geq \mathbf s+\mathbf e$ then the map is compact (\cite[Lemma 18.12]{Schneider-NFA}). By $\eqref{eqn:defn-locally-anal}$ and \cite[Proposition 1.1.32(i)]{Emerton-LocalAnalMemoir} we deduce a topological identification
\begin{equation}\label{eqn:Rdefn-locally-anal}
\mathscr A(\mathcal O_p, R) = \dirlim_{|\mathbf s| \rightarrow +\infty} \mathbf A^{\mathbf s}(\mathcal O_p, R).
\end{equation}
Finally, we write $\mathbf D^{\mathbf s}(\mathcal O_p,R)$ for $R$-Banach dual $\mathbf A^{\mathbf s}(\mathcal O_p,R)'$ equipped with its operator topology and convolution product. The $R$-Banach algebras $\mathbf D^{\mathbf s}(\mathcal O_p, R)$ are co-final with the Banach algebras in Section \ref{subsec:cpa-groups} (for the same reasons as above) and thus we have a topological identification 
\begin{equation*}
\mathscr D(\mathcal O_p,R) = \invlim_{|\mathbf s| \rightarrow +\infty} \mathbf D^s(\mathcal O_p, R).
\end{equation*}

\begin{rmk}
The $R$-Banach modules $\mathbf D^{\mathbf s}(\mathcal O_p,R)$ are {\em not} the same as $\mathbf D^{\mathbf s}(\mathcal O_p,\mathbf Q_p)\widehat{\otimes}_{\mathbf Q_p} R$ and thus not obviously potentially orthornormalizable.
\end{rmk}

We now recall the following definition.
\begin{defn}
If $R$ is a $\mathbf Q_p$-Banach algebra, a ring of definition $R_0$ for $R$ is a subring $R_0 \subset R$ which is open and bounded.
\end{defn}
We note that this implies as well that $R_0$ is $p$-adically separated and complete, and $R_0[1/p] = R$.  After fixing $R_0 \subset R$ a ring of definition, we now define
\begin{equation*}
\mathbf A^{\mathbf s,\circ}(\mathcal O_p, R) := \mathbf A^{\mathbf s,\circ}(\mathcal O_p, \mathbf Q_p) \widehat{\otimes}_{\mathbf Z_p} R_0.
\end{equation*}
The $R_0$-algebra $\mathbf A^{\mathbf s,\circ}(\mathcal O_p, R)$ is naturally an open and bounded $R_0$-subalgebra $\mathbf A^{s}(\mathcal O_p,R)$ and we have an equality after inverting $p$. For the distributions, still with $R_0$ fixed, we define $\mathbf D^{s,\circ}(\mathcal O_p, R)$ as the $R_0$-linear dual
\begin{equation*}
\mathbf D^{\mathbf s,\circ}(\mathcal O_p, R) := \Hom_{R_0}(\mathbf A^{s,\circ}(\mathcal O_p, R), R_0).
\end{equation*}
\begin{rmk}\label{remark:notation}
The notations $\mathbf A^{\mathbf s,\circ}$ and $\mathbf D^{\mathbf s,\circ}$ are misleading in that they obviously depend on $R_0$. If $R$ is reduced, then we may take $R_0$ to be the subring of power-bounded elements in $R$. In any case, the reader may also notice that we never make ``natural use'' of the lattices (as opposed to the functors $\mathbf A^{\mathbf s}(\mathcal O_p,-)$ and $\mathbf D^{\mathbf s}(\mathcal O_p,-)$).
\end{rmk}

\subsection{Actions by the monoid $\Delta$}\label{subsec:monoid-action}

We maintain the notations of the previous subsection and we also fix a $\mathbf Q_p$-Banach algebra $R$ and a ring of definition $R_0 \subset R$. If $h(z)$ is a function on $\mathcal O_p^\times$ valued in a ring, then write $h(z)_{!}$ for its extension by zero to  $\mathcal O_p$.

\begin{lem}\label{lem:character-analytic}
If $\chi: \mathcal O_p^\times \rightarrow R^\times$ is a continuous character, then there exists $\mathbf s(\chi) \in \mathbf Z_{\geq 0}^{\{v\mid p\}}$ such that $f \in \mathbf A^{\mathbf s(\chi)}(\mathcal O_p,R)$ when $f$ is a function of either of the following two forms.
\begin{enumerate}
\item $f(z) = \chi(d+cz)$ with $c \in \varpi_p\mathcal O_p$ and $d \in \mathcal O_p^\times$.
\item $f(z) = \chi(z)_!$.
\end{enumerate}
If $\chi(\mathcal O_p^\times) \subset R_0^\times$, then there exists $\mathbf s^{\circ}(\chi) \in \mathbf Z_{\geq 0}^{\{v\mid p\}}$ depending on $R_0$ so that $f \in \mathbf A^{\mathbf s^{\circ}(\chi),\circ}(\mathcal O_p,R)$ for the same functions.
\end{lem}
\begin{proof}
If $c \in \varpi_p\mathcal O_p$ and $d \in \mathcal O_p^\times$ then $\chi(d+cz) = \chi(d+cz)_{!}$. Since $z \mapsto d + cz$ is polynomial in $z$, we only need to prove the lemma where $f(z) = \chi(z)_!$. In the case where $p$ is inverted, this is well-known. We now deduce the $R_0$-case from the $R$-case.

First, we observe that if $g \in \mathbf A(\mathcal O_p,R)$ and $g(0) \in R_0$, then there exist $\mathbf s(g)$ so that $g(\varpi_p^{\mathbf s(g)}z) \in \mathbf A^{\circ}(\mathcal O_p,R)$ (expand the series defining $g$). Now write $f(z) = \chi(z)_{!}$. For $a$ running over a (finite) set of coset representatives for $\mathcal O_p/\varpi_p^{\mathbf s}\mathcal O_p$, there exists $g_a \in \mathbf A(\mathcal O_p,R)$ such that $f(a+\varpi_{p}^{\mathbf s(\chi)}z) = g_a(z)$. Since $g_a(0) = f(a) \in R_0$, the first sentence of this paragraph applies to each $g_a$ and the lemma follows.
\end{proof}

Recall that $T \subset {\GL_2}_{/\mathbf Z}$ denotes the diagonal torus. Thus $T(\mathcal O_p) \simeq (\mathcal O_p^\times)^2$ is naturally a CPA group.
\begin{defn}
The space of $p$-adic weights is $\mathscr W = \mathscr X(T(\mathcal O_p))$.
\end{defn} 
If $\Omega = \Sp(R)$ and $\lambda_{\Omega}:\Omega\rightarrow \mathscr W$ is a point then we often confuse it with the corresponding pair  $\lambda_{\Omega}=(\lambda_{\Omega,1},\lambda_{\Omega,2})$ where $\lambda_{\Omega,i}: \mathcal O_p^\times \rightarrow R^\times$ are continuous character. If $R$ is a finite extension of $\mathbf Q_p$ we write just $\lambda$. In either case, we generally refer to both the point and the character as a $p$-adic weight.

Now consider the submonoid of $\GL_2(F_p)$ defined by
\begin{equation*}
\Delta := \left\{\begin{pmatrix} a & b \\ c & d \end{pmatrix} \in \GL_2(F_p) \cap M_2(\mathcal O_p) \mid c \in \varpi_p \mathcal O_p \text{ and } d \in \mathcal O_p^\times \right\}.
\end{equation*}
If $g = \begin{smallpmatrix} a & b \\ c & d\end{smallpmatrix} \in \Delta$ then $cz + d \in \mathcal O_p^\times$ and so the left action $g\cdot z = {az + b \over cz + d}$ of $\Delta$ on $\mathcal O_p$ is well-defined and it is clearly continuous.

Now consider $\Omega = \Sp(R)$ and let $\lambda_{\Omega}:\Omega\rightarrow \mathscr W$ be a $p$-adic weight. Set $\mathbf s(\Omega) := \max\{\mathbf s(\lambda_{\Omega,1}\lambda_{\Omega,2}^{-1}),\mathbf s(\lambda_{\Omega,2}^{-1})\}$ as above.\footnote{Inserting $\mathbf s(\lambda_{\Omega,2}^{-1})$ into the maximum is purely for convenience of notation later on (see Lemma \ref{lem:construction-Qlambda}).} Then, for $\mathbf s \geq \mathbf s(\Omega)$ we may endow $\mathbf A^{\mathbf s}(\mathcal O_p,R)$ with a continuous $R$-linear right action of $\Delta$ via
\begin{equation}\label{eqn:function-action}
f\big|_{g}(z) = \lambda_{\Omega,1}\lambda_{\Omega,2}^{-1}(cz+d)\lambda_{\Omega,2}\left(\det g \cdot \varpi_p^{-v(\det g)}\right)f(g \cdot z)
\end{equation}
where $g=\begin{smallpmatrix}a & b\\ c& d\end{smallpmatrix} \in \Delta$, $f \in \mathbf A^{\mathbf s}(\mathcal O_p,R)$ and $z \in \mathcal O_p$.\footnote{To be clear, we recall that $\varpi_p^{-v(\det g)}$ means $\prod_{v \mid p} \varpi_v^{-v(\det g_v)}$.} This definition is well-posed by Lemma \ref{lem:character-analytic}. We then equip $\mathbf D^{\mathbf s}(\mathcal O_p,R)$ with the dual left action:\ $(g \cdot \mu)(f) = \mu(f|_g)$. Either action is referred to as a ``weight $\lambda$-action.''

\begin{rmk}\label{rmk:difference-with-daves-paper}
The monoid $\Delta$ and the action \eqref{eqn:function-action} differ from their definitions in \cite[Section 2.2]{Hansen-Overconvergent} by conjugation by $\begin{smallpmatrix} & 1 \\ \varpi_p & \end{smallpmatrix} \in \GL_2(F_p)$. Compare with Proposition \ref{prop:implications-extremely-non-critical}(1).
\end{rmk}

The above action of $\Delta$ is compatible with the injective restriction map $\mathbf A^{\mathbf s}(\mathcal O_p,R) \rightarrow \mathbf A^{\mathbf s'}(\mathcal O_p,R)$ when $\mathbf s' \geq \mathbf s$, so we get a continuous action of $\Delta$ on $\mathscr A(\mathcal O_p,R)$. On the dual side, $\mathbf D^{\mathbf s}(\mathcal O_p,R)$ is equipped with a continuous $R$-linear left action by $\Delta$ and the compatibility extends this to a continuous action on $\mathscr D(\mathcal O_p,R)$. Finally, when the image of $\lambda_{\Omega}$ is contained in $R_0$, then  \eqref{eqn:function-action} defines an action of $\Delta$ on $\mathbf A^{\mathbf s,\circ}(\mathcal O_p,R)$ as well as a left action on $\mathbf D^{\mathbf s,\circ}(\mathcal O_p,R)$ for all $\mathbf s \geq \mathbf s^{\circ}(\Omega) := \max\{\mathbf s^{\circ}(\lambda_{\Omega,1}\lambda_{\Omega,2}^{-1}), \mathbf s^{\circ}(\lambda_{\Omega,2}^{-1})\}$.

We summarize the notations presented  above as follows.

\begin{defn}\label{defn:actions}
\leavevmode
\begin{enumerate}
\item If $\Omega=\Sp(R)$ is a $\mathbf Q_p$-affinoid space, $\lambda_{\Omega}:\Omega \rightarrow \mathscr W$ is a $p$-adic weight and $\mathbf s \geq \mathbf s(\Omega)$, then we write $\mathbf A^{\mathbf s}_{\Omega} := \mathbf A^{\mathbf s}(\mathcal O_p,R)$, $\mathbf D^{\mathbf s}_{\Omega} := \mathbf D^{\mathbf s}(\mathcal O_p,R)$, $\mathscr A_{\Omega} := \mathscr A(\mathcal O_p,R)$, and $\mathscr D_{\Omega} := \mathscr D(\mathcal O_p,R)$ for the above $R$-modules equipped with their continuous actions of $\Delta$ via $\lambda_{\Omega}$ above. When $R_0$ is a ring of definition containing the image of $\lambda_{\Omega}$ and $\mathbf s \geq \mathbf s^{\circ}(\Omega)$ then we write $\mathbf A^{\mathbf s,\circ}_{\Omega} = \mathbf A^{\mathbf s,\circ}(\mathcal O_p,R)$ and $\mathbf D^{\mathbf s,\circ} = \mathbf D^{\mathbf s, \circ}(\mathcal O_p,R)$ for the $R_0$-modules equipped with their action of $\Delta$ above.
\item If $\lambda \in \mathscr W(\overline{\mathbf Q}_p)$ with residue field $k_{\lambda}$, we write $\mathbf A^{\mathbf s}_{\lambda}$, $\mathbf D^{\mathbf s}_{\lambda}$, $\mathscr A_{\lambda}$, and $\mathscr D_{\lambda}$ in place of $\mathbf A^{\mathbf s}_{\Sp k_\lambda}$, $\mathbf D^{\mathbf s}_{\Sp k_\lambda}$, $\mathscr A_{\Sp k_\lambda}$, and $\mathscr D_{\Sp k_\lambda}$.
\end{enumerate}
\end{defn}

\subsection{The integration map for cohomological weights}\label{subsec:integration-map}
Throughout this subsection we fix $L \subset \overline{\mathbf Q}_p$ and assume it contains the Galois closure of $F$ inside $\overline{\mathbf Q}_p$. We also consider a fixed  cohomological weight $\lambda=(\kappa,w)$. (The notations of the previous two subsections also remain in force.)

Recall we defined the $L$-vector space $\mathscr L_\lambda(L)$, equipped with a left action of $\GL_2(F_p)$ in \eqref{eqn:padic-Llambda}. It thus inherits an action of the monoid $\Delta \subset \GL_2(F_p)$ from Section \ref{subsec:monoid-action}. We also view $\lambda$ as a $p$-adic weight $\lambda = (\lambda_1,\lambda_2)$ where $\lambda_i$ is given by
\begin{equation*}
\lambda_i(z) = \prod_{v \mid p} \prod_{\sigma \in \Sigma_v} \sigma(z)^{e_i(\sigma)}
\end{equation*}
where $e_1(\sigma) = {1\over 2}(w+\kappa_\sigma)$ and $e_2(\sigma) = {1\over 2}(w - \kappa_\sigma)$. The residue field $k_{\lambda}$ of $\lambda \in \mathscr W$ is contained in the Galois closure of $F$ inside $\overline{\mathbf Q}_p$. Thus to a cohomological weight $\lambda$ we also have a $\Delta$-module of distributions $\mathscr D_{\lambda}\otimes_{k_\lambda} L$.

\begin{defn}
The integration map is the $L$-linear map $I_\lambda: \mathscr D_\lambda\otimes_{k_\lambda} L \rightarrow \mathscr L_\lambda(L)$ given by
\begin{equation}\label{eqn:I-lambda-defn}
I_\lambda(\mu)(X) = \mu((z+X)^{\kappa}) := \sum_{0\leq j\leq \kappa} {\kappa \choose j} \mu(z^j) X^{\kappa - j}.
\end{equation}
\end{defn}
It is elementary to check the action of $\Delta$ has the following relationship to the integration map:\ if $g \in \Delta$ and $\mu \in \mathscr D_{\lambda}\otimes_{k_\lambda} L$, then
\begin{equation}\label{eqn:I-lambda-action}
I_\lambda(g \cdot \mu) = \left(\varpi_p^{-v(\det g)}\right)^{w -\kappa \over 2}g\cdot I_{\lambda}(\mu).
\end{equation} 

\begin{defn}\label{defn:sharp-notation}
$\mathscr L_{\lambda}^{\sharp}(L) := \mathscr L_{\lambda}(L)\otimes (\varpi_p^{-v(\det g)})^{w-\kappa \over 2}$ (as a left $\Delta$-module).
\end{defn}
Thus $\mathscr L_\lambda^{\sharp}$ coincides with $\mathscr L_\lambda^{\sharp}$ as a representation of the Iwahori subgroup of $\mathrm{GL}_2(\mathcal{O}_p)$, but the full action of $\Delta$ has been twisted so that $I_\lambda$ becomes equivariant (point (1) below). Before stating the next proposition, we note that any left $\Delta$-module becomes a left $\mathcal O_p^\times$-module via the inclusion  $\begin{smallpmatrix} \mathcal O_p^\times \\ & 1 \end{smallpmatrix} \subset \Delta$.
\begin{prop}\label{prop:change-to-sharp}
\leavevmode
\begin{enumerate}
\item $I_\lambda:  \mathscr D_{\lambda}\otimes_{k_\lambda} L \rightarrow \mathscr L_{\lambda}^{\sharp}(L)$ is $\Delta$-equivariant.
\item If $\mathcal O_L \subset L$ denotes the ring of integers and $\mathscr L_{\lambda}^{\sharp}(\mathcal O_L)$ are those polynomials with $\mathcal O_L$-coefficients then $\mathscr L_{\lambda}^{\sharp}(\mathcal O_L)$ is $\Delta$-stable.
\item The identity map $\mathscr L_{\lambda}(L) \rightarrow \mathscr L^{\sharp}_{\lambda}(L)$ is an isomorphism of left $\mathcal O_p^\times$-modules.
\end{enumerate}
\end{prop}
\begin{proof}
Point (1) is immediate from \eqref{eqn:I-lambda-action}. The second point is straightforward from the definition. The third point is because if $x \in \mathcal O_p^\times$ and $g = \begin{smallpmatrix} x \\ & 1 \end{smallpmatrix}$ then $\det(g) \in \mathcal O_p^\times$, so $\lambda_2(\varpi_p^{-v_p(\det g)}) = 1$.
\end{proof}

\subsection{$p$-adic twisting}\label{subsec:padic-twisting}

In this subsection we consider two $p$-adic analogs of the twisting studied in Section \ref{subsec:twisting-archimedean}. Recall that $\Gamma_F$ is the Galois group of the maximal abelian extension of $F$ unramified away from  $p$ and $\infty$. Global class field theory defines an isomorphism
\begin{equation*}
\Gamma_F \simeq F^\times \backslash \mathbf A_F^\times / H
\end{equation*}
where $H$ is the closure of the subgroup generated by $(F_\infty^\times)^{\circ} \widehat{\mathcal O}_F^{(p),\times}$. Thus there is a natural short exact sequence
\begin{equation}\label{eqn:ses-GammaF}
1 \rightarrow \mathcal O_p^\times \big/ \overline{\mathcal O_{F,+}^\times} \rightarrow \Gamma_F \rightarrow \Cl_F^+ \rightarrow 1,
\end{equation}
where $\Cl_F^+$ is the narrow class group, $\mathcal O_{F,+}^\times$ are the totally positive units, and the bar indicates the $p$-adic closure under the natural inclusion $\mathcal O_{F,+}^\times\subset\mathcal O_p^\times$. By Lemma \ref{lem:CPA-groups} and \eqref{eqn:ses-GammaF}, $\Gamma_F$ is a CPA group. We write $\mathscr X(\Gamma_F)$ for the rigid analytic space parameterizing continuous $p$-adic characters on $\Gamma_F$.

\begin{defn}\label{eqn:defn-twisting}
Suppose that $R$ is a $\mathbf Q_p$-Banach algebra and $N$ is an $R$-module equipped with an $R$-linear left action $g\cdot n$ of the monoid $\Delta$. If $\vartheta: \Gamma_F \rightarrow R^\times$ is an $R$-valued point of $\mathscr X(\Gamma_F)$ then we define a new left $\Delta$-module by
\begin{equation*}
N(\vartheta) = N \otimes \vartheta^{-1}|_{\mathcal O_p^\times}(\det g \cdot \varpi_p^{-v(\det g)}).
\end{equation*}
\end{defn}

We note that $\mathscr X(\Gamma_F)$ also acts on $\mathscr W$ by central twists:\ if $\lambda=(\lambda_1,\lambda_2)$ is a character on $(\mathcal O_p^\times)^{\oplus 2}$ then we define we define $\vartheta \cdot \lambda := (\vartheta|_{\mathcal O_p^\times} \lambda_1, \vartheta|_{\mathcal O_p^\times} \lambda_2)$. 

For the next three results, let $\Omega \rightarrow \mathscr W$ be a $p$-adic weight. The previous paragraph allows us to define a new $p$-adic weight $\vartheta^{-1}\cdot \Omega$ whenever $\vartheta \in \mathscr X(\Gamma_F)(\Omega)$.
\begin{lem}\label{lemma:distr-twist}
If $\vartheta \in \mathscr X(\Gamma_F)(\Omega)$, then the identity map is an isomorphism $\mathscr D_{{\Omega}}(\vartheta) \simeq \mathscr D_{\vartheta^{-1}\cdot {\Omega}}$.
\end{lem}
\begin{proof}
This follows immediately from the definitions.
\end{proof}

Now consider a compact open subgroup $K$ of $\GL_2(\mathbf A_{F,f})$ such that $K_p \subset \Delta$. If $N$ is a left $\Delta$-module then we define a local system on $Y_K$ as in Section \ref{subsec:adelic-cochains}, with $\GL_2^+(F)$ acting trivially and $k \in K_p$ acting on the right as $k^{-1}$ acts on the left. We view $\mathscr O(\Omega)$ as a trivial left $\Delta$-module.
\begin{lem}\label{lem:theta-det}
If $\vartheta \in \mathscr X(\Gamma_F)(\Omega)$, then $\vartheta_{\det} : \GL_2(\mathbf A_F)\rightarrow \mathscr O(\Omega)^\times$ given by $g\mapsto \vartheta(\det g)$ defines an element of $H^0(Y_K, \mathscr O(\Omega)(\vartheta))$.
\end{lem}
\begin{proof}
Since $\vartheta$ is trivial on $(F_\infty^\times)^{\circ}$,  $\vartheta_{\det}$ is trivial on $\GL_2^+(F_\infty)$. So, $\vartheta_{\det}$ is a locally constant on $\GL_2(\mathbf A_F)$ and invariant under multiplication by $K_\infty^{\circ}$. Further, $\vartheta_{\det}$ trivial on $\GL_2(F)$ since $\vartheta$ is trivial on $F^\times$. Finally, if $k \in K$ then $\vartheta(\det k) = \vartheta(\det k_p)$ because $\vartheta$ vanishes on the units away from $p$. So finally, if $g \in \GL_2(\mathbf A_{F})$ and $k \in K$, then
\begin{equation*}
\vartheta_{\det} (gk) = \vartheta(\det k_p) \vartheta_{\det}(g) = \vartheta_{\det}(g)|_{k}.
\end{equation*}
This concludes the proof.
\end{proof}

Following Lemma \ref{lem:theta-det}, the following definition is well-posed.

\begin{defn}\label{defn:p-adic-big-twisting}
If $\vartheta \in \mathscr X(\Gamma_F)(\Omega)$ and $N$ is a left $\mathscr O(\Omega)[\Delta]$-module then we define the twisting map
\begin{equation*}
\tw_{\vartheta}: H^{\ast}_c(Y_K,N) \rightarrow H^{\ast}_c(Y_K,N(\vartheta))
\end{equation*}
to be cup product with $\vartheta_{\det}$.
\end{defn}

Twisting defined this was is directly related to twisting the action of Hecke operators $[K\delta K]$ acting on $H^{\ast}_c(Y_K,N)$, as the next proposition shows. Its relation to twisting automorphic forms is explained afterward.
\begin{prop}\label{prop:twisting-hecke-eigenvalues}
Assume that $\begin{smallpmatrix} \mathcal O_p^\times \\ & 1 \end{smallpmatrix} \subset K_p$. Then, for each finite place $v$ of $F$ we have
\begin{equation*}
 [K\begin{smallpmatrix} \varpi_v \\ & 1\end{smallpmatrix} K] \circ \tw_{\vartheta}  = \vartheta(\varpi_v)\tw_{\vartheta} \circ [K\begin{smallpmatrix} \varpi_v \\ & 1\end{smallpmatrix} K].
\end{equation*}
\end{prop}
\begin{proof}
First, we claim that we can write $K\begin{smallpmatrix} \varpi_v \\ & 1 \end{smallpmatrix} K = \bigcup \delta_i K$ with $\delta_i \in \GL_2(F_v)$ such that $\det \delta_i = \varpi_v$ if $v \mid p$ and $\vartheta(\det \delta_i) = \vartheta(\varpi_v)$ in general. This is true for any $\delta_i$ if $v \nmid p$ since $\vartheta$ is trivial on $\det(K^p)$. But if $v \mid p$ and $\delta_i$ is any choice then $\det(\delta_i) = \varpi_v u_i^{-1}$ for some $u_i \in \mathcal O_v^\times$. By the assumption on $K$, we can replace $\delta_i$ by $\delta_i \begin{smallpmatrix} u_i \\ & 1 \end{smallpmatrix} \in \delta_i K$.

Now to prove the proposition we fix a choice of $\delta_i$ as above. By Proposition \ref{prop:canonical} it suffices to individually calculate $\tw_{\vartheta}([K\delta K]\phi)$ and $[K\delta K]\tw_{\vartheta}(\phi)$ on the level of adelic cochains, from Section \ref{subsec:adelic-cochains}. For clarity, let us write $\delta \cdot n$ for the action of $\Delta$ on $N$ and $\delta \star n$ for the action of $\Delta$ on $N(\vartheta)$. 

Let $\phi \in C_{\ad}^{\bullet}(K,N)$. Each adelic chain (an element of $C_{\bullet}(D_{\mathbf A})$ as in Section \ref{subsec:adelic-cochains}) is a finite linear combination of the simple tensors $\sigma \otimes [g_f]$ where $\sigma \in C_{\bullet}(D_\infty)$ and $g_f \in \GL_2(\mathbf A_{F,f})$. For such simple tensors we have:
\begin{align*}
\tw_{\vartheta}([K\delta K]\phi)(\sigma\otimes [g_f]) &= \vartheta(\det g_f) ([K\delta K]\phi)(\sigma \otimes [g_f])\\
&= \vartheta(\det g_f) \sum_i \delta_i \cdot \phi(\sigma \otimes [g_f \delta_i]).
\end{align*}
On the other hand, since $\vartheta(\det \delta_i) = \vartheta(\varpi_v)$, and $\det \delta_{i,p} \cdot \varpi_p^{-v(\det \delta_{i,p})} = 1$ we get
\begin{align*}
[K\delta K] \tw_{\vartheta}(\phi)(\sigma\otimes [g_f]) &= \sum_i \delta_i \star (\tw_{\vartheta}\phi)(\sigma \otimes[g_f\delta_i]) \\
&= \sum_i \vartheta(\det g_f \det \delta_i)\vartheta(\det \delta_{i,p} \cdot \varpi_p^{-v(\det \delta_{i,p})})^{-1} \delta_i \cdot \phi(\sigma \otimes [g_f \delta_i])\\
&= \vartheta(\varpi_v)\vartheta(\det g_f) \sum_i \delta_i \cdot \phi(\sigma \otimes [g_f\delta_i]).
\end{align*}
Thus, the proposition is proven.
\end{proof}
So, under the mild hypothesis of Proposition \ref{prop:twisting-hecke-eigenvalues} (which is satisfied in practice), we can twist distribution-valued Hecke eigenclasses by $p$-adic characters of $\Gamma_F$ and obtain new Hecke eigenclasses  of a possibly different weight. But the twisting maps $\tw_{\vartheta}$ do not preserve the cohomology of the finite-dimensional spaces $\mathscr L_{\lambda}$, so we also need a second kind of twisting analogous to Section \ref{subsec:twisting-archimedean}. 

As before, write $\theta:\mathbf A_F^\times \rightarrow \mathbf C^\times$ for a finite order Hecke character but we assume now that it is unramified away from $p$. Write $\mathfrak f$ for its conductor. Then $\theta^{\iota} := \iota \circ \theta$ defines a finite order character $\theta^{\iota} : \mathbf A_F^\times \rightarrow \overline{\mathbf Q}_p^\times$ which descends to a character of $\Gamma_F$. Suppose that $L$ is a subfield of $\overline{\mathbf Q}_p$ containing the Galois closure of $F$ and the values of $\theta^{\iota}$ and also let $\mathfrak n$ be an integral ideal of $\mathcal O_F$. In analogy with Section \ref{subsec:twisting-archimedean} we define a linear map
\begin{equation}\label{eqn:padic-twisting-finite-order}
\tw_{\theta^{\iota}}^{\cl} : H^\ast_c(Y_{1}(\mathfrak n),{\mathscr L}_{\lambda}(L)) \rightarrow H^\ast_c(Y_{1}(\mathfrak{n f}^2), {\mathscr L}_{\lambda}(L))
\end{equation} 
by 
\begin{equation*}
\tw_{\theta^{\iota}}^{\cl} = \theta^{\iota}_{\det}\cup  \sum_{t \in \Upsilon_{\mathfrak f}^\times} \theta^{\iota}(t)v_{t_0,p}^{\ast}.
\end{equation*}
Here the notation is just as in Section \ref{subsec:twisting-archimedean}. Note, however, that because the local systems $\mathscr L_\lambda(L)$ are defined with respect to a right action of $\GL_2(F_p)$, we no longer have an isomorphism between $v_t^{\ast}\mathscr L_\lambda(L)$ and $\mathscr L_\lambda(L)$. In fact, the map written $v_{t,p}^{\ast}$ above is the map on cohomology fitting into the diagram
\begin{equation}\label{eqn:p-vtstar-diagram}
\xymatrix{
H^\ast_c(Y_{1}(\mathfrak n),{\mathscr L}_{\lambda}(L)) \ar[r]^-{v_{t,p}^{\ast}} \ar[d]_-{\pr^{\ast}} &  H^\ast_c(Y_{11}(\mathfrak {n f^2}), {\mathscr L}_{\lambda}(L)) \\
H^\ast_c(Y_{K_{11}(\mathfrak{nf^2})_t}, {\mathscr L}_{\lambda}(L)) \ar[r]_-{r_{u_t}^{\ast}} & H^\ast_c(Y_{11}(\mathfrak{nf^2}), {\mathscr L}_\lambda(L)(u_t)), \ar[u]_-{\simeq}
}
\end{equation}
where the right vertical arrow is induced by the isomorphism $P \mapsto u_t \cdot P$ of local systems $\mathscr L_{\lambda}(L)(u_t) \rightarrow \mathscr L_\lambda(L)$ in the opposite direction of the diagonal arrow in \eqref{eqn:iota-transfer}.

The image of $\tw_{\theta^{\iota}}^{\cl}$ is contained in $H^\ast_c(Y_{1}(\mathfrak{nf})^2,\mathscr L_{\lambda}(L))$ just as in the proof of Lemma \ref{lemma:twisting-compatibilities}. And, if $E$ is a subfield of $\mathbf C$ containing the Galois closure of $F$ and the values of $\theta$ and $L = \mathbf Q_p(\iota(E))$, then \eqref{eqn:iota-transfer} implies that the diagram
\begin{equation}
\xymatrix{
H^d_c(Y_{1}(\mathfrak n),{\mathscr L}_{\lambda}(L)) \ar[r]^-{\tw_{\theta^{\iota}}^{\cl}} & H^d_c(Y_{1}(\mathfrak{n f}^2), {\mathscr L}_{\lambda}(L))\\
H^d_c(Y_{1}(\mathfrak n),{\mathscr L}_{\lambda}(E)) \ar[r]_-{\tw_{\theta}}  \ar[u]^-{\iota} & H^d_c(Y_{1}(\mathfrak{n f}^2), {\mathscr L}_{\lambda}(E)) \ar[u]_-{\iota}.
}
\end{equation}
is commutative. 

We also record another adelic cochain computation, which will be used in the proof of Lemma \ref{lemma:adelic-twisting-expansion}.

\begin{prop}\label{prop:twisting-adelic-cochains}
If $\psi \in H^\ast_c(Y_{1}(\mathfrak n),{\mathscr L}_{\lambda}(L))$ is represented by $\widetilde \psi \in C_{\ad,c}^{\bullet}(K_1(\mathfrak n), \mathscr L_{\lambda}(L))$, then $\tw_{\theta^\iota}^{\cl}(\psi) \in H^\ast_c(Y_{1}(\mathfrak n\mathfrak f^2),{\mathscr L}_{\lambda}(L))$ is represented by $\tw_{\theta^{\iota}}^{\cl}(\widetilde \psi) \in C_{\ad,c}^{\bullet}(K_{1}(\mathfrak{n f}^2), \mathscr L_{\lambda}(L))$ whose value on a singular chain $\sigma = \sigma_\infty\otimes[g_f]$ is given by
\begin{equation*}
\tw_{\theta^\iota}^{\cl}(\widetilde \psi)(\sigma) = \theta^{\iota}(\det g_f) \sum_{t \in \Upsilon_{\mathfrak f}^\times} \theta^{\iota}(t) \begin{smallpmatrix} 1 & t_0 \\ & 1 \end{smallpmatrix}\cdot \widetilde \psi(\sigma \begin{smallpmatrix} 1 & t_0 \\ & 1 \end{smallpmatrix}).
\end{equation*}
\end{prop}
\begin{proof}
First, $\theta_{\det}^{\iota} \in H^0(Y_{11}(\mathfrak nf^2), L)$ is given by $g \mapsto \theta^{\iota}(\det g)$ and it is clearly represented on the level of adelic cochains by $\sigma_\infty\otimes[g_f] \mapsto \theta^{\iota}(\det g_f)$ (since $\theta^{\iota}$ is trivial on $(F_\infty^\times)^{\circ}$). Comparing our claim with the definition of $\tw_{\theta^{\iota}}$, it is enough to show that $v_{t,p}^{\ast}(\psi)$ is represented by the adelic cochain
\begin{equation}\label{eqn:formula-to-show}
v_{t,p}^{\ast}(\widetilde \psi)(\sigma) = \begin{smallpmatrix} 1 & t\\ & 1\end{smallpmatrix} \cdot \widetilde \psi(\sigma \begin{smallpmatrix} 1 & t\\ & 1\end{smallpmatrix})
\end{equation} 
for any $t \in \mathbf A_{F,f}$. According to the definition \eqref{eqn:p-vtstar-diagram} above, $v_{t,p}^{\ast}$ is the composition of three maps. The first map is the pullback of a projection. The second is the map induced by right multiplication by $u_t$. The third map is the map $P \mapsto u_t \cdot P$ on the level of local systems $\mathscr L_\lambda(L)(u_t) \mapsto \mathscr L_\lambda(L)$. Thus the computation \eqref{eqn:formula-to-show} of $v_{t,p}^{\ast}(\widetilde \psi)$ is immediate from the explanation following Proposition \ref{prop:canonical}.
\end{proof}

\begin{rmk}\label{rmk:twisting-compatibility}
The classical twisting \eqref{eqn:padic-twisting-finite-order} defined here compares directly with the twisting in Definition \ref{defn:p-adic-big-twisting}. Suppose that $\vartheta = \theta^{\iota}$ is a finite order $p$-adic Hecke character of $\Gamma_F$. We can apply the above discussion to $\mathfrak n \cap \mathbf p$ and then deduce is a commuting diagram
\begin{equation*}
\xymatrix{
H^{\ast}_c(Y_1(\mathfrak n\cap \mathbf p), \mathscr D_{\lambda}) \ar[r]^-{\tw_{\vartheta}} \ar[d]_-{I_{\lambda}} & H^{\ast}_c(Y_1(\mathfrak n\cap \mathbf p), \mathscr D_{\lambda}(\vartheta)) \ar[r]^-{I_{\lambda}} & H^{\ast}_c(Y_{1}(\mathfrak n\cap  \mathbf p), \mathscr L_{\lambda}^{\sharp}(\vartheta)) \ar[d]^-{\sum \theta^{\iota}(t)v_{t_0}^{\ast}}\\
H^{\ast}_c(Y_{1}(\mathfrak n\cap \mathbf p), \mathscr L_\lambda^{\sharp}) \ar[rr]^-{\tw_{\theta^{\iota}}^{\cl}}& & H^{\ast}_c(Y_{1}((\mathfrak n \cap \mathbf p) \mathfrak f^2), \mathscr L_{\lambda}^{\sharp}),
}
\end{equation*}
where the right vertical arrow makes implicit use of the identity map inducing an isomorphism $\mathscr L_{\lambda}(\vartheta) \simeq \mathscr L_{\lambda}$ of local systems on $Y_{11}(\mathfrak{(n\cap \mathbf p) f^2})$. (The $\sharp$-notation is defined in Definition \ref{defn:sharp-notation}.)
\end{rmk}

%% file: eigen.tex
In this section we assume that $\mathfrak n$ is an integral ideal that is co-prime to $p$. Our goal is to define a certain eigenvariety of tame level $\mathfrak n$ and then show that reasonable classical points are smooth on this eigenvariety. Since this section is likely perceived as more technical than others in this article, let us elaborate on our motivations before continuing.

Eigenvarieties are families, parametrized by weights, of systems of Hecke eigenvalues (`eigensystems') that generalize $p$-adic systems of eigenvalues appearing in classical spaces of automorphic forms. For our purposes, the relevant eigensystems appear in the $\mathscr D_{\lambda}$-valued cohomology as $\lambda$ varies over $p$-adic weights and are required to have non-vanishing eigenvalues on Hecke operators at places dividing the prime $p$.

The eigenvariety is necessary in this article, first of all, to make sense of the variation statement in part e.\ of Theorem \ref{thm:intro-main-theorem}. It is also required for a less tautological reason. The $p$-adic $L$-functions constructed in Section \ref{sec:padicL-functions} arise from applying the period maps of Section \ref{sec:period-maps} to distribution-valued cohomology classes. If each classical eigensystem obviously gave rise to a unique distribution-valued cohomology class, the story would end there. However, the supposition is unclear. In fact, by definition, it is only the ``non-critical'' eigensystems (see Definition \ref{defn:classical-noncritical}) that lift via the integration map. For other eigensystems, which are called critical, the situation is more subtle. Bella\"iche's work when $F = \mathbf Q$ (\cite{Bellaiche-CriticalpadicLfunctions}) suggests that identifying a canonical cohomology class is linked to the smoothness of an eigenvariety.\footnote{In {\em loc.\ cit.}, Bella\"iche clearly also credits Chenevier with the particular smoothness argument and its consequences.} This link is made precise by commutative algebra in Section \ref{subsec:consequences}.

So, in this section we start by defining the eigenvariety we use and then establishing its basic features, before moving on to justifying the smoothness statement. The essential source of all difficulty is that, in general, distribution-valued cohomology is supported in many degrees, not just the middle degree $d = (F:\mathbf Q)$. Related to this, the canonical eigenvariety is constructed from all cohomology classes may be non-reduced and may contain irreducible components of varying dimensions. (See \cite{Hansen-Overconvergent} or \cite{Urban-Eigenvarieties} for some discussion on the very interesting problem of dimensions of components of eigenvarieties.) Both features are problematic for applying various ``soft" $p$-adic analytic arguments from the literature, which usually only work well for eigenvarieties of maximal dimension.  The purpose of Sections \ref{subsec:wt-space} through \ref{subsec:middle-degree} is to define what we call the middle-degree eigenvariety $\mathscr E(\mathfrak n)_{\rmmid}$, which is an eigenvariety that parametrizes eigensystems supported only in middle degree. Conjecturally, $\mathscr E(\mathfrak n)_{\rmmid}$ is the complement of the irreducible components of non-maximal dimension within the larger canonical eigenvariety, and in principle $\mathscr E(\mathfrak n)_{\rmmid}$ contains all points corresponding to classical cohomological $p$-refined automorphic representations. The main technical requirements for the definition are the introduction of an auxiliary eigenvariety constructed from Borel--Moore homology and related spectral sequences from \cite{Hansen-Overconvergent}. We prove $\mathscr E(\mathfrak n)_{\rmmid}$ is equidimensional of the maximal possible dimension and reduced (see Proposition \ref{prop:density-results} and Theorem \ref{theorem:reducedness}).

We then prove the key smoothness statement in Sections \ref{subsec:interlude-galoisrep} and \ref{subsec:smoothness}. The high-level strategy goes back to Hida and Mazur and the earliest connections between $p$-adic modular forms and deformations of Galois representations. Namely, we  aim to control the rigid local rings on the eigenvariety by defining and controlling a deformation problem on Galois representations instead. The deformation condition imposed, which is called either being refined or weakly-refined depending on the reference, was discovered in the setting of $F = \mathbf Q$ by Kisin (\cite{Kisin-OverconvergentModularForms}) and sometimes is also referred to as the $p$-adic interpolation of crystalline eigenvalues. In the case at hand, the interpolation of crystalline eigenvalues is due to Liu (\cite{Liu-Triangulations}) and Kedlaya, Pottharst, and Xiao (\cite{KedlayaPottharstXiao-Finiteness}). The relevant smoothness theorems for $F = \mathbf Q$ (for cuspidal cohomological automorphic forms) were established by Bella\"iche (\cite{Bellaiche-CriticalpadicLfunctions}). We take his argument as a model for our own. A portion of the argument requires $p$-adic Hodge theory, including the theory of $(\varphi, \Gamma)$-modules, in the background. We have tried to provide many references and suggestions for the reader unfamiliar with these theories, though nearly all that we need is already contained in the literature (albeit sometimes in an obscured form). Appendix \ref{app:semistable} discusses technical points on $p$-adic Galois representations and the main result there is required in order to include, in our key smoothness statement, the possibility of automorphic representations that are special at some $p$-adic places.

The non-critical versus critical dichotomy is important to be aware of, although it is ultimately immaterial for the smoothness statement. If the reader would like a geometric way to understand the distinction, Proposition \ref{prop:etaleness} shows the weight map is smooth at non-critical points, whereas it is only the eigenvariety itself that is smooth at critical points. For $p$-adic automorphic forms in more general settings, the dichotomy is more striking and the geometry of the attendant eigenvarieties is more interesting. The interested reader could begin by focusing their attention on Remark \ref{rmk:critical-miracle}, leading to the reference \cite{Bergdall-Smoothness} or the work of Breuil, Hellmann, and Schraen (see \cite{BreuilHellmannSchraen-Classicality, BreuilHellmannSchraen-LocalModel}, for instance).

\subsection{A weight space}\label{subsec:wt-space}
Recall the notation from the start of Section \ref{subsec:padic-twisting}. View  $\overline{\mathcal O_{F,+}^\times} \subset T(\mathcal O_p)$ as a closed subgroup via the diagonal embedding.

\begin{defn}
$\mathscr W(1) := \mathscr X(T(\mathcal O_p)/\overline{\mathcal O_{F,+}^\times})$.
\end{defn}
The dimension of $\mathscr W(1)$ as a rigid analytic space is $1+ d + \delta_{F,p}$ where $\delta_{F,p}$ is the Leopoldt defect, defined here to be one less than the dimension of $\mathcal O_p^\times/\overline{\mathcal O_F^\times}$ as a CPA group. There is a natural closed immersion $\mathscr W(1) \rightarrow \mathscr W$ and every cohomological weight defines a point in $\mathscr W(1)(\overline{\mathbf Q}_p)$.\footnote{We could have also considered a more general $p$-adic weight space. Namely, we could also take $\mathscr W(\mathfrak n)$ defined to be those continuous characters of $T(\mathcal O_p)$ which vanish on the finite index subgroup $\Gamma(\mathfrak n) \subset \mathcal O_{F,+}^\times$ of units $u$ which are congruent to $1 \bmod \mathfrak n$. Then $\mathscr W(1) \subset \mathscr W(\mathfrak n)$ is an open and closed embedding onto a union of connected components containing all the cohomological weights. But the local systems $\mathscr D_{\lambda}$ at level $\mathfrak n \mathbf p$ considered below are non-trivial exactly for $\lambda \in \mathscr W(\mathfrak n)$.} There is also a natural action of $\mathscr X(\mathcal O_p^\times/\overline{\mathcal O_{F,+}^\times})$ on $\mathscr W(1)$ by central twisting (compare with Section \ref{subsec:padic-twisting}). We denote this action by $\eta \cdot \lambda$ for $\eta \in \mathscr X(\mathcal O_p^\times/\overline{\mathcal O_{F,+}^\times})$ and $\lambda \in \mathscr W(1)$.

\begin{defn}\label{defn:classical-weight}
A weight $\lambda \in \mathscr W(1)(\overline{\mathbf Q}_p)$ is called twist cohomological if it is in the $\mathscr X(\mathcal O_p^\times/\overline{\mathcal O_{F,+}^\times})(\overline{\mathbf Q}_p)$-orbit of the cohomological weights.
\end{defn}
The ambiguity in being simultaneously twist cohomological and cohomological is easy to control.
\begin{lem}\label{lem:cohomological-weights}
If $\lambda=(\kappa,w)$ and $\lambda'=(\kappa',w')$ are two cohomological weights and $\eta \in \mathscr X(\mathcal O_p^\times/\overline{\mathcal O_{F,+}^\times})(\overline{\mathbf Q}_p)$ such that $\lambda = \eta \cdot \lambda'$, then $\eta$ is of the form $z \mapsto z^n$ for some $n \in \mathbf Z$, $\kappa=\kappa'$, and $w = w'+2n$.
\end{lem}
We clarify before the proof that $z \mapsto z^n$ means the character on $\mathcal O_p^\times$ given by $z=(z_v) \mapsto \prod_{v \mid p} \prod_{\sigma \in \Sigma_v} \sigma(z_v)^n$.
\begin{proof}[Proof of Lemma \ref{lem:cohomological-weights}]
Write $\lambda = (\lambda_1,\lambda_2)$ and similarly for $\lambda'$. By assumption, we have $\lambda_i = \eta \lambda_i'$ for $i=1,2$. In particular $z^{\kappa} = \lambda_1 \lambda_2^{-1} = \lambda_1'\lambda_2'^{-1} = z^{\kappa'}$, so $\kappa = \kappa'$. Since $\kappa$ determines the parity of $w$ (and the same for $\kappa'$ and $w'$) we conclude that $w - w'$ is an even integer, say $w - w' = 2n$. We finally deduce $\eta = \lambda_1 \lambda_1'^{-1} = z^{w - w'\over 2} = z^n$, as claimed.
\end{proof}

Recall that if $X$ is a rigid analytic space and $Z \subset X(\overline{\mathbf Q}_p)$ is a subset then $Z$ is said to be accumulating if for each $z \in Z$ and $U$ a connected admissible open neighborhood of $z$, $Z \cap U$ is Zariski-dense in $U$.
\begin{lem}\label{lem:weight-density}
The twist cohomological weights in $\mathscr W(1)$ are Zariski-dense and accumulating.
\end{lem}
\begin{proof}
Clear.
\end{proof}

\subsection{Distribution-valued cohomology and eigenvarieties}\label{subsec:eigenvarieties}

We write $I \subset \GL_2(\mathcal O_p)$ for the subgroup of matrices that are upper triangular modulo $\varpi_p \mathcal O_p$. Since $I \subset \Delta$, each point $\Omega \rightarrow \mathscr W(1)$ defines a local system $\mathscr D_{\Omega}$ on $Y_{K_1(\mathfrak n)I}$ and so we get  associated $\mathscr O(\Omega)$-modules $H^{\ast}_c(\mathfrak n, \mathscr D_{\Omega}) := H^{\ast}_c(Y_{K_1(\mathfrak n)I}, {\mathscr D}_{\Omega})$ and $H^{\ast}_c(\mathfrak n, \mathbf D_{\Omega}^{\mathbf s}):=H^{\ast}_c(Y_{K_1(\mathfrak n)I},  \mathbf D_{\Omega}^{\mathbf s})$ (for $\mathbf s \geq \mathbf s(\Omega)$). We define $H^{\BM}_{\ast}(\mathfrak n, \mathscr A_{\Omega})$ and $H^{\BM}_{\ast}(\mathfrak n, \mathbf A_{\Omega}^{\mathbf s})$ similarly. Denote by $\mathbf T(\mathfrak n) \subset \mathbf T_{\mathbf Q_p}(K_1(\mathfrak n) I)$ the $\mathbf Q_p$-subalgebra generated just by the operators $T_v,S_v$ for $v \nmid \mathfrak n \mathbf p$ and $U_v$ for $v \mid p$. Because $\Delta$ contains the elements $\begin{smallpmatrix} \varpi_v \\ & 1 \end{smallpmatrix}$ for $v \mid p$, the algebra $\mathbf T(\mathfrak n)$ acts by  $\mathscr O(\Omega)$-linear endomorphisms on $H^{\ast}_c(\mathfrak n, \mathscr D_{\Omega})$, $H_{\ast}^{\BM}(\mathfrak n, \mathscr A_{\Omega})$, $H^{\ast}_c(\mathfrak n, \mathbf D_{\Omega}^{\mathbf s})$,  and $H_{\ast}^{\BM}(\mathfrak n, \mathbf A_{\Omega}^{\mathbf s})$. Finally we set $U_p := \prod_{v \mid p} U_v^{e_v} \in \mathbf T(\mathfrak n)$.

\begin{rmk} Before moving forward, we acknowledge that we will reference many results from \cite{Hansen-Overconvergent} below that are, strictly speaking, written with ordinary (co)homology rather than (co)homology with supports. The changes required in \cite{Hansen-Overconvergent} are either explained there, implicit there, or they are inconsequential and transparent. We will directly reference \cite{Hansen-Overconvergent} without further warning.
\end{rmk}

For the rest of this subsection, $\Omega$ will denote an affinoid open subdomain in $\mathscr W(1)$ and $\mathbf s$ will implicitly mean $\mathbf s \geq \mathbf s(\Omega)$. Since $\mathbf D_{\Omega}^{\mathbf s}$ and $\mathbf A_{\Omega}^{\mathbf s}$ are $\mathbf Q_p$-vector spaces, the homology $H_{\ast}^{\BM}(\mathfrak n, \mathbf A_{\Omega}^{\mathbf s})$ is computed by a Borel--Serre complex $C_{\bullet}^{\BM}(\mathfrak n, \mathbf A_{\Omega}^{\mathbf s})$. The cohomology $H^{\ast}_c(\mathfrak n, \mathbf D_{\Omega}^{\mathbf s})$ is also computed by a Borel--Serre cochain complex $C^{\bullet}_c(\mathfrak n,\mathbf D_{\Omega}^{\mathbf s})$ (similarly for $\mathscr A_{\Omega}$ and $\mathscr D_{\Omega}$). These are complexes whose terms are finite direct sums of copies of the coefficients, or possibly the invariants of such a complex by the action of a finite group (see \cite[Section 2.1]{Hansen-Overconvergent}). 

The operator $U_p$ lifts to a compact operator (which we abusively write using the same symbol) on $C_{\bullet}^{\BM}(\mathfrak n, \mathbf A_{\Omega}^{\mathbf s})$. The Fredholm series $f_{\Omega}(t) = \det\left(1 - tU_p | C_{\bullet}^{\BM}(\mathfrak n, \mathbf A_{\Omega}^{\mathbf s})\right)$ is an entire function in $t$ over $\mathscr O(\Omega)$, by \cite[Proposition 3.1.1]{Hansen-Overconvergent} it is independent of $\mathbf s$, and it behaves naturally under base change $\Omega \rightarrow \Omega'$. Write $f(t) \in \mathscr O(\mathscr W(1))\{\{t \}\}$ for the unique function whose restriction to each $\Omega$ is $f_{\Omega}$. Following \cite[Section 4.1]{Hansen-Overconvergent}, we say that a pair $(\Omega,h)$, with $h\geq 0$ a real number, is slope adapted if the series $f_{\Omega}$ admits a slope-$\leq h$ decomposition $f_{\Omega} = Q_{\Omega,h} R_{\Omega,h}$ (where $Q_{\Omega,h}$ is a polynomial; see \cite[Section 4]{AshStevens-OverconvergentCohomology}). In that case, $\mathscr Z_{\Omega,h} := \Sp(\mathscr O(\Omega)[t]/Q_{\Omega,h}\mathscr O(\Omega)[t])$ is naturally an affinoid open subdomain of the spectral variety $\mathscr Z \subset \mathscr W(1) \times \mathbf G_m$ for $f$. By \cite[Proposition 4.1.4]{Hansen-Overconvergent}, the $\mathscr Z_{\Omega,h}$ form an admissible covering of $\mathscr Z$, as $(\Omega,h)$ runs over slope adapted pairs. We summarize the facts we will need from \cite[Section 3.1]{Hansen-Overconvergent}.
\begin{prop}\label{prop:salient-slope-decomps}
Suppose that $(\Omega,h)$ is slope adapted.
\begin{enumerate}
\item $C_{\bullet}^{\BM}(\mathfrak n, \mathscr A_{\Omega})$  and $C_c^{\bullet}(\mathfrak n, \mathscr D_{\Omega})$ admit slope-$\leq h$ decompositions
\begin{align*}
C_{\bullet}^{\BM}(\mathfrak n,\mathscr A_{\Omega}) &\simeq C_{\bullet}^{\BM}(\mathfrak n, \mathscr A_{\Omega})_{\leq h} \oplus C_{\bullet}^{\BM}(\mathfrak n, \mathscr A_{\Omega})_{>h}\\
C_c^{\bullet}(\mathfrak n, \mathscr D_\Omega) &\simeq C_c^{\bullet}(\mathfrak n, \mathscr D_\Omega)_{\leq h} \oplus C_c^{\bullet}(\mathfrak n, \mathscr D_\Omega)_{>h}.
\end{align*}
\item $C^{\bullet}_c(\mathfrak n, \mathscr D_{\Omega})_{\leq h} \simeq \Hom_{\mathscr O(\Omega)}(C_{\bullet}^{\BM}(\mathfrak n, \mathscr A_{\Omega})_{\leq h},\mathscr O(\Omega))$.
\item The homology $H^{\BM}_{\ast}(\mathfrak n, \mathscr A_{\Omega})$ and cohomology $H_c^{\ast}(\mathfrak n, \mathscr D_{\Omega})$ also admit slope-$\leq h$ decompositions
\begin{align*}
H^{\BM}_{\ast}(\mathfrak n, \mathscr A_{\Omega}) &\simeq H^{\BM}_{\ast}(\mathfrak n, \mathscr A_{\Omega})_{\leq h} \oplus H^{\BM}_{\ast}(\mathfrak n, \mathscr A_{\Omega})_{>h}\\
H_c^{\ast}(\mathfrak n, \mathscr D_{\Omega}) &\simeq H_c^{\ast}(\mathfrak n, \mathscr D_{\Omega})_{\leq h} \oplus H_c^{\ast}(\mathfrak n, \mathscr D_{\Omega})_{>h}.
\end{align*}
\item $H^{\BM}_{\ast}(\mathfrak n, \mathscr A_{\Omega})_{\leq h} = H_\ast(C_\bullet^{\BM}(\mathfrak n, \mathscr A_{\Omega})_{\leq h})$ and $H_c^{\ast}(\mathfrak n, \mathscr D_{\Omega})_{\leq h} = H^{\ast}(C_c^{\bullet}(\mathfrak n, \mathscr D_{\Omega})_{\leq h})$.
\item If $\Omega' \subset \Omega$ is an affinoid subdomain, then the slope-$\leq h$ parts in (1) and (3) naturally commute with base change $\mathscr O(\Omega) \rightarrow \mathscr O(\Omega')$.
\end{enumerate}
\end{prop}
\begin{proof}
See the second through the fifth propositions of \cite[Section 3.1]{Hansen-Overconvergent}.
\end{proof}
The complexes $C_\bullet^{\BM}(\mathfrak n, \mathscr A_{\Omega})_{\leq h}$ and $C^{\bullet}_c(\mathfrak n, \mathscr D_{\Omega})_{\leq h}$ are naturally complexes $\mathscr O(\mathscr Z_{\Omega,h})$-modules where $t \in \mathscr O(\mathscr Z_{\Omega,h})$ acts via $U_p^{-1}$.
\begin{prop}
There exists complexes of coherent $\mathscr O_{\mathscr Z}$-modules $\mathscr K^{\BM}_{\bullet}$ and $\mathscr K_c^{\bullet}$ on $\mathscr Z$ uniquely determined by the property that
\begin{align*}
\mathscr K^{\BM}_{\bullet}(\mathscr Z_{\Omega,h}) &\simeq C_{\bullet}^{\BM}(\mathfrak n, \mathscr A_{\Omega})_{\leq h}\\
\mathscr K_c^{\bullet}(\mathscr Z_{\Omega,h}) &\simeq C_c^{\bullet}(\mathfrak n, \mathscr D_\Omega)_{\leq h}
\end{align*}
for any slope adapted pair $(\Omega,h)$.
\end{prop}
\begin{proof}
This is proven just like \cite[Proposition 4.3.1]{Hansen-Overconvergent} (the essential point is Proposition \ref{prop:salient-slope-decomps}(5)).
\end{proof}

\begin{defn}
$\mathscr M^{\BM}_{\ast}$  (resp.\ $\mathscr M_c^{\ast}$) is the homology (resp.\ cohomology) sheaf of the complex $\mathscr K_\bullet^{\BM}$ (resp.\ $\mathscr K_c^{\bullet}$).
\end{defn}

Thus, $\mathscr M^{\BM}_{\ast}$ and $\mathscr M_c^{\ast}$ are graded coherent $\mathscr O_{\mathscr Z}$-modules and if $(\Omega,h)$ is a slope adapted pair, then $\mathscr M^{\BM}_{\ast}(\mathscr Z_{\Omega,h}) \simeq H_\ast^{\BM}(\mathfrak n, \mathscr A_{\Omega})_{\leq h}$ and $\mathscr M_c^{\ast}(\mathscr Z_{\Omega,h}) \simeq H^{\ast}_c(\mathfrak n, \mathscr D_{\Omega})_{\leq h}$. We further have natural ring morphisms
\begin{equation*}
\xymatrix{
 & \End_{\mathscr O(\mathscr Z_{\Omega,h})}\left(H_c^{\ast}(\mathfrak n, \mathscr D_{\Omega})_{\leq h}\right)\\
\mathbf T(\mathfrak n) \ar[dr]_-{\psi_{\Omega,h}'} \ar[ur]^-{\psi_{\Omega,h}}\\
 & \End_{\mathscr O(\mathscr Z_{\Omega,h})}\left(H_\ast^{\BM}(\mathfrak n, \mathscr A_{\Omega})_{\leq h}\right),
}
\end{equation*}
which glue to define morphisms of algebras $\psi: \mathbf T(\mathfrak n) \rightarrow \End_{\mathscr O_{\mathscr Z}}(\mathscr M_c^{\ast})$ and $\psi': \mathbf T(\mathfrak n) \rightarrow \End_{\mathscr O_{\mathscr Z}}(\mathscr M_\ast^{\BM})$. (Compare with the text prior to \cite[Definition 4.3.2]{Hansen-Overconvergent}. Notice also that it is the same if we replace $\End_{\mathscr O(\mathscr Z_{\Omega,h})}(-)$ with $\End_{\mathscr O(\Omega)}(-)$, the former being a subring of the latter.)

\begin{defn}
The eigenvariety $\mathscr E(\mathfrak n)$ (resp.\ $\mathscr E'(\mathfrak n)$) is the $\mathbf Q_p$-rigid analytic space associated to the eigenvariety datum $(\mathscr W(1), \mathscr Z, \mathscr M_c^{\ast}, \mathbf T(\mathfrak n), \psi)$ (resp.\ $(\mathscr W(1), \mathscr Z, \mathscr M_\ast^{\BM}, \mathbf T(\mathfrak n), \psi')$) as in \cite[Definition 4.3.2]{Hansen-Overconvergent}.
\end{defn}

\begin{rmk}
By  calling one $\mathscr E(\mathfrak n)$ and the other $\mathscr E'(\mathfrak n)$, we indicate our focus on the distribution-valued cohomology. The function-valued homology is only a technical tool used later (see Section \ref{subsec:middle-degree}). Thus, in what follows, we will only indicate homology versions of results when strictly necessary (the reader should not infer a lack of truth from their lack of exposition).
\end{rmk}

By definition (see \cite[Definition 4.2.1 and Theorem 4.2.2]{Hansen-Overconvergent}), the space $\mathscr E(\mathfrak n)$ is a $\mathbf Q_p$-rigid analytic space that comes equipped with a pair of maps $\upsilon : \mathscr E(\mathfrak n) \rightarrow \mathscr Z$, which is finite, and $\lambda: \mathscr E(\mathfrak n) \rightarrow \mathscr W(1)$, and a coherent (graded) sheaf $\mathscr M_c^{\ast,\dagger}$ of $\mathscr O_{\mathscr E(\mathfrak n)}$-modules equipped with a ring morphism $\psi: \mathbf T(\mathfrak n) \rightarrow \End_{\mathscr O_{\mathscr E(\mathfrak n)}}(\mathscr M_c^{\ast,\dagger})$ such that $\mathscr M_c^{\ast} \simeq \upsilon_{\ast}\mathscr M_c^{\ast,\dagger}$, with the isomorphism being compatible with the two possible morphisms we have called $\psi$.\footnote{The morphisms $\upsilon$ and $\lambda$ are, respectively, labeled $\pi$ and $w$ in \cite{Hansen-Overconvergent}.} The morphism $\upsilon$ and $\lambda$ factorize
\begin{equation}\label{eqn:weight-projection-factor}
\xymatrix{
\mathscr E(\mathfrak n) \ar[r]^-{\upsilon} \ar[dr]_-{\lambda} & \mathscr Z \ar[d]^-{\pr}\\
& \mathscr W(1)
}
\end{equation}
where $\pr: \mathscr Z \subset \mathscr W(1)\times \mathbf G_m \rightarrow \mathscr W(1)$ is the projection.  If $x \in \mathscr E(\mathfrak n)$ we prefer to write $\lambda_x \in \mathscr W(1)$ for its weight, rather than $\lambda(x)$. By \cite[Theorem 4.3.3]{Hansen-Overconvergent}, if $\lambda \in \mathscr W(1)$ is fixed, then the points $x \in \mathscr E(\mathfrak n)(\overline{\mathbf Q}_p)$ with $\lambda_x = \lambda$ are in bijection with the ring morphisms $\psi_x: \mathbf T_\lambda(\mathfrak n) \rightarrow \overline{\mathbf Q}_p$ where
\begin{equation*}
\mathbf T_\lambda(\mathfrak n) := \invlim_{h \rightarrow \infty} \im\left(\mathbf T(\mathfrak n) \rightarrow \End_{k_\lambda}(H^{\ast}_c(\mathfrak n, \mathscr D_{\lambda})_{\leq h})\right).
\end{equation*}
Given $x \in \mathscr E(\mathfrak n)(\overline{\mathbf Q}_p)$, we write $\mathfrak m_x \subset \mathbf T(\mathfrak n)$ for the maximal ideal 
\begin{equation*}
\mathfrak m_x := \ker\left(\mathbf T(\mathfrak n) \rightarrow \mathbf T_\lambda(\mathfrak n) \overset{\psi_x}{\longrightarrow} \overline{\mathbf Q}_p\right).
\end{equation*}
We also write $k_x$ for the residue field of $x$.

The rigid analytic spaces $\mathscr Z$ and $\mathscr W(1)$ are both equidimensional of the same dimension. Since the map $\upsilon$ in \eqref{eqn:weight-projection-factor} is finite, every irreducible component of $\mathscr E(\mathfrak n)$ has dimension at most $\dim \mathscr Z = \dim \mathscr W(1) = 1 + d + \delta_{F,p}$. The space $\mathscr E(\mathfrak n)$ is generally {\em not} equidimensional beyond the case $F = \mathbf Q$. For instance, if $d>1$ there is always an Eisenstein component of $\mathscr E(\mathfrak n)$ of dimension strictly smaller than $1+d+\delta_{F,p}$.

\begin{prop}\label{prop:openness-maximal-dimensional}
If $X \subset \mathscr E(\mathfrak n)$ is an irreducible component of (maximal) dimension $1+d+\delta_{F,p}$, then $\lambda(X) \subset \mathscr W(1)$ is Zariski-open.
\end{prop}
\begin{proof}
The map $\upsilon$ is finite and $X$ is closed in $\mathscr E(\mathfrak n)$, so $\upsilon(X) \subset \mathscr Z$ is closed. Moreover, it is evidently irreducible of dimension $\dim \mathscr Z$. Thus $\upsilon(X)$ is an irreducible component of $\mathscr Z$ (\cite[Corollary 2.2.7]{Conrad-IrredComponents}). Since the irreducible components of the Fredholm variety $\mathscr Z$ are all defined by Fredholm hypersurfaces (\cite[Proposition 4.1.2]{Hansen-Overconvergent}), we deduce $\lambda(X) = \pr(\upsilon(X))$ is Zariski-open in $\mathscr W(1)$ from \cite[Proposition 4.1.3]{Hansen-Overconvergent}.
\end{proof}

Having described $\mathscr E(\mathfrak n)$ by its defining characteristics, we will also need to briefly give an atlas. The eigenvariety $\mathscr E(\mathfrak n)$ is admissibly covered by affinoid subdomains $\mathscr E_{\Omega,h}:=\Sp(\mathbf T_{\Omega,h})$ where $\mathbf T_{\Omega,h}$ is the $\mathscr O(\Omega)$-algebra generated by the image of $\psi_{\Omega,h}$ inside $\End_{\mathscr O(\Omega)}(H^{\ast}_c(\mathfrak n, \mathscr D_{\Omega})_{\leq h})$ and $(\Omega,h)$ runs over slope adapted pairs. The sections $\mathscr M^{\ast,\dagger}_c(\mathscr E_{\Omega,h})$ are canonically identified with $\mathscr M^{\ast}_c(\mathscr Z_{\Omega,h}) = H^{\ast}_c(\mathfrak n, \mathscr D_{\Omega})_{\leq h}$.  This follows from the construction of eigenvarieties as in the proof of \cite[Theorem 4.2.2]{Hansen-Overconvergent}. By a slight abuse of notation, we will drop the $\dagger$ from the notation completely. Context should clarify whether an instance of the notation $\mathscr M_c^{\ast}$ is the sheaf on $\mathscr Z$ or $\mathscr E(\mathfrak n)$.

To set notations for an atlas on $\mathscr E'(\mathfrak n)$, it is covered by affinoid subdomains $\mathscr E'_{\Omega,h} := \Sp(\mathbf T_{\Omega,h}')$ where $\mathbf T_{\Omega,h}'$ is the $\mathscr O(\Omega)$-algebra generated by the image of $\psi_{\Omega,h}'$ inside $\End_{\mathscr O(\Omega)}(H_{\ast}^{\BM}(\mathscr A_{\Omega})_{\leq h})$ and $(\Omega,h)$ is a slope adapted pair.  There is also a graded sheaf $\mathscr M_{\ast}^{\BM}$ on $\mathscr E'(\mathfrak n)$ whose sections are given by $\mathscr M_{\ast}^{\BM}(\mathscr E_{\Omega,h}') \simeq H_{\ast}^{\BM}(\mathfrak n, \mathscr A_{\Omega})_{\leq h}$. 

\begin{defn}\label{defn:good-nbrhood}
Let $x \in \sce(\mathfrak{n})(\overline{\mathbf Q}_p)$.  A good neighborhood of $x$ is a connected affinoid open $U$ containing $x$ with the property that there exists a slope adapted pair $(\Omega, h)$ such that $U$ is a connected component of $\mathscr E_{\Omega,h}$.
\end{defn}

If $U$ is a good neighborhood of $x$ and $(\Omega,h)$ is as in the definition thereof, denote by $e_U \in \mathbf{T}_{\Omega,h}$ the idempotent so that $\sco(U)=e_U \mathbf{T}_{\Omega,h}$. Then, $\scm_c^{\ast}(U) \cong e_U H^{\ast}_c(\mathfrak{n},\scd_{\Omega})_{\leq h}$ is a Hecke-stable direct summand of $H^{\ast}_c(\mathfrak{n},\scd_{\Omega})_{\leq h}$.  The affinoid $U$ is completely determined by the triple $(\Omega,h,e_U)$, and we say that $U$ belongs to the slope adapted pair $(\Omega,h)$.

\begin{prop} \label{prop:good-neighborhoods}
For any $x \in \sce(\mathfrak{n})$, the collection of good neighborhoods of $x$ are cofinal in the collection of admissible opens containing $x$.
\end{prop}

\begin{proof} This proposition is a direct consequence of the construction of $\sce(\mathfrak{n})$.
\end{proof}

\subsection{Some special points}\label{subsec:special-points}
In this subsection, we catalog certain important points on $\mathscr E(\mathfrak n)$. Traditionally this would mean discussing ``classical points." Here we discuss, as well, twists of classical points by $p$-adic Hecke characters. Since we are not assuming the truth of Leopoldt's conjecture, we need to do this in order to unconditionally produce a dense set of points at which we have good \emph{a priori} control on the eigenvariety. 

For the moment, suppose that $\psi: \mathbf T(\mathfrak n) \rightarrow \overline{\mathbf Q}_p$ is a Hecke eigensystem and $\vartheta \in \mathscr X(\Gamma_F)(\overline{\mathbf Q}_p)$. Then we define a new Hecke eigensystem
\begin{equation}\label{eqn:twisting-hecke}
\tw_{\vartheta}(\psi)(T) := 
\begin{cases}
\vartheta(\varpi_v) \psi(T) & \text{if $T = T_v$ and $v \nmid \mathfrak n\mathbf p$ or $T=U_v$ and $v \mid p$;}\\
\vartheta(\varpi_v)^2\psi(T) & \text{if $T = S_v$ and $v \nmid \mathfrak n \mathbf p$.}
\end{cases}
\end{equation}
Let $\mathfrak m_{\psi} = \ker(\psi)$ and similarly set $\mathfrak m_{\tw_{\vartheta}(\psi)} = \ker(\tw_{\vartheta}(\psi))$. Recall that in Definition \ref{defn:p-adic-big-twisting} we introduced a linear map $\tw_{\vartheta}$ on the distribution-valued cohomology (see Lemma \ref{lemma:distr-twist} also).
\begin{lem}\label{lemma:twisting}
\leavevmode
\begin{enumerate}
\item $v_p(\psi(U_v)) = v_p(\tw_{\vartheta}(\psi)(U_v))$ for each $v \mid p$.
\item The linear map $\tw_{\vartheta}$ induces an isomorphism
\begin{equation*}
\tw_{\vartheta}: H^{\ast}_c(\mathfrak n, \mathscr D_{\lambda})_{\mathfrak m_\psi} \overset{\simeq}{\longrightarrow} H^{\ast}_c(\mathfrak n, \mathscr D_{\vartheta^{-1}\cdot \lambda})_{\mathfrak m_{\tw_{\vartheta}(\psi)}}.
\end{equation*}
\end{enumerate}
\end{lem}
\begin{proof}
The group $\Gamma_F$ is compact, so $\vartheta(\varpi_v)$ is a unit for all places $v$. That proves part (1). For part (2), $\tw_{\vartheta}$ defines an isomorphism on the level of vector spaces (before localizing) because its inverse is $\tw_{\vartheta^{-1}}$. The compatibility with the Hecke action follows from Proposition \ref{prop:twisting-hecke-eigenvalues}.
\end{proof}
Lemma \ref{lemma:twisting} implies the following is well-posed.
\begin{defn}
If $x \in \mathscr E(\mathfrak n)(\overline{\mathbf Q}_p)$ and $\vartheta \in \mathscr X(\Gamma_F)(\overline{\mathbf Q}_p)$, then we define $\tw_{\vartheta}(x) \in \mathscr E(\mathfrak n)(\overline{\mathbf Q}_p)$ to be the point corresponding to the Hecke eigensystem $\tw_{\vartheta}(\psi_x)$.
\end{defn}
One can view twisting by characters of $\Gamma_F$ as giving a group action of $\mathscr X(\Gamma_F)(\overline{\mathbf Q}_p)$ on $\mathscr E(\mathfrak n)(\overline{\mathbf Q}_p)$ compatible with the weight twisting in that
\begin{equation}\label{eqn:twisting-diagram}
\xymatrixrowsep{3pc}
\xymatrixcolsep{6pc}
\xymatrix{
\mathscr X(\Gamma_F)(\overline{\mathbf Q}_p) \times \mathscr E(\mathfrak n)(\overline{\mathbf Q}_p) \ar[r]^-{(\vartheta,x)\mapsto \tw_{\vartheta}(x)} \ar[d]_-{\left(\vartheta|_{\mathcal O_p^\times},  \lambda\right)} & \mathscr E(\mathfrak n)(\overline{\mathbf Q}_p) \ar[d]^-{\lambda}\\
\mathscr X(\mathcal O_p^\times/\overline{\mathcal O_{F,+}^\times})(\overline{\mathbf Q}_p) \times \mathscr W(1)(\overline{\mathbf Q}_p)  \ar[r]_-{(\eta,\lambda) \mapsto \eta^{-1}\cdot \lambda} & \mathscr W(1)(\overline{\mathbf Q}_p)
}
\end{equation}
is a commuting diagram. Of course, this is completely functorial and then gives actions on the level of rigid analytic groups. 

\begin{lem}\label{lem:twist-coh-weight-well-posed}
For $x\in \mathscr E(\mathfrak n)(\overline{\mathbf Q}_p)$, $x$ is in the $\mathscr X(\Gamma_F)(\overline{\mathbf Q}_p)$-orbit of a point of cohomological weight if and only if $\lambda_x$ is twist cohomological (Definition \ref{defn:classical-weight}).
\end{lem}
\begin{proof}
By \eqref{eqn:twisting-diagram}, if $x = \tw_{\vartheta}(x')$ and $x'$ has cohomological weight, then $x$ has twist cohomological weight. On the other hand, suppose that $\lambda_x = \eta \cdot \lambda$ where $\lambda$ is a cohomological weight and $\eta \in \mathscr X(\mathcal O_p^\times/\overline{\mathcal O_{F,+}^\times})(\overline{\mathbf Q}_p)$. Then, choose any one of the finite number of extensions $\vartheta$ of $\eta$ to a character of $\Gamma_F$ and set $x' = \tw_{\vartheta}(x)$. By \eqref{eqn:twisting-diagram} again, $x'$ has weight $\lambda$ and thus $x = \tw_{\vartheta^{-1}}(x')$ is in the $\mathscr X(\Gamma_F)(\overline{\mathbf Q}_p)$-orbit of a point of cohomological weight.
\end{proof}

Now suppose that $\pi$ is a cohomological cuspidal automorphic representation whose prime-to-$p$ conductor divides $\mathfrak n$. Then, each choice of $p$-refinement $\alpha$ for $\pi$ defines a Hecke eigensystem $\psi_{(\pi,\alpha)}:\mathbf T(\mathfrak n) \rightarrow \overline{\mathbf Q}_p$, depending on $\iota$. Write $\mathfrak m_{(\pi,\alpha)} = \ker(\psi_{(\pi,\alpha)})  \subset \mathbf T(\mathfrak n)$. If $L \subset \overline{\mathbf Q}_p$ denotes the residue field of $\psi_{(\pi,\alpha)}$ then $H^{\ast}_c(\mathfrak n, \mathscr L_{\lambda}(L))_{\mathfrak m_{(\pi,\alpha)}} \neq (0)$. 

Recall the $\sharp$-twisting in Definition \ref{defn:sharp-notation}, which allows for comparison between $\mathscr L_\lambda$ and $\mathscr D_\lambda$. Given $(\pi,\alpha)$ we define $\psi_{(\pi,\alpha)}^{\sharp}:\mathbf T(\mathfrak n) \rightarrow \overline{\mathbf Q}_p$ to be the ring morphism where $\psi_{(\pi,\alpha)}^{\sharp}(T) = \psi_{(\pi,\alpha)}(T)$ for $T = T_v$ or $T = S_v$ with $v \nmid \mathfrak n\mathbf p$ and
\begin{equation*}
\psi_{(\pi,\alpha)}^{\sharp}(U_v) = \alpha_v^{\sharp} = \varpi_v^{\kappa-w\over 2}\alpha_v = \varpi_v^{\kappa-w\over 2}\psi_{(\pi,\alpha)}(U_v) \;\;\;\;\; (\text{if $v\mid p$}).
\end{equation*}
We write $\mathfrak m_{(\pi,\alpha)}^{\sharp} =\ker(\psi_{(\pi,\alpha)}^{\sharp})$. Thus, $H^{\ast}_c(\mathfrak n, \mathscr L_{\lambda}^{\sharp}(L))_{\mathfrak m^{\sharp}_{(\pi,\alpha)}} \neq (0)$  and there is a canonical isomorphism \[H^{\ast}_c(\mathfrak n, \mathscr L_{\lambda}^{\sharp}(L))_{\mathfrak m^{\sharp}_{(\pi,\alpha)}} \cong H^{\ast}_c(\mathfrak n, \mathscr L_{\lambda}(L))_{\mathfrak m_{(\pi,\alpha)}}\]
of $L$-vector spaces that is equivariant for the prime-to-$p$ Hecke operators, and twisted-equivariant (in the evident sense) for the Hecke operators at $p$.

\begin{defn}\label{defn:classical-noncritical}
Let $x \in \mathscr E(\mathfrak n)(\overline{\mathbf Q}_p)$ be a point of cohomological weight $\lambda=(\kappa,w)$.
\begin{enumerate}
\item $x:=x(\pi,\alpha)$ is called classical if $\psi_x=\psi^{\sharp}_{(\pi,\alpha)}$ for some (unique) $p$-refined cuspidal automorphic representation $(\pi,\alpha)$ of weight $\lambda$ and prime-to-$p$ conductor dividing $\mathfrak n$. In this case we write $x = x(\pi,\alpha)$. We refer to the prime-to-$p$ conductor of $x$ as the prime-to-$p$ conductor of $\pi$.
\item $x$ is called non-critical if $x$ is classical and the integration map
\begin{equation*}
I_{\lambda}: H^{\ast}_c(\mathfrak n, \mathscr D_{\lambda}\otimes_{k_{\lambda}} k_x)_{\mathfrak m_x} \rightarrow H^{\ast}_c(\mathfrak n, \mathscr L_\lambda^{\sharp}(k_x))_{\mathfrak m_x}
\end{equation*}
is an isomorphism.
\end{enumerate}
\end{defn}

We stress that $(\pi,\alpha)$ being $p$-refined, for us, includes the condition that $\pi$ is Iwahori-spherical at places dividing $p$, i.e. either an unramified twist of a special representation or an unramified principal series.

We will extend these definitions below, and then we will also give numerical criteria for point to be non-critical. First, we check that being non-critical is stable (among classical points) under twisting.

\begin{lem}\label{lem:twist-noncritical-well-posed}
Suppose that $x,x' \in \mathscr E(\mathfrak n)(\overline{\mathbf Q}_p)$ are classical points and $x = \tw_{\vartheta}(x')$ for some $\vartheta \in \mathscr X(\Gamma_F)(\overline{\mathbf Q}_p)$. Then, the following conclusions hold.
\begin{enumerate}
\item $\vartheta = \mathbf N_p^ n \vartheta'$ for $\vartheta'$ an unramified Artin character and $n \in \mathbf Z$.
\item $x$ is non-critical if and only if $x'$ is non-critical.
\end{enumerate}
\end{lem}
\begin{proof}
We first prove (1). By Lemma \ref{lem:twist-coh-weight-well-posed} and Lemma \ref{lem:cohomological-weights}, there exists an $n \in \mathbf Z$ such that $\vartheta|_{\mathcal O_p^\times}$ is $z \mapsto z^n$. Thus $\vartheta':= \vartheta \mathbf N_p^{-n}$ is trivial on $\mathcal O_p^\times$. We deduce from \eqref{eqn:ses-GammaF} that it factors through a character of the narrow class group, as promised.

For point (2) we use the notation of the previous paragraph, and we also write $\lambda_x = \lambda$ and $\lambda_{x'} = \lambda'$. We can write $\vartheta' = (\theta')^{\iota}$ where $\theta'$ is a finite order, unramified Hecke character. So, it follows from Remark \ref{rmk:twisting-compatibility} that the diagram
\begin{equation*}
\xymatrixcolsep{6pc}
\xymatrix{
H^{\ast}_c(\mathfrak n, \mathscr D_{\lambda'}) \ar[r]^-{\tw_{\vartheta}}_-{\simeq} \ar[d]_-{I_\lambda'} & H^{\ast}_c(\mathfrak n, \mathscr D_{\lambda}) \ar[d]^-{I_\lambda}\\
H^{\ast}_c(\mathfrak n, \mathscr L_{\lambda'}^{\sharp}) \ar[r]_-{\tw_{\mathbf N_p^n\theta^{\iota}}}^-{\simeq} & H^{\ast}_c(\mathfrak n,\mathscr L_{\lambda}^{\sharp})
}
\end{equation*}
is commutative (see Remark \ref{rmk:adelic-norm-twisting} for including twists by the adelic norm).  Localizing at Hecke eigensystems, this proves the claim.
\end{proof}

Now consider a twist cohomological weight $\lambda$. Thus there exists a cohomological weight $\lambda_0 = (\kappa_0,w_0)$ and $\lambda = \eta\cdot \lambda_0$ for some $\eta$. If $\lambda_1 = (\kappa_1,w_1)$ is another cohomological weight that can twisted to $\lambda$, then Lemma \ref{lem:cohomological-weights} implies that $\kappa_0 = \kappa_1$. Thus we can always write a twist cohomological weight $\lambda = (\kappa,\ast)$ to mean $\lambda = \eta \cdot (\kappa,w)$ for some $w$. This allows us to define numerical criteria at points $x \in \mathscr E(\mathfrak n)(\overline{\mathbf Q}_p)$ of twist cohomological, not just cohomological, weight.

\begin{defn}\label{defn:numerical-criteria}
Let $x \in \mathscr E(\mathfrak n)(\overline{\mathbf Q}_p)$ be of twist cohomological weight $\lambda_x = (\kappa,\ast)$. We say that:
\begin{enumerate}
\item $x$ is twist classical if there exists a classical point $x' \in \mathscr E(\mathfrak n)(\overline{\mathbf Q}_p)$ and $\vartheta \in \mathscr X(\Gamma_F)(\overline{\mathbf Q}_p)$ such that $x = \tw_{\vartheta}(x')$.
\item $x$ is twist non-critical if $x=\tw_{\vartheta}(x')$ with $x'$ a classical, non-critical point.
\item $x$ has non-critical slope if $v_p(\psi_x(U_p)) < \inf_\sigma (1+\kappa_\sigma)$.
\item $x$ is extremely non-critical if $v_p(\psi_x(U_p)) < {1\over 2}\inf_{\sigma}(1+\kappa_\sigma)$.
\end{enumerate}
\end{defn}

Note that Definition \ref{defn:numerical-criteria} applies in particular to points of cohomological weight. Further, Lemma \ref{lem:twist-noncritical-well-posed} implies that whether or not $x$ is twist non-critical is independent of the choice of classical point in the definition thereof. Finally, whether or not a point has non-critical slope (resp.\ is extremely non-critical) can be checked before or after twisting (by Lemma \ref{lemma:twisting}).

By definition a twist non-critical point is twist classical, but {\em a priori} the points (3) and (4) do not assume classicality. Proposition \ref{prop:implications-extremely-non-critical} below fills in the only non-trivial implication in the chain:
\begin{equation*}
\text{extremely non-critical $\implies$ non-critical slope $\implies$ twist non-critical $\implies$ twist classical.}
\end{equation*}
To prove this, we need a lemma.

\begin{lem}\label{lemma:bound-below-eigenvalues}
If $\pi$ is a cohomological cuspidal automorphic representation and $\alpha$ is a $p$-refinement, then $0 \leq v_p(\alpha_v^{\sharp})$ for all $v \mid p$.
\end{lem}
\begin{proof}
If $L \subset \overline{\mathbf Q}_p$ is sufficiently large, then $H^{d}_c(\mathfrak n, \mathscr L_{\lambda}^{\sharp}(L))[\mathfrak m_{\pi,\alpha}^{\sharp}] \neq (0)$. But by Proposition \ref{prop:change-to-sharp}, the $U_v$-operator acting on $H^{d}_c(\mathfrak n, \mathscr L_{\lambda}^{\sharp}(L))$ preserves the integral lattice $H^{d}_c(\mathfrak n, \mathscr L_{\lambda}^{\sharp}(\mathcal O_L))$. Thus $\alpha_v^{\sharp}$ must be integral.
\end{proof}

\begin{prop}\label{prop:implications-extremely-non-critical}
Let $x \in \mathscr E(\mathfrak n)(\overline{\mathbf Q}_p)$ be of twist cohomological weight $\lambda$.
\begin{enumerate}
\item If $x$ has non-critical slope, then $x$ is twist non-critical.
\item If $x$ is extremely non-critical, then the action of $\mathbf T_{\lambda}(\mathfrak n)$ on $H^d_c(\mathfrak n, \mathscr D_{\lambda})_{\mathfrak m_x}$ is semi-simple.
\end{enumerate}
\end{prop}

Part (2) of this Proposition will play a key role in the proof that the eigenvariety $\mathscr{E}(\mathfrak{n})_{\mathrm{mid}}$ constructed in the next section is reduced, cf. Theorem \ref{theorem:reducedness}.

\begin{proof}
In case (1) (resp.\ (2)) we can write $x = \tw_{\vartheta}(x')$ where $x'$ has cohomological weight and $x'$ has non-critical slope (resp.\ is extremely non-critical). By Lemma \ref{lem:twist-noncritical-well-posed} in case (1) and Lemma \ref{lemma:twisting} in case (2), it suffices to replace $x$ by $x'$ and thus assume that $x$ has cohomological weight. In that case, point (1) follows from \cite[Theorem 3.2.5]{Hansen-Overconvergent}.\footnote{To make this calculation, one should take the Borel in \cite{Hansen-Overconvergent} to be the upper-triangular Borel and the element $t$ in \cite[Theorem 3.2.5]{Hansen-Overconvergent} to be $\begin{smallpmatrix} 1 &\\ & \varpi_p^{e_p}\end{smallpmatrix}$. Then, the $U_t$-operator in that reference is the $U_p$-operator in this paper (see Remark \ref{rmk:difference-with-daves-paper}).} 

We now prove (2) in the case $x$ has cohomological weight. First, by definition an extremely non-critical point has non-critical slope and so is non-critical by point (1). Thus $H^d_c(\mathfrak n, \mathscr D_{\lambda})_{\mathfrak m_x} \simeq H^d_c(\mathfrak n, \mathscr L_{\lambda}^{\sharp}(L))_{\mathfrak m_x}$. Now write $x = x(\pi,\alpha)$. It is known that the Hecke operators away from $\mathfrak n \mathbf p$ are semi-simple on the whole space $H^d_c(\mathfrak n, \mathscr L_{\lambda}^{\sharp}(L))$. If we localize at $\mathfrak m_x$ then the same is true for the operators $U_v$ when $\pi_v$ is Steinberg. Thus it remains to show that if $\pi_v$ is unramified, then the $U_v$ operator acts semi-simply. For that, it is sufficient to show that the two roots of $X^2 - a_v(\pi)X + \omega_\pi(\varpi_v)q_v$ are distinct. Here, $\omega_\pi(\varpi_v) = \zeta q_v^w$ where $\zeta$ is a root of unity, and $q_v = p^{f_v}$. In particular, it is enough to show that
\begin{equation}\label{eqn:sufficient-boundedness-condition}
v_p(\alpha_v) <  {f_v(1 + w)\over 2} = {1\over e_v} \sum_{\sigma \in \Sigma_v} {1 + w \over 2}.
\end{equation}
But $\alpha_v^{\sharp} = \psi_x(U_v)= \alpha_v \varpi_v^{\kappa-w\over 2}$ satisfies $v_p(\alpha_v^{\sharp}) \geq 0$ (Lemma \ref{lemma:bound-below-eigenvalues}) and, since $\psi_x(U_p) = \prod_{v \mid p} (\alpha_v^{\sharp})^{e_v}$ and $x$ is extremely non-critical, we see that 
\begin{equation*}
v_p(\alpha_v^{\sharp}) < {1\over e_v} \inf_{\sigma \in \Sigma_v} {1 + \kappa_{\sigma} \over 2} < {1\over e_v} \sum_{\sigma \in \Sigma_v} {1 + \kappa_{\sigma} \over 2}.
\end{equation*}
The bound \eqref{eqn:sufficient-boundedness-condition} follows immediately, completing the proof of (2).
\end{proof}

\begin{rmk}\label{rmk:extremely-non-critical-unramified-principal-series}
By Proposition \ref{prop:implications-extremely-non-critical}, any point $x$ that is extremely non-critical is of the form $x = \tw_{\vartheta}(x')$ where $x'$ is classical. With an assumption slightly stronger than extremely non-critical, slightly more can be said. Specifically, suppose that $x$ has twist cohomological weight $\lambda_x = (\kappa,\ast)$ and, moreover,
\begin{enumerate}
\item we have $\kappa_\sigma \geq 2$ for all $\sigma$, and
\item $v_p(\psi_x(U_p)) < \frac{1}{3}\inf_{\sigma}(1+\kappa_\sigma)$. 
\end{enumerate}
Then, we claim that $x = \tw_{\vartheta}(x')$ where $x'$ is a classical point such that the associated automorphic representation is an unramified principal series at each $v \mid p$. (The constants $2$ and $3$ in the conditions (1) and (2) are not particularly important. See the proof of Proposition \ref{prop:family-galois-properties}.)

Indeed, it suffices to show the claim holds when $x = (\pi,\alpha)$ is classical. Suppose instead that $\pi_v$ is Steinberg while (1) and (2) both hold. Since $\pi_v$ is Steinberg, $\alpha_v^2$ is a unit multiple of $q_v^w$. So, 
\begin{equation*}
v_p((\alpha_v^{\sharp})^{e_v}) = e_v \left(\frac{f_v w}{2} + \frac{1}{e_v}\sum_{\sigma \in \Sigma_v} \frac{\kappa_\sigma - w}{2}\right) = \sum_{\sigma \in \Sigma_v} \frac{\kappa_\sigma}{2}.
\end{equation*}
But, since $\kappa_\sigma \geq 2$ by (1) we have $\frac{\kappa_\sigma}{2} \geq \frac{1+\kappa_{\sigma}}{3}$ for any $\sigma \in \Sigma_v$ and thus
$$
v_p((\alpha_v^{\sharp})^{e_v}) \geq \sum_{\sigma \in \Sigma_v} \frac{1+\kappa_\sigma}{3} \geq \frac{1}{3}\inf_{\sigma}  (1 + \kappa_\sigma) > v_p(\psi_x(U_p)),
$$
using (2). But this contradicts that each $\alpha_v^{\sharp}$ is $p$-adically integral (Lemma \ref{lemma:bound-below-eigenvalues}), which completes the proof of the claim.
\end{rmk}

\subsection{The middle-degree eigenvariety}\label{subsec:middle-degree}

We now return to the eigenvarieties $\mathscr E(\mathfrak n)$. Recall the open affinoid charts $\mathscr E_{\Omega,h}=\Sp(\mathbf T_{\Omega,h})$ and $\mathscr E'_{\Omega,h} = \Sp(\mathbf T_{\Omega,h}')$ defined towards the end of Section \ref{subsec:eigenvarieties}. If $A$ is a commutative ring we write $A^{\red}$ for its nilreduction, and if $X$ is a rigid analytic space we write $X^{\red}$ for its nilreduction.

\begin{prop}\label{prop:reduced-factorization}
\leavevmode
\begin{enumerate}
\item If $(\Omega,h)$ is a slope adapted pair, then we have a natural commuting diagram
\begin{equation*}
\xymatrix{
\mathbf T(\mathfrak n) \otimes_{\mathbf Q_p} \mathscr O(\Omega) \ar@{>>}[r]^-{\psi_{\Omega,h}'} \ar@{>>}[d]_-{\psi_{\Omega,h}} & \mathbf T'_{\Omega,h} \ar@{.>}[d]\\
\mathbf T_{\Omega,h} \ar@{>>}[r] & \mathbf T_{\Omega,h}^{\red}
}
\end{equation*}
\item The morphisms $\mathbf T'_{\Omega,h} \rightarrow \mathbf T_{\Omega,h}^{\red}$ in part (1) glue to a canonical morphism $\tau: \mathscr E(\mathfrak n)^{\red} \rightarrow \mathscr E'(\mathfrak n)$.
\end{enumerate}
\end{prop}
We emphasize that the Hecke action on homology defined in \cite{Hansen-Overconvergent} is perhaps slightly nonstandard, but this normalization of the action is chosen exactly so the equivariance claim in the following proof holds true.
\begin{proof}
By \cite[Theorem 3.3.1]{Hansen-Overconvergent} there is a first quadrant spectral sequence
\begin{equation}\label{eqn:sp-seq}
E_2^{i,j} = \Ext_{\mathscr O(\Omega)}^{i}(H_j^{\BM}(\mathfrak n, \mathscr A_{\Omega})_{\leq h},\mathscr O(\Omega)) \Rightarrow H^{i+j}_c(\mathfrak n, \mathscr D_{\Omega})_{\leq h}
\end{equation}
which is equivariant for the action of $\mathbf T(\mathfrak n) \otimes_{\mathbf Q_p} \mathscr O(\Omega)$. Thus, if $T \in \ker(\psi_{\Omega,h}')$, then acts trivially on every term in the $E_2$-page for the spectral sequence \eqref{eqn:sp-seq}. In particular, that means that $T$ acts nilpotently on the abutment $H^{\ast}_c(\mathfrak n, \mathscr D_{\Omega})_{\leq h}$, which is what we wanted to show in (1).

The second part of the proposition is immediate from the construction of the eigenvariety and the local nature of the nilreduction.
\end{proof}

Now consider the graded sheaves $\mathscr M_\ast^{\BM} = \bigoplus_j \mathscr M_j^{\BM}$ on $\mathscr E'(\mathfrak n)$. Le $\tau$ be as in Proposition \ref{prop:reduced-factorization}(2). Since $\mathscr M_j^{\BM}$ is a coherent sheaf on $\mathscr E'(\mathfrak n)$, its pullback $\tau^{\ast}\mathscr M_j^{\BM}$ to $\mathscr E(\mathfrak n)^{\red}$ is also coherent. The natural map $i : \mathscr E(\mathfrak n)^{\red} \rightarrow \mathscr E(\mathfrak n)$ is a closed immersion, so $i_{\ast}\tau^{\ast}\mathscr M_j^{\BM}$ is thus a coherent sheaf on $\mathscr E(\mathfrak n)$. In particular, its support is a closed analytic subset. In general, we write $\supp(\mathscr M)$ for the support of a sheaf $\mathscr M$.

\begin{defn}\label{defn:middle-degree}
\begin{equation*}
\mathscr E(\mathfrak n)_{\rmmid} := \mathscr E(\mathfrak n) - \left[\left(\bigcup_{j=d+1}^{2d} \supp(\mathscr M_c^{j})\right) \cup \left(\bigcup_{j=0}^{d-1} \supp(i_{\ast}\tau^{\ast}\mathscr M_j^{\BM})\right)\right]
\end{equation*}
 \end{defn}
We immediately give a separate characterization of $\mathscr E(\mathfrak n)_{\rmmid}$. The entire reason for introducing the homology-based eigenvariety was to give Definition \ref{defn:middle-degree} because it is not clear that condition (2) in the next proposition gives a well-defined affinoid open subspace.

\begin{prop}\label{prop:middle-alternate-characterization}
If $x \in \mathscr E(\mathfrak n)(\overline{\mathbf{Q}}_p)$, then the following conditions are equivalent.
\begin{enumerate}
\item $x \in \mathscr E(\mathfrak n)_{\rmmid}(\overline{\mathbf Q}_p)$.
\item $H^j_c(\mathfrak n, \mathscr D_{\lambda_x}\otimes_{k_{\lambda_x}} k_x)_{\mathfrak m_x} \neq (0)$ if and only if $j = d$.
\end{enumerate}
Moreover, $\mathscr E(\mathfrak n)^{\rmmid} \cap \supp(\mathscr M^j_c)$ is empty if $0 \leq j \leq d-1$ also.
\end{prop}
\begin{proof}
This follows from \cite[Proposition 4.5.2]{Hansen-Overconvergent} and elementary manipulations of supports.
\end{proof}
Using Proposition \ref{prop:middle-alternate-characterization}, we can affirm that $\mathscr E(\mathfrak n)_{\rmmid}$ is non-empty, and afterwards we will go deeper into its properties.
\begin{lem}\label{lem:non-critical-pts-middle-degree}
Every non-critical point on $\mathscr E(\mathfrak n)$ belongs to $\mathscr E(\mathfrak n)_{\rmmid}$.
\end{lem}
\begin{proof}
If $x \in \mathscr E(\mathfrak n)(\overline{\mathbf Q}_p)$ is non-critical of cohomological weight $\lambda$, then $H^{\ast}_c(\mathfrak n, \mathscr D_{\lambda}\otimes_{k_\lambda} k_x)_{\mathfrak m_x}\simeq H^{\ast}_c(\mathfrak n, \mathscr L_{\lambda}^{\sharp}(k_x))_{\mathfrak m_x}$. Since cuspidal eigensystems in $H^{\ast}_c(\mathfrak n, \mathscr L_{\lambda})$ appear only in middle degree (see \cite{Harder}), we have $x \in \mathscr E(\mathfrak n)_{\rmmid}(\overline{\mathbf Q}_p)$ by Proposition \ref{prop:middle-alternate-characterization}.
\end{proof}

We note that $\mathscr E(\mathfrak n)_{\rmmid}$ is (non-empty) Zariski-open in $\mathscr E(\mathfrak n)$. In particular, if $x \in \mathscr E(\mathfrak n)_{\rmmid}$ then any sufficiently small good neighborhood $U$ of $x$ in $\mathscr E(\mathfrak n)$ is actually contained in $\mathscr E(\mathfrak n)_{\rmmid}$ (Proposition \ref{prop:good-neighborhoods}).

\begin{prop}\label{prop:looks-like-eigen}
\leavevmode
\begin{enumerate}
\item The coherent sheaf $\mathscr M^d_c|_{\mathscr E(\mathfrak n)_{\rmmid}}$ is flat over $\mathscr W(1)$. 
\item $\mathscr E(\mathfrak n)_{\rmmid}$ is admissibly covered by good neighborhoods $U$ belonging to slope adapated pairs $(\Omega,h)$  such that $\mathscr O(U)$ acts faithfully on the finite projective $\mathscr O(\Omega)$-module $\mathscr M^d_c(U) = e_UH^d_c(\mathfrak n, \mathscr D_{\Omega})_{\leq h}$.
\end{enumerate}
\end{prop}
\begin{proof}
For (1), we want to show that if $x \in \mathscr E(\mathfrak n)_{\rmmid}$ is of weight $\lambda=\lambda_x$, then for any slope adapated pair $(\Omega,h)$ the module $(H^d_c(\mathfrak n, \mathscr D_{\Omega})_{\leq h})_{\mathfrak m_x} = \mathscr M^d_c(\mathscr E_{\Omega,h})_{\mathfrak m_x}$ is finite free over $\mathscr O(\Omega)_{\mathfrak m_\lambda}$. To do this, we consider a second quadrant spectral sequence (\cite[Theorem 3.3.1]{Hansen-Overconvergent})
\begin{equation}\label{eqn:tor-sp-seq}
E_2^{i,j} = \Tor_{-i}^{\mathscr O(\Omega)_{\mathfrak m_\lambda}}(\mathscr M^j_c(\mathscr E_{\Omega,h})_{\mathfrak m_x}, k_\lambda) \Rightarrow (H^{i+j}_c(\mathfrak n, \mathscr D_{\lambda})_{\leq h})_{\mathfrak m_x}.
\end{equation}
If $j\neq d$ then, since $x \in \mathscr E(\mathfrak n)_{\rmmid}$, the $E_2^{i,j}$-term in \eqref{eqn:tor-sp-seq} vanishes for all $i$. Thus we deduce canonical isomorphisms
\begin{equation}\label{eqn:tor-isomorphisms}
\Tor_n^{\mathscr O(\Omega)_{\mathfrak m_\lambda}}(\mathscr M_c^d(\mathscr E_{\Omega,h})_{\mathfrak m_x},k_\lambda) \simeq (H^{d-n}_c(\mathfrak n, \mathscr D_\lambda)_{\leq h})_{\mathfrak m_x}
\end{equation}
for all $n \geq 0$. By Proposition \ref{prop:middle-alternate-characterization} we further deduce that either side of \eqref{eqn:tor-isomorphisms} vanishes for $n > 0$. By the local criterion for flatness (\cite[Section 22]{Matsumura-CommutativeRingTheory}), $\mathscr M^d_c(\mathscr E_{\Omega,h})_{\mathfrak m_x}$ is free over $\mathscr O(\Omega)_{\mathfrak m_\lambda}$. This proves (1).

Now we prove (2). First, it is immediate that $\mathscr E(\mathfrak n)_{\rmmid}$ is admissibly covered by good neighborhoods $U$ of $\mathscr E(\mathfrak n)$. By definition, $\mathscr O(U) = e_U\mathbf T_{\Omega,h}$ acts faithfully on $\mathscr M^{\ast}_c(U) = e_UH^{\ast}_c(\mathfrak n, \mathscr D_{\Omega})_{\leq h}$. But if $U \subset \mathscr E(\mathfrak n)_{\rmmid}$ and $j \neq d$, then $\Ann_{\mathscr O(U)}(\mathscr M^j_c(U)) = \mathscr O(U)$ by Proposition \ref{prop:middle-alternate-characterization}. We thus deduce that $\mathscr O(U)$ acts faithfully on $\mathscr M^d_c(U)$. Since $\mathscr M_c^d(U)$ is finite projective over $\mathscr O(\Omega)$ by part (1), we have completed the proof of  (2).
\end{proof}

\begin{rmk}\label{eqn:base-change-middle-degree}
Equation \eqref{eqn:tor-isomorphisms} shows, for instance, that if $x \in \mathscr E(\mathfrak n)_{\rmmid}$ is of weight $\lambda$ then the natural map $(H^d_c(\mathfrak n, \mathscr D_{\Omega})_{\leq h})_{\mathfrak m_x} \otimes_{\mathscr O(\Omega)} k_\lambda \rightarrow (H^d_c(\mathfrak n, \mathscr D_{\lambda})_{\leq h})_{\mathfrak m_x}$ is an isomorphism.
\end{rmk}

\begin{prop}\label{prop:density-results}
\leavevmode
\begin{enumerate}
\item $\mathscr E(\mathfrak n)_{\rmmid}$ is stable under twisting by $\mathscr X(\Gamma_F)$.
\item Every twist non-critical point on $\mathscr E(\mathfrak n)$ belongs to $\mathscr E(\mathfrak n)_{\rmmid}$.
\item If $X \subset \mathscr E(\mathfrak n)_{\rmmid}$ is an irreducible component then $\dim X = \dim \mathscr W(1)$ and $X$ is contained in a unique irreducible component of $\mathscr E(\mathfrak n)$.
\item The extremely non-critical points are a Zariski-dense accumulation subset of $\mathscr E(\mathfrak n)_{\rmmid}$.
\end{enumerate}
\end{prop}

Recall that if $X$ is a rigid space, a subset $Z \subset X$ is an \emph{accumulation subset} if every $z \in Z$ admits a basis of affinoid neighborhoods $U \subset X$ such that $U \cap Z$ is Zariski-dense in $U$.

\begin{proof}
Part (1) follows immediately from Proposition \ref{prop:middle-alternate-characterization} and Lemma \ref{lemma:twisting}. Part (2) then follows from part (1) and Lemma \ref{lem:non-critical-pts-middle-degree}.

From \cite[Theorem 1.1.6]{Hansen-Overconvergent} and Proposition \ref{prop:middle-alternate-characterization} we deduce that if $x$ is a point on $\mathscr E(\mathfrak n)_{\rmmid}$ then any irreducible component of $\mathscr E(\mathfrak n)$ passing through $x$ has dimension equal to $\dim \scw(1)$. Thus the claim (3) follows from \cite[Corollary 2.2.9]{Conrad-IrredComponents}.

Finally we prove (4). First, if $X \subset \mathscr E(\mathfrak n)_{\rmmid}$ is an irreducible component then $\lambda(X)$ is Zariski-open in $\mathscr W(1)$ (by part (3) and Proposition \ref{prop:openness-maximal-dimensional}). By Lemma \ref{lem:weight-density} we deduce that $X$ contains a point $x_0$ of twist cohomological weight. This reduces the statement of (4) to proving that extremely non-critical points are accumulating on a neighborhood near any point $x_0$ of twist cohomological weight.

Consider a good neighborhood $U \subset \mathscr E(\mathfrak n)_{\rmmid}$ of $x_0$. Say $U$ belongs to a slope adapted pair $(\Omega,h)$. First, $U$ is the rigid analytic spectrum of $\mathscr O(U)$. Second, Proposition  \ref{prop:looks-like-eigen} implies $\mathscr O(U)$ acts faithfully on the finite projective $\mathscr O(\Omega)$-module $\mathscr M^d_c(U)$. So, by \cite[Lemme 6.2.10]{Chenevier-pAdicAutomorphicForm}, the irreducible components of $U$ map surjectively onto $\Omega$, and by \cite[Lemme 6.2.8]{Chenevier-pAdicAutomorphicForm} we deduce that the pre-image $(\lambda|_U)^{-1}(Z) \subset U$ of any Zariski-dense subset $Z \subset \Omega$ is still Zariski-dense in $U$. Since $x_0$ has twist cohomological weight we conclude that $U$ contains a Zariski-dense accumulating set of points of twist cohomological weight. On the other hand, we can easily shrink $U$ so that $x \mapsto v_p(\psi_x(U_p))$ is constant on $U$ as well, and thus see clearly that in fact we can take a Zariski-dense accumulating subset of extremely non-critical points as claimed.
\end{proof}

We now pause for a lemma of commutative algebra.

\begin{lem}\label{lem:reducedness-lemma-john}
Suppose that $A$ is a noetherian integral domain of characteristic zero and $A \rightarrow B$ is a finite morphism with $B$ torsion free over $A$. Then, the following conditions are equivalent.
\begin{enumerate}
\item $B$ is reduced.
\item $A \rightarrow B$ is generically \'etale.
\item The support of $\Omega^{1}_{B/A}$ in $\Spec(B)$ has positive codimension. (See the beginning of the proof.)
\end{enumerate}
If furthermore $M$ is a finite projective $A$-module and $B$ is actually a commutative $A$-subalgebra of $\End_A(M)$ then these conditions are all equivalent to:
\begin{enumerate}
\setcounter{enumi}{3}
\item There exists a Zariski-dense subset $X \subset \Spec(A)$ such that $B$ has reduced image inside $\End_{A_{\mathfrak p}/\mathfrak p A_{\mathfrak p}}(M_{\mathfrak p}/\mathfrak p M_{\mathfrak p})$ for all $\mathfrak p \in X$.
\end{enumerate}
\end{lem}

Here we say a finite map of Noetherian rings $A\to B$ is generically \'etale if it satisfies either of the following two equivalent conditions: 
\begin{enumerate}[a.]
\item For all minimal primes $\mathfrak{p} \subset A$, the ring $B\otimes_A \mathrm{Frac}(A/\mathfrak{p})$ is a finite \'etale $\mathrm{Frac}(A/\mathfrak{p})$-algebra.
\item There is an open dense subsecheme $U \subset \Spec(A)$ such that $\Spec(B)\times_{\Spec(A)} U \rightarrow U$ is finite \'etale. 
\end{enumerate}

   These conditions are equivalent because the locus where $\Spec(B) \rightarrow \Spec(A)$ is not \'etale is closed in $\Spec(B)$ (see \cite[Proposition 3.8]{Milne-EC} for instance) and this locus has closed image in $\Spec(A)$ because $\Spec(B) \rightarrow \Spec(A)$ is proper ($B$ being finite over $A$).

\begin{proof}[Proof of Lemma \ref{lem:reducedness-lemma-john}]
If $\mathfrak p \in \Spec(A)$ write $k(\mathfrak p)$ for its residue field. When $\mathfrak p$ is the generic point, we write $K = k(\mathfrak p)$. 

We note first that the hypotheses imply that $B$ is equidimensional of the same dimension as $A$. This gives meaning to condition (3).  Now we will show that (1) and (2) are equivalent. Since $B$ is torsion free over $A$,  $B$ is reduced if and only if $B\otimes_A K$ is reduced. Thus it suffices to show that $B \otimes_A K$ is reduced if and only if $B\otimes_A K$ is a finite \'etale $K$-algebra. Since $K$ has characteristic zero, this follows from Wedderburn's theorem (see \cite[Prop.\ 3, Chap.\ VIII]{Bourbaki-Algebre8} for instance).

Our second claim is that (2) and (3) are equivalent. Since $A$ is reduced, noetherian and $A \rightarrow B$ is finite we have that $A \rightarrow B$ is generically flat (\cite[Theorem 6.9.1]{EGA4-2}). So being generically \'etale and generically unramified are equivalent, the latter being clearly equivalent to condition (3).

For the rest of the proof we will assume that $B$ is as in the ``furthermore''. It is elementary to check that $B$ is then a finite torsion free $A$-algebra, so that (1) through (3) are all equivalent. We will show that (2) implies (4) and (4) implies (1).

Begin by assuming (2) and choose a dense open subscheme $U \subset \Spec(A)$ that $A_{\mathfrak p} \rightarrow B\otimes_A A_{\mathfrak p}$ is finite \'etale for each $\mathfrak p \in U$. Then the fiber $B \otimes_A k(\mathfrak p)$ is a finite \'etale $k(\mathfrak p)$-algebra; in particular it is reduced. Since the natural map $B\rightarrow \End_{k(\mathfrak p)}(M\otimes_A k(\mathfrak p))$ factors through $B\otimes_A k(\mathfrak p)$ we see that $B$ has reduced image as in (4) for all $\mathfrak p \in U$ meaning we can take $X = U$ to witness (4).

Finally assume that (4) holds and consider such a set $X$. Since $M$ is projective over $A$ and $X$ is Zariski-dense in $\Spec(A)$, the natural map
\begin{equation*}
\End_A(M) \rightarrow \prod_{\mathfrak p \in X} \End_{k(\mathfrak p)}(M\otimes_A k(\mathfrak p))
\end{equation*}
is injective. Thus we deduce that
\begin{equation}\label{eqn:B-diagonal-embedding}
B \rightarrow \prod_{\mathfrak p \in X} \End_{k(\mathfrak p)}(M\otimes_A k(\mathfrak p))
\end{equation}
is also injective. On the other hand, $B$ has reduced image in each coordinate of \eqref{eqn:B-diagonal-embedding} by our assumption (4), so it follows that $B$ is reduced.
\end{proof}

The previous lemma is applied to prove the following theorem.

\begin{thm}\label{theorem:reducedness}
$\sce(\mathfrak{n})_{\rmmid}$ is reduced. 
\end{thm}
\begin{proof}
We proved in Proposition \ref{prop:looks-like-eigen} that $\mathscr E(\mathfrak n)_{\rmmid}$ is admissibly covered by good affinoid opens $U$ belonging to slope adapated pairs $(\Omega,h)$ such that $\mathscr O(U)$ is an $\mathscr O(\Omega)$-subalgebra of the endomorphism $\End_{\mathscr O(\Omega)}(\mathscr M_c^d(U))$, and $\mathscr M_c^d(U)$ is finite projective over $\mathscr O(\Omega)$. So, Lemma \ref{lem:reducedness-lemma-john} provides criteria to check that each $\mathscr O(U)$ is reduced, which is what we will do. If $U$ itself contains an extremely non-critical point, then condition (4) of Lemma \ref{lem:reducedness-lemma-john} holds by Proposition \ref{prop:implications-extremely-non-critical}. So, $\mathscr O(U)$ is reduced when $U$ contains an extremely non-critical point.

In general, let $Z$ be the support of $\Omega^1_{\mathscr E(\mathfrak n)_{\rmmid}/\mathscr W(1)}$. Point (3) of Lemma \ref{lem:reducedness-lemma-john} and the prior paragraph implies that $Z$ meets any $U$ containing an extremely non-critical  in a closed subspace of positive codimension. By Proposition \ref{prop:density-results}, such $U$ are Zariski-dense and accumulating on each irreducible component of $\mathscr E(\mathfrak n)_{\rmmid}$. Thus $Z$ does not contain any irreducible component of $\mathscr E(\mathfrak n)_{\rmmid}$. 

But now it follows that $Z$ itself has positive codimension in $\mathscr E(\mathfrak n)_{\rmmid}$ (see the argument in \cite[Corollary 2.2.7]{Conrad-IrredComponents} for instance) and {\em a fortiori} meets any good neighborhood $U$ (all of which are equidimensional), regardless of whether $U$ contains an extremely non-critical point, in a closed subspace of positive codimension. Thus, the equivalence between conditions (1) and (3) in Lemma \ref{lem:reducedness-lemma-john} prove that $\mathscr O(U)$ is reduced in general.
\end{proof}

\subsection{Interlude on Galois representations}\label{subsec:interlude-galoisrep}
If $K$ is a field and $\overline K$ is a fixed algebraic closure we write $G_K$ for the Galois group of $\overline K$ over $K$. Recall that if $K/\mathbf Q_{\ell}$ is finite extension, and if $\ell \neq p$, then any continuous representation $\rho: G_K \rightarrow \GL_2(\overline{\mathbf Q}_p)$ has a corresponding Weil--Deligne representation $\WD(\rho)$ (\cite{Tate-NumberTheoreticBackground}). When $\ell = p$ we use the language (and standard notations like $D_{\dR}$, $D_{\crys}$, etc.) developed within the $p$-adic Hodge theory of Galois representations by Fontaine (\cite{Fontaine-RepresentationSemiStable}). In particular, if $\ell = p$ and $\rho$ is potentially semistable then it too has an associated Weil--Deligne representation $\WD(\rho)$. For each embedding $\sigma: K \rightarrow \overline{\mathbf Q}_p$, we also write $\HT_\sigma(\rho)$ for $\sigma$-th Hodge--Tate weight which is defined to be the jumps in the Hodge filtration on the $\overline{\mathbf Q}_p$-vector space $D_{\dR}(\rho)\otimes_{K,\sigma} \overline{\mathbf Q}_p$.

Recall that we defined a normalized local Langlands correspondence $r^{\iota}$ over $\overline{\mathbf Q}_p$ (Section \ref{subsec:notation}). If $\rho$ is a representation of $G_F$ then and $v$ is a place of $F$ then we write $\rho_v$ for its restriction to a decomposition group at $v$. The previous paragraph then applies to the various $\rho_v$.

Let $G_{F,\mathfrak{n}p}$ denote the Galois group of the maximal extension of $F$ unramified away from all infinite places and all places dividing $\mathfrak{n}p$.

\begin{thm}\label{thm:classical-galois-representations}
Let $\pi$ be a cohomological cuspidal automorphic representation of $\mathrm{GL}_2/F$ of conductor $\mathfrak n$. Then there exists a unique continuous and irreducible representation
\begin{equation*}
\rho_{\pi} : G_F \rightarrow \GL_2(\overline{\mathbf Q}_p)
\end{equation*}
such that $\rho_{\pi,v}$ is potentially semi-stable at all $v \mid p$ and $\WD(\rho_{\pi,v}) = r^{\iota}(\pi_v)$ for all $v$. In particular, $\rho_{\pi,v}$ factors over the surjection $G_F \to G_{F,\mathfrak{n}p}$.

Furthermore, if $\pi$ has weight $\lambda=(\kappa,w)$ and $v \mid p$ then the following conclusions hold:
\begin{enumerate}
\item If $\sigma \in \Sigma_v$, then $\HT_{\sigma}(\rho_{\pi,v}) = \{{w-\kappa_{\sigma}\over 2}, {w + \kappa_{\sigma}\over 2} + 1\}$.
\item If $\pi_v$ is an unramified special representation then $\rho_{\pi,v}$ is semistable non-crystalline.
\item If $\pi_v$ is an unramified principal series representation then $\rho_{\pi,v}$ is crystalline.
\end{enumerate}
(The second two are deduced from the equation $\WD(\rho_{\pi,v}) = r^{\iota}(\pi_v)$, already.)
\end{thm}
\begin{proof}
The construction of $\rho_{\pi}$ and proof that it satisfies local-global compatibility away from $p$ can be deduced from independent work of Carayol (\cite{Carayol-HMF}), Wiles (\cite{Wiles-OrdinaryModular}), Blasius and Rogawski (\cite{BlasiusRogawski-HMFMotives}), and Taylor (\cite{Taylor-GaloisHMF}). The local-global compatibility at the $p$-adic places is due to Saito (\cite{Saito-padicHodgeTheory,Saito-HMFpAdic}), Blasius and Rogawski as before, and Skinner (\cite{Skinner-HMF}).
\end{proof}
\begin{rmk}\label{rmk:crystalline-eigenvalues}
If $\pi_v$ is an unramified principal series, then the characteristic polynomial of $\varphi^{f_v}$ acting on $D_{\crys}(\rho_{\pi,v})$ is equal to the characteristic polynomial of $r^{\iota}(\pi_v)(\Frob_v)$ or, what is the same, the image of the $v$-th Hecke polynomial $p_v(X) = X^2 - a_v(\pi)X  + \omega_{\pi}(\varpi_v)q_v$ under $\iota$. In particular, in the same case, if $\alpha$ is a $p$-refinement of $\pi$, then $\iota(\alpha_v)$ is an eigenvalue of $\varphi^{f_v}$ acting on $D_{\crys}(\rho_{\pi,v})$. Note that if $\beta_v$ is the second roots of $p_v(X)$ then the Ramanujan--Petersson conjecture for Hilbert modular forms (see \cite{Blasisu-HMFRamanujan} and the references there) implies that $\alpha_v$ and $\beta_v$ are Weil numbers of the same weight.
\end{rmk}

We will now globalize the construction of Galois representations in Theorem \ref{thm:classical-galois-representations} over $\mathscr E(\mathfrak n)_{\rmmid}$. Write $\psi: \mathbf T(\mathfrak n) \rightarrow \mathscr O(\mathscr E(\mathfrak n)_{\rmmid})$ to denote the universal Hecke eigensystem on $\mathscr E(\mathfrak n)_{\rmmid}$. 

\begin{prop}\label{prop:pseudo-representation}
There exists a unique two-dimensional pseudorepresentation
\begin{equation*}
T : G_{F,\mathfrak n \mathbf p} \rightarrow \mathscr O(\mathscr E(\mathfrak n)_{\rmmid})
\end{equation*}
such that if $v \nmid \mathfrak n \mathbf p$ then $T(\Frob_v) = \psi(T_v)$.
\end{prop}
\begin{proof}
First, Theorem \ref{theorem:reducedness} implies that $\mathscr E(\mathfrak n)_{\rmmid}$ is reduced. Second,  Theorem \ref{thm:classical-galois-representations} and Proposition \ref{prop:density-results} implies that we have a Zariski-dense subset $Z \subset \mathscr E(\mathfrak n)_{\rmmid}(\overline{\mathbf Q}_p)$ such that if $z \in Z$ then there is a  Galois representations $\rho_{z}: G_{F,\mathfrak n \mathbf p} \rightarrow \GL_2(\overline{\mathbf Q}_p)$ with $\tr(\rho_z(\Frob_v)) = \psi_z(T_v)$ for all $v \nmid \mathfrak n \mathbf p$. Specifically, we take $Z$ to be all those points which are twist classical and for $z \in Z$ of the form $z = \tw_{\vartheta}(x)$, with $x=x(\pi,\alpha)$ classical, we take $\rho_{z} = \rho_{\pi}\otimes \vartheta$ with $\rho_\pi$ as in Theorem \ref{thm:classical-galois-representations}. This tautologically gives the Hecke eigensystem $\psi_{z'}$ away from $\mathfrak n \mathbf p$ by \eqref{eqn:twisting-hecke}. The Zariski-density of these points follows from Propositions  \ref{prop:implications-extremely-non-critical} and \ref{prop:density-results}. Thus this proposition follows from a result of Chenevier (\cite[Proposition 7.1.1]{Chenevier-pAdicAutomorphicForm}) once we check a boundedness condition. Specifically, the eigenvariety $\mathscr E(\mathfrak n)_{\rmmid}$ is reduced and nested (in the sense of \cite[Section 7.2]{BellaicheChenevier-Book}), so by \cite[Lemma 7.2.11]{BellaicheChenevier-Book} the power bounded functions on $\mathscr E(\mathfrak n)_{\rmmid}$ form a compact subring of $\mathscr O(\mathscr E(\mathfrak n)_{\rmmid})$. The Lemma \ref{lemma:boundedness}(2) below implies the Hecke eigenvalues away from $\mathfrak n \mathbf p$ lie in this compact subring, and so Chenevier's result applies for us.
\end{proof}
To fill the gap in the previous proposition we need a small bit of notation.  Define $\mathbf T_{\mathbf Z_p}^{\mathrm{nr}}(\mathfrak n)$ as the $\mathbf Z_p$-span of the Hecke operators $(T_v)_{v \nmid \mathfrak n \mathbf p}$ inside $\mathbf T(\mathfrak n)$. Let $\mathscr E_{\Omega,h}$ be an affinoid neighborhood on $\mathscr E(\mathfrak n)$ with $(\Omega,h)$ slope adapted. Let $R = \mathscr O(\Omega)$ and suppose that $R_0 \subset R$ is a ring of definition, so we get an $R_0$-lattice $D^{\mathbf s,\circ}_{\Omega} \subset D^{\mathbf s}_{\Omega}$ as in Section 5.2. We then define an $R_0$-module $H^{d}_c(\mathfrak n, \mathbf D^{\mathbf s,\circ}_{\Omega})_{\leq h}$ by
\begin{equation*}
H^{d}_c(\mathfrak n, \mathbf D^{\mathbf s,\circ}_{\Omega})_{\leq h} := \im\left(H^{d}_c(\mathfrak n, \mathbf D^{\mathbf s,\circ}_{\Omega}) \rightarrow H^{d}_c(\mathfrak n, \mathbf D^{\mathbf s}_{\Omega}) \twoheadrightarrow H^{d}_c(\mathfrak n, \mathbf D^{\mathbf s}_{\Omega})_{\leq h} \right).
\end{equation*}
Thus $H^d_c(\mathfrak n, \mathbf D_{\Omega}^{\mathbf s,\circ})_{\leq h}$ is an $R_0$-submodule of the finite Banach $R$-module $H^d_c(\mathfrak n, \mathbf D_{\Omega}^{\mathbf s})_{\leq h}$.

\begin{lem}\label{lemma:boundedness}
Assume the notations of the previous paragraph.
\begin{enumerate}
\item $H^d_c(\mathfrak n, \mathbf D_{\Omega}^{\mathbf s,\circ})_{\leq h}$ is bounded and stable under the natural action of $\mathbf T_{\mathbf Z_p}^{\nr}(\mathfrak n)$.
\item $\psi(\mathbf T_{\mathbf Z_p}^{\mathrm{nr}}(\mathfrak n)) \subset \mathscr O(\mathscr E(\mathfrak n))$ consists of power bounded elements.
\end{enumerate}
\end{lem}
\begin{proof}
First, we will show that (2) follows from (1). Since, $\psi$ is an algebra morphism, it is enough to check that $\psi(\mathbf T_{\mathbf Z_p}^{\nr}(\mathfrak n))$ is bounded. Part (1) of this lemma implies that the induced endomorphisms on  $H^d_c(\mathfrak n, \mathbf D_{\Omega}^{\mathbf s})_{\leq h}$ are bounded and that is enough because the topology on $\mathscr O(\mathscr E(\mathfrak n))$ is the weakest topology making all of the natural maps $\mathscr O(\mathscr E(\mathfrak n)) \rightarrow \mathscr O(\mathscr E_{\Omega,h})=\mathbf T_{\Omega,h}$ continuous.

Now we prove (1). Write $K = K_1(\mathfrak n)I$. If $K' \subset K$ is an open and normal subgroup then we consider the diagram
\begin{equation*}
\xymatrix{
H^{d}_c(Y_K,\mathbf D_{\Omega}^{\mathbf s,\circ}) \ar[dd] \ar[r] & H^{d}_c(Y_K,\mathbf D_{\Omega}^{\mathbf s}) \ar@{>>}[r] \ar@{=}[d] & H^{d}_c(Y_K,\mathbf D_{\Omega}^{\mathbf s})_{\leq h}  \ar@{=}[d]\\
 & H^{d}_c(Y_{K'},\mathbf D_{\Omega}^{\mathbf s})^{K/K'} \ar[r] \ar@{^{(}->}[d] & H^{d}_c(Y_{K'},\mathbf D_{\Omega}^{\mathbf s})_{\leq h}^{K/K'}  \ar@{^{(}->}[d]\\
H^{d}_c(Y_{K'},\mathbf D_{\Omega}^{\mathbf s,\circ})  \ar[r] & H^{d}_c(Y_{K'},\mathbf D_{\Omega}^{\mathbf s}) \ar@{>>}[r]  & H^{d}_c(Y_{K'},\mathbf D_{\Omega}^{\mathbf s})_{\leq h}.
}
\end{equation*}
The two equalities are because $\mathbf D_{\Omega}^{\mathbf s}$ is a $\mathbf Q$-vector space and $K/K'$ is a finite group. The right-hand column consists of finite $R$-modules and thus the inclusion is continuous for the unique Banach $R$-module topologies. So, to check that the image of the top horizontal row is bounded, it is enough to check that the image of the bottom horizontal row is bounded. Replacing $\mathfrak n$ by a smaller ideal we can assume $Y_K$ is a neat level (Proposition \ref{prop:neatness}). In that case, the cohomology $H^d_c(Y_K,M)$ is computed by Borel--Serre complexes $C_c^{\bullet}(Y_K,M)$ for $M = \mathbf D_{\Omega}^{\mathbf s, \circ}$ or $M = \mathbf D_{\Omega}^{\mathbf s}$ (see the start of Section \ref{subsec:eigenvarieties} or \cite[p.15-16]{Hansen-Overconvergent}). In that case, the image of $H^d_c(Y_K,\mathbf D_{\Omega}^{\mathbf s,\circ}) \rightarrow H^d_c(Y_K,\mathbf D_{\Omega}^{\mathbf s})_{\leq h}$ is obviously bounded as it is the image, in cohomology, of the bounded subcomplex $C_c^{\ast}(K,\mathbf D_{\Omega}^{\mathbf s,\circ}) \subset C_c^{\ast}(K,\mathbf D_{\Omega}^{\mathbf s})$ under the quotient map $C_c^{\ast}(K,\mathbf D_{\Omega}^{\mathbf s}) \rightarrow C_c^{\ast}(K,\mathbf D_{\Omega}^{\mathbf s})_{\leq h}$.
\end{proof}

The lemma completes the proof of Proposition \ref{prop:pseudo-representation}. So now, for $x \in \mathscr E(\mathfrak n)_{\rmmid}(\overline{\mathbf Q}_p)$, we write $T_x$ for the specialization of the pseudorepresentation in Proposition \ref{prop:pseudo-representation} to the residue field $k_x$. A theorem of Taylor (\cite[Theorem 1(2)]{Taylor-Pseudoreps}) implies that for each $x$ there exists a unique continuous and semi-simple representation $\rho_x: G_F \rightarrow \GL_2(\overline{\mathbf Q}_p)$ so that $\tr(\rho_x) = T_x$. Note that if $x$ is a classical point then in fact $\rho_x$ may be defined over $k_x$ by the unicity, and the construction of the classical $\rho_x$ (as in the proofs of Theorem \ref{thm:classical-galois-representations}).

We now turn towards the important properties of $\rho_x$ at the $p$-adic places. If $\lambda =(\lambda_1,\lambda_2) \in \mathscr W$ is any $p$-adic weight then we can restrict to each $\lambda_i$ to $\lambda_{i,v}$ along $\mathcal O_v^\times \hookrightarrow \mathcal O_p^\times$. We then define characters $\eta_{i,v}$ on $\mathcal O_v^\times$ by
\begin{align}\label{defn:eta-map}
\eta_{1,v}(\lambda) &:= \lambda_{2,v}^{-1}\\
\eta_{2,v}(\lambda) &:= (\lambda_{1,v}\prod_{\sigma \in \Sigma_v} \sigma)^{-1}.\nonumber
\end{align}
Writing $\eta(\lambda)_v = \eta_{1,v}(\lambda)\eta_{2,v}(\lambda)$, which is a character on $\mathcal O_v^\times$, for each $\lambda$ we thus have a $\eta(\lambda) \in \mathscr X(\mathcal O_p^\times)$ given by the collection $(\eta(\lambda)_v)_{v \mid p}$. All together, $\lambda \mapsto \eta(\lambda)$ defines a morphism of rigid analytic spaces
\begin{equation*}
\eta : \mathscr W \rightarrow \mathscr X(\mathcal O_p^\times),
\end{equation*}
which induces a morphism on sub-rigid analytic spaces
\begin{equation}\label{eqn:eta-morphism}
\eta : \mathscr W(1) \rightarrow \mathscr X(\mathcal O_p^\times/\overline{\mathcal O_{F,+}^\times}).
\end{equation}
The next lemma will only be used later (see the proof of Theorem \ref{thm:smoothness}). If $x \in \mathscr E(\mathfrak n)_{\rmmid}(\overline{\mathbf Q}_p)$ then $\det(\rho_x)$, which is a character on $G_F$, can be restricted to a character $\det(\rho_x)|_{\mathcal O_p^\times}$ using the local Artin map at the $p$-adic places.
\begin{lem}\label{lem:relevant-determinant}
If $x \in \mathscr E(\mathfrak n)_{\rmmid}(\overline{\mathbf Q}_p)$, then $\det(\rho_{x})|_{\mathcal O_p^\times} = \eta(\lambda_x)$. In particular, the map 
\begin{align*}
\mathscr E(\mathfrak n)_{\rmmid} &\overset{\det_p}{\longrightarrow} \mathscr X(\mathcal O_p^\times)\\
x &\longmapsto \det(\rho_x)|_{\mathcal O_p^\times}
\end{align*}
factors through $\mathscr X(\mathcal O_p^\times/\overline{\mathcal O_{F,+}^\times})$.
\end{lem}
\begin{proof}
The lemma is true at classical $x$ by Theorem \ref{thm:classical-galois-representations}, twist classical $x$ by the definition of twisting, and all $x$ by interpolation.
\end{proof}

\begin{lem}\label{lemma:lift-to-O}
Suppose that $x \in \mathscr E(\mathfrak n)_{\rmmid}(\overline{\mathbf Q}_p)$ is a classical point. Then, there exists a good affinoid neighborhood $x \in U \subset \mathscr E(\mathfrak n)_{\rmmid}$ and a continuous linear representation $\rho_U: G_{F,\mathfrak n \mathbf p} \rightarrow \GL_2(\mathscr O(U))$ such that $\rho_U\otimes_{\mathscr O(U)} k_u = \rho_u$ for each $u \in U$.
\end{lem}
\begin{proof}
Write $x = x(\pi,\alpha)$. Since $\pi$ is cuspidal, the Galois representation $\rho_x = \rho_\pi$ is absolutely irreducible. Write $\mathscr O_x$ for the rigid local ring of $x$ on $\mathscr E(\mathfrak n)_{\rmmid}$. Then $\mathscr O_x$ is a Henselian local ring (\cite[Theorem 2.1.5]{Berkovich-EtaleCohomology}), so by \cite[Corollarie 5.2]{Rouqier-Jalgebra96-Pseudocharacters} there exists a continuous lift $\rho_{\mathscr O_x}$ of $\rho_x$ to $\mathscr O_x$ such that $\tr(\rho_{\mathscr O_x})$ is equal to the specialization of the pseudorepresentation $T$ as in Theorem \ref{prop:pseudo-representation} to the ring $\mathscr O_x$. By \cite[Lemma 4.3.7]{BellaicheChenevier-Book} we can extend $\rho_{\mathscr O_x}$ to a continuous representation $\rho_U$ over some affinoid neighborhood of $U$ in a manner compatible with the pseudorepresentation $T$. Being absolutely irreducible is a Zariski-open condition on $U$ (\cite[Section 7.2.1]{Chenevier-pAdicAutomorphicForm}) and so we may, if necessary, shrink $U$ and assume that $\rho_u$ is absolutely irreducible at each $u \in U$. At that point the equality $\tr(\rho_u) = \tr(\rho_U\otimes_{\mathcal O(U)} k_u)$ becomes an equality of true representations by the theorem of Brauer and Nesbitt. This proves the lemma.
\end{proof} 

\begin{lem}\label{lem:newform-accumulation}
Suppose that $x \in \mathscr E(\mathfrak n)_{\rmmid}(\overline{\mathbf Q}_p)$ is a classical point of prime-to-$p$ conductor $\mathfrak n$. Then, if $U$ is a good neighborhood of $x$ in $\mathscr E(\mathfrak n)_{\rmmid}$, then $U$ contains a Zariski-dense and accumulating subset of points $y$ that are twist classical of the form $y = \tw_{\vartheta}(x')$ where $x'$ is classical and also has prime-to-$p$ conductor $\mathfrak n$.
\end{lem}
\begin{proof}
For $\mathfrak n \subsetneq \mathfrak n'$, write $\widetilde{\mathscr E}(\mathfrak n')$ for the eigenvariety constructed out of the finite slope subspaces $H^{\ast}(\mathfrak n', \mathscr D_{\lambda})_{\leq h}$ except only with endomorphisms by $\mathbf T(\mathfrak n)$ (i.e.\ ignore the Hecke operators at primes dividing $\mathfrak n/\mathfrak n'$). Then the construction we outlined gives a natural closed immersion $\widetilde{\mathscr E}(\mathfrak n') \hookrightarrow \mathscr E(\mathfrak n)$.

Let $x$ be in the statement of the lemma, so $x$ is a classical point of prime-to-$p$ conductor. Then, $x$ is not in the image of any of the finitely many such embeddings by the same argument as \cite[Lemma 2.7]{Bellaiche-CriticalpadicLfunctions} (which relies on the family of Galois representations we've just established). So, you can shrink $U$ and assume in fact that $U$ misses any of these embeddings. Then, $U$ contains a Zariski-dense and accumulating subset of points $y$ as above, where the claim on the conductors follows because the quantity ``prime-to-$p$ conductor of $x'$'' is actually independent of choosing $x'$, since any chosen $\vartheta$ is unramified away from $p$.
\end{proof}

For any character $\eta$ on $\mathcal O_v^\times$, we write $\LT_{\varpi_v}(\eta)$ for the extension of $\eta$ to a character of $F_v^\times$ defined by stipulating that $\varpi_v \mapsto 1$. The character $\LT_{\varpi_v}(\eta)$ is unitary, so we use the same notation to denote its continuous extension to a Galois character on $G_{F_v}$.\footnote{As a character of $G_{F_v}$, $\LT_{\varpi_v}(\eta)$ coincides with the composition $\eta \circ \chi_{\varpi_v}$, where $\chi_{\varpi_v}:G_{F_v} \to \mathcal{O}_{v}^{\times}$ is the character obtained from the $G_{F_v}$-action on the Tate module of the Lubin-Tate formal $\mathcal{O}_v$-module associated with the uniformizer $\varpi_v$. This explains the notation.} These normalizations are designed so that $\lambda=(\kappa,w)$ is a cohomological weight, then $\HT_{\sigma}(\LT_{\varpi_v}(\eta_{1,v}(\lambda))) = {w - \kappa_\sigma\over 2}$ and $\HT_{\sigma}(\LT_{\varpi_v}(\eta_{2,v}(\lambda))) = {w + \kappa_\sigma\over 2} + 1$ for all $\sigma \in \Sigma_v$ (compare with Theorem \ref{thm:classical-galois-representations}). 

\begin{prop}\label{prop:family-galois-properties}
Let $x = x(\pi,\alpha)\in \mathscr E(\mathfrak n)_{\rmmid}(\overline{\mathbf Q}_p)$ be a classical point with prime-to-$p$ conductor $\mathfrak n$. Choose  $U$ and $\rho_U$ as in Lemma \ref{lemma:lift-to-O}. Write $\mathscr O_x$ for the rigid local ring on $\mathscr E(\mathfrak n)_{\rmmid}$ at $x$ and $\rho_{\mathscr O_x}$ for the specializiation of $\rho_U$ along $\mathscr O(U) \rightarrow \mathscr O_x$.
\begin{enumerate}
\item If $w \nmid p$ and $I_w$ is the choice of an inertia subgroup at $w$ then $\rho_{\mathscr O_x}|_{I_w} \simeq \rho_x|_{I_w} \otimes_{k_x} \mathscr O_x$.
\item Assume further that if $v \mid p$ and $\pi_v$ is an unramified principal series then the $v$-th Hecke polynomial has distinct roots.\footnote{Compare with condition 2(c) in Definition \ref{defn:decent} below.} Then, if $v \mid p$, then 
\begin{equation*}
D_{\crys}^+(\rho_{U,v}\otimes \LT_{\varpi_v}(\eta_{1,v}(\lambda_U))^{-1})^{\varphi^{f_v} = \psi(U_v)}
\end{equation*}
is locally free of rank one over $F_{v}^{\nr}\otimes_{\mathbf Q_p} \mathscr O(U)$ and commutes with base change on $U$.
\end{enumerate}
\end{prop}
In part (2), $F_v^{\nr} \subset F_v$ means the maximal unramified extension of $\mathbf Q_p$ inside $F_v$. If $\rho$ is an $R$-linear representation of $G_{F_v}$, then $D_{\crys}(\rho)$ is an $(F_v^{\nr}\otimes_{\mathbf Q_p} R)$-module.
\begin{proof}[Proof of Proposition \ref{prop:family-galois-properties}]
The argument for part (1) follows exactly as in the argument for ``property (iii)'' in the proof of \cite[Theorem 2.16]{Bellaiche-CriticalpadicLfunctions}, including when $w$ is a ramified place for $\rho_x$, once \cite[Lemma 2.7]{Bellaiche-CriticalpadicLfunctions} is replaced by Lemma \ref{lem:newform-accumulation}. The details of that argument, which we provide here, require referencing the text \cite{BellaicheChenevier-Book}. Let $\mathscr K_x$ be the total field of fractions of the (reduced, by Theorem \ref{theorem:reducedness}) ring $\mathscr O_x$. Thus $\mathscr K_x = \prod_{s(x)} \mathscr K_{s(x)}$ where $s(x)$ runs over germs of irreducible components of $\mathscr E(\mathfrak n)_{\rmmid}$ passing through $x$, and $\mathscr K_{s(x)}$ is the fraction field of the germ $s(x)$. We write $\rho_{s(x)}^{\operatorname{gen}} = \rho_{\mathscr O_x} \otimes \mathscr K_{s(x)}^{\operatorname{gen}}$ for the extension of scalars to $\mathscr K_{s(x)}^{\operatorname{gen}}$. We now fix the place $w$ as in (1) and consider the three Weil--Deligne representations $(r_x,N_x)$, $(r_{\mathscr O_x},N_{\mathscr O_x})$, and $(r_{s(x)}^{\operatorname{gen}},N_{s(x)}^{\operatorname{gen}})$ associated with $\rho_{x,w}$, $\rho_{\mathscr O_{x},w}$, and $\rho_{s(x),w}^{\operatorname{gen}}$.  Since $x$ is classical, $\rho_x$ is irreducible and defined over $k_x$, so $\rho_{s(x)}^{\operatorname{gen}}$ is the unique continuous and irreducible representation $G_{F,\mathfrak n \mathbf p} \rightarrow \GL_2(\mathscr K_{s(x)})$ whose trace agrees with the pseudorepresentation $T$ (compare with the text prior to \cite[Corollary 7.5.10]{BellaicheChenevier-Book}). Thus $(r_{s(x)}^{\operatorname{gen}},N_{s(x)}^{\operatorname{gen}})$ is the projection onto $\mathscr K_{s(x)}$ of the generic Weil--Deligne representation as in \cite[Definition 7.8.16]{BellaicheChenevier-Book}, which validates using the references below. First, by \cite[Lemma 7.8.17]{BellaicheChenevier-Book}, each $r_{s(x)}^{\operatorname{gen}}|_{I_w}$ is isomorphic to the constant $I_w$-representation $r_x|_{I_w} \otimes_{k_x} \mathscr K_{s(x)}^{\operatorname{gen}}$, which implies $r_{\mathscr O_x}|_{I_w}$ is also constant. To prove (1) it thus suffices to identify $N_x$ and $N_{\mathscr O_x}$. Using \cite[Proposition 7.8.9(ii)]{BellaicheChenevier-Book}, we can replace $\mathscr O_x$ by each $\mathscr K_{s(x)}$ and show that $N_{s(x)}^{\operatorname{gen}}$ and $N_x$ have the same Jordan normal forms. There are only two such forms, since each operator is nilpotent on a two-dimensional vector space. By \cite[Proposition 7.8.19(iii)]{BellaicheChenevier-Book}, if $N_{s(x)}^{\operatorname{gen}} = 0$, then $N_x = 0$ as well. On the other hand, if $N_{s(x)}^{\operatorname{gen}}$ is non-trivial, then \cite[Proposition 7.8.19(ii)]{BellaicheChenevier-Book} and Lemma \ref{lem:newform-accumulation} implies that $N_{x'}$ is non-trivial on a Zariski dense set of a twist classical points $x'$, with prime-to-$p$ conductor $\mathfrak n$, that accumulate near at $x$. In particular, this forces $w$ to be a place dividing $\mathfrak n$, so $N_x \neq 0$ in this case. This completes the proof of (1).

Now we prove part (2). Fix $v \mid p$. We claim that the family $r_U:=\rho_{U,v} \otimes \mathrm{LT}_{\varpi_K}(\eta_{1,v}(\lambda_U))^{-1}$ of two-dimensional Galois representations over the reduced rigid analytic space $U$ is a weakly-refined family in the sense of \cite[Definition 1.5]{Liu-Triangulations}. Namely, we choose as the definining data $\{h_1,h_2,F,Z\}$, where
\begin{itemize}
\item $h_1 = 0$ and $h_2 = \frac{\partial}{\partial z}\big|_{z = 1}\left( \eta_{2,v}(\lambda_U)\eta_{1,v}(\lambda_U)^{-1}\right)$,
\item $F = \psi(U_v)$, and
\item $Z \subseteq U$ is the Zariski-dense subset of the extremely non-critical points $z$ that satisfy the slightly stronger conditions in Remark \ref{rmk:extremely-non-critical-unramified-principal-series}. Namely, writing the weight of $z$ as $(\kappa,\ast)$, then $\kappa_{\sigma} \geq 1$ for all $\sigma \in \Sigma_F$, and $v_p(\psi_z(U_p)) < \frac{1}{3}\inf_{\sigma} (1+\kappa_\sigma)$. 
\end{itemize}
The subset of extremely non-critical points is known to be dense by part (4) of Proposition \ref{prop:density-results}, but the same proof shows that our chosen subset $Z$ is also Zariski-dense. By assumption in (2), the eigenvalues of $\varphi^{f_v}$ acting on $D_{\crys}^+(r_x)$ are distinct, and so part (2) of the proposition follows from \cite[Proposition 5.13]{Liu-Triangulations}, once we verify that the axioms in \cite[Definition 1.5]{Liu-Triangulations}.\footnote{There are two minor alerts for the reader. First, the notation in \cite{Liu-Triangulations} is for the functions $h_i$ to be written $\kappa_i$, notation we did not use since it collides with our prior notations in this paper. Second, technically, the axioms cannot literally be verified because \cite{Liu-Triangulations} uses the convention that the Hodge--Tate weights are the negatives of the ones we use here. However, the only change, given that $h_1 = 0$ for us, is that axiom (c) in \cite[Definition 1.5]{Liu-Triangulations} should have the word smallest in place of the word biggest.} 

We pause the proof now for an observation on twisting. Suppose that $z \in \mathscr E(\mathfrak n)_{\rmmid}(\overline{\mathbf Q}_p)$ and $\vartheta \in \mathscr X(\Gamma_F)(\overline{\mathbf Q}_p)$ and consider $z' = \tw_{\vartheta}(z)$. Then $\rho_{z'} \simeq \rho_{z} \otimes \vartheta$ and $\lambda_{z'}|_{\mathcal O_v^\times} = \lambda_z|_{\mathcal O_v^\times} \otimes  \vartheta_v^{-1}$ (see \eqref{eqn:twisting-diagram}). Since $\eta_{1,v} = \lambda_{2,v}^{-1}$, by definition \eqref{defn:eta-map}, we see that
\begin{equation}\label{eqn:twist-refined}
r_{z'} \simeq r_z \otimes \LT_{\varpi_v}(\vartheta_v^{-1})\vartheta_v.
\end{equation}
The factor $\LT_{\varpi_v}(\vartheta_v^{-1}) \vartheta_v$ appearing here is the unramified (hence crystalline, with Hodge--Tate weights all zero) character of $\mathcal O_v^\times$ sending $\varpi_v$ to $\vartheta_v(\varpi_v)$.

With this observation, we can verify the axioms of \cite[Definition 1.5]{Liu-Triangulations} (which we label (a) thru (f) as in {\em loc.\ cit.}). For instance, axiom (b) asks that $r_z$ is crystalline for $z \in Z$. When $z$ is classical, Remark \ref{rmk:extremely-non-critical-unramified-principal-series} implies that $z$ is associated with an unramified principal series. Thus $r_z$ is crystalline by Theorem \ref{thm:classical-galois-representations}. When $z$ is only twist classical, the same conclusion follows from \eqref{eqn:twist-refined}. Similarly, by Theorem \ref{thm:classical-galois-representations} and \eqref{eqn:twist-refined}, the set of Hodge--Tate weights $\HT_{\sigma}(r_z)$ is equal to $\{h_1(z),h_2(z)\}$ for $z \in Z$ and $\sigma \in \Sigma_v$. On the one hand, \cite[Lemma 7.5.12]{BellaicheChenevier-Book} implies that $\HT_{\sigma}(r_u) = \{h_1(u),h_2(u)\}$ for all $u \in U$, so axiom (a) is confirmed. On the other hand, if $z \in Z$ is classical then Theorem \ref{thm:classical-galois-representations} implies $h_1(z) = 0 < h_2(z)$, and inequality continues to hold on all of $Z$ by \eqref{eqn:twist-refined}. So, axiom (c) is confirmed. The most crucial axiom is (d), which states that
$$
D_{\crys}^+(r_z)^{\varphi^{f_v} = \psi_z(U_v)} \neq 0
$$
for all $z \in Z$. By \eqref{eqn:twist-refined}, and because $\psi_{\tw_{\vartheta}(z)}(U_v) = \vartheta(\varpi_v)\psi_z(U_v)$, we may assume that $z \in Z$ is classical. As mentioned above, this implies that $z$ is associated with an unramified principal series. Thus, the fact that $\psi_z(U_v)$ is a crystalline eigenvalue for $r_z$ follows from Remark \ref{rmk:crystalline-eigenvalues}. Axiom (e) follows from Lemma \ref{lem:weight-density}, Proposition \ref{prop:openness-maximal-dimensional}, and part (3) of Proposition \ref{prop:density-results} (which is the key point that components of $\mathscr E(\mathfrak n)_{\rmmid}$ have maximal dimension). Finally, axiom (f) is true by construction of the $h_i$. This completes the proof.
\end{proof}

\subsection{Smoothness at $p$-distinguished decent classical points}\label{subsec:smoothness}
We now generalize the definition of non-critical.
\begin{defn}\label{defn:decent}
A classical point $x = x(\pi,\alpha) \in \mathscr E(\mathfrak n)(\overline{\mathbf Q}_p)$ is decent if either:
\begin{enumerate}
\item It is non-critical as in Definition \ref{defn:classical-noncritical}, or 
\item The following three conditions hold.
\begin{enumerate}
\item $H^{\ast}_c(\mathfrak n,\mathscr D_{\lambda})_{\mathfrak m_x}$ is concentrated only in degree $d$,
\item The Selmer group $H^1_f(G_{F},\ad \rho_{\pi})$ vanishes, and
\item For each $v \mid p$, $\alpha_v$ is a simple root of $X^2 - a_v(\pi)X + \omega_{\pi}(\varpi_v)q_v$.
\end{enumerate}
\end{enumerate}
\end{defn}
In condition 2(b) of Definition \ref{defn:decent}, $\ad \rho_\pi$ is the adjoint representation $\rho_\pi \otimes \rho_\pi^{\vee} \simeq \End(\rho_\pi)$.

\begin{lem}
If $x \in \mathscr E(\mathfrak n)(\overline{\mathbf Q}_p)$ is decent, then $x \in \mathscr E(\mathfrak n)_{\rmmid}(\overline{\mathbf Q}_p)$.
\end{lem}
\begin{proof}
If $x$ is non-critical, then this follows from Lemma \ref{lem:non-critical-pts-middle-degree}. Otherwise, condition (2a) implies the claim by Proposition \ref{prop:middle-alternate-characterization}.
\end{proof}

We will see later (Theorem \ref{thm:socle-one-d}) that the Hecke eigensystem corresponding to a decent point $x$ has multiplicity one in the distribution-valued cohomology. When $x$ is a non-critical point, this is a classical automorphic fact. But if $x$ satisfies condition (2) of Definition \ref{defn:decent}, we deduce it from the following geometric theorem on the eigenvariety. The proof occupies the rest of this subsection.
\begin{thm}\label{thm:smoothness}
Suppose that $x\in \mathscr E(\mathfrak n)_{\rmmid}(\overline{\mathbf Q}_p)$ is decent, the prime-to-$p$ conductor of $x$ is $\mathfrak n$, and condition 2(c) in Definition \ref{defn:decent} is satisfied. Then, $\mathscr E(\mathfrak n)_{\rmmid}$ is smooth at $x$.
\end{thm}
To be clear, the assumption on $x$ in Theorem \ref{thm:smoothness} is that either $x$ satisfies condition (2) of Definition \ref{defn:decent} or $x$ is non-critical and further satisfies condition 2(c) of Definition \ref{defn:decent}. The proof in case $x$ satisfies (2) is at the end of the subsection. In case $x$ is non-critical, the proof is in Proposition \ref{prop:etaleness} below.

We now fix some notation that will remain in force for the rest of this section. We will write $x \in \mathscr E(\mathfrak n)_{\rmmid}(\overline{\mathbf Q}_p)$ and $\lambda=\lambda_x$ for its weight. Write $L = k_x$ for the residue field at $x$. We write $\mathscr O_x$ for the rigid local ring on $\mathscr E(\mathfrak n)_{\rmmid}$ at $x$ and $\mathscr O_{\lambda}$ for the rigid local ring on $\mathscr W(1)$ at $\lambda$.

We first prove Theorem \ref{thm:smoothness} in the non-critical case.
\begin{prop}\label{prop:etaleness}
If $x \in \mathscr E(\mathfrak n)_{\rmmid}(\overline{\mathbf Q}_p)$ is as in Theorem \ref{thm:smoothness} and non-critical, then $\lambda:\mathscr E(\mathfrak n)_{\rmmid} \rightarrow \mathscr W(1)$ is \'etale at $x$.
\end{prop}
\begin{proof}
This argument is essentially due to Chenevier (\cite[Theorem 4.8]{Chenevier-InfiniteFern}).\footnote{In the case of $F=\mathbf Q$ there is also an argument given by Bella\"iche (\cite[Lemma 2.8]{Bellaiche-CriticalpadicLfunctions}) that relies on  {\em a priori} knowing that the weight map is flat. In general, this is only observed at decent points and only after the arguments in this section. See Section \ref{subsec:consequences}.}

Let $U$ be a sufficiently small good neighborhood of $x$, belonging to a slope adapted pair $(\Omega,h)$, such that $x$ is the unique reduced point of $U$ lying above $\lambda \in \Omega$. Set $M = \mathscr M_c^d(U)$. For each $\epsilon \in \{\pm 1\}^{\Sigma_F}$, let $M^{\epsilon}$ be the $\epsilon$-component, so $M = \bigoplus_{\epsilon} M^{\epsilon}$. Since these are $\mathscr O(\Omega)$-direct summands of $M$ they are each finite projective over $\mathscr O(\Omega)$ (see Proposition \ref{prop:looks-like-eigen}) and $U$ is the rigid analytic spectrum of the image of $\mathbf T(\mathfrak n)\otimes_{\mathbf Q_p} \mathscr O(\Omega) \rightarrow \End_{\mathscr O(\Omega)}(M^{\epsilon})$ (for any $\epsilon$). Further, if $\lambda' \in \Omega$ is any weight then 
\begin{equation}\label{eqn:module-fiber}
M^{\epsilon}/\mathfrak m_{\lambda'}M^{\epsilon} = \bigoplus_{\substack{y\in U \\ \lambda_y = \lambda'}} H^d_c(\mathfrak n, \mathscr D_{\lambda'})^{\epsilon}_{\mathfrak m_{y'}}.
\end{equation}
Remember that we have assumed $x$ is the unique point above $\lambda$. So, since $x$ is assumed to be non-critical, the prime-to-$p$ conductor of $\pi$ is $\mathfrak n$, and part (c) of Definition \ref{defn:decent} holds, we deduce that \eqref{eqn:module-fiber} is in fact $1$-dimensional. If $\lambda'$ is any other weight near to $\lambda$ over which all the points $y'$ are extremely non-critical  with prime-to-$p$ conductor $\mathfrak n$ (such weights are accumulating at $\lambda$) then $H^d_c(\mathfrak n, \mathscr D_{\lambda'})_{\mathfrak m_{y'}}^{\epsilon}$ is also $1$-dimensional. Since the dimension of \eqref{eqn:module-fiber} is constant with respect to $\lambda'$ we deduce $M^{\epsilon}$ is projective of rank one over $\mathscr O(\Omega)$. So, the composition $\mathscr O(\Omega) \rightarrow \mathscr O(U) \rightarrow \End_{\mathscr O(\Omega)}(M^{\epsilon})$ becomes an isomorphism after a finite field extension, meaning $\mathscr O(\Omega) \rightarrow \mathscr O(U)$ is \'etale.
\end{proof}
For the remainder of this subsection we fix a decent classical point $x \in \mathscr E(\mathfrak n)_{\rmmid}(\overline{\mathbf Q}_p)$ of weight $\lambda$ as in Theorem \ref{thm:smoothness}. Because Proposition \ref{prop:etaleness} deals with the non-critical case of Theorem \ref{thm:smoothness}, we will further assume that $x$ satisfies condition (2) of Definition \ref{defn:decent}. Write $L = k_x$ for the residue field at $x$.\label{page:start-deformation-theory}

We now begin to use the language of deformation theory of Galois representations. Let $\operatorname{Set}$ be the category of sets and $\mathfrak{AR}_L$ be the category of local Artinian $L$-algebras with residue field $L$. Recall that a functor $\mathfrak X : \mathfrak{AR}_L \rightarrow \operatorname{Set}$ is called (pro-)representable if there exists a complete local noetherian $L$-algebra $R$ with residue field $L$ such that $\mathfrak X \cong \Hom(R,-)$. Recall also that a morphism $\mathfrak X' \rightarrow \mathfrak X$ of functors is called relatively representable if for any morphism $\mathfrak Y \rightarrow \mathfrak X$, with $\mathfrak Y$ representable, the fibered product $\mathfrak Y \times_{\mathfrak X} \mathfrak X'$ is representable (see \cite[Section 19]{Mazur-Fermat-Deformations}).

Write $\rho = \rho_x$ for the global Galois representation and $\rho_v$ for its restriction to a place $v$. Let $\mathfrak X_v:\mathfrak{AR}_L \rightarrow \mathrm{Set}$ be the functor  parameterizing deformations $\widetilde \rho_v$ of $\rho_v$ up to strict equivalence, i.e.\ for $A \in \mathfrak {AR}_L$, the set $\mathfrak X_v(A)$ consists of lifts $\widetilde \rho_v : G_{F_v} \rightarrow \GL_2(A)$ of $\rho_v$ along $A \rightarrow L$ and two lifts are equivalent if they are conjugate by an invertible matrix congruent to the identity modulo $\mathfrak m_A$. The Zariski tangent space $\mathfrak X_v(L[u]/u^2L[u])$ to the functor $\mathfrak X_v$ is canonically isomorphic to $H^1(G_{F_v},\ad\rho_v)$. We will next describe relatively representable subfunctors of $\mathfrak X_v$, for each $v$, and use them to define a global deformation functor. 

First, suppose that $v \nmid p$ and write $I_v$ for the inertia group at $v$. Then, we define $\mathfrak X_{v,f}$ as the subfunctor of minimally ramified deformations of $\rho_v$, which means $\mathfrak X_{v,f}(A) \subseteq \mathfrak X_v(A)$ are those deformations such that $\widetilde \rho_v \simeq \rho_v\otimes_L A$ as $I_v$-representations. The relative representability of $\mathfrak X_{v,f} \subseteq \mathfrak X_v$ follows from Schlessinger's criterion as in \cite[Section 23]{Mazur-Fermat-Deformations} (see \cite[Proposition 7.6.3]{BellaicheChenevier-Book} for a proof). The Zariski tangent space $\mathfrak X_{v,f}(L[u]/u^2L[u])$ is canonically isomorphic to the local Bloch--Kato Selmer group 
$$
H^1_f(G_{F_v},\ad\rho_v) = \ker\left(H^1(G_{F_v},\ad\rho_v) \rightarrow H^1(I_v,\ad \rho_v)\right) \subset H^1(G_{F_v},\ad\rho_v),
$$
defined in \cite[Section 3]{BlochKato-TamagawaNumbersOfMotives}.

Now suppose that $v \mid p$. If $\sigma \in \Sigma_v$, part (1) of Theorem \ref{thm:classical-galois-representations} implies that the Hodge--Tate weights $\HT_\sigma(\rho_v)$ are $\{\frac{w-\kappa_\sigma}{2}, \frac{w+\kappa_\sigma}{2} + 1\}$, which in particular is a pair of distinct integers since $\kappa_\sigma \geq 0$. Thus for each $\widetilde \rho_v$, the Hodge--Tate--Sen weights of $\widetilde \rho_v$ are also distinct. We can thus choose characters $\widetilde \eta_{i,v} : \mathcal O_v^\times \rightarrow A^\times$ (for $i=1,2$) such that the Hodge--Tate--Sen weights of $\widetilde \rho_v$ are $\{\HT_{\sigma}(\LT_{\varpi_v}(\widetilde \eta_{1,v})),\HT_{\sigma}(\LT_{\varpi_v}(\widetilde \eta_{2,v}))\}$ and
$$
\HT_{\sigma}(\LT_{\varpi_v}(\widetilde \eta_{i,v})) \equiv \HT_{\sigma}(\LT_{\varpi_v}(\eta_{i,v}(\lambda))) \bmod \mathfrak m_A,
$$
where $\mathfrak m_A$ is the maximal ideal of $A$ (see \cite[Section 2.3]{Bergdall-Smoothness}, for instance).

Recall that $\alpha_v^{\sharp} = \psi_x(U_v)$ is an eigenvalue for $\varphi^{f_v}$ acting on $D_{\crys}^+(\rho_v\otimes \LT_{\varpi_v}(\eta_{1,v}(\lambda))^{-1})$. The functor of weakly-refined deformations $\mathfrak X_v^{\Ref}$ is defined by: 
$$
\mathfrak X_v^{\Ref}(A) = \{\widetilde \rho_v \in \mathfrak X_v(A) \mid D_{\crys}^+(\widetilde \rho_v \otimes \LT_{\varpi_v}(\widetilde \eta_{1,v})^{-1})^{\varphi^{f_v} = \widetilde \Phi} \text{ is free of rank one for some lift $\widetilde \Phi$ of $\alpha_v^{\sharp}$}\}.
$$
Recall now that assumption 2(c) in Definition \ref{defn:decent} implies $\alpha_v^{\sharp}$ is a simple eigenvalue of $\varphi^f$ acting on $D_{\crys}^+(\rho_v\otimes \LT_{\varpi_v}(\eta_{1,v}(\lambda))^{-1})$. For this reason, the natural inclusion $\mathfrak X_v^{\Ref} \subset \mathfrak X_v$ is relatively representable. This was first proven by Kisin (\cite[Proposition 8.13]{Kisin-OverconvergentModularForms}) when $F_v = \mathbf Q_p$ and later generalized by Tan (\cite[Section 5.2.2]{Tan-Thesis}). A third proof, using the language of $(\varphi,\Gamma)$-modules, is given in \cite[Section 3.1]{Bergdall-Smoothness}. 

Let $\mathfrak t_v^{\Ref} = \mathfrak X_v^{\Ref}(L[u]/u^2L[u])$ be the Zariski tangent space to $\mathfrak X_v^{\Ref}$. This is a subspace of $H^1(G_{F_v},\ad\rho_v)$, as is the local Bloch--Kato Selmer group 
$$
H^1_f(G_{F_v},\ad\rho_{v}) = \ker\left(H^1(G_{F_v},\ad\rho_v) \rightarrow H^1(G_{F_v},\ad\rho_v \otimes_{\mathbf Q_p} B_{\crys})\right)
$$
defined in \cite{BlochKato-TamagawaNumbersOfMotives}. If $\rho_v$ is crystalline (i.e.\ $\pi_v$ is an unramified principal series; see Theorem \ref{thm:classical-galois-representations}) then $H^1_f(G_{F_v},\ad\rho_{v})$ parameterizes infinitesmial crystalline deformations. 

The next proposition provides a crucial bound for the dimension of $\mathfrak t_v^{\Ref}$. Note that everything written thus far is stated in terms of Galois representations, but the proof of the next result genuinely requires using $(\varphi,\Gamma)$-modules. When $\rho_v$ is crystalline, the proposition is proven in \cite{Bergdall-Smoothness} and when $\rho_v$ is semi-stable and non-crystalline, we provide a proof in Appendix \ref{app:semistable} in order to (i) include a proof, but (ii) limit mentioning $(\varphi,\Gamma)$-modules within this text. It is possible to avoid the language of $(\varphi,\Gamma)$-modules when $\rho_v = \chi_1 \oplus \chi_2$ is a sum of characters. The reader interested in understanding this simpler case prior to reading either \cite{Bergdall-Smoothness} or Appendix \ref{app:semistable} can examine the final two paragraphs of the proof of \cite[Proposition 2.16]{Bellaiche-CriticalpadicLfunctions}.

\begin{prop}\label{prop:local-bound}
\leavevmode
\begin{enumerate}
\item $H^1_f(G_{F_v},\ad\rho_v) \subset \mathfrak t_{v}^{\Ref}$.
\item $\dim_L \mathfrak t_v^{\Ref}/H^1_f(G_{F_v},\ad\rho_v) \leq 2[F_v:\mathbf Q_p]$.
\end{enumerate}
\end{prop}

\begin{proof}
If $\rho_v$ is semi-stable but non-crystalline, see Lemma \ref{lem:containment} for part (1) and  Corollary \ref{corollary:semi-stable-bound} for part (2). In the case $\rho_v$ is crystalline, we just noted that $H^1_f(G_{F_{v}},\ad\rho_v)$ is the tangent space of crystalline deformations, all of which are weakly-refined, and so (1) is clear. Part (2) follows from \cite[Corollary 3.19]{Bergdall-Smoothness}, except there are hypotheses in {\em loc.\ cit.}\ that need to be checked in the present context.

In Appendix \ref{app:semistable}, we recall the definition of the Robba ring $\mathcal R_{F_v,L}$ and give details on $(\varphi,\Gamma)$-modules, triangulations, and their connection to Galois representations through the $D_{\rig}$ functor in the two-dimensional setting. The choice of $p$-refinement $\alpha_v$ induces a choice of triangulation $P_{\bullet}$ on $D_{\rig}(\rho_v)$. The functor of weakly-refined deformations of $\rho_v$ as defined in the prior paragraphs is isomorphic, via the $D_{\rig}$-functor, to the functor of weakly-refined deformations of $D_{\rig}(\rho_v)$ with respect to $P_{\bullet}$ that is defined in \cite[Section 3.2]{Bergdall-Smoothness}. The hypotheses of \cite[Corollary 3.19]{Bergdall-Smoothness} are that $P_{\bullet}$ is a {\em regular} and {\em generic} triangulation whose {\em critical type} is a collection of distinct simple transposes. We now explain what these terms means in terms of the refinement $\alpha_v$ and, at the same time, establish that they hold.

\begin{itemize}
\item {\em Regular}. The triangulation $P_{\bullet}$ is regular as in \cite[Definition 3.4]{Bergdall-Smoothness} if and only if $\alpha_v$ is a simple eigenvalue for $\varphi^f$ acting on $D_{\crys}(\rho_v)$. This holds by our assumption 2(c) of Definition \ref{defn:decent}.
\item {\em Generic}. The triangulation $P_{\bullet}$ is generic as in \cite[Definition 3.4]{Bergdall-Smoothness} if a certain second cohomology group of a rank one $(\varphi,\Gamma)$-module vanishes. In our case, it is sufficient to check the eigenvalues $\alpha_v$ and $\beta_v$ of $\varphi^{f_v}$ acting on $D_{\crys}(\rho_v)$ satisfy $\alpha_v\beta_v^{-1} \neq p^{f_v}$. But this follows because $\alpha_v$ and $\beta_v$ are Weil numbers of the same weight (see Remark \ref{rmk:crystalline-eigenvalues}).
\item {\em Critical type}. The critical type of $P_{\bullet}$ is a collection of elements $(\pi_{\sigma})_\sigma \in S_2^{\Sigma_v}$, where $S_2$ is the permutation group on two letters, that depends on $\alpha$ (see \cite[Definition 1.1]{Bergdall-Smoothness}). Regardless of their definition, the $\pi_{\sigma}$ are products of distinct simple transpositions and so the hypothesis on critical type in \cite[Corollary 3.19]{Bergdall-Smoothness} holds.
\end{itemize}
This completes the proof.
\end{proof}
\label{page:end-deformation-theory}

\begin{rmk}\label{rmk:critical-miracle}
The reality that the critical type in the prior proof is a collection of products of simple transpositions is a tautology, yet a miracle, limited to two-dimensional Galois representations. In the $n$-dimensional case, a  critical type is a permutation of $n$ letters and thus not always products of distinct simple transpositions. In particular, the bounds in \cite[Corollary 3.19]{Bergdall-Smoothness} only hold for {\em some} choices of refinements and so the smoothness statement we are building towards can also only be established on higher-dimensional eigenvarieties for {\em some} choices of refinement. Breuil, Hellmann, and Schraen, in fact, established systematic singularities occur in higher dimensions \cite[Theorem 5.4.2]{BreuilHellmannSchraen-LocalModel}.
\end{rmk}

We also need a minor result on the variation of determinants over $\mathfrak t_v^{\Ref}$. If $\chi: G_{F_v} \rightarrow L^\times$ is a character then its universal deformation functor $\mathfrak X_{\chi}$ is representable and its Zariski tangent space $\mathfrak t_{\chi}$ is $H^1(G_{F_v},\ad \chi) \cong H^1(G_{F_v}, L)$. Given $v \mid p$, the next lemma concerns the determinant morphism of functors $\det: \mathfrak X_{v} \rightarrow \mathfrak X_{\det \rho_v}$.
\begin{lem}\label{lem:det-surjective}
If $v \mid p$ then $\det: \mathfrak t_v^{\Ref} \rightarrow \mathfrak t_{\det \rho_v}$ is surjective.
\end{lem}
\begin{proof}
Write $d = \det \rho_v : G_{F_v} \rightarrow L^\times$ and suppose that $\widetilde d : G_{F_v} \rightarrow L[\varepsilon]^\times$ is an infinitesmial deformation. Thus we can write $\widetilde d = d\cdot (1 + a\varepsilon)$ where $a : G_{F_v} \rightarrow L$ is a continuous group morphism; the map $\widetilde d \mapsto a$ is the isomorphism $\mathfrak t_{\det \rho_v} \cong H^1(G_{F_v},L)$.

We let $\chi_a  = 1 + \frac{a}{2}\varepsilon$, which is now a character $\chi_a : G_{F_v} \rightarrow L[\varepsilon]^\times$ deforming the trivial character on $G_{F_v}$ whose square equals the character $1 + a\varepsilon$.  Let $\rho_{v,L[\varepsilon]} = \rho_v \otimes_L L[\varepsilon]$ be the constant deformation of $\rho_v$ to $L[\varepsilon]$ and then set $\widetilde \rho = \rho_{v,L[\varepsilon]} \otimes \chi_a \in \mathfrak t_v$. We have $\det(\widetilde \rho) = \widetilde d$ and so to prove the lemma it suffices to show that $\widetilde \rho \in \mathfrak t_v^{\Ref}$.

For that, consider $\chi_a|_{\mathcal O_v^\times}$ (via local class field theory). Since $\chi_a$ deforms the trivial character, the characters $\widetilde \eta_{i,v}$  above, for $\widetilde \rho$, can be taken to be $\widetilde \eta_{i,v} = \chi_a|_{\mathcal O_v^\times} \cdot \eta_{i,v}(\lambda)$ and we have
$$
\widetilde \rho \otimes \LT_{\varpi_v}(\widetilde \eta_{i,v})^{-1} = \rho_{v,L[\varepsilon]}\otimes \LT_{\varpi_v}(\eta_{i,v}(\lambda))^{-1} \otimes \left[\chi_a \otimes \LT_{\varpi_v}(\chi_a|_{\mathcal O_v^\times})^{-1}\right].
$$
The character $\chi_a \otimes \LT_{\varpi_v}(\chi_a|_{\mathcal O_v^\times})^{-1}$ is trivial on $\mathcal O_v^\times$, and in particular it is crystalline, and thus 
\begin{equation}\label{eqn:Dcrys}
D_{\crys}^+(\widetilde \rho \otimes \LT_{\varpi_v}(\widetilde \eta_{i,v})^{-1}) \cong  D_{\crys}^+(\rho_{v,L[\varepsilon]}\otimes \LT_{\varpi_v}(\eta_{i,v}(\lambda))^{-1})
\end{equation}
up to a twisting on the Frobenius operator $\varphi$ whch is trivial modulo $\varepsilon$ . Since the trivial deformation $\rho_{v,L[\varepsilon]}$ clearly lies in $\mathfrak t_v^{\Ref}$, we deduce from \eqref{eqn:Dcrys}  that $D_{\crys}^+(\widetilde \rho \otimes \LT_{\varpi_v}(\widetilde \eta_{i,v})^{-1})$ contains a free rank one submodule on which $\varphi^{f_v}$ acts through an eigenvalue deforming $\alpha_v^{\sharp}$ and thus $\widetilde \rho$ lies in $\mathfrak t_v^{\Ref}$ also.
\end{proof}

We now have the technical ingredients required to prove Theorem \ref{thm:smoothness}. Let $\mathfrak X_{\rho}$ denote the deformation functor of the global Galois representation $\rho$. The functor $\mathfrak X_{\rho}$ is representable because $\rho$ is (absolutely) irreducible by Theorem \ref{thm:classical-galois-representations}. So, if we define $X_{\rho}^{\Ref} \subseteq \mathfrak X_{\rho}$ to be the subfunctor of deformations that are weakly-refined at $v \mid p$ and minimally ramified at $v \nmid p$, then since each local functor is relatively representable we deduce that $\mathfrak X_{\rho}^{\Ref}$ is representable as well. Write $R_{\rho}^{\Ref}$ for the universal deformation ring representing $\mathfrak X_{\rho}^{\Ref}$.

From now on, we use the notation $H^1_{/f}(\ast,\ast)$ for the quotient $H^1(\ast,\ast)/H^1_f(\ast,\ast)$. The tangent space $\mathfrak t_\rho^{\Ref}$ to $\mathfrak X_{\rho}^{\Ref}$  sits in a an exact sequence
\begin{equation}\label{eqn:trhoref-ses}
0 \rightarrow \mathfrak t_{\rho}^{\Ref} \rightarrow H^1(G_{F},\ad \rho) \rightarrow \left(\prod_{v\mid p} H^1(G_{F_v},\ad\rho_v)/\mathfrak t_{v}^{\Ref}\right) \oplus \left(\prod_{v\nmid p} H^1_{/f}(G_{F_v},\ad\rho_v)\right).
\end{equation}
Let $H^1_f(G_F,\ad\rho)$ be the global adjoint Selmer group. Note that by part (a) of Proposition \ref{prop:local-bound} we have that $H^1_f(G_F,\ad\rho) \subseteq \mathfrak t_{\rho}^{\Ref}$ and then by \eqref{eqn:trhoref-ses} we have a canonical short exact sequence
\begin{equation}\label{eqn:selmer-about-to-vanish}
0 \rightarrow H^1_f(G_F,\ad\rho) \rightarrow \mathfrak t_{\rho}^{\Ref} \rightarrow \bigoplus_{v \mid p} \mathfrak t_v^{\Ref}/H^1_f(G_{F_v},\ad\rho_v).
\end{equation}
We are now ready to prove Theorem \ref{thm:smoothness}.

\begin{proof}[Proof of Theorem \ref{thm:smoothness} when $x$ satisfies condition (2) in Definition \ref{defn:decent}]
Let $T_x\mathscr E(\mathfrak n)_{\rmmid}$ be the tangent space to $\mathscr E(\mathfrak n)_{\rmmid}$ at $x$. Since $\mathscr E(\mathfrak n)_{\rmmid}$ is equidimensional of dimension $1 + d + \delta_{F,p}$ (Proposition \ref{prop:density-results}), we have a lower bound $1 + d + \delta_{F,p} \leq \dim T_x\mathscr E(\mathfrak n)_{\rmmid}$. To prove the theorem we need to show the reverse inequality holds.

Lemma \ref{lemma:lift-to-O} defines a lift $\rho_{\mathscr O_x}$ of $\rho$ to $\mathscr O_x$ and Proposition \ref{prop:family-galois-properties} shows that $\rho_{\mathscr O_x}$ satisfies the local deformation conditions for $\mathfrak X_{\rho}^{\Ref}$. So, by universality, we have a  canonical map $R_{\rho}^{\Ref} \rightarrow \widehat{\mathscr O}_x$ which is surjective, by a standard argument (see \cite[Proposition 4.3]{Bergdall-Smoothness} for instance). Thus, there is an induced inclusion $T_x\mathscr E(\mathfrak n)_{\rmmid} \subset \mathfrak t_{\rho}^{\Ref}$ on tangent spaces, sending a tangent vector to the attendant infinitesimal deformation of $\rho$. 

Here we have used part (c) of condition (2) in Definition \ref{defn:decent} in order to use the deformation functor $\mathfrak X_{\rho}^{\Ref}$. Part (a) of condition (2) is implicit in the entire discussion as we have assumed that $x$ lies on $\mathscr E(\mathfrak n)_{\rmmid}$. Finally, we use part (b) of condition (2), which states that $H^1_f(G_F,\ad \rho) = (0)$. Combining the vanishing with \eqref{eqn:selmer-about-to-vanish} we have natural containments
$$
T_x\mathscr E(\mathfrak n)_{\rmmid} \subset \mathfrak t_{\rho}^{\Ref} \subset \bigoplus_{v \mid p} \mathfrak t_v^{\Ref}/H^1_f(G_{F_v},\ad\rho_v).
$$
By part (2) of Proposition \ref{prop:local-bound}, the sum of the local spaces on the right of these containments has dimension $2d$. However, on $\mathfrak t_{\rho}^{\Ref}$ and its local avatars there is no restriction placed on a determinants such as the restriction that exists over $\mathscr E(\mathfrak n)_{\rmmid}$ (compare Lemma \ref{lem:relevant-determinant}. and Lemma \ref{lem:det-surjective}). After taking determinants into account, we see that $T_x\mathscr E(\mathfrak n)_{\rmmid}$ has dimension at most $1 + d + \delta_{F,p}$.

More precisely, recall that at $v \mid p$ we write $\mathfrak t_{\det \rho_v} \cong H^1(G_{F_v},\ad \det \rho_v)$ for the tangent space for deformations of the character $\rho_v$. Restriction $\mathcal O_v^\times$ defines, via the local Artin map, a canonical isomorphism $\mathfrak t_{\det \rho_v}/H^1_f(G_{F_v},\ad \rho\det_v) \cong T_{\det \rho_v|_{\mathcal O_v^\times}}\mathscr X(\mathcal O_v^\times)$. Recalling the notations of Lemma \ref{lem:relevant-determinant}, we then have a commuting diagram
\begin{equation*}
\xymatrix{
T_x\mathscr E(\mathfrak n)_{\rmmid} \ar[r] \ar[dd]_-{\det_p}  & \bigoplus_{v \mid p} \mathfrak t_{\rho_v}^{\Ref}/H^1_f(G_{F_v},\ad \rho_v) \ar[d]^-{(\det v)_{v \mid p}}\\
& \bigoplus_{v \mid p} \mathfrak t_{\det \rho_v}/H^1_f(G_{F_v},\ad \det \rho_v) \ar[d]^-{\cong}\\
T_{\eta(\lambda_x)}\mathscr X(\mathcal O_p^\times/\overline{\mathcal O_{F,+}^\times}) \ar[r] & T_{\eta(\lambda_x)}\mathscr X(\mathcal O_p^\times)
}
\end{equation*}
where the horizontal rows arrows are injections and, moreover, Lemma \ref{lem:det-surjective} implies the arrows in the right column are surjections. Thus, $T_x\mathscr E(\mathfrak n)_{\rmmid} \subset K$ where $K$ is defined by the short exact sequence
$$
0 \rightarrow K \rightarrow \bigoplus_{v \mid p} \mathfrak t_{\rho_v}^{\Ref}/H^1_f(G_{F_v},\ad\rho_v) \rightarrow T_{\eta(\lambda_x)}\mathscr X(\mathcal O_p^\times)/T_{\eta(\lambda_x)}\mathscr X(\mathcal O_p^\times/\overline{\mathcal O_{F,+}^\times}) \rightarrow 0.
$$
Combining this with part (2) of Proposition \ref{prop:local-bound}, we have
\begin{equation*}
\dim_L T_x\mathscr E(\mathfrak n)_{\rmmid} \leq 2d - (d - 1 - \delta_{F,p}) = d + 1 + \delta_{F,p}.
\end{equation*}
This completes the proof.
\end{proof}

%% file: period.tex
Recall that we write $\Gamma_F$ for the maximal abelian extension of $F$ unramified away from $p$ and $\infty$. This is a CPA group and hence we have $R$-valued distributions $\mathscr D(\Gamma_F,R)$ for any affinoid point $\Sp(R) = \Omega \rightarrow \mathscr W$. The goal of this section is to define, and study, canonical morphisms
\begin{equation*}
\mathscr P_{\Omega}: H^d_c(\mathfrak n,\mathscr D_{\Omega}) \rightarrow \mathscr D(\Gamma_F,R) 
\end{equation*}
which we call period maps. Amice's theorem then links the period maps to $p$-adic $L$-functions.

\subsection{Analytic distributions on $\Gamma_F$}\label{subsec:distribution-GammaF}
Consider the canonical exact sequence
\begin{equation}\label{eqn:GammaF-exact}
1 \rightarrow \overline{\mathcal O_{F,+}^{\times}} \rightarrow \mathcal O_p^\times \overset{j_p}{\longrightarrow} \Gamma_F \rightarrow \Cl_F^+ \rightarrow 1
\end{equation}
where $\Cl_F^+$ is the narrow class group, and the map $j_p$ is induced by the natural inclusion $\mathcal O_p^\times \hookrightarrow \mathbf A_F^\times$. We will need to make explicit some presentations of rings of analytic functions as limits of Banach algebras (LB-structures).

We begin with $\mathcal O_p^\times$. In Section \ref{subsec:monoid-action} we defined, for $f \in \mathscr A(\mathcal O_p^\times,\mathbf Q_p)$, the ``extension by zero'' function  $f_{!}: \mathcal O_p \rightarrow \mathbf Q_p$ 
\begin{equation*}
f_{!}(a) = \begin{cases}
f(a) & \text{if $a \in \mathcal O_p^\times$},\\
0 & \text{otherwise.}
\end{cases}
\end{equation*}
The map $f \mapsto f_{!}$ defines a closed embedding $\mathscr A(\mathcal O_p^\times,\mathbf Q_p) \hookrightarrow \mathscr A(\mathcal O_p,\mathbf Q_p)$. For $\mathbf s \in \mathbf Z_{\geq 0}^{\{v \mid p \}}$ we set $\mathbf A^{\mathbf s, \circ}(\mathcal O_p^\times, \mathbf Q_p) := \mathscr A(\mathcal O_p^\times, \mathbf Q_p) \cap \mathbf A^{\mathbf s, \circ}(\mathcal O_p,\mathbf Q_p)$ and 
\begin{equation*}
\mathbf A^{\mathbf s}(\mathcal O_p^\times, \mathbf Q_p) :=  \mathbf A^{\mathbf s, \circ}(\mathcal O_p^\times,\mathbf Q_p)[1/p] = \mathscr A(\mathcal O_p^\times, \mathbf Q_p) \cap \mathbf A^{\mathbf s}(\mathcal O_p,\mathbf Q_p),
\end{equation*} 
all the intersections happening within $\mathscr A(\mathcal O_p^\times, \mathbf Q_p)$. By \eqref{eqn:defn-locally-anal}, and because $\mathscr A(\mathcal O_p^\times, \mathbf Q_p)$ is closed inside $\mathscr A(\mathcal O_p,\mathbf Q_p)$, we deduce from \cite[Proposition 1.1.41]{Emerton-LocalAnalMemoir} that there is a natural topological identification
\begin{equation}\label{eqn:Optimes-ident}
\mathscr{A}(\mathcal O_p^\times, \mathbf Q_p) \simeq \dirlim_{|\mathbf s|\rightarrow +\infty} \mathbf A^{\mathbf s}(\mathcal O_p^\times, \mathbf Q_p).
\end{equation}

Now consider $\Gamma_F$. If $\gamma \in \Gamma_F$ write $r_{\gamma}: \Gamma_F \rightarrow \Gamma_F$ for multiplication by $\gamma$. Then, if $\gamma \in \Gamma_F$ and $f \in \mathscr A(\Gamma_F,\mathbf Q_p)$ we define
\begin{equation*}
f|_{\gamma \mathcal O_p^\times} := f \circ r_{\gamma} \circ j_p
\end{equation*}
which is an element of $\mathscr A(\mathcal O_p^\times,\mathbf Q_p)$. For each $\mathbf s \in \mathbf Z_{\geq 0}^{\{v \mid p \}}$ we define
\begin{equation}\label{eqn:definition-GammaF-s-analytic}
\mathbf A^{\mathbf s,\circ}(\Gamma_F,\mathbf Q_p) := \{f \in \mathscr A(\Gamma_F,\mathbf Q_p) \mid f|_{\gamma \mathcal O_p^\times} \in \mathbf A^{\mathbf s,\circ}(\mathcal O_p^\times,\mathbf Q_p) \text{ for each $\gamma\in\Gamma_F$}\},
\end{equation}
and
\begin{equation*}
\mathbf A^{\mathbf s}(\Gamma_F,\mathbf Q_p) := \mathbf A^{\mathbf s,\circ}(\Gamma_F,\mathbf Q_p)[1/p] = \{f \in \mathscr A(\Gamma_F,\mathbf Q_p) \mid f|_{\gamma \mathcal O_p^\times} \in \mathbf A^{\mathbf s}(\mathcal O_p^\times,\mathbf Q_p)\text{ for each $\gamma\in\Gamma_F$}\}.
\end{equation*}
\begin{lem}\label{lem:GammaF-dirlim}
The natural map 
\begin{equation}\label{eqn:GammaF}
\dirlim_{|\mathbf s|\rightarrow + \infty} \mathbf A^{\mathbf s}(\Gamma_F,\mathbf Q_p) \rightarrow\mathscr A(\Gamma_F,\mathbf Q_p)
\end{equation}
is a topological isomorphism.
\end{lem}
\begin{proof}
Note that $H= \im(j_p)$ is a CPA group and the natural map $\mathscr A(H,\mathbf Q_p) \rightarrow \mathscr A(\mathcal O_p^\times, \mathbf Q_p)$ is closed embedding. By the same argument above (especially \eqref{eqn:Optimes-ident} and \cite[Proposition 1.1.41]{Emerton-LocalAnalMemoir}) we deduce that $\mathbf A^{\mathbf s}(H,\mathbf Q_p) := \mathscr A(H,\mathbf Q_p) \cap \mathbf A^{\mathbf s}(\mathcal O_p^\times, \mathbf Q_p)$ presents $\mathscr A(H,\mathbf Q_p)$ topologically as a locally convex inductive limit 
\begin{equation}\label{eqn:H-isomorphism}
\mathscr A(H,\mathbf Q_p) \simeq \dirlim_{|\mathbf s|\rightarrow +\infty} \mathbf A^{\mathbf s}(H,\mathbf Q_p).
\end{equation}
Choose coset representatives $\gamma_1,\dots,\gamma_h$ for $\Gamma_F/H$. Then, the natural topological isomorphism
\begin{align*}
\mathscr A(\Gamma_F,\mathbf Q_p) &\overset{\simeq}{\longrightarrow} \bigoplus_{i=1}^h \mathscr A(H,\mathbf Q_p)\\
f &\mapsto \left(h \mapsto f(\gamma_i h)\right)
\end{align*}
 identifies the subspace $\mathbf A^{\mathbf s}(\Gamma_F,\mathbf Q_p)$ defined above with the direct sum of the subspaces $\mathbf A^{\mathbf s}(H,\mathbf Q_p)$ we just defined. So the map \eqref{eqn:GammaF} being a topological isomorphism is a consequence of the same fact for \eqref{eqn:H-isomorphism} and the fact that locally convex inductive limits commute with finite products. This completes the proof.
\end{proof}

Now suppose that $R$ is a $\mathbf Q_p$-Banach algebra and $R_0$ is a ring of definition. Then, for any of the CPA groups $G$ which appear above, we set $\mathbf A^{\mathbf s,\circ}(G,R) := \mathbf A^{\mathbf s, \circ}(G,\mathbf Q_p)\widehat{\otimes}_{\mathbf Z_p} R_0$ and $\mathbf A^{\mathbf s}(G,R) := \mathbf A^{\mathbf s,\circ}(G,R)[1/p] = \mathbf A^{\mathbf s}(G,\mathbf Q_p)\widehat{\otimes}_{\mathbf Q_p}R$. We define distribution algebras $\mathbf D^{\mathbf s}(G,R) = \mathbf A^{\mathbf s}(G,R)'$ and $\mathbf D^{\mathbf s,\circ}(G,R) = \Hom_{R_0}(\mathbf A^{\mathbf s,\circ}(\ast,R),R_0)$, with the same caveat as in Remark \ref{remark:notation}.

We note the following analogue of Lemma \ref{lem:character-analytic}, which illustrates the compatibility of our notations of $\mathbf s$-analytic.
\begin{lem}\label{lem:trivial-analog}
Suppose that $\chi: \mathcal O_p^\times \rightarrow R$ is a continuous character and $R_0 \subset R$ is a ring of definition containing the image of $\chi$. Then for $\mathbf s^{\circ}(\chi)$ as in Lemma \ref{lem:character-analytic}, we have $\chi \in \mathbf A^{\mathbf s^{\circ}(\chi)+1,\circ}(\mathcal O_p^\times,R)$ (similarly for $\mathbf s(\chi)$).
\end{lem}
\begin{proof}
This follows immediately from the following observation whose proof we omit:\ if $f: \mathcal O_p \rightarrow R$ is a function and $z\mapsto f(a+\varpi_p z)$ defines an element of $\mathbf A^{\mathbf s,\circ}(\mathcal O_p,R)$ for each $a \in \mathcal O_p$, then $f$ itself defines an element of $\mathbf A^{\mathbf s + \mathbf 1, \circ}(\mathcal O_p,R)$.
\end{proof}

\subsection{Definition of period maps}\label{subsec:period-definition}
Recall (Section \ref{subsec:sym-spaces}) that $\mathrm C_\infty$ denotes the Shintani cone. If $\Omega=\Sp(R) \rightarrow \mathscr W$ is a $\mathbf Q_p$-affinoid with corresponding weight $\lambda_{\Omega}$, then we write $\mathrm t^{\ast}{\mathbf A}^{\mathbf s}_{\Omega}$ for the local system on $\mathrm C_\infty$ induced by the right action of $\mathcal O_p^\times$
\begin{equation*}
f\big|_{u_p}(z) := f\big|_{\begin{smallpmatrix} u_p \\ & 1 \end{smallpmatrix}}(z) = \lambda_{\Omega,2}(u_p)f(u_pz)
\end{equation*}
for each $f \in \mathbf A^{\mathbf s}(\mathcal O_p,R)$, $u_p \in \mathcal O_p^\times$, and $z \in \mathcal O_p$ (here $\mathbf s \geq \mathbf s(\Omega)$). The action is compatible with changing $\mathbf s\geq \mathbf s(\Omega)$, and if $R_0 \subset R$ is a ring of definition for $R$ containing the values of $\lambda_{\Omega}$ and $\mathbf s \geq \mathbf s^{\circ}(\Omega)$ then it preserves the $R_0$-submodule $\mathrm t^{\ast}\mathbf A^{\mathbf s,\circ}_{\Omega}$.

\begin{lem}\label{lem:construction-Qlambda}
Fix a ring of definition $R_0 \subset R$ and $\mathbf s \geq \mathbf s^{\circ}(\Omega)$.  For $f \in \mathbf A^{\mathbf s,\circ}(\Gamma_F,R)$, $x \in \mathbf A_{F}^\times$, and $z \in \mathcal O_p$ define
\begin{equation}\label{eqn:definitinon-Q-map}
Q_{\Omega}^{\mathbf s,\circ}(f)(x)(z) = \begin{cases}
\lambda_{\Omega,2}^{-1}(z)\cdot f(xz) & \text{if $z \in \mathcal O_p^\times$};\\
0 & \text{otherwise.}
\end{cases}
\end{equation}
Then, $f \mapsto Q_{\Omega}^{\mathbf s,\circ}$ defines an $R_0$-module morphism
\begin{equation*}
Q_{\Omega}^{\mathbf s,\circ}: \mathbf A^{\mathbf s,\circ}(\Gamma_F,R) \rightarrow H^0(\mathrm C_\infty, \mathrm t^{\ast}{\mathbf A}^{\mathbf s,\circ}_{\Omega}).
\end{equation*}
Moreover, the induced map $Q_{\Omega}^{\mathbf s} : \mathbf A^{\mathbf s}(\Gamma_F,R) \rightarrow H^0(\mathrm C_\infty,\mathrm t^{\ast}{\mathbf A}^{\mathbf s})$ is independent of $R_0$ and if $\mathbf s' \geq \mathbf s$ then fits naturally into a commuting diagram
\begin{equation*}
\xymatrix{
\mathbf A^{\mathbf s}(\Gamma_F,R) \ar[d] \ar[r]^-{Q_{\Omega}^{\mathbf s}} & H^0(\mathrm C_\infty, \mathrm t^{\ast}{\mathbf A}^{\mathbf s}_{\Omega}) \ar[d]\\
\mathbf A^{\mathbf s'}(\Gamma_F,R) \ar[r]^-{Q_{\Omega}^{\mathbf s'}} & H^0(\mathrm C_\infty, \mathrm t^{\ast}{\mathbf A}^{\mathbf s'}_{\Omega})
}
\end{equation*}
and these extend to a map
\begin{equation*}
Q_{\Omega} : \mathscr A(\Gamma_F,R) \rightarrow H^0(\mathrm C_\infty, \mathrm t^{\ast}{\mathscr A}_{\Omega}).
\end{equation*}
This map is natural in the weight $\Omega$, in the sense that if $\Sp(R) = \Omega \rightarrow \mathscr W$ factors through some $\Sp(R') =\Omega'\to \mathscr W$ then we have a commuting diagram
\begin{equation*}
\xymatrix{
\mathscr A(\Gamma_F,R) \ar[r]^-{Q_{\Omega}} & H^0(\mathrm C_\infty, \mathrm t^{\ast}\mathscr A_{\Omega})\\
\mathscr A(\Gamma_F,R') \ar[r]_-{Q_{\Omega'}} \ar[u] & H^0(\mathrm C_\infty, \mathrm t^{\ast}\mathscr A_{\Omega'}). \ar[u]
}
\end{equation*}

\end{lem}
\begin{proof}
All the claims after inverting $p$ are clear, so we just prove the first statement.

Let $f \in \mathbf A^{\mathbf s,\circ}(\Gamma_F,R)$ and set $q = Q_{\Omega}^{\mathbf s,\circ}(f)$ defined in \eqref{eqn:definitinon-Q-map}. It follows from Lemma \ref{lem:trivial-analog} and the precise definitions of the radii that $q(x) \in \mathbf A_{\Omega}^{\mathbf s,\circ}$ for each $x \in \mathbf A_F^\times$, giving us a continuous function $q: \mathbf A_F^\times \rightarrow \mathbf A_{\Omega}^{\mathbf s, \circ}$ which we want to show it is a section in $H^0(\mathrm C_\infty, \mathrm t^{\ast}\mathbf A_{\Omega}^{\mathbf s, \circ})$. 

First, $q$ is locally constant on $F_\infty^\times$ because the function $f$ itself factors through $(F_{\infty}^\times)^{\circ}$. It remains to show that $q(\xi x u) = q(x)|_{u_p}$ for all $\xi \in F^\times$, $x \in \mathbf A_F^\times$ and $u \in \widehat{\mathcal O}_F^\times$. If $z\in \mathcal O_p - \mathcal O_p^\times$ then both $q(\xi xu)$ and $q(x)|_{u_p}$ vanish on $z$. If $z \in \mathcal O_p^\times$ though, then
\begin{equation*}
q(x)|_{u_p}(z) = \lambda_{\Omega,2}(u_p)q(x)(u_p z)= \lambda_{\Omega,2}^{-1}(z)f(xu_pz) = \lambda_{\Omega,2}^{-1}(z)f(\xi x u z) = q(\xi x u)(z).
\end{equation*}
For the second to last equality we used that $f$ is a function on $\Gamma_F$. This completes the proof.
\end{proof}

Throughout the rest of this subsection we consider an integral ideal $\mathfrak m \subset \mathcal O_F$ and we assume that $\mathfrak m \subset \mathbf p$. Since $K_1(\mathfrak m)$ is $\mathrm t$-good, we have a proper embedding $\mathrm t : \mathrm C_\infty \hookrightarrow Y_1(\mathfrak m)$ as in \eqref{eqn:shintani-torus}.

For each $\Omega$ as above, $\mathrm t^{\ast}\mathbf A_{\Omega}^{\mathbf s,\circ}$ is the pullback of the local system $\mathbf A_{\Omega}^{\mathbf s,\circ}$ on $Y_1(\mathfrak m)$ (which is well-posed because $\mathfrak m \subset \mathbf p$). There are similar obvious comments regarding $\mathbf A_{\Omega}^{\mathbf s}$ and $\mathscr A_{\Omega}$. Thus, by Lemma \ref{lem:construction-Qlambda}, we get a composition $\mathscr Q_{\Omega} = \mathrm t_{\ast} \circ \PD \circ Q_{\Omega}$
\begin{equation}\label{eqn:QOmega-defn}
\xymatrix{
\mathscr A(\Gamma_F,R) \ar@/_2pc/[rrr]_-{\mathscr Q_{\Omega}} \ar[r]^-{Q_\Omega} & H^0(\mathrm C_\infty, \mathrm t^{\ast}{\mathscr A}_\Omega) \ar[r]^-{\PD} & H^{\BM}_d(\mathrm C_\infty, \mathrm t^{\ast}{\mathscr A}_\Omega) \ar[r]^-{\mathrm t_{\ast}} & H^{\BM}_d(Y_{1}(\mathfrak m), {\mathscr A}_\Omega).
}
\end{equation}
We also have natural analogs $\mathscr Q_{\Omega}^{\mathbf s}$ and $\mathscr Q_{\Omega}^{\mathbf s,\circ}$. 

Now recall that the natural pairing $\mathscr D_{\Omega} \otimes_R \mathscr A_{\Omega} \rightarrow R$ together with the cap product defines a canonical $R$-bilinear pairing
\begin{equation*}
\langle - , - \rangle : H^d_c(Y_{1}(\mathfrak m), {\mathscr D}_\Omega) \otimes_R H^{\BM}_d(Y_{1}(\mathfrak m), {\mathscr A}_\Omega)  \rightarrow R.
\end{equation*}
Thus we define $\mathscr P_{\Omega}: H^d_c(Y_{1}(\mathfrak m), {\mathscr D}_\Omega) \rightarrow \Hom_R(\mathscr A(\Gamma_F,R),R)$ to be given by
\begin{equation}\label{eqn:pairing-defn}
\mathscr P_{\Omega}(\Psi)(f) = \langle \Psi, \mathscr Q_{\Omega}(f) \rangle.
\end{equation}
Replacing $\mathscr Q_{\Omega}$ with $\mathscr Q_{\Omega}^{\mathbf s}$ or $\mathscr Q_{\Omega}^{\mathbf s,\circ}$, we also get analogous morphisms $\mathscr P_{\Omega}^{\mathbf s}$ and $\mathscr P_{\Omega}^{\mathbf s, \circ}$. The rest of this subsection is devoted to proving the following theorem.

\begin{thm}\label{theorem:period-image}
The image of $\mathscr P_\Omega$ is contained in $\mathscr D(\Gamma_F,R) \subset \Hom_R(\mathscr A(\Gamma_F,R),R)$.
\end{thm}
Omitting the proof, we record precisely the definition of the period map(s).
\begin{defn}\label{defn:period-map}
If $\Omega=\Sp(R) \rightarrow \mathscr W$ is a point, then the period map $\mathscr P_{\Omega}$ is the $R$-linear map
\begin{align*}
\mathscr P_\Omega : H^d_c(Y_1(\mathfrak m), {\mathscr D}_\Omega) &\rightarrow \mathscr D(\Gamma_F,R)\\
\mathscr P_{\Omega}(\Psi)(f) &= \langle \Psi, \mathscr Q_{\Omega}(f)\rangle
\end{align*}
defined above.
\end{defn}
To prove Theorem \ref{theorem:period-image}, we note the following lemma on recognizing when certain linear functions are continuous.
\begin{lem}\label{lem:abstract-invert-p}
Suppose that $R$ is a $\mathbf Q_p$-Banach algebra and $R_0$ a ring of definition for $R$. If $M$ is a potentially orthonormalizable $R$-Banach module with $R$-Banach dual $M'$, and $M_0$ is any open and bounded $R_0$-submodule of $M$, then the natural map $\Hom_{R_0}(M_0,R_0)[1/p] \rightarrow \Hom_R(M,R)$ factors through an isomorphism
\begin{equation*}
\Hom_{R_0}(M_0,R_0)[1/p] \simeq M' ,
\end{equation*}
and the topology on $M'$ is the gauge topology defined by the $R_0$-submodule $\Hom_{R_0}(M_0,R_0)$ (i.e., the submodules $p^n\Hom_{R_0}(M_0,R_0)$ define a neighborhood basis of $0$).
\end{lem}
\begin{proof}
We first set some notation. If $I$ is a set we write $c(I,R)$ for the set of sequences $(r_i)_{i \in I}$ with $r_i \in R$ and such that for each $\varepsilon > 0$, $|r_i| < \varepsilon$ for all but finitely many $i$ (cf. \cite[Section 1]{Serre-CompactOperators}). We let $c(I,R_0)$ be those sequences with $r_i \in R_0$ for each $i$. Finally, we let $b(I,R)$ but those sequences $r_i$ which are bounded. Note that $c(I,R)' \simeq b(I,R)$.

By definition, we can choose a $R$-Banach module isomorphism $f:c(I,R)\simeq M$ for some set $I$.  Then $c(I,R_0) \subset c(I,R)$ is open and bounded, and $M_0 = f(c(I,R_0))$ is then an open and bounded $R_0$-submodule of $M$ (boundedness is clear, and openness follows from the open mapping theorem).  For this particular choice of $M_0$, the lemma follows by direct inspection, since $f$ induces compatible isomorphisms $\Hom_{R_0}(M_0,R_0) \simeq \prod_{I} R_0$ and $M'  \simeq b(I,R) \simeq (\prod_{I}R_0)[1/p]$.  The case of a general $M_0$ then reduces to this special case upon noting that any two open bounded $R_0$-submodules $M_{0,1},M_{0,2}$ satisfy $p^N M_{0,1} \subset M_{0,2} \subset p^{-N} M_{0,1}$ for $N \gg 0$.
\end{proof}
\begin{proof}[Proof of Theorem \ref{theorem:period-image}]
By Lemma \ref{lem:GammaF-dirlim} we have
\begin{equation}\label{eqn:dual-projective-limit}
 \Hom_{R}(\mathscr A(\Gamma_F,R) ,R) \cong \invlim_{|\mathbf{s}|\to \infty}  \Hom_{R}(\mathbf A^{\mathbf{s}}(\Gamma_F,R) , R)
 \end{equation}
and
 \begin{equation}\label{eqn:distr-projective-limit}
 \mathscr D(\Gamma_F,R) = \invlim_{|\mathbf s|\rightarrow +\infty} \mathbf D^{\mathbf s}(\Gamma_F,R).
 \end{equation}
 Choose now a ring of definition $R_0 \subset R$ containing the image of $\lambda_{\Omega}$. By definition, $R_0$ is open and bounded in $R$ and $\mathbf A^{\mathbf s,\circ}(\Gamma_F,R) \subset \mathbf A^{\mathbf s}(\Gamma_F,R)$ is also open and bounded. Furthermore, $\mathbf A^{\mathbf s}(\Gamma_F,R)$ is potentially orthonormalizable for each $\mathbf s$ since it is the completed scalar extension of a $\mathbf Q_p$-Banach space, which is always potentially orthornormalizable (see \cite[Proposition 1]{Serre-CompactOperators} and \cite[Lemma 2.8]{Buzzard-Eigenvarieties}). Thus, Lemma \ref{lem:abstract-invert-p}, together with \eqref{eqn:dual-projective-limit} and \eqref{eqn:distr-projective-limit}, implies that 
 \begin{equation}\label{eqn:distribution-isomorphism}
 \mathscr D(\Gamma_F,R) \simeq \invlim_{|\mathbf s|\rightarrow +\infty} \Hom_{R_0}(\mathbf A^{\mathbf s,\circ}(\Gamma_F,R),R_0)[1/p]\subset \Hom_R(\mathscr A(\Gamma_F,R),R).
 \end{equation}
Now consider the commuting diagram
 \begin{equation*}
 \xymatrix{
 \mathscr A(\Gamma_F,R) \ar[r]^-{\mathscr Q_{\Omega}} & H^{\BM}_d(Y_1(\mathfrak m), {\mathscr A}_{\Omega}) \\
 \mathbf A^{\mathbf s}(\Gamma_F,R) \ar[r]^-{\mathscr Q_{\Omega}^{\mathbf s}}\ar[u] & H^{\BM}_d(Y_1(\mathfrak m), {\mathbf A}_{\Omega}^{\mathbf s})\ar[u]\\
 \mathbf A^{\mathbf s,\circ}(\Gamma_F,R)\ar[r]^-{\mathscr Q_{\Omega}^{\mathbf s,\circ}} \ar[u] & H^{\BM}_d(Y_1(\mathfrak m), {\mathbf A}_{\Omega}^{\mathbf s,\circ}).\ar[u]
 }
 \end{equation*}
Since $\mathbf D_{\Omega}^{\mathbf s}$ is the $R$-Banach dual of $\mathbf A_{\Omega}^{\mathbf s}$ and $\mathbf D_{\Omega}^{\mathbf s,\circ} \subset \mathbf D_{\Omega}^{\mathbf s}$ is the $R_0$-linear dual of $\mathbf A^{\mathbf s,\circ}_{\Omega}$ (similarly for $\Gamma_F$), the naturality of the pairings $\langle -, - \rangle$ implies that
\begin{equation}\label{eqn:distribution-diagram}
\xymatrix{
H^d_c(Y_1(\mathfrak m),{\mathscr D}_{\Omega}) \ar[r]^-{\mathscr P_{\Omega}} \ar[d]& \Hom_R(\mathscr A(\Gamma_F,R),R) \ar[d]\\
H^d_c(Y_1(\mathfrak m),{\mathbf D}^{\mathbf s}_{\Omega}) \ar[r]^-{\mathscr P_{\Omega}^{\mathbf s}} & \Hom_R(\mathbf A^{\mathbf s}(\Gamma_F,R),R)\\
H^d_c(Y_1(\mathfrak m),{\mathbf D}_{\Omega}^{\mathbf s,\circ}) \ar[u] \ar[r]^-{\mathscr P_{\Omega}^{\mathbf s, \circ}} & \Hom_{R_0}(\mathbf A^{\mathbf s,\circ}(\Gamma_F,R),R_0) \ar[u]
}
\end{equation}
is also a commuting diagram.

Finally, consider $\Psi \in H^d_c(Y_1(\mathfrak m),\mathscr D_{\Omega})$ and write $\Psi^{\mathbf s} \in H^d_c(Y_1(\mathfrak m), \mathbf D_{\Omega}^{\mathbf s})$ for its restriction to $\mathbf D_{\Omega}^{\mathbf s}$. Since sheaf cohomology commutes with flat scalar extension in the coefficients, and $\mathbf D_{\Omega}^{\mathbf s,\circ}[1/p] = \mathbf D_{\Omega}^{\mathbf s}$, the bottom left vertical arrow in \eqref{eqn:distribution-diagram} is an isomorphism after inverting $p$. Following the diagram \eqref{eqn:distribution-diagram} around, we deduce that
\begin{equation*}
\mathscr P_{\Omega}^{\mathbf s}(\Psi^{\mathbf s}) \in \Hom_{R_0}(\mathbf A^{\mathbf s,\circ}(\Gamma_F,R), R_0)[1/p] \subset \Hom_R(\mathbf A^{\mathbf s}(\Gamma_F,R),R).
\end{equation*}
Since $\mathbf s$ is arbitrary, \eqref{eqn:distribution-isomorphism} shows that $\mathscr P_{\Omega}(\Psi) \in \mathscr D(\Gamma_F,R)$ by \eqref{eqn:distribution-isomorphism}.
\end{proof}

\subsection{Compatibilities}\label{subsec:compatibilities}
In this brief subsection we catalog some straightforward features of the period maps. We let $\mathfrak m \subset \mathbf p$ be an integral ideal and we generally let $\Omega=\Sp(R) \rightarrow \mathscr W$ be an affinoid point of weight space. An example of a straightforward feature is, for instance:\ if $\Omega \rightarrow \mathscr W$ factors through $\Sp(R')=\Omega'$ then the natural diagram
\begin{equation*}
\xymatrix{
H^d_c(Y_1(\mathfrak m), \mathscr D_{\Omega}) \ar[r]^-{\mathscr P_{\Omega}} & \mathscr D(\Gamma_F,R)\\
H^d_c(Y_1(\mathfrak m), \mathscr D_{\Omega'}) \ar[r]_-{\mathscr P_{\Omega'}} \ar[u] & \mathscr D(\Gamma_F,R') \ar[u]
}
\end{equation*}
commutes. Here is another compatibility.

\begin{lem}\label{lem:independence-of-level}
If $\mathfrak m' \subset \mathfrak m$ and $\pr: Y_1(\mathfrak m') \rightarrow Y_1(\mathfrak m)$ is the projection map, then $\mathscr P_{\Omega}(\Psi) = \mathscr P_{\Omega}(\pr^{\ast}\Psi)$ for all $\Psi \in H^d_c(Y_1(\mathfrak m), \mathscr D_{\Omega})$.
\end{lem}
\begin{proof}
Temporarily denote $\mathscr P_{\Omega}^{\mathfrak m}$ and $\mathscr Q_{\Omega}^{\mathfrak m}$ for the maps defined above with the level specified. We want to show $\mathscr P_{\Omega}^{\mathfrak m}(\Psi) = \mathscr P_{\Omega}^{\mathfrak m'}(\pr^{\ast}\Psi)$ for all $\Psi \in H^d_c(Y_1(\mathfrak m),\mathscr D_{\Omega})$. 

What  is clear is that $\pr$ is compatible with the two possible embeddings $\mathrm t$. So, it follows from the definition \eqref{eqn:QOmega-defn} that $\pr_{\ast}(\mathscr Q_{\Omega}^{\mathfrak m'}(f)) = \mathscr Q_{\Omega}^{\mathfrak m}(f)$ for all $f \in \mathscr A(\Gamma_F,R)$. And now if $f \in \mathscr A(\Gamma_F,R)$ and $\Psi \in H^d_c(Y_1(\mathfrak m),\mathscr D_{\Omega})$ then we see that
\begin{equation*}
\langle \pr^{\ast} \Psi , \mathscr Q_{\Omega}^{\mathfrak m'}(f) \rangle = \langle \Psi, \pr_{\ast}(\mathscr Q_{\Omega}^{\mathfrak m'}(f))\rangle = \langle \Psi, \mathscr Q_{\Omega}^{\mathfrak m}(f)\rangle.
\end{equation*}
This proves the lemma.
\end{proof}

We also note the following truly tautological relationship between the period map and the Amice transform (Proposition \ref{prop:amice}).
\begin{prop}\label{prop:amice-pairing-compat}
If $\chi : \Gamma_F \rightarrow R^\times$ is a continuous character and $\Psi \in H^d_c(Y_1(\mathfrak m),\mathscr D_{\Omega})$, then 
\begin{equation}
\mathscr P_{\Omega}(\Psi)(\chi) = \mathcal A_{\mathscr P_{\Omega}(\Psi)}(\chi).
\end{equation}
\end{prop}

Finally, it will be helpful to note the interaction between the period map and the Archimedean Hecke operators from Section \ref{subsec:archimedean}, where the notations $\pi_0(F_\infty^\times)$, $T_{\zeta}$, and so on are defined. A more involved calculation with the $U_v$-operators is the subject of Section \ref{subsec:abstract-interpolation} below.

\begin{prop}\label{prop:support-calculation}
Let $\Psi \in H^d_c(Y_1(\mathfrak m),\mathscr D_{\Omega})$. Then, 
\begin{enumerate}
\item If $\chi: \Gamma_F \rightarrow R^\times$ is a continuous character and $\zeta \in \pi_0(F_\infty^\times)$, then $\mathscr P_{\Omega}(T_\zeta \Psi)(\chi) = \chi(\zeta)\mathscr P_{\Omega}(\Psi)(\chi)$.
\item If $\varepsilon \in \{\pm 1\}^{\Sigma_F}$ and $\Psi \in H^d_c(Y_1(\mathfrak m),\mathscr D_{\Omega})^{\epsilon}$ then $\mathscr P_{\Omega}(\Psi)(\chi) = 0$ unless $\chi(\zeta) = \epsilon(\zeta)$ for all $\zeta \in \pi_0(F_\infty^\times)$.
\end{enumerate}
\end{prop}
\begin{proof}
$\mathscr P_{\Omega}(\Psi)$ is linear in $\Psi$. In particular, part (2) clearly follows from part (1). To prove (1), we set some notation. Write $\rho_{\zeta}: Y_1(\mathfrak m) \rightarrow Y_1(\mathfrak m)$ for right multiplication by $\begin{smallpmatrix} \zeta \\ & 1 \end{smallpmatrix}$, so $T_{\zeta}$ is the pullback $\rho_\zeta^{\ast}$. On the other hand, write $r_\zeta:  \mathrm C_\infty \rightarrow \mathrm C_\infty$ for right multiplication by $\zeta$.

It is trivial to check from the definition in Lemma \ref{lem:construction-Qlambda} that $r_\zeta^{\ast}(Q_{\Omega}(\chi)) = \chi(\zeta) Q_{\Omega}(\chi)$. Since $(r_\zeta)_{\ast} \circ \PD \circ r_\zeta^{\ast} = \PD$ (see Proposition \ref{prop:oritentation-preserving} and \eqref{eqn:naturality-PD}) we deduce that
\begin{equation}
((r_{\zeta})_{\ast} \circ \PD)(Q_{\Omega}(\chi)) = \chi(\zeta) \PD(Q_{\Omega}(\chi)).
\end{equation}
But $\mathscr Q_{\Omega} = \mathrm t_{\ast} \circ \PD \circ Q_\Omega$, and $(\rho_\zeta)_{\ast} \circ \mathrm t_{\ast}  = \mathrm t_{\ast} \circ (r_\zeta)_{\ast}$, so we get
\begin{equation*}
\mathscr P_{\Omega}(T_\zeta \Psi)(\chi) = \langle \rho_\zeta^{\ast} \Psi, \mathscr Q_{\Omega}(\chi)\rangle = \langle \Psi, (\rho_\zeta)_{\ast}\mathscr Q_{\Omega}(\chi) \rangle = \chi(\zeta)\langle \Psi, \mathscr Q_{\Omega}(\chi)\rangle,
\end{equation*}
as we promised in part (1).
\end{proof}

\begin{rmk}\label{rmk:direct-sum-decomps}
If $\epsilon \in \{\pm 1\}^{\Sigma_F}$ then write $\mathscr X(\Gamma_F)^{\epsilon}$ for those characters $\chi$ on $\Gamma_F$ such that $\chi(\zeta) = \epsilon(\zeta)$ for all $\zeta \in \pi_0(F_\infty^\times)$. Then, $\mathscr X(\Gamma_F)$ is a disjoint union 
\begin{equation*}
\mathscr X(\Gamma_F) = \bigcup_{\epsilon} \mathscr X(\Gamma_F)^{\epsilon}
\end{equation*}
and so $\mathscr O(\mathscr X(\Gamma_F)) = \bigoplus_{\epsilon} \mathscr O(\mathscr X(\Gamma_F)^{\epsilon})$. The previous two lemmas say that $\mathcal A \circ \mathscr P_{\Omega}$ respects the direct sum decompositions in the following diagram
\begin{equation*}
\xymatrix{
H^d_c(Y_1(\mathfrak m), \mathscr D_{\Omega}) \ar@{=}[d] \ar[r]^-{\mathscr P_{\Omega}} & \mathscr D(\Gamma_F,R) \ar[r]^-{\mathcal A} & \mathscr O(\mathscr X(\Gamma_F)) \widehat{\otimes}_{\mathbf Q_p} R \ar@{=}[d]\\
\bigoplus_{\epsilon} H^d_c(Y_1(\mathfrak m), \mathscr D_{\Omega})^{\epsilon} \ar[rr] & & \bigoplus_{\epsilon} \mathscr O(\mathscr X(\Gamma_F)^{\epsilon}) \widehat{\otimes}_{\mathbf Q_p} R.
}
\end{equation*}
\end{rmk}

\subsection{Growth properties}\label{subsec:growth-properties}

In this subsection we analyze the growth properties of our period morphisms $\mathscr P_\lambda$ (over a field). If $L/\mathbf Q_p$ is a finite extension then we always take the ring of integers $\mathcal O_L \subset L$ to be a ring of definition. We also fix an integral ideal $\mathfrak m \subset \mathbf p$ as in the previous subsection.

\begin{defn}\label{defn:order-of-growth}
Let $L/\mathbf Q_p$ be a finite extension and $h\geq 0$. If $\mu \in \mathscr D(\Gamma_F,L)$, then we say that $\mu$ has growth \emph{of order} $\leq h$ if
\begin{equation*}
\sup_{\mathbf s} \left(\sup_{f\in A^{\mathbf s, \circ}(\Gamma_F,L)} p^{-|\mathbf s| h}|\mu(f)| \right) < +\infty.
\end{equation*}
\end{defn}

\begin{prop}\label{prop:growth-condition}
If $\Psi \in H^d_c(Y_1(\mathfrak m),\mathscr D_{\lambda}\otimes_{k_\lambda} L)_{\leq h}$ then $\mathscr P_{\lambda}(\Psi)$ is a distribution with growth of order $\leq h$.
\end{prop}
\begin{proof}
By Lemma \ref{lem:independence-of-level} and Proposition \ref{prop:neatness} we may assume that $Y_1(\mathfrak m)$ is a manifold (compare with the proof of Lemma \ref{lemma:boundedness}). 

With $h$ fixed, this means that for some $\mathbf s_0$, the slope-$\leq h$ part 
\begin{equation*}
H^d_c(Y_1(\mathfrak m), \mathscr D_{\lambda}\otimes_{k_\lambda} L)_{\leq h} \simeq H^d_c(Y_1(\mathfrak m), \mathbf D^{\mathbf s_0}_{\lambda}\otimes_{k_\lambda} L)_{\leq h}
\end{equation*}
is equal to the slope-$\leq h$ part of the $d$-th cohomology of a Borel--Serre complex
\begin{equation*}
C^{\bullet}_c(\mathbf D_{\lambda}^{\mathbf s_0}\otimes_{k_\lambda} L) \simeq C^{\bullet}_c(\mathbf D_{\lambda}^{\mathbf s_0}\otimes_{k_\lambda} L)_{\leq h} \oplus C^{\bullet}_c(\mathbf D_{\lambda}^{\mathbf s_0}\otimes_{k_\lambda} L)_{>h}.
\end{equation*}
The terms which make up the complex $C^{\bullet}_c(\mathbf D_{\lambda}^{\mathbf s_0}\otimes_{k_\lambda} L)$ are finite direct sums of the Banach space $D_{\lambda}^{\mathbf s_0}\otimes_{k_\lambda} L$. Thus, the family of operators $\{ p^{|\mathbf s|h} U_p^{-|\mathbf s|}\}$ on $C^{\bullet}_c(\mathbf D_{\lambda}^{\mathbf s_0}\otimes_{k_\lambda} L)_{\leq h}$ is a family whose operator norms are bounded independent of $\mathbf s$.

Now choose $\Psi \in H^d_c(Y_1(\mathfrak m), \mathscr D_{\lambda}\otimes_{k_\lambda} L)_{\leq h}$, $\mathbf s_0$ as in the previous paragraph and write $\Psi^{\mathbf s_0}$ in $H^d_c(Y_1(\mathfrak m), \mathbf D^{\mathbf s_0}_{\lambda}\otimes_{k_\lambda} L)_{\leq h}$ for the restriction of $\Psi$ to radius $\mathbf s_0$. Given $\mathbf s$, write
\begin{equation*}
\Psi_{\mathbf s} := (p^{h}U_p^{-1})^{|\mathbf s|} (\Psi^{\mathbf s_0}) \in H^d_c(Y_1(\mathfrak m),\mathbf D_{\lambda}^{\mathbf s_0}\otimes_{k_\lambda} L).
\end{equation*}
By the boundedness discussion in the previous paragraph, we may choose a single $C > 0$ such that $p^C \Psi_{\mathbf s} \in H^d_c(Y_1(\mathfrak m),\mathbf D_{\lambda}^{\mathbf s_0,\circ}\otimes_{k_\lambda^{\circ}} \mathcal O_L)$ for all $\mathbf s \geq \mathbf s_0$. Here we are using the reduction in the first sentence of this proof so that $H^d_c(Y_1(\mathfrak m),\mathbf D_{\lambda}^{\mathbf s_0,\circ}\otimes_{\mathcal O_\lambda} \mathcal O_L)$ is the cohomology in degree $d$ of the bounded sub-complex $C^{\bullet}_c(\mathbf D^{\mathbf s_0,\circ}_\lambda \otimes_{k_\lambda^{\circ}} \mathcal O_L) \subset C^{\bullet}_c(\mathbf D^{\mathbf s_0}_\lambda \otimes_{k_\lambda} L)$.

Now let $\mathbf s \geq \mathbf s_0$ and $f \in \mathbf A^{\mathbf s,\circ}(\Gamma_F,L)$. Then we compute
\begin{equation}\label{eqn:growth-first-calculation}
\mathscr P_{\lambda}(\Psi)(f) = \mathscr P_{\lambda}^{\mathbf s_0}(\Psi^{\mathbf s_0})(f) = p^{-C}p^{-|\mathbf s|h} \mathscr P_{\lambda}^{\mathbf s_0}(U_p^{|\mathbf s|} p^C \Psi_{\mathbf s})(f).
\end{equation}
Now note that the Hecke operator $U_p$ is self-adjoint under $\langle - , - \rangle$, and so
\begin{equation}\label{eqn:OL-valued-pairing}
\mathscr P_{\lambda}^{\mathbf s_0}(U_p^{|\mathbf s|}p^C \Psi_{\mathbf s})(f) = \langle U_p^{|\mathbf s|} p^C \Psi_{\mathbf s}, \mathscr Q_{\lambda}^{\mathbf s, \circ}(f)\rangle = \langle p^C \Psi_{\mathbf s}, U_p^{|\mathbf s|} \mathscr Q_{\lambda}^{\mathbf s,\circ}(f)\rangle.
\end{equation}
Since \eqref{eqn:OL-valued-pairing} is the pairing between the element $U_p^{|\mathbf s|} \mathscr Q_{\lambda}^{\mathbf s,\circ}(f) \in H_d^{\BM}(Y_1(\mathfrak m), \mathbf A^{\mathbf s,\circ}\otimes_{k_\lambda^{\circ}} \mathcal O_L)$ and the image of $p^C \Psi_{\mathbf s}$ in $H_c^{d}(Y_1(\mathfrak m), \mathbf D^{\mathbf s,\circ}\otimes_{k_\lambda^{\circ}} \mathcal O_L)$, it is necessarily an element of $\mathcal O_L$. And so \eqref{eqn:growth-first-calculation} shows that
\begin{equation*}
\big|p^{-|\mathbf s|h} \mathscr P_{\lambda}(\Psi)(f)\big| < p^C,
\end{equation*}
independent of $\mathbf s$ and $f$, completing the proof.
\end{proof}

\subsection{The $p$-adic evaluation class}\label{subsec:padic-eval-class}
In this subsection we consider $L \subset \overline{\mathbf Q}_p$  finite over $\mathbf Q_p$ and containing the Galois closure of $F$. We also use $\lambda = (\kappa,w)$ to denote a cohomological weight, which we view as a $p$-adic weight as in Section \ref{subsec:integration-map}. 
\begin{defn}
If $m$ is an integer critical with respect to $\lambda$, then we define $\delta_{m,p}^{\star} \in \mathscr L_\lambda(L)^\vee$ by
\begin{equation*}
\delta_{m,p}^{\star}(X^j) =  \begin{cases}
{\kappa \choose j}^{-1} & \text{if $j = {\kappa + w \over 2} - m$},\\
0 & \text{otherwise.}
\end{cases}
\end{equation*}
\end{defn}
Now write $\mathbf N_p: \mathbf A_{F}^\times \rightarrow L^\times$ for the $p$-adic realization of the adelic norm $|\cdot|_{\mathbf A_F}$ via $\iota$. That is, $\mathbf N_p$ is given by the following formula
\begin{equation}\label{eqn:p-adic-realization}
\mathbf N_p(x) = |x_f|_{\mathbf A_F} \left(\prod_{v \mid \infty} \sgn(x_v)\right)\cdot\left(\prod_{v \mid p}\prod_{\sigma \in \Sigma_v} \sigma(x_v)\right).
\end{equation}
The character $\mathbf N_{p}$ is the adelic version of the cyclotomic character on $\Gamma_F$, but we also write $\mathbf N_p$ for the induced element of $\mathscr X(\Gamma_F)$. We also consider the local system $\mathrm t^{\ast}{\mathscr L}_{\lambda}(L)^\vee$ on $\mathrm C_\infty$ corresponding to the right $\mathcal O_p^\times$-module structure on $\mathscr L_{\lambda}(L)^{\vee}$ gotten by restricting to $\begin{smallpmatrix}\mathcal O_p^\times \\ & 1 \end{smallpmatrix} \hookrightarrow \GL_2(F_p)$.

\begin{lem}\label{lem:padic-class}
If $x_p \in F_p^\times$ then $\delta_{m,p}^{\star}|_{\begin{smallpmatrix}x_p & \\ & 1 \end{smallpmatrix}} = \left(\prod_{v \mid p}\prod_{\sigma \in \Sigma_v}\sigma(x_v)\right)^m\cdot \delta_{m,p}^{\star}$. Thus,
\begin{enumerate}
\item If $x_p \in \mathcal O_p^\times$ then $\delta_{m,p}^{\star}|_{\begin{smallpmatrix}x_p & \\ & 1 \end{smallpmatrix}} = \mathbf N_p^m(x_p)\delta_{m,p}^{\star}$.
\item The formula $\delta_{m,p}(x) = \mathbf N_p^m(x)\delta_{m,p}^{\star}$ defines an element of $H^0(\mathrm C_\infty, \mathrm t^{\ast}{\mathscr L}_{\lambda}(L)^{\vee})$.
\end{enumerate}
\end{lem}
\begin{proof}
By definition,
\begin{equation*}
\delta_{m,p}^{\star}|_{\begin{smallpmatrix} x_p & \\ & 1 \end{smallpmatrix}}(X^j) = \prod_{v \mid p} \prod_{\sigma \in \Sigma_v} \sigma(x_v)^{w-\kappa_{\sigma}\over 2}\sigma(x_v)^{\kappa_{\sigma}}\delta_{m,p}^{\star}\left((X_{\sigma}/\sigma(x_v))^{j_\sigma}\right).
\end{equation*}
The final term in the product is only non-zero if $j = {\kappa + w \over 2} - m$, in which case what we get is $\delta_{m,p}^{\star}(X^j)$ times the coefficient
\begin{equation*}
\prod_{v \mid p} \prod_{\sigma \in \Sigma_v} \sigma(x_v)^{w-\kappa_{\sigma}\over 2}\sigma(x_v)^{\kappa_{\sigma}}\sigma(x_v)^{m - {\kappa_{\sigma} + w \over 2} } = \prod_{v \mid p}\prod_{\sigma \in \Sigma_v} \sigma(x_v)^m.
\end{equation*}
This completes the proof point (1). To prove point (2) we first note that $\mathbf N_p$ is locally constant on $F_\infty^\times$ and thus to check $\delta_{m,p}$ actually defines a section we need to check that $\delta_{m,p}(\xi x u) = \delta_{m,p}(x)|_{u_p}$ if $\xi \in F^\times$, $x \in \mathbf A_F^\times$ and $u \in \widehat{\mathcal O}_F^\times$. But that follows immediately from point (1).
\end{proof}

Recall from Section \ref{subsec:integration-map} that we have the dual integration map $I_\lambda^{\vee} : \mathscr L_\lambda(L)^\vee \rightarrow \mathscr A_{\lambda} \otimes_{k_\lambda} L$, which is equivariant for the $\mathcal O_p^\times$-module structures on either side (Proposition \ref{prop:change-to-sharp}(3)). The reader expecting a $\sharp$ in the notation can look ahead to Remark \ref{rmk:what-twisting}.
\begin{lem}\label{eqn:transpose-eval-classes}
If $m$ is an integer critical with respect to $\lambda$, then
\begin{equation*}
I_\lambda^\vee(\delta_{m,p}^{\star}) = \prod_{v \mid p} \prod_{\sigma \in \Sigma_v} \sigma(-)^{{\kappa_\sigma - w \over 2} + m}.
\end{equation*}
In particular, if $z \in \mathcal O_p^\times$ then $I_\lambda^{\vee}(\delta_{m,p}^{\star})(z) = \mathbf N_p^m(z) \lambda_2^{-1}(z)$.
\end{lem}
\begin{proof}
Recall that $\delta_{m,p}^{\star}(X^j)$ is zero except if $j = {\kappa+w\over 2} - m$, in which case it takes the value ${\kappa \choose j}^{-1}$. Thus, if $\mu \in \mathscr D_\lambda(L)$, then
\begin{equation*}
\mu\left(I_{\lambda}^\vee(\delta_{m,p}^{\star})\right) = \delta_{m,p}^{\star}\left(I_{\lambda}(\mu)\right) = \delta_{m,p}^{\star}\left(\sum_j {\kappa \choose j} \mu(z^j) X^{\kappa-j}\right) = \mu\left(z^{{\kappa - w \over 2} + m}\right)
\end{equation*}
Since $\mu$ is arbitrary, we are finished.
\end{proof}
It is convenient here to calculate the interaction between $\delta_{m,p}$ and the map \[Q_{\lambda} : \mathscr A(\Gamma_F,k_{\lambda}) \rightarrow H^0(\mathrm C_\infty, \mathrm t^{\ast}{\mathscr A}_{\lambda})\] defined in Section \ref{subsec:period-definition}.

\begin{lem}\label{lem:Qlambda-p-adic-class}
Let $m$ be an integer critical with respect to $\lambda$.
\begin{enumerate}
\item If $x \in \mathbf A_F^\times$, then $Q_{\lambda}(\mathbf N_p^m)(x)|_{\mathcal O_p^\times} = I_\lambda^{\vee}(\delta_{m,p}(x))|_{\mathcal O_p^\times}$. 
\item If $f=(f_v) \in \mathbf Z_{\geq 1}^{\{v \mid p\}}$ and $a \in \mathcal O_p^\times$, then $Q_{\lambda}(\mathbf N_p^m)(x)\big|\begin{smallpmatrix} \varpi_p^{f} & a \\ & 1 \end{smallpmatrix} = I_\lambda^{\vee}(\delta_{m,p}(x))\big|\begin{smallpmatrix} \varpi_p^{f} & a \\ & 1 \end{smallpmatrix}$.
\end{enumerate}
\end{lem}
\begin{proof}
(2) follows from (1) because if $a \in \mathcal O_p^\times$ and $f_v \geq 1$ for all $v \mid p$, then $a + \varpi_p^{f}z \in \mathcal O_p^\times$ for all $z \in \mathcal O_p$.  It remains to prove (1). By definition, in Lemma \ref{lem:padic-class}, $\delta_{m,p}(x) = \mathbf N_p^m(x)\delta_{m,p}^{\star}$. Let $u \in \mathcal O_p^\times$. By Lemma \ref{eqn:transpose-eval-classes}, we have $I_{\lambda}^\vee(\delta_{m,p}^{\star})(u) = \mathbf N_p(u)^m\lambda_2^{-1}(u)$. Thus,  $I_{\lambda}^{\vee}(\delta_{m,p}(x))(u) = \mathbf N_p^m(x)\mathbf N_p^m(u) \lambda_2^{-1}(u)$. Since $\mathbf N_p(-)$ is multiplicitive and $u \in \mathcal O_p^\times$, this is also the value of $Q_\lambda(\mathbf N_p^m)(x)(u)$ (see \eqref{eqn:definitinon-Q-map}).
\end{proof}

In analogy with Definition \ref{defn:infinity-class} we make the following definition. 
\begin{defn}\label{defn:padic-class}
If $K \subset \GL_2(\mathbf A_{F,f})$ is a $\mathrm t$-good subgroup, then we define $\cl_p(m) := \mathrm t_{\ast}(\PD(\delta_{m,p})) \in H_d^{\BM}(Y_K,\mathscr L_{\lambda}(L)^{\vee})$ where $\delta_{m,p}$ is as in Lemma \ref{lem:padic-class}.
\end{defn}

The next proposition explains how the $p$-adic evaluation class is completely analogous to the Archimedean one previously defined in Definition \ref{defn:infinity-class}. (Note, that just as in Lemma \ref{lemma:infinity-class} there is no need to include the subgroup $K$ in the notation.) Namely, suppose that $E \subset \mathbf C$ is a subfield containing the Galois closure of $F$ in $\mathbf C$ and let $L = \mathbf Q_p(\iota(E))$. Then for any compact open subgroup $K \subset \GL_2(\mathbf A_{F,f})$ containing $\begin{smallpmatrix} \widehat{\mathcal O}_F^\times \\ & 1 \end{smallpmatrix}$ we have a natural commuting diagram
\begin{equation}\label{eqn:transfer-classes}
\xymatrix{
H_d^{\BM}(Y_K,\mathscr L_\lambda(E)^{\vee}) \ar[r]^-{\iota}_-{\simeq} & H_d^{\BM}(Y_K,\mathscr L_\lambda(L)^{\vee})\\
H_d^{\BM}(\mathrm C_\infty,\mathrm t^{\ast}\mathscr L_\lambda(E)^{\vee}) \ar[r]^-{\iota}_-{\simeq} \ar[u]^-{\mathrm t_{\ast}} & H_d^{\BM}(\mathrm C_\infty,\mathrm t^{\ast}\mathscr L_\lambda(L)^{\vee}) \ar[u]_-{\mathrm t_{\ast}}\\
H^0(\mathrm C_\infty,\mathrm t^{\ast}\mathscr L_\lambda(E)^{\vee}) \ar[r]^-{\iota}_-{\simeq} \ar[u]^-{\PD} & H^0(\mathrm C_\infty,\mathrm t^{\ast}\mathscr L_\lambda(L)^{\vee}) \ar[u]_-{\PD}
}
\end{equation}
The horizontal maps are all isomorphisms as indicated (Proposition \ref{prop:archi-shift}).
\begin{prop}\label{prop:p-adic-eval-class-defn}
If $m$ is an integral critical with respect to $\lambda$, then $\iota(\cl_\infty(m)) = \cl_p(m)$.
\end{prop}
\begin{proof}
By \eqref{eqn:transfer-classes} and the definitions it is enough to check that $\iota(\delta_m) = \delta_{m,p}$ (where $\delta_m$ is as in Proposition \ref{prop:eval-3} and $\delta_{m,p}$ is as in Lemma \ref{lem:padic-class}). 

To be clear, by the construction in Proposition \ref{prop:archi-shift},  $\iota(\delta_m)$ is the section $x \mapsto \iota(\delta_m(x))|_{x_p}$ where the $\iota$ on the right-hand side is the natural way of turning an element of $\mathscr L_\lambda(E)^{\vee}$ into an element of $\mathscr L_{\lambda}(L)^{\vee}$ via scalar extension along $\iota$. In particular, $\iota(\delta_m^{\star}) = \delta_{m,p}^{\star}$. Thus we can compute 
\begin{equation*}
\iota(\delta_m(x)) = \iota((|x_f|_{\mathbf A_{F}}\prod_{v \mid \infty} \sgn(x_v))^m \delta_m^{\star}) = (\mathbf N_p(x)\prod_{v \mid p} \prod_{\sigma \in \Sigma_v} \sigma(x_v)^{-1})^{m}\delta_{m,p}^{\star}
\end{equation*}
(compare with \eqref{eqn:p-adic-realization}). And now Lemma \ref{lem:padic-class} tells us that $\iota(\delta_m(x))|_{x_p} = \mathbf N_p(x) \delta_{m,p}^{\star} = \delta_{m,p}(x)$. This completes the proof.
\end{proof}

We will finally make a computation regarding the $p$-adic evaluation class that is used later in Corollary \ref{cor:abstract-interp-eigen}. (One could also give an analogous Archimedean computation and use Proposition \ref{prop:p-adic-eval-class-defn}.)

Let $v \mid p$ and denote $V_v^+ = \begin{smallpmatrix} \varpi_v \\ & 1 \end{smallpmatrix} \in \GL_2(\mathbf A_{F,f})$. Suppose that we fix a $\mathrm t$-good subgroup $K$. Write $K_{\varpi_v} := K \cap V_v^+ K (V_v^+)^{-1}$ and similarly $K_{\varpi_v^{-1}} = K \cap (V_v^+)^{-1} K V_v^+$. Then right multiplication by $V_v^+$ induces a map $V_v^+ : Y_{K_{\varpi_v}} \rightarrow Y_{K_{\varpi_v^{-1}}}$ that lifts to a map of local systems $\mathscr L_\lambda(L)^\vee \rightarrow \mathscr L_\lambda(L)^\vee$ given by $\delta \mapsto \delta|_{V_v^+}$. More precisely we are considering the composition of two maps on the level of local systems. The first is the map on the base given by $V_v^+$ and the identity map on the local system where $K_{\varpi_v^{-1}}$ acts on $\mathscr L_\lambda(L)$ by the twisted action $\mathscr L_\lambda(L)((V_v^+)^{-1})$ of $(V_v^+)^{-1} K V_v^+$. The second map is the identity on the base $Y_{K_{\varpi_v^{-1}}}$ and the right translation on the level of local systems. (Compare with \eqref{eqn:adelic-cochains}.)

In any case, we thus have a pushforward map 
\begin{equation}\label{eqn:V_v^+-pushforward}
(V_v^+)_{\ast} : H_d^{\BM}(Y_{K_{\varpi_v}}, \mathscr L_\lambda^\vee(L)) \rightarrow H_d^{\BM}(Y_{K_{\varpi_v^{-1}}}, \mathscr L_\lambda^\vee(L)).
\end{equation}
Note that both $K_{\varpi_v}$ and $K_{\varpi_v^{-1}}$ are still $\mathrm t$-good because $K$ is. Thus there is a $p$-adic evaluation class $\cl_p(m)$ on either side of  \eqref{eqn:V_v^+-pushforward}.

\begin{lem}\label{lemma:pushforward-class}
$(V_v^+)_{\ast} \cl_p(m) = q_v^{m} \cl_p(m)$.
\end{lem}
\begin{proof}
Consider the diagram
\begin{equation}\label{eqn:diagram-pushforward-crap}
\xymatrix{
H_d^{\BM}(Y_{K_{\varpi_v}}, \mathscr L_\lambda(L)^\vee) \ar[r]^-{(V_v^+)_{\ast}} & H_d^{\BM}(Y_{K_{\varpi_v^{-1}}}, \mathscr L_\lambda(L)^\vee)\\
H_d^{\BM}(\mathrm C_\infty, \mathrm t^{\ast}\mathscr L_\lambda(L)^\vee) \ar[u]^-{\mathrm t_{\ast}} \ar[r]^-{(r_{\varpi_v})_{\ast}} & H_d^{\BM}(\mathrm C_\infty, \mathrm t^{\ast}\mathscr L_\lambda(L)^\vee) \ar[u]_-{\mathrm t_{\ast}}\\
H^0(\mathrm C_\infty, \mathrm t^{\ast}\mathscr L_\lambda(L)^\vee)  \ar[u]_-{\PD}  & H^0(\mathrm C_\infty, \mathrm t^{\ast}\mathscr L_\lambda(L)^\vee)  \ar[l]_-{r_{\varpi_v}^{\ast}}. \ar[u]_-{\PD}
}
\end{equation}
Here we write $r_{\varpi_v}$ for the map on $\mathrm C_\infty$ which is right multiplication by $\varpi_v$ and with a non-trivial action to the level of local systems as above. The pullback map $r_{\varpi_v}^{\ast}$ is the map given by $(r_{\varpi_v}^{\ast}s)(x) = s(x\varpi_v)|_{(V_v^+)^{-1}}$ for all $s \in H^0(\mathrm C_\infty,\mathrm t^{\ast}\mathscr L_\lambda(L)^\vee)$ and $x \in \mathbf A_F^\times$. Taking $s = \delta_{m,p}$ we get
\begin{equation*}
r_{\varpi_v}^{\ast}(\delta_{m,p})(x) = \delta_{m,p}(x\varpi_v)|_{(V_v^+)^{-1}} = \mathbf N_p^m(x \varpi_v) \delta_{m,p}^{\star}|_{(V_v^+)^{-1}} = |\varpi_v|_{\mathbf A_F}^m \mathbf N_p^m(x)\delta_{m,p}^{\star} = q_v^{-m} \delta_{m,p}(x).
\end{equation*}
(The third equality used \eqref{eqn:p-adic-realization} and Lemma \ref{lem:padic-class}.) Thus, $r_{\varpi_v}^{\ast}\delta_{m,p} = q_v^{-m}\delta_{m,p}$. The conclusion now follows from Proposition \ref{prop:oritentation-preserving} and \eqref{eqn:naturality-PD}.
\end{proof}

\subsection{Abstract interpolation}\label{subsec:abstract-interpolation}

The main result in this subsection (Theorem \ref{thm:abstract-interpolation} below) is an ``abstract'' equality of functionals on a certain overconvergent cohomology group. It relates the Hecke action at $p$ to the $p$-adic evaluation classes via the period maps.
 
 In the remainder of this section we fix a finite order Hecke character $\theta$ of conductor $\mathfrak f = \prod_{v} \mathfrak p_v^{f_v}$ where $f_v = 0$ if $v \nmid p$. We write $\theta^{\iota}$ for $\iota\circ \theta$, which is thus a $\overline{\mathbf Q}_p$-valued Hecke character. We also fix a field $L \subset \overline{\mathbf Q}_p$ containing the Galois closure of $F$ in $\overline{\mathbf Q}_p$ and the values of $\theta^{\iota}$.  Thus $\theta^{\iota}$ is an element of $\mathscr A(\Gamma_F,L)$.
Set $f_{+,v} = \max(f_v,1)$ and let $f_+ =(f_{+,v}) \in \mathbf Z_{\geq 1}^{\{ v \mid p \}}$. We also fix a cohomological weight $\lambda$.   
 
 Recall the definition of
\begin{equation*}
Q_{\lambda} : \mathscr A(\Gamma_F,L) \rightarrow H^0(\mathrm C_\infty, \mathrm t^{\ast}\mathscr A_{\lambda}\otimes_{k_\lambda} L)
\end{equation*}
from Lemma \ref{lem:construction-Qlambda}. In particular, if $g \in \mathscr A(\Gamma_F,L)$ and $x \in \mathbf A_F^\times$ then $Q_{\lambda}(g)(x)$ is an element of $\mathscr A_{\lambda}\otimes_{k_\lambda} L$ and $\mathscr A_\lambda\otimes_{k_\lambda} L$ has a right action of $\Delta$.
\begin{lem}\label{lem:Q-lambda-theta}
If $a \in \mathcal O_p$, $g \in \mathscr A(\Gamma_F,L)$ and $x \in \mathbf A_{F}^\times$ then
\begin{equation*}
Q_{\lambda}(g \theta^{\iota})(x)\big|\begin{smallpmatrix} \varpi_p^{f_+} & a \\ & 1 \end{smallpmatrix} = 
\begin{cases}
\theta^{\iota}(ax)\cdot Q_{\lambda}(g)(x)\big|\begin{smallpmatrix} \varpi_p^{f_+} & a \\ & 1 \end{smallpmatrix} & \text{if $a \in \mathcal O_p^\times$},\\
0 & \text{if $a \not\in \mathcal O_p^\times$}.
\end{cases}
\end{equation*}
\end{lem}
\begin{proof}
If $z \in \mathcal O_p$, then $a + \varpi_p^{f_+}z \in \mathcal O_p^\times$ if and only if $a \in \mathcal O_p^\times$. Thus, by \eqref{eqn:function-action} and the definition of $Q_{\lambda}$ we deduce
\begin{multline}
Q_{\lambda}(g \theta^{\iota})(x)\big|\begin{smallpmatrix} \varpi_p^{f_+} & a \\ & 1 \end{smallpmatrix}(z)\\
=Q_{\lambda}(g\theta^{\iota})(x)(a + z\varpi_p^{f_+}) = \begin{cases}
(g\theta^{\iota})(x(a+\varpi_p^{f_+}z))\lambda_2^{-1}(a+z\varpi_p^{f_+}) & \text{if $a \in \mathcal O_p^\times$},\\
0 & \text{if $a \not\in \mathcal O_p^\times$}.
\end{cases}
\end{multline}
This already proves the case $a \not\in \mathcal O_p^\times$. When $a \in \mathcal O_p^\times$, $\theta^{\iota}(a+\varpi_p^{f_+}z) = \theta^{\iota}(a)$ by definition of the conductor of $\theta$ and thus the case $a \in \mathcal O_p^\times$ follows from multiplicativity of $\theta$.
\end{proof}
We now fix further notation. Set
\begin{align*}
S_0^\times &:= \prod_{\substack{v \mid p\\ f_v=0}} (\mathcal O_v/\varpi_v \mathcal O_v)^\times\\
S_1^\times &:= \prod_{\substack{v \mid p \\ f_v>0}} (\mathcal O_v/\varpi_v^{f_v} \mathcal O_v)^\times.
\end{align*}
If $b \in S_0^\times$ and $c \in S_1^\times$ then we write $b+c$ for the natural element of $(\mathcal O_p/\varpi_p^{f_+}\mathcal O_p)^\times \simeq S_0^\times \times S_1^\times$. Implicit in the notation below is that any choices of lifts are irrelevant. For instance, $\theta^{\iota}(b+c)$ makes perfect sense for $b \in S_0^\times$ and $c \in S_1^\times$.

As before, let $v \mid p$ and let $V_v^+$ be the matrix $\begin{smallpmatrix} \varpi_v \\ & 1 \end{smallpmatrix} \in \GL_2(\mathbf A_{F,f})$. In general, if $K \subset \GL_2(\mathbf A_{F,f})$ is a compact open subgroup and $\mathfrak m$ is an ideal then we have a natural map
\begin{equation*}
\begin{smallpmatrix} \varpi_v^{f_{v,+}} \\ & 1 \end{smallpmatrix} = (V_v^+)^{f_{v,+}} : H^{\ast}_c(Y_1(\mathfrak m),\mathscr L_{\lambda}(L)) \rightarrow H^{\ast}_c(Y_1^0(\mathfrak m; \mathfrak p_v^{f_v^+}), \mathscr L_{\lambda}(L))
\end{equation*}
where $Y_1^0(\mathfrak m; \mathfrak p_v^{f_{v,+}}) = Y_{K_1^0(\mathfrak m; \mathfrak p_v^{f_{v,+}})}$ and 
\begin{equation*}
K_1^0(\mathfrak m; \mathfrak p_v^{f_{v,+}}) = \{ g = \begin{smallpmatrix} a & b \\ c & d \end{smallpmatrix} \in K_1(\mathfrak m) \mid b \in \mathfrak p_v^{f_{v,+}}\widehat{\mathcal O}_F\}.
\end{equation*}
It is clear that this morphism is independent of the choice of uniformizer $\varpi_v$. Furthermore, since $K_1(\mathfrak m) \supset K_1^0(\mathfrak m; \mathfrak p_v^{f_{v,+}})$ if $\mathfrak p_v \mid \mathfrak m$ then we can also take the endomorphism $U_v$ of $H^{\ast}_c(Y_1(\mathfrak m), \mathscr L_{\lambda}(L))$ and post-compose it with pullback along $Y_1^0(\mathfrak m; \mathfrak p_v^{f_{v,+}}) \rightarrow Y_1(\mathfrak m)$. This discussion gives meaning to the following lemma.

\begin{lem}\label{lemma:second-half-of-computations}
Let $\mathfrak m \subset \mathbf p$. If $\psi \in H^{\ast}_c(Y_1(\mathfrak m), \mathscr L_\lambda(L))$ is represented by an adelic cochain $\widetilde \psi$ and $W := \prod_{f_v=0} (U_v-V_v^+) \prod_{f_v > 0} (V_v^+)^{f_v}$, then $W(\psi) \in H^{\ast}_c(Y_1^0(\mathfrak m;\mathbf p^{f_+}), \mathscr L_{\lambda}(L))$ is represented by the adelic cochain
\begin{equation*}
W(\widetilde \psi)(\sigma) = \sum_{b \in S_0^\times} \begin{smallpmatrix} \varpi_p^{f_+} & b \\ & 1\end{smallpmatrix} \cdot \widetilde \psi\left( \sigma \begin{smallpmatrix} \varpi_p^{f_+} & b \\ & 1 \end{smallpmatrix}\right).
\end{equation*}
\end{lem}
\begin{proof}
According to the definitions (Section \ref{subsec:adelic-cochains}) we have
\begin{align*}
((V_v^+)^{f_v} \widetilde \psi)(\sigma) &= \begin{smallpmatrix} \varpi_v^{f_v} & \\  & 1 \end{smallpmatrix} \cdot \widetilde \psi \left(\sigma \begin{smallpmatrix} \varpi_v^{f_v} & \\  & 1 \end{smallpmatrix}\right)\\ 
((U_v-V_v^+)\widetilde \psi)(\sigma) &= \sum_{b_v \in (\mathcal O_v/\varpi_v\mathcal O_v)^\times} \begin{smallpmatrix} \varpi_v & b_v \\  & 1 \end{smallpmatrix} \cdot \widetilde \psi \left(\sigma \begin{smallpmatrix} \varpi_v & b_v \\  & 1 \end{smallpmatrix}\right).
\end{align*}
Here we are using $\mathfrak m \subset \mathbf p$ to use the given description of the $U_v$-operator (see Remark \ref{rmk:U-operator-notation}). In the second formula, we are free to choose coset representatives in $\widehat{\mathcal O}_F$ for $(\mathcal O_v/\varpi_v \mathcal O_v)^\times$ that are supported only on $v$. But then the matrices in the two formulas above, as one ranges over all $v \mid p$,  necessarily commute and the formula for $W(\widetilde \psi)$ is clear.
\end{proof}
We make a similar calculation for the next lemma. We do not specify the level at which the result ends up, but this omission is harmless because we will apply Lemma \ref{lemma:adelic-twisting-expansion} only through  Lemma \ref{lemma:second-half-of-computations} at which point we know precisely the resulting level subgroup.
\begin{lem}\label{lemma:adelic-twisting-expansion}
Let $\mathfrak m \subset \mathbf p$. If $\psi \in H^{\ast}_c(Y_1(\mathfrak m), \mathscr L_\lambda(L))$ is represented by an adelic cochain $\widetilde \psi$ and $b \in S_0^\times$, then $\begin{smallpmatrix} \varpi_p^{f_+} & b \\ & 1 \end{smallpmatrix}\cdot \tw_{\theta^{\iota}}^{\cl}(\psi)$ is represented by the adelic cochain
\begin{equation*}
\sigma = \sigma_\infty \otimes [g_f] \mapsto (\prod_{\substack{v \mid p\\ f_v=0}}\theta^{\iota}(\varpi_v))\theta^{\iota}(\det g_f)\sum_{c \in S_1^\times} \theta^{\iota}(c + b)\begin{smallpmatrix} \varpi_p^{f_+} & b + c \\ & 1 \end{smallpmatrix}\cdot \widetilde \psi\left(\sigma \begin{smallpmatrix} \varpi_p^{f_+} & b + c \\ & 1 \end{smallpmatrix}\right).
\end{equation*}
\end{lem}
\begin{proof}
First, by definition we have
\begin{equation}\label{eqn:first-action}
\left(\begin{smallpmatrix} \varpi_p^{f_+} & b \\ & 1 \end{smallpmatrix}\cdot \tw_{\theta^{\iota}}^{\cl}(\widetilde \psi)\right)(\sigma) = \begin{smallpmatrix} \varpi_p^{f_+} & b \\ & 1 \end{smallpmatrix}\cdot \tw_{\theta^{\iota}}^{\cl}(\widetilde \psi)\left(\sigma \begin{smallpmatrix} \varpi_p^{f_+} & b \\ & 1 \end{smallpmatrix}\right).
\end{equation}
Set $\varpi^{(0)} = \prod_{\substack{v \mid p\\ f_v=0}} \varpi_v$ and $\varpi^{(1)} := \prod_{\substack{v \mid p\\ f_v>0}} \varpi_v^{f_v}$, so that $\mathfrak f \widehat{\mathcal O}_F = \varpi^{(1)} \widehat{\mathcal O}_F$. If $c \in S_1^\times$ then  choose a lift $\widehat c$ to $\widehat{\mathcal O}_F^\times$ so that $\widehat{c} \mapsto c$ in $S_1^\times$ but $\widehat{c} \mapsto b$ in $S_0^\times$. Then, $\{\widehat{c}/\varpi^{(1)}\}_{c \in S_1^\times}$ is a set of representatives for $\Upsilon_{\mathfrak f}^\times$, so  Proposition \ref{prop:twisting-adelic-cochains} implies that 
\begin{multline}\label{eqn:second-action}
\tw_{\theta^{\iota}}^{\cl}(\widetilde \psi)\left(\sigma \begin{smallpmatrix} \varpi_p^{f_+} & b \\ & 1 \end{smallpmatrix}\right)\\ = \theta^{\iota}(\varpi_p^{f_+}\det g_f) \sum_{c \in S_1^\times} \theta^{\iota}(\widehat{c}/\varpi^{(1)}) \begin{smallpmatrix} 1 & \widehat{c}_0/\varpi^{(1)} \\ & 1 \end{smallpmatrix} \widetilde \psi\left(\sigma \begin{smallpmatrix} \varpi_p^{f_+} & b \\ & 1 \end{smallpmatrix}\begin{smallpmatrix} 1 & \widehat{c}_0/\varpi^{(1)} \\ & 1 \end{smallpmatrix}\right),
\end{multline}
where as before $\widehat{c}_0$ is zero at places $v \nmid \mathfrak f$. In particular, $\varpi_p^{f_+} \widehat{c}_0/\varpi^{(1)} = \widehat{c}_0$ and so
\begin{equation*}
\begin{smallpmatrix} \varpi_p^{f_+} & \ast \\ & 1 \end{smallpmatrix}\begin{smallpmatrix} 1 & \widehat{c}_0/\varpi^{(1)} \\ & 1\end{smallpmatrix} = \begin{smallpmatrix} \varpi_p^{f_+} & \ast + \widehat{c}_0\\ & 1 \end{smallpmatrix}.
\end{equation*}
On the other hand, $\varpi_p^{f_+}\widehat{c}/\varpi^{(1)} = \varpi^{(0)} \widehat{c}$, whence $\theta^{\iota}(\varpi_p^{f_+} \widehat c / \varpi^{(1)}) = \theta^{\iota}(\varpi^{(0)})\theta^{\iota}(\widehat c)$. We finally remark that $\theta^{\iota}(\widehat c) = \theta^{\iota}(c+b)$ by construction of $\widehat c$. Putting these observations into \eqref{eqn:first-action} and \eqref{eqn:second-action}, we have completed the proof.
\end{proof}
In the statement of the next theorem, we write $U_{\varpi_p}$ for the Hecke operator defined by the double coset of $\begin{smallpmatrix} \varpi_p \\ & 1 \end{smallpmatrix}$. We could have also called it $U_{\mathbf p}$ but we fear it looks too close to $U_p$. In any case, the point is that $U_{\varpi_p}^{f_+} = \prod_{v \mid p} U_v^{f_{v,+}}$.

\begin{thm}\label{thm:abstract-interpolation}
Let $\mathfrak m \subset \mathbf p$,  $\Psi \in H^d_c(Y_1(\mathfrak m),{\mathscr{D}}_{\lambda}\otimes_{k_\lambda} L)$, and let $m$ be an integer which is critical with respect to $\lambda$.  Then,
\begin{multline}\label{eqn:abstract-interpolation}
\langle U_{\varpi_p}^{f_+}\Psi, \mathscr Q_{\lambda}(\mathbf N_p^m \theta^{\iota}) \rangle\\
= \varpi_{p}^{-f_+ {w-\kappa \over 2}}\bigg\langle \left(\prod_{\substack{v \mid p\\ f_v=0}} \theta^{\iota}(\varpi_v)^{-1}(U_v-V_v^+)\prod_{\substack{v \mid p\\ f_v>0}}(V_v^+)^{f_v}\right) \tw_{\theta^{\iota}}^{\cl}I_{\lambda}(\Psi), \cl_p(m)\bigg\rangle
\end{multline}
\end{thm}
Before giving the proof, we want to clarify two points about the statement of the theorem.
\begin{rmk}\label{rmk:what-twisting}
On the right-hand side, the element $I_\lambda(\Psi)$ is meant to be an $\mathscr L_\lambda(L)$-valued cohomology class, not an $\mathscr L_\lambda^{\sharp}(L)$-valued one. (See Definition \ref{defn:sharp-notation} for the definition of $\mathscr L_\lambda^{\sharp}(L)$.) The same remark applies to its twist by $\theta^{\iota}$. The only difference is the Hecke action at $p$, and if you want an $\mathscr L_\lambda^{\sharp}(L)$-valued class, which is arguably more a more natural choice, then of course you remove the $\varpi_p$-factor from the front of the formula. See \eqref{eqn:see-above} below. 

But for the sake of comparing to classical $L$-values, if we make the switch in the previous paragraph then we also have to remember to view $\cl_p(m)$ as a $(\mathscr L_{\lambda}^{\vee})^{\sharp}$-valued homology class and take this into account during computations. (Compare with Corollary \ref{cor:abstract-interp-eigen} below).
\end{rmk}

\begin{rmk}\label{rmk:clarification-meaning}
In the proof below we are going to work at the level of adelic cochains (as indicated by the previous lemmas). Since we elide the actual cohomology in the arguments, and thus omit making precise the levels, let us clarify further the two sides of the formula \eqref{eqn:abstract-interpolation}.

We hope that the left-hand side of \eqref{eqn:abstract-interpolation} is clear:\ we are taking the class $U_{\varpi_p}^{f_+}\Psi$  in $H^d_c(Y_1(\mathfrak m), \mathscr D_\lambda \otimes_{k_\lambda} L)$ and pairing it with the class $\mathscr Q_{\lambda}(\mathbf N_p^{m}\theta^{\iota}) \in H_d^{\BM}(Y_1(\mathfrak m), \mathscr A_{\lambda}\otimes_{k_\lambda} L)$. 

Let's unwind the right-hand side of \eqref{eqn:abstract-interpolation}. First, $\tw_{\theta^{\iota}}I_{\lambda}(\Psi)$ is a class in $H^d_c(Y_{1}(\mathfrak m \mathfrak f^2), \mathscr L_{\lambda}(L))$. If we write $W$ for the operator acting on this class in  \eqref{eqn:abstract-interpolation} (and the proof below), then $W\tw_{\theta^{\iota}}(I_\lambda(\Psi))$ defines a class in $H^d_c(Y_1^0(\mathfrak{mf}^2; \mathfrak f^2); \mathscr L_{\lambda}(L))$ by the discussion preceding Lemma \ref{lemma:second-half-of-computations}. And since $\begin{smallpmatrix} \widehat{\mathcal O}_F^\times \\ & 1 \end{smallpmatrix} \subset K_1^0(\mathfrak{mf}^2; \mathfrak f^2)$, we can make sense of the evaluation class $\cl_p(m) \in H_d^{\BM}(Y_1^0(\mathfrak{mf}^2;\mathfrak f^2), \mathscr L_{\lambda}^\vee(L))$ which was carefully and universally defined in Definition \ref{defn:padic-class}. We then pair these classes, and this is what we mean by the right-hand side of \eqref{eqn:abstract-interpolation}.
\end{rmk}
\begin{proof}[Proof of Theorem \ref{thm:abstract-interpolation}]
For the purposes of the proof, write
\begin{equation*}
W := \prod_{\substack{v \mid p\\ f_v = 0}} \theta^{\iota}(\varpi_v)^{-1}(U_v-V_v^+) \prod_{\substack{v\mid p\\ f_v > 0}} (V_v^+)^{f_v}
\end{equation*}
for the operator appearing on the right-hand side of \eqref{eqn:abstract-interpolation}, as in Remark \ref{rmk:clarification-meaning}. (It is a scaling of the operator ``$W$'' in Lemma \ref{lemma:second-half-of-computations}.)

Recall that $\mathscr Q_{\lambda} = \mathrm t_{\ast}\circ \PD \circ Q_\lambda$. Thus, according to \eqref{eqn:push-pull-duality-formula} we have
\begin{equation}\label{eqn:Up-pairing}
\langle U_{\varpi_p}^{f_+}\Psi, \mathscr Q_\lambda(\mathbf N_p^m \theta^{\iota}) \rangle = \langle \mathrm t^{\ast}(U_{\varpi_p}^{f_+}\Psi) \cup Q_\lambda(\mathbf N_p^m \theta^{\iota}), [\mathrm C_\infty] \rangle,
\end{equation}
where $[\mathrm C_\infty]$ is the Borel--Moore fundamental class for $\mathrm C_\infty$. For the purposes of this equation, the cup-product $\mathrm t^{\ast}(U_{\varpi_p}^{f_+}\Psi) \cup Q_\lambda(\mathbf N_p^m \theta^{\iota}) \in H^d_c(\mathrm C_\infty, \mathrm t^{\ast}(\mathscr D_{\lambda} \otimes_L \mathscr A_{\lambda}))$ is implicitly its image in $H^d_c(\mathrm C_\infty, L)$ under the natural map. 

Similarly, since $\cl_p(m) = \mathrm t_{\ast}(\PD(\delta_{m,p}))$ (Definition \ref{defn:padic-class}) we have
\begin{equation}\label{eqn:W-pairing}
\varpi_p^{-f_+\cdot {w- \kappa\over 2}}\langle W \tw_{\theta^{\iota}} I_\lambda(\Psi), \cl_p(m)\rangle =  \varpi_p^{-f_+\cdot {w-\kappa \over 2}} \langle \mathrm t^{\ast}\left(W \tw_{\theta^{\iota}} I_\lambda(\Psi)\right) \cup\delta_{m,p}, [\mathrm C_\infty]\rangle
\end{equation}
(with the same caveat on the meaning of the cup product). Comparing \eqref{eqn:Up-pairing} and \eqref{eqn:W-pairing}, it is enough to show that the cup products appearing define the same elements of $H^d_c(\mathrm C_\infty, L)$. For that, we will explicitly compute using adelic cochains. 

Fix a singular $d$-chain $\sigma = \sigma_\infty \otimes [x]$ on $(F_{\infty}^\times)^{\circ}\times \mathbf A_{F,f}^\times$, and a representative $\widetilde \Psi$ for $\Psi$ in the adelic cochains $C_{\ad,c}^{\bullet}(K_{1}(\mathfrak m), \mathscr D_{\lambda}\otimes_{k_\lambda} L)$. To cut down on parentheses, let us write $\mathrm t_{\sigma} := \mathrm t(\sigma)$ for the image of $\sigma$ under $\mathrm t$. Then, the definition of the cup product on the level of cochains means that we want to show
\begin{equation}\label{eqn:interpolation-goal}
\underset{\in \mathscr D_{\lambda}(L)}{\underbrace{(U_{\varpi_p}^{f_+}\widetilde \Psi)(\mathrm t_\sigma)}}\bigl(\underset{\in \mathscr A_{\lambda}(L)}{\underbrace{Q_{\lambda}(\mathbf N_p^m\theta^{\iota})(x)}}\bigr) = \varpi_p^{-f_+\cdot{w-\kappa\over 2}}\underset{\in \mathscr L_{\lambda}(L)}{\underbrace{\left(W\tw_{\theta^{\iota}}I_{\lambda}(\widetilde \Psi)\right)(\mathrm t_\sigma)}}\bigl(\underset{\in \mathscr L_{\lambda}(L)^{\vee}}{\underbrace{\delta_{m,p}(x)}}\bigr).
\end{equation}
(To aid the reader, we have indicated where each object lives with underbraces.)

We begin computing the left-hand side of \eqref{eqn:interpolation-goal}. In general, if $s \in H^0(\mathrm C_\infty, \mathrm t^{\ast}\mathscr A_{\lambda})$, then 
\begin{equation}\label{eqn:cup-product-summation}
(U_{\varpi_p}^{f_+}\widetilde{\Psi})(\mathrm t_\sigma)\left(s(x)\right) = \sum_{a \in \mathcal O_p/\varpi_p^{f_+}\mathcal O_p} \widetilde \Psi\left(\mathrm t_\sigma\begin{smallpmatrix} \varpi_p^{f_+} & a \\ & 1 \end{smallpmatrix}\right)\left(s(x)|\begin{smallpmatrix} \varpi_p^{f_+} & a \\ & 1 \end{smallpmatrix} \right).
\end{equation}
Consider $s = Q_{\lambda}(\mathbf N_p^m \theta^{\iota})$. By Lemma \ref{lem:Q-lambda-theta}, the term in the sum on the right-hand side of \eqref{eqn:cup-product-summation} is zero if $a \not\in (\mathcal O_p/\varpi_p^{f_+}\mathcal O_p)^{\times}$, but otherwise we have
\begin{align}
Q_{\lambda}(\mathbf N_p^m \theta^{\iota})(x)|\begin{smallpmatrix} \varpi_p^{f_+} & a \\ & 1 \end{smallpmatrix} &= \theta^{\iota}(ax)I_{\lambda}^{\vee}(\delta_{m,p}(x))|\begin{smallpmatrix} \varpi_p^{f_+} & a \\ & 1 \end{smallpmatrix} & \text{(by Lemmas \ref{lem:Qlambda-p-adic-class} \& \ref{lem:Q-lambda-theta})}\label{eqn:see-above} \\
&= \varpi_p^{-f_+\cdot{w-\kappa\over 2}}\theta^{\iota}(ax) I_{\lambda}^{\vee}\left(\delta_{m,p}(x)|\begin{smallpmatrix} \varpi_p^{f_+} & a \\ & 1 \end{smallpmatrix}\right) & \text{(by \eqref{eqn:I-lambda-action})}\nonumber.
\end{align}
Combining this with \eqref{eqn:cup-product-summation}, and transposing $I_{\lambda}$, we see that 
\begin{multline}\label{eqn:half-way-one-side}
(U_{\varpi_p}^{f_+}\widetilde \Psi)(\mathrm t_\sigma)\left(Q_\lambda(\mathbf N_p^m \theta^{\iota})(x)\right)\\
= \varpi_p^{-f_+\cdot {w - \kappa\over 2}}\sum_{a \in (\mathcal O_p/\varpi_p^{f_+}\mathcal O_p)^\times} \theta^{\iota}(ax) I_{\lambda}(\widetilde \Psi)\left(\mathrm t_\sigma\begin{smallpmatrix} \varpi_p^{f_+} & a \\ & 1 \end{smallpmatrix}\right)\left(\delta_{m,p}(x)|\begin{smallpmatrix} \varpi_p^{f_+} & a \\ & 1 \end{smallpmatrix}\right)\\
= \varpi_p^{-f_+\cdot {w-\kappa \over 2}} \sum_{a \in (\mathcal O_p/\varpi_p^{f_+}\mathcal O_p)^\times} \theta^{\iota}(ax) \left(\begin{smallpmatrix} \varpi_p^{f_+} & a \\ & 1 \end{smallpmatrix}\cdot I_{\lambda}(\widetilde \Psi)\right)(\mathrm t_\sigma)(\delta_{m,p}(x)).
\end{multline}

We want to see that this expression is the same as the right-hand side of \eqref{eqn:interpolation-goal}. For that, let $\widetilde{\psi}= I_{\lambda}(\widetilde{\Psi})$ and then Lemma \ref{lemma:second-half-of-computations} and Lemma \ref{lemma:adelic-twisting-expansion} combine to show that 
\begin{align}\label{eqn:most-of-work}
W \tw_{\theta^{\iota}} \widetilde \psi(\mathrm t_\sigma)  &= \bigl(\prod_{\substack{v\mid p\\f_v=0}}\theta^{\iota}(\varpi_v)^{-1}\bigr)\sum_{b \in S_0^\times} \begin{smallpmatrix} \varpi_p^{f_+} & b \\ & 1\end{smallpmatrix} \cdot \widetilde \psi\left( \mathrm t_\sigma \begin{smallpmatrix} \varpi_p^{f_+} & b \\ & 1 \end{smallpmatrix}\right).\\
&= \theta^{\iota}(x)\sum_{c \in S_1^\times} \sum_{b \in S_0^\times} \theta^{\iota}(c+b) \begin{smallpmatrix} \varpi_p^{f_+} & c+b \\ & 1 \end{smallpmatrix} \widetilde \psi \left(\mathrm t_\sigma\begin{smallpmatrix} \varpi_p^{f_+} & c+b \\ & 1 \end{smallpmatrix} \right) \nonumber\\
&= \theta^{\iota}(x)\sum_{a \in (\mathcal O_p/\varpi_p^{f_+}\mathcal O_p)^\times} \theta^{\iota}(a) \left(\begin{smallpmatrix} \varpi_p^{f_+} & a \\ & 1 \end{smallpmatrix}\cdot \widetilde \psi\right)(\mathrm t_\sigma).\nonumber 
\end{align}
Multiplying \eqref{eqn:most-of-work} by $\varpi_p^{-f_+\cdot{w-\kappa\over 2}}$ and evaluating at $\delta_{m,p}(x)$, we see exactly \eqref{eqn:half-way-one-side}. This completes the proof.
\end{proof}
Our interest is in eigenclasses, so we separate out the following corollary of Theorem \ref{thm:abstract-interpolation}.
\begin{cor}\label{cor:abstract-interp-eigen}
Suppose that $\Psi \in H^d_c(Y_1(\mathfrak m),\mathscr D_\lambda\otimes_{k_\lambda} L)$ is an eigenvector for each operator $U_v$,with eigenvalue $\alpha_v^{\sharp}$. Set $\alpha_v = \varpi_v^{w-\kappa\over 2}\alpha_v^{\sharp}$. Then, for all integers $m$ critical with respect to $\lambda$,
\begin{equation*}
\mathscr P_\lambda(\Psi)(\mathbf N_p^m \theta^{\iota}) = \prod_{\substack{v \mid p\\f_v > 0}} (\alpha_v^{-1} q_v^m)^{f_v}\cdot \prod_{\substack{v \mid p\\f_v=0}} (1 - \theta^{\iota}(\varpi_v)^{-1}\alpha_v^{-1}q_v^m)\cdot \langle \tw_{\theta^{\iota}}^{\cl}(I_\lambda(\Psi)), \cl_p(m)\rangle.
\end{equation*}
\end{cor}
\begin{proof}
To summarize our assumptions:\ we are assuming that $U_{\varpi_p}^{f_+}\Psi = \prod_{v \mid p} (\alpha_v^{\sharp})^{f_{v,+}} \Psi$ and hence $U_v \tw_{\theta^{\iota}}^{\cl} I_{\lambda}(\Psi) = \theta^{\iota}(\varpi_v) \alpha_v \tw_{\theta^{\iota}}^{\cl} I_{\lambda}(\Psi)$ (see Remark \ref{rmk:what-twisting}). Then, by Theorem \ref{thm:abstract-interpolation} we get that
\begin{multline*}
\mathscr P_{\lambda}(\Psi)(\mathbf N_p^m \theta^{\iota}) = \left(\prod_{v\mid p} (\alpha_v^{\sharp})^{-f_{v,+}}\right) \langle U_{\varpi_p}^{f_+}\Psi, \mathscr Q_\lambda(\mathbf N_p^m \theta^{\iota})\rangle \\
= \left(\prod_{v \mid p} \alpha_v^{-f_{v,+}}\right) \bigg\langle \left(\prod_{\substack{v \mid p\\ f_v=0}} \alpha_v -\theta^{\iota}(\varpi_v)^{-1} V_v^+)\prod_{\substack{v \mid p\\ f_v>0}}(V_v^+)^{f_v}\right) \tw_{\theta^{\iota}}^{\cl}I_{\lambda}(\Psi), \cl_p(m)\bigg\rangle.
\end{multline*}
But here $V_v^+$ is the pullback along $\begin{smallpmatrix} \varpi_v \\ & 1 \end{smallpmatrix}$ and so it is adjoint to the pushforward of the same matrix under the pairing $\langle - , - \rangle$. By Lemma \ref{lemma:pushforward-class}, we can  thus replace each instance of $V_v^+$ with $q_v^m$. The result follows.
\end{proof}

%% file: final_section.tex
\subsection{Consequences of smoothness}\label{subsec:consequences}

We begin by proving a lemma in commutative algebra. If $(R,\mathfrak m_R)$ is a Noetherian local ring and $M$ is a module over $R$ then we write $\pd_R(M)$ for its projective dimension over $R$ and $\depth_R(M)$ for its $\mathfrak m_R$-depth. (These terms are defined in either \cite{Matsumura-CommutativeRingTheory} or \cite[Section 16-17]{EGA4-1}, and we refer the reader to the citations in the proof for definitions.)

\begin{lem}\label{lem:comm-alg-conclusions}
Suppose that $(R,\mathfrak m_R)$ and $(T,\mathfrak m_T)$ are Noetherian local rings with $R$ regular and $T$ Cohen--Macaulay and $R \rightarrow T$ is a finite injective local morphism. The following conclusions hold.
\begin{enumerate}
\item $T$ is flat over $R$.
\item If $T$ is regular then $T/\mathfrak m_RT$ is a local complete intersection.
\end{enumerate}
Suppose that $M$ is a finite $T$-module such that $\pd_T(M) < \infty$.
\begin{enumerate}
\setcounter{enumi}{2}
\item $\pd_R(M) = \pd_T(M)$.
\item $M$ is projective over $T$ if and only if $M$ is projective over $R$, in which case the natural map $T/\mathfrak m_RT \rightarrow \End_{R/\mathfrak m_R R}(M/\mathfrak m_RM)$ is injective.
\end{enumerate}
\end{lem}
\begin{proof}
Part (1) follows from \cite[Theorem 23.1]{Matsumura-CommutativeRingTheory}. For (2), since $R$ is regular and $R \rightarrow T$ is flat by (1), the ideal $\mathfrak m_R T$ is generated by a $T$-regular sequence. Thus $T/\mathfrak m_R T$ is a local complete intersection by \cite[Theorem 21.2(iii)]{Matsumura-CommutativeRingTheory}.

Now write $n = \dim R = \dim T$. Since $R$ and $T$ are both Cohen--Macaulay, $n = \depth_R(R) = \depth_T(T)$. 
Since $R$ is regular, $\pd_R(M) < \infty$ by \cite{Serre-RegularDimension}. So, if $\pd_T(M) < \infty$ as well, the Auslander--Buchsbaum formula (\cite[Theorem 3.7]{AuslanderBuchsbaum-Formula}) implies that
\begin{equation}\label{eqn:ab-formula}
\depth_R(M) + \pd_R(M) = n = \depth_T(M) + \pd_T(M).
\end{equation}
Since $R \rightarrow T$ is a local morphism, \cite[Proposition 16.4.8]{EGA4-1} implies that $\depth_R(M) = \depth_T(M)$ and thus \eqref{eqn:ab-formula} reduces to $\pd_R(M) = \pd_T(M)$ as we claimed in (3).

For (4), the first clause immediately follows from (3). For the second clause, if $M$ is projective over $R$ then $M/\mathfrak m_RM$ is finite projective over $T/\mathfrak m_RT$ and so clearly $T/\mathfrak m_RT$ acts faithfully on $M/\mathfrak m_RM$.
\end{proof}

\begin{rmk}\label{rmk:regular-case}
If $T$ is regular in Lemma \ref{lem:comm-alg-conclusions} (which will always be the case below) then the hypothesis on projective dimension before (3) is automatic by \cite{Serre-RegularDimension}. The Auslander--Buchsbaum formula for regular local rings is also proven in \cite[Proposition 17.3.4]{EGA4-1}, in case the reader is interested.
\end{rmk}

We now return to the setting and notation of Section \ref{subsec:middle-degree}. Let $x \in \mathscr E(\mathfrak n)_{\rmmid}(\overline{\mathbf Q}_p)$ be a point of weight $\lambda$ and $h = v_p(\psi_x(U_p))$. Choose an affinoid neighborhood $\Omega \subset \mathscr W(1)$ containing $\lambda$ so that $(\Omega,h)$ is slope adapted. Thus, $x$ defines a maximal ideal $\mathfrak m_x \subset \mathbf T_{\Omega,h}$. Now define $R_x := \mathscr O(\Omega)_{\mathfrak m_{\lambda}}$, $\mathbf T_x := (\mathbf T_{\Omega,h})_{\mathfrak m_x}$, and $M_x :=  \left(H^d_c(\mathfrak n, \mathscr D_{\Omega})_{\leq h}\right)_{\mathfrak m_x} = H^d_c(\mathfrak n, \mathscr D_{\Omega})_{\mathfrak m_x}$. Note that $\mathbf T_x$ is naturally a subring of $\End_{R_x}(M_x)$. Recall that since $x\in \mathscr E(\mathfrak n)_{\rmmid}(\overline{\mathbf Q}_p)$, there is a natural isomorphism $M_x/\mathfrak m_{\lambda}M_x \cong H^d_c(\mathfrak n, \mathscr D_{\lambda})_{\mathfrak m_{x}}$ by Remark \ref{eqn:base-change-middle-degree}. Finally, we also write $T_x$ for the image of $\mathbf T(\mathfrak n)$ in $\End_{k_{\lambda}}(M_x/\mathfrak m_{\lambda}M_x) = \End_{k_{\lambda}}\left(H^d_c(\mathfrak n, \mathscr D_{\lambda})_{\mathfrak m_{x}}\right)$.

\begin{prop}\label{prop:smoothness-general}
If $\mathscr E(\mathfrak n)_{\rmmid}$ is smooth at $x$, then $M_x$ is finite projective over $\mathbf T_x$, and the natural surjection $\mathbf T_x/\mathfrak m_{\lambda}\mathbf T_x \to T_x$ is an isomorphism.
\end{prop}
\begin{proof}
Since $\mathscr E(\mathfrak n)_{\rmmid}$ is equidimensional of dimension equal to the dimension of $\mathscr W(1)$, the map $R_{x} \rightarrow \mathbf T_x$ is a finite injective map of local noetherian rings with $R_x$ regular. Moreover, $M_x$ is finite projective over $R_x$ (Proposition \ref{prop:looks-like-eigen}). So, given that $\mathbf T_x$ is also regular,  Lemma \ref{lem:comm-alg-conclusions}(4) implies that $M_x$ is finite projective over $\mathbf T_x$ and the natural map $\mathbf T_x/\mathfrak m_\lambda \mathbf T_x \to \End_{k_\lambda}(M_x/\mathfrak m_\lambda M_x)$ is injective. Since this map factors over the natural surjection $\mathbf T_x/\mathfrak m_{\lambda}\mathbf T_x \to T_x$, we deduce that the latter is an isomorphism, as desired.
\end{proof}

If $\epsilon \in \{\pm 1\}^{\Sigma_F}$ then write $M_x^{\epsilon} = H^d_c(\mathfrak n, \mathscr D_{\Omega})_{\mathfrak m_x}^{\epsilon}$. In the next proposition we write $\soc_T(M)$ for the socle of $M$ as a $T$-module, i.e.\ the sum of the simple $T$-submodules.
\begin{thm}\label{thm:socle-one-d}
Suppose that $(\pi,\alpha)$ is a $p$-refined cuspidal automorphic representation of cohomological weight $\lambda$ and conductor $\mathfrak n$. If $\alpha$ is a decent refinement, $x = x(\pi,\alpha) \in \mathscr E(\mathfrak n)_{\rmmid}(\overline{\mathbf Q}_p)$, and $\epsilon \in \{\pm 1\}^{\Sigma_F}$, then
\begin{enumerate}
\item $\soc_{T_x}(H^d_c(\mathfrak n,\mathscr D_{\lambda}\otimes_{k_\lambda} k_x)^{\epsilon}_{\mathfrak m_x})$ is one-dimensional over $k_x$.
\end{enumerate}
If, further, condition 2(c) in Definition \ref{defn:decent} holds, then
\begin{enumerate}
\setcounter{enumi}{1}
\item The $\mathbf T_x$-module $M_x^\epsilon$ is free of rank one.
\end{enumerate}
\end{thm}
\begin{proof}
We will actually check the second claim first. Suppose that $x$ is decent and satisfies condition 2(c) of Definition \ref{defn:decent}. Then, $x$ is a smooth point on $\mathscr E(\mathfrak n)_{\rmmid}$ by Theorem \ref{thm:smoothness}. By Proposition \ref{prop:smoothness-general}, $M_x$ is projective over $\mathbf T_x$, and hence so is its direct summand $M_x^{\epsilon}$, and furthermore its rank is equal to the rank of $M_x^{\epsilon}/\mathfrak m_\lambda M_x^{\epsilon}$ over $T_x$. By \eqref{eqn:tor-isomorphisms}, $M_x^{\epsilon}/\mathfrak m_\lambda M_x^{\epsilon} \simeq H^d_c(\mathfrak n, \mathscr D_\lambda)^{\epsilon}_{\mathfrak m_x}$ (as $T_x$-modules). Now, set $M^\epsilon=H^d_c(\mathfrak n, \mathscr D_{\Omega})_{\leq h}^{\epsilon}$, which we regard as a coherent sheaf over $X=\mathrm{Sp}\,\mathbf{T}_{\Omega,h}$. Since $M_{x}^{\epsilon}$ is free over $(\mathbf{T}_{\Omega,h})_{\mathfrak{m}_x}$, $M^\epsilon$ is free over some connected (Zariski-)open neighborhood $U$ of $x$ in $X$. In particular, to calculate the rank of $M_{x}^{\epsilon}$, it suffices to calculate the rank of the fiber of $M^\epsilon$ at any closed point $y\in U$; but by Proposition \ref{prop:density-results} we can assume that $y$ is extremely non-critical classical, in which case the rank is one.  So this completes the proof of (2).

Now we check point (1) is true. If $x$ is non-critical, this is a purely automorphic calculation. Otherwise, since $x$ is decent, point (2) applies to $x$. Thus we are reduced to showing that $\dim_{k_x} \soc_{T_x}(T_x) = 1$. But $T_x \simeq \mathbf T_x / \mathfrak m_\lambda \mathbf T_x$ by Proposition \ref{prop:smoothness-general}, so $T_x$ is a local complete intersection ring by Lemma \ref{lem:comm-alg-conclusions}(2). In particular, $T_x$ is Gorenstein (and of dimension zero) and our result follows from \cite[Theorem 18.1]{Matsumura-CommutativeRingTheory}.
\end{proof}

\subsection{$p$-adic $L$-functions}\label{subsec:padicLfunctions}

Throughout this subsection, we fix a cuspidal automorphic representation $\pi$ of weight $\lambda$ and conductor $\mathfrak n$. We make the following choices:
\begin{enumerate}
\item $\alpha$ is a decent $p$-refinement for $\pi$.
\item For each $\epsilon \in \{\pm 1\}^{\Sigma_F}$ we choose $\Omega_{\pi}^{\epsilon} \in \mathbf C^\times$ as in Theorem \ref{thm:normalize-by-periods}, which by definition also fixes the choice as in Theorem \ref{thm:algebraic-normalization}.
\end{enumerate}
We write $E$ for the subfield of $\mathbf C$ generated by $\mathbf Q(\pi)$, $\mathbf Q(\alpha)$, and the Galois closure of $F$. Let $L = \mathbf Q_p(\iota(E))$. Recall that $\iota$ induces an isomorphism $H^d_c(\mathfrak n, \mathscr L_\lambda(E)) \simeq H^d_c(\mathfrak n, \mathscr L_\lambda(L))$.

Given (1) and (2) we define $\Phi_{\pi,\alpha}^{\epsilon} \in H^d_c(Y_1(\mathfrak n),\mathscr L_{\lambda}(L))^{\epsilon}$ to be
\begin{equation*}
\Phi_{\pi,\alpha}^{\epsilon} = \iota\left({\pr^{\epsilon} \ES(\phi_{\pi,\alpha}) \over \Omega_{\pi}^{\epsilon}}\right),
\end{equation*}
where $\phi_{\pi,\alpha}$ is the $p$-refined eigenform associated to $(\pi,\alpha)$. In the notation of Section \ref{subsec:special-points} we have $\Phi_{\pi,\alpha}^{\epsilon} \in H^d_c(\mathfrak n, \mathscr L_\lambda(L))^{\epsilon}[\mathfrak m_{\pi,\alpha}]$.  On the other hand, since $\alpha$ is a decent $p$-refinement for $\pi$, Theorem \ref{thm:socle-one-d} above implies that $\dim H^d_c(\mathfrak n, \mathscr D_{\lambda} \otimes_{k_\lambda} L)^{\epsilon}[\mathfrak m_{\pi,\alpha}^{\sharp}] = 1$ and there is a natural integration map
\begin{equation}\label{eqn:map-on-eigen}
I_\lambda: H^d_c(\mathfrak n, \mathscr D_{\lambda} \otimes_{k_\lambda} L)^{\epsilon}[\mathfrak m_{\pi,\alpha}^{\sharp}] \rightarrow H^d_c(\mathfrak n, \mathscr L_\lambda(L))^{\epsilon}[\mathfrak m_{\pi,\alpha}].
\end{equation}
 (Recall the $\sharp$-notation from Definition \ref{defn:sharp-notation}.) We note the following lemma.
\begin{lem}\label{lem:dont-be-crazy}
$I_\lambda(H^d_c(\mathfrak n, \mathscr D_{\lambda} \otimes_{k_\lambda} L)^{\epsilon}[\mathfrak m_{\pi,\alpha}^{\sharp}]) \neq (0)$ if and only if $\alpha$ is non-critical.
\end{lem}
\begin{proof}
If $\alpha$ is non-critical then $I_\lambda$ is an isomorphism, so one implication is clear.

Now suppose that $\alpha$ is not non-critical, but recall that $\alpha$ is decent. Thus condition 2(c) of Definition \ref{defn:decent} holds. This implies that $H^d_c(\mathfrak n, \mathscr L_\lambda(L))_{\mathfrak m_{\pi,\alpha}}^{\epsilon} \simeq H^d_c(\mathfrak n, \mathscr L_\lambda(L))^{\epsilon}[\mathfrak m_{\pi,\alpha}]$, and part (2) of Theorem \ref{thm:socle-one-d} implies that $M = H^d_c(\mathfrak n, \mathscr D_\lambda \otimes_{k_\lambda} L)^{\epsilon}_{\mathfrak m_{\pi,\alpha}^{\sharp}}$ is free of rank one over $T$, where $T$ is the largest quotient of $\mathbf T(\mathfrak n)$ acting faithfully on $M$. We note that $T$ is a local complete intersection (by the above discussion).

Since $\alpha$ is not non-critical, the map
\begin{equation}\label{eqn:source-crazy-lemma}
I_\lambda:M \rightarrow H^d_c(\mathfrak n, \mathscr L_\lambda(L))^{\epsilon}[\mathfrak m_{\pi,\alpha}]
\end{equation}
is not an isomorphism. If it is zero we are done. If it is non-zero, then the target is a simple $T$-module and thus \eqref{eqn:source-crazy-lemma} is the surjection of $M$ onto its largest $T$-simple quotient (the co-socle). In particular, the socle $M[\mathfrak m_{\pi,\alpha}^{\sharp}] \subset M$ maps to zero under $I_\lambda$, as claimed.
%
%
\end{proof}
Now recall that we defined a period map
\begin{equation*}
\mathscr P_\lambda: H^d_c(\mathfrak n, \mathscr D_\lambda \otimes_{k_\lambda} L) \rightarrow \mathscr D(\Gamma_F,L)
\end{equation*}
in Definition \ref{defn:period-map} and we may post-compose it with the Amice transform $\mathcal A$ to get elements in $\mathscr O(\mathscr X(\Gamma_F))\otimes_{\mathbf Q_p} L$ (Proposition \ref{prop:amice}).

For the next definition and the results afterward, we assume that $(\pi,\alpha)$ is a decently $p$-refined cohomological cuspidal automorphic representation of weight $\lambda$ and conductor $\mathfrak n$.

\begin{defn}\label{defn:padic-Lfunction}
$L_p^{\epsilon}(\pi,\alpha) = \mathcal A\left(\mathscr P_{\lambda}(\Psi_{\pi,\alpha}^{\epsilon})\right)$ where $\Psi_{\pi,\alpha}^{\epsilon} \in H^d_c(\mathfrak n, \mathscr D_{\lambda} \otimes_{k_\lambda} L)^{\epsilon}[\mathfrak m_{\pi,\alpha}^{\sharp}]$ is any choice of non-zero vector that, if $\alpha$ is non-critical, we assume satisfies $I_\lambda(\Psi_{\pi,\alpha}^{\epsilon}) = \Phi_{\pi,\alpha}^{\epsilon}$.
\end{defn}

Note that, by Proposition \ref{prop:amice-pairing-compat}, if $\chi$ is a continuous character on $\Gamma_F$ then it defines a locally analytic function on $\Gamma_F$ and $L^{\epsilon}_p(\pi,\alpha)(\chi) = \mathscr P_\lambda(\Psi_{\pi,\alpha}^{\epsilon})(\chi) = \langle \Psi_{\pi,\alpha}^{\epsilon},\mathscr Q_{\lambda}(\chi)\rangle$ as in Section \ref{subsec:period-definition}.

With this definition, we can catalog the properties of these $p$-adic $L$-functions.

\begin{prop}[Canonicity]\label{prop:final-canoncial}
$L_p^\epsilon(\pi,\alpha)$ is naturally defined up to an element of $L^\times$ in general, and an element of $\iota(E^\times)$ if $\alpha$ is non-critical.
\end{prop}
\begin{proof}
Obviously there is a choice of $L^\times$-multiple in Definition \ref{defn:padic-Lfunction} in general. But if $\alpha$ is non-critical then the ambiguity is up to the construction of $\Phi_{\pi,\alpha}^{\epsilon}$, which is only up to $\iota(E^\times)$ through the choice of periods $\Omega_{\pi}^{\epsilon}$ as in Theorem \ref{thm:normalize-by-periods} (see Remark \ref{rmk:normalization}).
\end{proof}
Given $\epsilon \in \{\pm 1\}^{\Sigma_F}$ we write $\mathscr X(\Gamma_F)^{\epsilon}$ for the union of components of $\mathscr X(\Gamma_F)$ consisting of characters $\chi$ for which $\chi(\zeta) = \epsilon(\zeta)$ for all $\zeta \in \pi_0(F_\infty^\times)$ (see Remark \ref{rmk:direct-sum-decomps}).
\begin{prop}[Support]\label{prop:final-support}
If $\epsilon \neq \epsilon'$, then $L_p^{\epsilon}(\pi,\alpha)\big|_{\mathscr X(\Gamma_F)^{\epsilon'}} = 0$.
\end{prop}
\begin{proof}
See Proposition \ref{prop:support-calculation}.
\end{proof}

If $h \geq 0$ is a real number and $f \in \mathscr O(\mathscr X(\Gamma_F))\otimes_{\mathbf Q_p} L$ then we say $f$ has order of growth $\leq h$ if $f = \mathcal A(\mu)$ for some (unique) distribution $\mu$ that has order of growth $\leq h$ as in Definition \ref{defn:order-of-growth}.
\begin{prop}[Growth]\label{prop:final-growth}
If $h_v = v_p(\iota(\alpha_v))$ and $h = \sum_{v \mid p} e_v h_v + \sum_{\sigma \in \Sigma_F} {\kappa_\sigma-w\over 2}$, then $L_{p}^{\epsilon}(\pi,\alpha)$ has order of growth $\leq h$.
\end{prop}
\begin{proof}
Proposition \ref{prop:growth-condition} implies $L_p^{\epsilon}(\pi,\alpha)$ has order of growth $\leq h$ where $h = \sum_{v \mid p} e_v v_p(\alpha_v^{\sharp})$. The translation to the claimed statement is clear.
\end{proof}

Before the next proposition, we recall the notation:
\begin{equation*}
\Lambda(\pi\otimes \theta,m+1)^{\alg} := {\sgn(\theta_\infty)  i^{1 + m + {\kappa-w\over 2}} \Delta_{F/\mathbf Q}^{m+1} \Lambda(\pi\otimes \theta,m+1) \over G(\theta)\Omega_\pi^{\epsilon}}.
\end{equation*}
Here $\theta$ is a finite order Hecke character, and $\epsilon$ is chosen so that  $\theta(\zeta)\zeta^m = \epsilon(\zeta)$ for all $\zeta \in \pi_0(F_\infty^\times)$. We have $\Lambda(\pi\otimes \theta,m+1)^{\alg} \in E(\theta)$ (it is only off by the absolute norm of the conductor of $\theta$ from the value in Theorem \ref{thm:algebraic-normalization}). We also recall that if $\mathfrak p_v \nmid \mathfrak n$ then $\alpha_v$ is a root of a quadratic polynomial (Definition \ref{defn:refinement-intext}) and we write $\beta_v = a_v(\pi) - \alpha_v$ for the other root. To save notation, in what follows, we stress that $\alpha_v$ and $\beta_v$ are viewed as $p$-adic numbers under the isomorphism $\iota: \mathbf C \simeq \overline{\mathbf Q}_p$.

\begin{prop}[Interpolation]\label{prop:final-interpolation}
Suppose that $m$ is an integer that is critical with respect to $\lambda$, $\theta$ is a finite order Hecke character of conductor $\prod_{v \mid p} \mathfrak p_{v}^{f_v}$ and $\epsilon(\zeta) = \theta(\zeta)\zeta^m$ for all $\zeta \in \pi_0(F_\infty^\times)$. Then,
\begin{enumerate}
\item If $\alpha$ is critical, then $L_p^{\epsilon}(\pi,\alpha)(\mathbf N_p^m \theta^{\iota}) = 0$.
\item If $\alpha$ is non-critical, then
\begin{multline*}
L_p^{\epsilon}(\pi,\alpha)(\mathbf N_p^m \theta^{\iota})\\
 =  \prod_{f_v > 0} \left(q_v^{m+1} \over \alpha_v\right)^{f_v} \prod_{f_v=0} (1 - \theta^{\iota}(\varpi_v)\alpha_v^{-1} q_v^m) \prod_{\substack{v \mid p\\ \mathfrak p_v \nmid \mathfrak n\\ f_v=0}} (1-\beta_v\theta^{\iota}(\varpi_v)q_v^{-(m+1)}) \cdot \iota\left(\Lambda(\pi\otimes \theta,m+1)^{\alg}\right).
\end{multline*}
\end{enumerate}
\end{prop}
\begin{proof}
Choose $\Psi_{\pi,\alpha}^{\epsilon}$ as in Definition \ref{defn:padic-Lfunction}. Then, by Proposition \ref{prop:amice-pairing-compat} we want to compute $\mathscr P_\lambda(\Psi_{\pi,\alpha}^{\epsilon})(\mathbf N_p^m \theta^{\iota})$ with the notations as in Section \ref{subsec:period-definition}. For each $v \mid p$, $\Psi_{\pi,\alpha}^{\epsilon}$ is a $U_v$-eigenform with eigenvalue $\alpha_v^{\sharp} = \varpi_v^{\kappa-w\over 2} \alpha_v$. Thus Corollary \ref{cor:abstract-interp-eigen} implies that 
\begin{equation}\label{eqn:first-step}
\mathscr P_\lambda(\Psi_{\pi,\alpha}^{\epsilon})(\mathbf N_p^m \theta^{\iota}) = \prod_{f_v > 0} \left(q_v^m \over \alpha_v\right)^{f_v} \prod_{f_v=0} (1 - \theta^{\iota}(\varpi_v)\alpha_v^{-1} q_v^m) \cdot \langle \tw_{\theta^{\iota}}^{\cl}(I_\lambda(\Psi_{\pi,\alpha}^{\epsilon})), \cl_p(m)\rangle.
\end{equation}
If $\alpha$ is critical, then the right-hand side vanishes by Lemma \ref{lem:dont-be-crazy}. This proves (1). If $\alpha$ is non-critical though, we have $I_\lambda(\Psi_{\pi,\alpha}^{\epsilon}) = \Phi_{\pi,\alpha}^{\epsilon}$, by definition. Thus
\begin{align*}
\iota^{-1}\left(\langle \tw_{\theta^{\iota}}^{\cl}(I_\lambda(\Psi_{\pi,\alpha}^{\epsilon})), \cl_p(m)\rangle\right) &= \iota^{-1}\left(\langle \tw_{\theta^{\iota}}^{\cl}(\Phi_{\pi,\alpha}^{\epsilon}), \cl_p(m) \rangle\right) \\
&= {1\over \Omega_\pi^{\epsilon}}\langle \tw_{\theta} {\pr^{\epsilon} \ES(\phi_{\pi,\alpha})}, \cl_\infty(m) \rangle & \text{(by Proposition \ref{prop:p-adic-eval-class-defn})}\\
&= {1\over \Omega_\pi^{\epsilon}}\langle \tw_{\theta}(\ES(\phi_{\pi,\alpha})), \cl_\infty(m)\rangle & \text{(by Lemma \ref{lem:allowed-to-project})}\\
&= {G(\theta^{-1})\over \Omega_\pi^{\epsilon}}\langle \ES(\phi_{\pi,\alpha} \otimes \theta), \cl_\infty(m)\rangle.
\end{align*}
Combining this calculation with \eqref{eqn:first-step}, we are finished by Corollary \ref{cor:mellin-twist-pair}. (The Gauss sum can be moved to the denominator using \eqref{eqn:gauss-sum-relation}; this is where the $m$'s in the $q_v$ exponents of \eqref{eqn:first-step} becomes $m+1$'s.)
\end{proof}

Finally, we have a many-variable version of the above constructions. It follows easily from the functorial nature of our construction of the period maps. The proof is directly inspired from \cite[Remark 4.16]{Bellaiche-CriticalpadicLfunctions}.

\begin{prop}[Variation]
Let $x = x_{\pi,\alpha}$ be a smooth classical point on $\mathscr E(\mathfrak n)_{\rmmid}$. Then, for each sufficiently small good open neighborhood $U$ of $x$ in $\mathscr E(\mathfrak n)_{\rmmid}$ there exists an element $\mathbf L^{\epsilon} \in \mathscr O(U)\widehat{\otimes}_{\mathbf Q_p} \mathscr O(\mathscr X(\Gamma_F))$ specified up to $\mathscr O(U)^\times$-multiple and such that for each decent point $x' \in U$ associated with a $p$-refined cohomological cuspidal automorphic representation $(\pi',\alpha')$ we have
\begin{equation*}
\mathbf L^{\epsilon}_p|_{u=x'} = c_{x'}L^{\epsilon}_p(\pi,\alpha)
\end{equation*}
for some constant $c_{x'} \in k_{x'}^\times$.
\end{prop}
\begin{proof}
Given $x$, every sufficiently small good open neighborhood $U$ of $x$ is regular (Theorem \ref{thm:smoothness}). Fix such a neighborhood, and assume that it is belongs to a slope adapted pair $(\Omega,h)$. By Proposition \ref{prop:looks-like-eigen} we may assume that $\mathscr O(U)$ acts faithfully on the finite projective $\mathscr O(\Omega)$-module $\mathscr M_c^d(U) = e_U H^d_c(\mathfrak n, \mathscr D_{\Omega})$. By Lemma \ref{lem:comm-alg-conclusions}, $\mathscr M^d_c(U)$ is finite projective over $\mathscr O(U)$. Furthermore, for each $\epsilon$, $M=\mathscr M_c^d(U)^{\epsilon}$ is free of rank one over $\mathscr O(U)$ by the same argument as in Theorem \ref{thm:socle-one-d}.

On the other hand, in Section \ref{subsec:period-definition} we constructed a canonical period map
\begin{equation*}
\mathscr P_{\Omega} : H^d_c(\mathfrak n, \mathscr D_{\Omega}) \rightarrow \mathscr D(\Gamma_F,\mathscr O(\Omega)).
\end{equation*}
We can then specialize this to 
\begin{equation*}
\mathscr P_{\Omega}|_{M} \in M^\vee \widehat{\otimes}_{\mathscr O(\Omega)} \mathscr D(\Gamma_F,\mathscr O(\Omega)) \simeq M^\vee \widehat{\otimes}_{\mathbf Q_p} \mathscr D(\Gamma_F,\mathbf Q_p)
\end{equation*}
where $M^\vee = \Hom_{\mathscr O(\Omega)}(M,\mathscr O(\Omega))$ is the dual $\mathscr O(U)$-module.

We now combine the previous two paragraphs. Since $U$ is smooth, $\mathscr O(U)$ is regular. In particular, it is Gorenstein. Since $M \simeq \mathscr O(U)$ as an $\mathscr O(U)$-module we deduce that $M^\vee$ is also free of rank one over $\mathscr O(U)$. Choose an $\mathscr O(U)$-linear isomorphism $M^\vee \simeq \mathscr O(U)$ and then we get
\begin{equation*}
\mathscr P_{\Omega}|_M \in M^\vee \widehat{\otimes}_{\mathbf Q_p} \mathscr D(\Gamma_F,\mathbf Q_p) \simeq \mathscr O(U)\widehat{\otimes}_{\mathbf Q_p} \mathscr D(\Gamma_F,\mathbf Q_p).
\end{equation*} 
We finally define $\mathbf L_p^{\epsilon} := \mathcal A(\mathscr P_{\Omega}|_M)$ where $\mathcal A$ is the Amice transform, as usual.

From the construction, $\mathbf L_p^{\epsilon}$ was uniquely defined up to $\mathscr O(U)^\times$-multiple and it is an exercise to see that it specializes the construction(s) given above.
\end{proof}

%% file: semistable-appendix.tex
\numberwithin{equation}{subsection}

The goal of this appendix is to prove Proposition \ref{prop:local-bound} in the cases not already available from \cite[Section 3]{Bergdall-Smoothness}. Namely, we calcuate the dimension of the Zariski tangent space to a certain deformation functor of certain two-dimensional semi-stable, but non-crystalline, Galois representations. The calculations are made using the associated $(\varphi,\Gamma)$-module. We recommend the reader who is not familiar with $(\varphi,\Gamma)$-modules over the Robba ring and their relationship to Galois representations also consult \cite{Berger-Representationp-adique}, along with \cite{KedlayaPottharstXiao-Finiteness} for further details (\cite[Chapter 2]{BellaicheChenevier-Book} also contains many details, in the case of $\Gal(\overline{\mathbf Q}_p/\mathbf Q_p)$-representations).

The appendix is divided as follows. In Section \ref{subsec:app-general} we recall $(\varphi,\Gamma)$-modules and their connection to Galois representations, and we describe general results. In Section \ref{subsec:app-2d-specific} we fix a rank two $(\varphi,\Gamma)$-module that is semi-stable, but not crystalline, and we establish results specific to that situation. Finally, In Section \ref{subsec:app-refined} we recall the notion of weakly-refined (infinitesimal) deformations and calculate the dimension of the corresponding Zariski tangent space.

We fix the following notations throughout Appendix \ref{app:semistable}. Let $K$ be a finite extension of $\mathbf Q_p$ with ring of integers $\mathcal O_K$ and fix a uniformizer $\varpi_K \in \mathcal O_K$. Let $K_0$ be maximal unramified subextension of $K$ and let $f = (K_0:\mathbf Q_p)$ be the residue degree of $K/\mathbf Q_p$. Let $L$ be a finite extension of $\mathbf Q_p$ such that $\# \Sigma_K = (K:\mathbf Q_p)$, where  $\Sigma_K := \Hom_{\mathbf Q_p}(K,L)$.

\subsection{($\varphi,\Gamma)$-modules and Galois representations}\label{subsec:app-general}

Let $K_\infty = \dirlim K(\zeta_{p^n})$ be the field extension of $K$ obtained by adjoining all $p$-th power roots of unity and $\Gamma_K = \Gal(K_\infty/K)$. The Robba ring $\mathcal R_K$ is the ring of Laurent series $f(T)$ in a single variable whose coefficients lie in the maximal absolutely unramified extension of $K_\infty$ and which converge on an annulus $r < |T| < 1$, with $r$ depending on $f(T)$. The ring $\mathcal R_K$ is equipped with a continuous operator $\varphi$ (the Frobenius) and a continuous action of the group $\Gamma_K$ that commutes with $\varphi$. We write $\mathcal R_{K,L} = \mathcal R_K \otimes_{\mathbf Q_p} L$ for the Robba ring extended linearly to $L$.

A $(\varphi,\Gamma_K)$-module (over $\mathcal R_{K,L}$) is a finite free $\mathcal R_{K,L}$-module $D$ equipped with a continuous operator $\varphi$ and a continuous action of the group $\Gamma_K$ that commutes with $\varphi$, and these actions are assumed to be semi-linear with respect to the $(\varphi,\Gamma_K)$-action on $\mathcal R_{K,L}$. There is a fully faithful functor
\begin{align*}
\{\text{continuous $L$-linear representations of $G_{K}$}\} &\longrightarrow \{\text{$(\varphi,\Gamma_K)$-modules over $\mathcal R_{K,L}$}\}\\
r &\longmapsto D_{\rig}(r)
\end{align*}
whose essential image is the full subcategory of \'etale $(\varphi,\Gamma_K)$-modules. Even if $r$ itself is irreducible, the $(\varphi,\Gamma_K)$-module $D_{\rig}(r)$ may become reducible.

Galois representation-theoretic notions and constructions extend to the category of $(\varphi,\Gamma_K)$-modules over $\mathcal R_{K,L}$ through the functor $D_{\rig}$. For instance, each $(\varphi,\Gamma_K)$-module has a dual $D^{\vee}$, we may speak of Hodge--Tate weights, and there are functors $D_{\crys}$ and $D_{\mathrm{st}}$ that are used to define categories of (not necessarily \'etale) crystalline and semi-stable $(\varphi,\Gamma_K)$-modules. For instance, $D_{\crys}(D) = D[1/t]^{\Gamma_K}$ where $t \in \mathcal R_{K,L}$ is Fontaine's ``$p$-adic $2\pi i$'' (and $D_{\crys}^+(D) = D^{\Gamma_K}$). In the ring $\mathcal R_{K,L}$, the element $t$ factors into a product $t = \prod_{\sigma \in \Sigma_K} t_\sigma$ and $t_{\sigma}\mathcal R_{K,L}$ is crystalline $(\varphi,\Gamma_K)$-module with Hodge--Tate weights are $\HT_{\sigma'}(t_{\sigma}\mathcal R_{K,L})) = -1$ if $\sigma' = \sigma$ and $0$ otherwise. The Frobenius acts trivially $D_{\crys}(t_\sigma \mathcal R_{K,L})$.

Rank one $(\varphi,\Gamma_K)$-modules are parametrized by continuous characters $\delta: K^\times \rightarrow L^\times$. We write $\mathcal R_{K,L}(\delta)$ for the $(\varphi,\Gamma_K)$-module corresponding to $\delta$. It is \'etale exactly when $\delta(\varpi_K)$ is a $p$-adic unit, which is to say that $\delta$ extends via the local Artin map to a continuous $L^\times$-valued character on $G_K^{\ab}$. 

Another class of characters is
$$
z^{s} := \left(z \mapsto \prod_{\sigma \in \Sigma_K} \tau(z)^{s_\sigma} \right),
$$
where $s = (s_{\sigma})_{\sigma \in \Sigma_K}$ is a collection of integers. The cyclotomic character $\chi_{\cycl}$ corresponds to the character on $K^\times$ given by 
$$
\chi_{\cycl}(z) = |N_{K/\mathbf Q_p}(z)|_pN_{K/\mathbf Q_p}(z) = |z^{1}|_p z^{1},
$$
where $1 = (1_\sigma)_{\sigma \in \Sigma_K}$ is the constant tuple. The $(\varphi,\Gamma_K)$-module $t_\sigma \mathcal R_{K,L}$ corresponds to the character $z^s$ where $s_\sigma = 1$ and $s_{\sigma'} = 0$ if $\sigma'\neq \sigma$.  As a final example, given a character $\delta_0 : \mathcal O_K^\times \rightarrow L^\times$ we write $\LT_{\varpi_K}(\delta_0)$ for the character $\delta_0$ extended to $K^\times$ by taking the value 1 on $\varpi_K$ (compare with the main text prior to Proposition \ref{prop:family-galois-properties}).

Given by $(\varphi,\Gamma_K)$-module and any character $\delta$ we write $D(\delta) = D \otimes_{\mathcal R_{K,L}} \mathcal R_{K,L}(\delta)$. We note the following lemma for later use. Note that $D$ is specifically not assumed to be crystalline in the lemma.
\begin{lem}\label{lem:annoying-issue}
If $D$ is a $(\varphi,\Gamma_{K})$-module and $\delta:K^\times \rightarrow L^\times$ is a crystalline character, the natural map $D_{\crys}(D)\otimes D_{\crys}(\mathcal R_{K,L}(\delta)) \rightarrow D_{\crys}(D(\delta))$ is an isomorphism.
\end{lem}
\begin{proof}
If $E$ is any $(\varphi,\Gamma_K)$-module, the natural map $D_{\crys}(D) \otimes D_{\crys}(E) \rightarrow D_{\crys}(D\otimes E)$ is injective. Thus, for any character $\delta$, the composition
$$
D_{\crys}(D)\otimes D_{\crys}(\mathcal R_{K,L}(\delta))\otimes D_{\crys}(\mathcal R_{K,L}(\delta^{-1})) \rightarrow D_{\crys}(D(\delta)) \otimes D_{\crys}(\mathcal R_{K,L}(\delta^{-1})) \rightarrow D_{\crys}(D)
$$
is the injective at each step. Since $\delta$ is crystalline, the domain of the composition is canonically isomorphic to $D_{\crys}(D)$ and the resulting composition $D_{\crys}(D) \rightarrow D_{\crys}(D)$ is the identity map. It follows that $D_{\crys}(D)\otimes D_{\crys}(\mathcal R_{K,L}(\delta))$ and $D_{\crys}(D(\delta))$ have the same $L$-dimension and thus the natural (injective) map $D_{\crys}(D)\otimes D_{\crys}(\mathcal R_{K,L}(\delta)) \rightarrow D_{\crys}(D(\delta))$ is an isomorphism.
\end{proof}

Every $(\varphi,\Gamma_K)$-module has Galois cohomology, written $H^{\bullet}(D)$, and Selmer groups, written $H^1_f(D)$ and $H^1_g(D)$ (see \cite[Section 1.4.1]{Benois-GreenbergL} and \cite[Section 3A]{Pottharst-FiniteSlope}). There is a canonical identification $H^0(D) = \Hom(\mathcal R_{K,L},D)$. The cohomology of rank one $(\varphi,\Gamma_K)$-modules is known (see \cite[Proposition 6.4.8]{KedlayaPottharstXiao-Finiteness}). For instance, $H^0(\mathcal R_{K,L}(\delta)) = (0)$ unless $\delta = z^{s}$ where $s = (s_\sigma)_{\sigma \in \Sigma_K}$ and $s_\sigma \leq 0$ for all $\sigma \in \Sigma_K$.

\begin{lem}\label{lem:app-rank-one}
If $\delta$ and $\delta'$ are any two characters, then any non-zero morphism of $(\varphi,\Gamma_K)$-modules $\mathcal R_{K,L}(\delta) \rightarrow \mathcal R_{K,L}(\delta')$ induces an isomorphism $D_{\crys}(\mathcal R_{K,L}(\delta)) \cong D_{\crys}(\mathcal R_{K,L}(\delta'))$.
\end{lem}
\begin{proof}
Let $f: \mathcal R_{K,L}(\delta) \rightarrow \mathcal R_{K,L}(\delta')$ be a non-zero morphism of $(\varphi,\Gamma_K)$-module. By the previous two paragraphs, $f$ induces a $(\varphi,\Gamma_K)$-equivariant isomorphism $\mathcal R_{K,L}(\delta)[1/t] \cong \mathcal R_{K,L}(\delta')[1/t]$. Thus $f$ also induces an isomorphism at the level of $D_{\crys}(-)$.
\end{proof}

\begin{exam}\label{exam:phi=1}
Suppose that $D$ a $(\varphi,\Gamma_K)$-module such that $D_{\crys}(D)^{\varphi = 1} \neq (0)$. Choose a non-zero $v \in D_{\crys}(D)^{\varphi = 1}$. By definition, we get an inclusion 
\begin{equation}\label{eqn:trivial-inclsion}
\mathcal R_{K,L} \cong \mathcal R_{K,L}\cdot v \hookrightarrow D[1/t]
\end{equation}
that is equivariant for the action of $\varphi$ and $\Gamma_K$. After clearing powers of the $t_\sigma$, the image of \eqref{eqn:trivial-inclsion} generates a rank one $(\varphi,\Gamma_K)$-submodule of $D$. Specifically, there exists (unique) integers $s = (s_\sigma)_{\sigma \in \Sigma_K}$ such that \eqref{eqn:trivial-inclsion} induces an inclusion $\mathcal R_{K,L}(z^s) \hookrightarrow D$ of $(\varphi,\Gamma_K)$-modules whose cokernel is also a $(\varphi,\Gamma_K)$-module (i.e.\ free over $\mathcal R_{K,L}$).
\end{exam}

Cohomology classes in $H^1(D)$ may be interpreted as extensions $0 \rightarrow D \rightarrow E \rightarrow \mathcal R_{K,L} \rightarrow 0$ in the category of $(\varphi,\Gamma_K)$-modules. For the Selmer groups, $H^1_f(D)$ parametrizes extensions $E$ such that
$$
0 \rightarrow D_{\crys}(D) \rightarrow D_{\crys}(E) \rightarrow D_{\crys}(\mathcal R_{K,L}) \rightarrow 0
$$
is exact, rather than only left exact. If $D$ is crystalline, $H^1_f(D)$ parametrizes crystalline extensions. The analogous definition and statement hold for $H^1_g$ after replacing ``crystalline'' with ``de Rham'' and $D_{\crys}$ with $D_{\dR}$. Berger proved that every de Rham $(\varphi,\Gamma)$-module is potentially semi-stable and that a potentially semi-stable extension of semi-stable $(\varphi,\Gamma_K)$-modules is also semi-stable (\cite{Berger-Representationp-adique}, but compare with \cite[Theorem 3.1, Corollary 3.3]{Pottharst-FiniteSlope}). Thus, if $D$ is semi-stable, then  $H^1_g(D)$ parametrizes semi-stable extensions $E$. In particular, $H^1_g(D) = H^1_{\mathrm{st}}(D)$ in the notations of \cite[Section 1.4.1]{Benois-GreenbergL}. We recall the following proposition, limited to the setting in which we will need to reference it.

\begin{prop}[{\cite[Proposition 1.4.4, Corollary 1.4.5, Corollary 1.4.6]{Benois-GreenbergL}}]\label{prop:app-benois-dimensions}
\leavevmode
\begin{enumerate}
\item Suppose that $D$ is a semi-stable $(\varphi,\Gamma_K)$-module. Then,
$$
\dim_L H^1_f(D) = \dim_L H^0(D) + h^-(D),
$$ 
where $h^-(D)$ is the number (with multiplicity) of negative Hodge--Tate weights of $D$, and
$$
\dim_L H^1_g(D) = \dim_L H^1_f(D) + \dim_L D_{\crys}(D^\vee(\chi_{\cycl})^{\varphi = 1}.
$$
\item Suppose that $0 \rightarrow E_1 \rightarrow E \rightarrow E_2 \rightarrow 0$ is an exact sequence of semi-stable $(\varphi,\Gamma_K)$-modules and the connecting map $H^0(E_2) \rightarrow H^1(E_1)$ is zero. Then, the natural sequence
$$
0 \rightarrow H^1_f(E_1) \rightarrow H^1_f(E) \rightarrow H^1_f(E_2) \rightarrow 0
$$
is exact.
\end{enumerate}
\end{prop}

\subsection{Triangulated, semi-stable, but non-crystalline, $(\varphi,\Gamma_K)$-modules of rank two}\label{subsec:app-2d-specific} 

We now begin to limit the discussion to certain rank two $(\varphi,\Gamma_K)$-modules. Suppose that $D$ is a $(\varphi,\Gamma_K)$-module of rank two. A triangulation of $D$ is an extension
\begin{equation}\label{eqn:appendix-triangulated}
0 \rightarrow \mathcal R_{K,L}(\delta_1) \rightarrow D \rightarrow \mathcal R_{K,L}(\delta_2) \rightarrow 0
\end{equation}
where the $\delta_i$ are characters. For the rest of this appendix we fix a $(\varphi,\Gamma_K)$-module $D$ and a triangulation \eqref{eqn:appendix-triangulated}. We also assume:
\begin{equation}\label{eqn:st}
\tag{st}
\text{$D$ is semi-stable but non-crystalline.}
\end{equation}
and
\begin{equation}\label{eqn:HT-reg}
\tag{HT-reg}
\text{for $\sigma \in \Sigma_K$, the $\sigma$-Hodge--Tate weights of $D$ are distinct}.
\end{equation}
Following \eqref{eqn:HT-reg} we write $h_{1,\sigma} < h_{2,\sigma}$ for the $\sigma$-Hodge--Tate weights of $D$ in the direction $\sigma \in \Sigma_K$. 

\begin{exam}
Consider the global setting of Section \ref{subsec:smoothness}, where a classical point $x = (\pi,\alpha)$ on $\mathscr E(\mathfrak n)_{\rmmid}$ has been selected. For a $p$-adic place $v$ of the number field $F$, let $K = F_v$ and $D = D_{\rig}(\rho_{x,v})$.  The Hodge--Tate weights $h_{i,\tau}$ are equal to $h_{1,\tau} = \frac{w -\kappa_\tau}{2}$ and $h_{2,\tau} = \frac{w+\kappa_\tau}{2} + 1$, as in Theorem \ref{thm:classical-galois-representations}. Thus $D$ satisfies \eqref{eqn:HT-reg}.  The choice of refinement $\alpha_v$ at $v \mid p$ defines a character $\delta_1'$ by
\begin{align*}
\delta_1' &= z^{-h_1}\cdot \nr(\alpha_v),
\end{align*}
where $\nr(\alpha_v)$ is trivial on $\mathcal O_K^\times$ and maps $\varpi_K$ to $\alpha_v$. Unwinding definitions (see Proposition \ref{prop:family-galois-properties}, for instance) we have $D_{\crys}(D\otimes\mathcal R_{K,L}(\delta_1'^{-1}))^{\varphi = 1}$ is non-zero. Thus, by Example \ref{exam:phi=1}, $D$ has a triangulation \eqref{eqn:appendix-triangulated} with characters $(\delta_1,\delta_2)$ where $\delta_1 = z^s \delta_1'$ (and $\delta_2 = \det(D)\cdot \delta_1^{-1}$). In a triangulation \ref{eqn:appendix-triangulated}, the Hodge--Tate weights of the $\delta_i$ together give the Hodge--Tate weights of $D$. Thus, the integers $s = (s_\sigma)_{\sigma \in \Sigma_K}$ are necessarily either $s_\sigma = 0$ or $s_\sigma = h_{1,\sigma} - h_{2,\sigma}$. The condition \eqref{eqn:st} arises only if $x$ is associated with an unramified special representations (see part (2) of Theorem \ref{thm:classical-galois-representations}).
\end{exam}
 
Since $D$ is semi-stable, each $\delta_i$ in \eqref{eqn:appendix-triangulated} is crystalline, and $\varphi^f$ acts on $D_{\crys}(\mathcal R_{K,L}(\delta_1\LT_{\varpi_K}(z^{h_1}))$ by $\Phi := \delta_1(\varpi_K)\prod_{\sigma \in \Sigma_K} \sigma(\varpi_K)^{\HT_{\sigma}(\delta_1) - h_{1,\sigma}}$. Since $D_{\crys}(-)$ is left exact, we have
\begin{equation}\label{eqn:non-zero}
D_{\crys}(D\left( \LT_{\varpi_K}(z^{h_1})\right))^{\varphi^f = \Phi} \neq 0.
\end{equation}
The rest of this section collects, in three lemmas, some basic facts about $D$ that follow from the hypothesis \eqref{eqn:st}. In order, the detail properties of $D_{\crys}$, the adjoint $(\varphi,\Gamma_K)$-module, and, results in Galois cohomology.

\begin{lem}\label{lem:appendix-iso-lemma}
\leavevmode
\begin{enumerate}
\item The natural injective map $D_{\crys}(\mathcal R_{K,L}(\delta_1)) \rightarrow D_{\crys}(D)$ is an isomorphism. 
\item There exists a collection of integers $s = (s_\sigma)_{\sigma \in \Sigma_K}$ such that $\delta_2 \delta_1^{-1} = \chi_{\cycl}^{-1} z^s$. In particular, if $\phi_i$ is the eigenvalue of $\varphi^f$ acting on $D_{\crys}(\mathcal R_{K,L}(\delta_i))$, then $\phi_2\phi_1^{-1} = p^f \neq 1$.
\end{enumerate}
\end{lem}
\begin{proof}
To prove (1), note that $D_{\crys}(D) = D_{\sstable}(D)^{N=0}$ is a $K_0\otimes_{\mathbf Q_p} L$-direct summand of $D_{\sstable}(D)$. Since $D_{\sstable}(D)$ is free of rank two over $K_0 \otimes_{\mathbf Q_p} L$, but $D$ is not crystalline, $D_{\crys}(D)$ is free of rank one over $K_0\otimes_{\mathbf Q_p} L$. In particular, the domain and range of the injective map $D_{\crys}(\mathcal R_{K,L}(\delta_1)) \rightarrow D_{\crys}(D)$ both have the same $L$-dimension, forcing the map to be an isomorphism.

To prove (2), note that \eqref{eqn:appendix-triangulated} and \eqref{eqn:st} implies $D$ defines a non-zero class in $H^1_g(\mathcal R_{K,L}(\delta_1\delta_2^{-1}))$ that does not lie in $H^1_f(\mathcal R_{K,L}(\delta_1\delta_2^{-1}))$. By Proposition \ref{prop:app-benois-dimensions}, applied to $\mathcal R_{K,L}(\delta_1\delta_2^{-1})$, we have that $\varphi$ acts trivially on $D_{\crys}(\mathcal R_{K,L}(\delta_2\delta_1^{-1}\chi_{\cycl}))$. The result then follows from Example \ref{exam:phi=1}.
\end{proof}

Now we write $\ad(D) = D \otimes D^\vee \cong \Hom(D,D)$ for the adjoint $(\varphi,\Gamma_K)$-module associated with $D$.

\begin{lem}\label{lem:appendix-adjoint-lemma}
\leavevmode
\begin{enumerate}
\item The natural surjective map $\ad(D) \rightarrow \mathcal R_{K,L}(\delta_2\delta_1^{-1})$ is not split.
\item We have $H^1_g(\ad(D)) = H^1_f(\ad(D))$.
\end{enumerate}
\end{lem}
\begin{proof}
Label the maps in \eqref{eqn:appendix-triangulated} as
\begin{equation*}
0 \rightarrow \mathcal R_{K,L}(\delta_1) \overset{\iota}{\rightarrow} D \overset{\pi}{\rightarrow}  \mathcal R_{K,L}(\delta_2) \rightarrow 0.
\end{equation*}
The natural surjective map in (1) is $\pi \otimes \iota^{\ast}$, which factors as
\begin{equation*}
\xymatrix{
\ad(D) \ar@{-->}[dr]_-{\pi\otimes \iota^{\ast}} \ar[r]^-{1 \otimes \iota^{\ast}} & D\otimes \mathcal R_{K,L}(\delta_1^{-1}) \ar[d]^-{\pi \otimes 1}\\
 & \mathcal R_{K,L}(\delta_2\delta_1^{-1}).
 }
\end{equation*}
If $\pi\otimes \iota^{\ast}$ were to split, then $\pi \otimes 1$ would also. Twisting by $\delta_1$, $\pi$ would have to split. But then $D$ would be a a sum of crystalline characters and thus crystalline itself, contradicting \eqref{eqn:st}. This proves (1).

We now prove (2). Note that $\ad(D)$ is semi-stable and $\ad(D)^\vee \cong \ad(D)$. Thus, Proposition \ref{prop:app-benois-dimensions}(1), it suffices to show that $D_{\crys}(\ad(D)\left(\chi_{\cycl}\right))^{\varphi = 1} = (0)$. Suppose not, and we will contradict (1). On the one hand, Example \ref{exam:phi=1} implies there is an inclusion
\begin{equation}\label{eqn:inclusion-sprime}
\mathcal R_{K,L}(\chi_{\cycl}^{-1} z^{r}) \hookrightarrow \ad(D),
\end{equation}
for some collection of integers $r = (r_\sigma)_{\sigma \in \Sigma_K}$. On the other, $\ad(D)$ sits in an exact sequence
$$
0 \rightarrow P \rightarrow \ad(D) \rightarrow \mathcal R_{K,L}(\delta_2\delta_1^{-1}) \rightarrow 0
$$
where $P$ itself has a composition series of $(\varphi,\Gamma_K)$-modules with rank one successive quotients $\mathcal R_{K,L}(\delta_1\delta_2^{-1})$ and $\mathcal R_{K,L}$ (with multiplicity two). Note that $\Hom(\mathcal R_{K,L}(\chi_{\cycl}^{-1} z^{r}), \mathcal R_{K,L}) = (0)$ regardless of $r$. Note also, by Lemma \ref{lem:appendix-iso-lemma}(2), we have $\delta_1\delta_2^{-1} = \chi_{\cycl}z^{-s}$, for some integers $s = (s_\sigma)_{\sigma \in \Sigma_K}$, and so
$$
\Hom(\mathcal R_{K,L}(\chi_{\cycl}^{-1} z^{r}), \mathcal R_{K,L}(\delta_1\delta_2^{-1})) = \Hom(\mathcal R_{K,L}(\chi_{\cycl}^{-2} z^{r+s}), \mathcal R_{K,L}) = (0),
$$
regradless of $s$ and $r$. So, our notes show \eqref{eqn:inclusion-sprime} factors through $\ad(D)/P$ and defines an isomorphism $\mathcal R_{K,L}(\chi_{\cycl}^{-1} z^{r}) \cong \mathcal R_{K,L}(\delta_2\delta_1^{-1})$ (since the cokernel of \eqref{eqn:inclusion-sprime} must be torsion free). But then \eqref{eqn:inclusion-sprime} provides a splitting of $\ad(D) \twoheadrightarrow \mathcal R_{K,L}(\delta_2\delta_1^{-1})$, which contradicts part (1).
\end{proof}
Finally we have a lemma on Galois cohomology.

\begin{lem}\label{lem:appendix-cohomology}
\leavevmode
\begin{enumerate}
\item We have $H^2(D\left(\delta_2^{-1}\right)) = (0)$.
\item The natural map $H^0(\ad(D)) \rightarrow H^0(D\left((\delta_1^{-1}\right))$ is an isomorphism.
\item The natural diagram
\begin{equation}\label{eqn:appendix-diagram}
\xymatrix{
0 \ar[r] & H^1_{f}(D\left(\delta_2^{-1}\right)) \ar[r] \ar[d]& \ar[r]\ar[d] H^1_{f}(\ad D) \ar[r] & H^1_{f}(D\left(\delta_1^{-1}\right)) \ar[r] \ar[d]& 0\\
0 \ar[r] & H^1(D\left(\delta_2^{-1}\right)) \ar[r] & \ar[r] H^1(\ad D) \ar[r] & H^1(D\left(\delta_1^{-1}\right)) \ar[r] & 0
}
\end{equation}
has exact rows.
\end{enumerate}
\end{lem}
\begin{proof}
First we prove (1). By local Tate duality (\cite[Theorem 1.2]{Liu-CohomologyDuality}), it is enough to show that $\Hom(D,\mathcal R_{K,L}(\delta_2\chi_{\cycl})) = (0)$. By way of contradiction, suppose that a non-zero $(\varphi,\Gamma_K)$-module morphism $f : D\rightarrow \mathcal R_{K,L}(\delta_2\chi_{\cycl})$ exists. Note that $\Hom(\mathcal R_{K,L}(\delta_2),\mathcal R_{K,L}(\delta_2\chi_{\cycl})) = (0)$ and thus $f$ induces a non-zero map $\mathcal R_{K,L}(\delta_1) \rightarrow \mathcal R_{K,L}(\delta_{2}\chi_{\cycl})$. Applying $D_{\crys}(-)$ and using Lemma \ref{lem:app-rank-one} and part (1) of Lemma \ref{lem:appendix-iso-lemma} we have a diagram
\begin{equation*}
\xymatrix{
D_{\crys}(\mathcal R_{K,L}(\delta_1)) \ar[r]^-{f} \ar[d] & D_{\crys}(\mathcal R_{K,L}(\delta_2\chi_{\cycl})) \\
D_{\crys}(D) \ar[ur]_-{f}
}
\end{equation*}
where every arrow is an isomorphism. This is a contradiction, though. Indeed, $\ker(f)$ is a rank one $(\varphi,\Gamma_K)$-submodule which is necessarily crystalline, since $D$ is semi-stable and a semi-stable character is also crystalline. Thus $D_{\crys}(\ker(f))$ is non-zero and certainly lies in the kernel of the map $f$ induces on the level of $D_{\crys}$. Thus no such $f$ exists.

Now we prove (2). We first note that $H^0(D\left(\delta_2^{-1}\right)) = (0)$. Indeed, a non-zero map $\mathcal R_{K,L}(\delta_2) \rightarrow D$ would induce an inclusion $D_{\crys}(\mathcal R_{K,L}(\delta_2)) \subseteq D_{\crys}(D)$, which contradicts parts (1) and (2) of Lemma \ref{lem:appendix-iso-lemma}. Thus we have  $H^0(\ad(D)) \subseteq H^0(D\left(\delta_1^{-1}\right))$. However, the same reasoning also explains that $H^0(D\left(\delta_1^{-1}\right))$ is one-dimensional, and since $H^0(\ad(D))$ is at least one-dimensional, claim (2) is settled.

The exactness of the bottom row of \eqref{eqn:appendix-diagram} is a consequence of parts (1) and (2) of this lemma. For the top row, the hypothesis of Proposition \ref{prop:app-benois-dimensions}(2), for the sequence
$$
0 \rightarrow D\left(\delta_2^{-1}\right) \rightarrow \ad(D) \rightarrow D\left(\delta_1^{-1}\right) \rightarrow 0,
$$
applies by part (2) of this lemma. Part (3) follows immediately from that proposition.
\end{proof}

\subsection{Weakly-refined deformations}\label{subsec:app-refined}

The goal of this final section is to prove Corollary \ref{corollary:semi-stable-bound}. We begin by recalling the setup. First, we assumed $D$ is triangulated as in \eqref{eqn:appendix-triangulated} and satisfies \eqref{eqn:HT-reg} and \eqref{eqn:st}. Next, recall that $\mathfrak{AR}_L$ denotes the category of local Artinian $L$-algebras with residue field $L$. Given a $(\varphi,\Gamma_K)$-module $D$ we write $\mathfrak X_D : \mathfrak{AR}_L \rightarrow \mathrm{Set}$ for its deformation functor (see \cite[Section 2.2]{Bergdall-Smoothness}). As above, we write $\ad(D) = D\otimes D^{\vee}$ and 
$$
\mathfrak t_D = \mathfrak X_D(L[\varepsilon]) \cong \Ext^1_{(\varphi,\Gamma_K)}(D,D) \cong H^1(\ad(D))
$$ 
for the Zariski tangent space to $\mathfrak X_{D}$. 

Let $\widetilde D \in \mathfrak t_D$. Since the Hodge--Tate weights of $D$ are distinct, at each embedding $\sigma \in \Sigma_K$, by \eqref{eqn:HT-reg}, we may write the Hodge--Sen--Tate weights of $\widetilde D$ according to $\eta_{i,\sigma} = h_{i,\sigma} + \varepsilon d\eta_{i,\sigma} \in L[\varepsilon]$ (see \cite[Section 2.4]{Bergdall-Smoothness}, for instance). Write $\log: \mathcal O_L^\times \rightarrow L$ for the logarithm defined on $1+p\mathcal O_L$ as usual, extended by zero on torsion elements, and homomorphically otherwise. If $\eta = (\eta_{\sigma})_{\sigma \in \Sigma_K)} \in L[\varepsilon]^{\Sigma_K}$ is given by $\eta_\sigma = h_\sigma + \varepsilon d\eta_\sigma$ with $h_\sigma \in \mathbf Z$, then write $z^{\eta}$ for the character $\mathcal O_L^\times \rightarrow L^\times$ given by
\begin{equation*}
 z \overset{z^{\eta}}{\longmapsto} \prod_{\tau \in \Sigma_K} \tau(z)^{h_\tau}(1 + \varepsilon d\eta_\tau \log(\tau(z))).
\end{equation*}
The Hodge--Sen--Tate weights of $z^{\eta}$ is $-\eta$ (meaning $\sigma$-by-$\sigma$). Recall that $\Phi_{\varpi_K} \in L^\times$ is the (only) eigenvalue of $\varphi^f$ acting on $D_{\crys}(D\left(\LT_{\varpi_K}(z^{h_1})\right))$.

\begin{defn}\label{defn:weakly-refined}
Let $\widetilde D \in \mathfrak t_D$.
\begin{enumerate}
\item $\widetilde D$ is called weakly-refined if $D_{\crys}^+(\widetilde D\left(\LT_{\varpi_K}(z^{\widetilde \eta_1})\right))^{\varphi^f = \tilde \Phi}$ is free of rank one over $K_0 \otimes_{\mathbf Q_p} L[\varepsilon]$ for some $\widetilde \Phi \in L[\varepsilon]$ with $\widetilde \Phi \equiv \Phi_{\varpi_K} \bmod \varepsilon$.
\item $\widetilde D$ is called Hodge--Tate if $\widetilde \eta_i = h_i$ for each $i$.
\end{enumerate}
\end{defn}
Since $\Phi_{\varpi_K}$ is a simple eigenvalue for $\varphi^{f}$  acting on $D_{\crys}(D\left(\LT_{\varpi_K}(z^{h_1})\right))$, the subset $\mathfrak t_D^{\Ref} \subset \mathfrak t_D$ of weakly refined deformations is an $L$-linear subspace, equal to the tangent space to the deformation functor $\mathfrak X_{D}^{\phi}$ as in \cite[Section 3.2]{Bergdall-Smoothness}.\footnote{An examination of the proof of Proposition 3.1 in {\em loc.\ cit.} shows that $D$ is not required to be crystalline, despite the context in which that result is proven, which means the reference is still valid without the crystalline hypothesis. Further, $\phi$ here is $\delta_1(\varpi_K)\prod_{\sigma \in \Sigma_K} \sigma(\varpi_K)^{\HT_{\sigma}(\delta_1)}$, i.e.\ the eigenvalue of $\varphi^f$ acting on $D$, not $D\left(\LT_{\varpi_K}(z^{h_1})\right)$.} Also write $\mathfrak t_D^{\HT}$ for the subspace of Hodge--Tate deformations of $D$ and then the intersection of $\mathfrak t_D^{\Ref}$ and $\mathfrak t_D^{\HT}$ is written $\mathfrak t_D^{\Ref,\HT}$. Note that $\widetilde D$ is a Hodge--Tate deformation if and only the underlying rank four $(\varphi,\Gamma_K)$-module is Hodge--Tate in the usual sense (compare with the proof of Lemma \ref{eqn:key-sequence} below).

The Selmer group $H^1_f(\ad(D))$, by definition, parametrizes those deformations $\widetilde D \in \mathfrak t_D$ such that the extension
\begin{equation*}
0 \rightarrow D[1/t] \rightarrow \widetilde D[1/t] \rightarrow D[1/t] \rightarrow 0
\end{equation*}
is split as $\Gamma_K$-modules. In particular, if $\widetilde D \in H^1_f(\ad D)$, then the {\em a priori} left-exact sequence
\begin{equation}\label{eqn:condition}
0 \rightarrow D_{\crys}(D) \rightarrow D_{\crys}(\widetilde D) \rightarrow D_{\crys}(D) \rightarrow 0
\end{equation}
is exact.\footnote{Warning:\ the exactness follows from $\widetilde D$ lying in $H^1_f(\ad(D))$ in general, but the converse does not hold unless $D$ is crystalline.} On the other hand, since $D$ is semi-stable by \eqref{eqn:st}, the Selmer group $H^1_g(\ad(D))$  parametrizes semi-stable deformations.

\begin{lem}\label{lem:containment}
We have $H^1_g(\ad(D)) = H^1_f(\ad(D)) \subset \mathfrak t_{D}^{\Ref,\HT}$.
\end{lem}
\begin{proof}
The first equality is part (2) of Lemma \ref{lem:appendix-adjoint-lemma}. It is clear that $H^1_g(\ad(D)) \subset \mathfrak t_D^{\HT}$ because a semi-stable $(\varphi,\Gamma_K)$-module is also a Hodge--Tate one. Thus it suffices to prove $H^1_f(\ad(D)) \subset \mathfrak t_D^{\Ref}$.

Consider $\widetilde D \in H^1_f(\ad(D))$. We will show that $M=D_{\crys}(\widetilde D\left(\LT_{\varpi_K}(z^{\widetilde \eta_1})\right))$ is free of rank one over $K_0\otimes_{\mathbf Q_p} L[\varepsilon]$, at which point the eigenvalue through which $\varphi^f$ acts must deform $\Phi_{\varpi_K}$ by part (1) of Lemma \ref{lem:appendix-iso-lemma}). Since $H^1_f(\ad(D)) = H^1_g(\ad(D)) \subseteq \mathfrak t_D^{\HT}$, as already noted,  we have $\widetilde \eta_1 = h_1$ is constant, so in fact $M = D_{\crys}(\widetilde D\left(\LT_{\varpi_k}(z^{h_1})\right))$. Since $\widetilde D$ lies in $H^1_f(\ad(D))$, the sequence
\begin{equation*}
0 \rightarrow D_{\crys}(D\left(\LT_{\varpi_k}(z^{h_1})\right)) \rightarrow M \rightarrow D_{\crys}(D\left(\LT_{\varpi_k}(z^{h_1})\right)) \rightarrow 0
\end{equation*}
is exact, by \eqref{eqn:condition} and Lemma \ref{lem:annoying-issue}. Thus $M$ is a $K_0\otimes_{\mathbf Q_p} L[\varepsilon]$-module, and $M/\varepsilon M$ is free of rank one over $K_0\otimes_{\mathbf Q_p} L$. If $m$ is the lift to $M$ of any basis vector of $M/\varepsilon M$, then the submodule $(K_0\otimes_{\mathbf Q_p} L[\varepsilon])\cdot m \subset M$ may be checked to be free of rank one over $K_0\otimes_{\mathbf Q_p} L[\varepsilon]$ (compare with the proof of ``(d) implies (b)'' in \cite[Lemma 3.3]{Bergdall-Smoothness}, which references \cite[Footnote 18, p.\ 78]{BellaicheChenevier-Book}). Since $M$ and $(K_0\otimes_{\mathbf Q_p} L[\varepsilon])\cdot m$ have the same length over $K_0\otimes_{\mathbf Q_p} L$, they must be equal. This completes the proof.
\end{proof}

By Lemma \ref{lem:containment} we now have an exact sequence of $L$-vector spaces
\begin{equation}\label{eqn:key-sequence}
0 \rightarrow \mathfrak t_{D}^{\Ref,\HT}/H^1_f(\ad(D)) \rightarrow \mathfrak t_D^{\Ref}/H^1_f(\ad(D)) \overset{d\widetilde\eta}{\longrightarrow} \bigoplus_{\sigma \in \Sigma_K} L^{\oplus 2}.
\end{equation}
Let $S_2$ be the permutations on the set $\{1,2\}$. We define the critical type of the triangulation \eqref{eqn:appendix-triangulated} as the collection of permutations $c = (c_\sigma)_{\sigma \in \Sigma_K} \in S_2^{\Sigma_K}$ such that $\HT_\sigma(\delta_i) = h_{c_\sigma(i),\sigma}$. (See \cite[Definition 2.2]{Bergdall-Smoothness}.) The next lemma controls the image of $d\widetilde \eta$ in the sequence \eqref{eqn:key-sequence}.

\begin{lem}\label{lem:ramification-lemma}
If $\widetilde D \in \mathfrak t_D^{\Ref}$ then $d\widetilde \eta_{i,\sigma} = d\widetilde \eta_{c_\sigma(i),\sigma}$ for each $i=1,2$ and all $\sigma \in \Sigma_K$. In particular,
\begin{equation*}
\dim_L \mathfrak t_D^{\Ref}/H^1_f(\ad D) \leq \dim_L \mathfrak t_D^{\Ref,\HT}/H^1_f(\ad D) + 2(K:\mathbf Q_p) - \# \{ \sigma \in \Sigma_K \mid c_\sigma \neq 1\}.
\end{equation*}
\end{lem} 
\begin{proof}
The main claim of the lemma implies, for $d\widetilde \eta$ as in \eqref{eqn:key-sequence}, that
$$
\dim_L \im(d\widetilde \eta) \leq \sum_{c_\sigma = 1} 2 + \sum_{c_\sigma \neq 1} 1 = 2(K:\mathbf Q_p) - \#\{\sigma \in \Sigma_K \mid c_\sigma \neq 1\}.
$$
Thus, the second claim follows from the first and \eqref{eqn:key-sequence}. The main claim of the lemma is is proven in  \cite[Theorem 7.1 and Lemma 7.2]{Bergdall-ParabolineVariation},  with some unnecessary hypotheses. We give a proof here for convenience. 

By replacing $D$ with $D' = D\left(\LT_{\varpi_K}(z^{h_1}\right))$ (and $\delta_i$ by $\delta_i \LT_{\varpi_K}(z^{h_1})$) we may assume that $D_{\crys}^+(D)^{\varphi^f = \widetilde \Phi}$ is free of rank one over $K_0\otimes_{\mathbf Q_p} L[\varepsilon]$. Write $D_1$ for the image of $\widetilde D \in H^1(\ad(D))$ under the natural map $H^1(\ad D) \rightarrow H^1(D\left(\delta_1^{-1}\right))$.

To prove the main claim, it is enough to show that $D_1$ belongs to $H^1_f(D\left({\delta_1}^{-1}\right))$. To see this, first note that $H^1_f(D\left(\delta_1^{-1}\right)) = H^1_g(D\left(\delta_1^{-1}\right))$ by the same logic as in part (2) of Lemma \ref{lem:appendix-adjoint-lemma}. Thus $D_1$ is assumed to be a semi-stable (since $D\left(\delta_1^{-1}\right)$ is) and, in particular, Hodge--Tate. Second, consider the matrix of Sen's operator on $\widetilde D$ (viewed as a rank four $(\varphi,\Gamma_K)$-module over $\mathcal R_{K,L}$) in the basis induced from \eqref{eqn:appendix-triangulated}. It is given by
\begin{equation*}
\begin{pmatrix}
\HT_\sigma(\delta_1) & & d\widetilde\eta_{c_\sigma(1),\sigma}-d\widetilde\eta_{1,\sigma}\\
 & \HT_\sigma(\delta_2)& & d\widetilde\eta_{c_\sigma(2),\sigma}-d\widetilde\eta_{1,\sigma}\\
  & & \HT_\sigma(\delta_1) & \\
  & & & \HT_\sigma(\delta_2) 
\end{pmatrix},
\end{equation*}
whereas the matrix of the Sen operator on $\widetilde D_1$ is the upper $3\times 3$-block
\begin{equation*}
\begin{pmatrix}
\HT_\sigma(\delta_1) & & d\widetilde\eta_{c_\sigma(1),\sigma}-d\widetilde\eta_{1,\sigma}\\
 & \HT_\sigma(\delta_2)&\\
  & & \HT_\sigma(\delta_1)& 
\end{pmatrix}.
\end{equation*}
Since $D_1$ is Hodge--Tate, Sen's operator is semi-simple and thus $d\widetilde\eta_{c_\sigma(1),\sigma} = d\widetilde\eta_{1,\sigma}$. This proves the main claim for $i = 1$. For $i = 2$, either it is trivial because $c_\sigma = 1$ or it is equivalent to the case of $i = 1$. Regardless, it remains to prove that $D_1 \in H^1_f(D \left(\delta_1^{-1})\right))$. 

Explicitly, $D_1$ is defined as
\begin{equation}\label{eqn:D1-definition}
D_1 = \ker\left(\widetilde D \twoheadrightarrow D \twoheadrightarrow \mathcal R_{K,L}(\delta_2)\right),
\end{equation}
and, since $D_{\crys}$ is left-exact, we also have an exact sequence
\begin{equation}\label{eqn:ses-in-proof}
0 \rightarrow D_{\crys}(D)^{\varphi^{f} = \Phi_{\varpi_K}} \rightarrow D_{\crys}(D_1)^{(\varphi^f = \Phi_{\varpi_K})} \rightarrow D_{\crys}(\mathcal R_{K,L}(\delta_1))^{\varphi^f = \Phi_{\varpi_K}}.
\end{equation}
Here $(-)^{(\ast)}$ means the generalized eigenspace for $(\ast)$. (Taking generalized eigenspaces is exact.) By \eqref{eqn:D1-definition} and part (2) of Lemma \ref{lem:appendix-iso-lemma}, we see that $D_{\crys}(D_1)^{(\varphi^f = \Phi_{\varpi_K})} = D_{\crys}(\widetilde D)^{(\varphi^f = \Phi_{\varpi_K})}$, which has $L$-dimension $2(K_0:\mathbf Q_p)$ because the right-hand side of free of rank one over $K_0\otimes_{\mathbf Q_p} L[\varepsilon]$. On the other hand, part (1) of Lemma \ref{lem:appendix-iso-lemma} implies the two outside terms of \eqref{eqn:ses-in-proof} each have $L$-dimension $(K_0:\mathbf Q_p)$. In particular, by counting dimensions we deduce  that \eqref{eqn:ses-in-proof} is surjective on the right. In particular, the natural sequence
$$
0 \rightarrow D_{\crys}(D) \rightarrow D_{\crys}(D_1) \rightarrow D_{\crys}(\mathcal R_{K,L}(\delta_1)) \rightarrow 0
$$
is exact. Since $\delta_1$ is crystalline, Lemma \ref{lem:annoying-issue} implies that $\widetilde D_1 \in H^1_f(D\left({\delta_1}^{-1})\right))$, which completes the proof.
\end{proof}

Following Lemma \ref{lem:ramification-lemma}, we see that in order to control $\dim_L \mathfrak t_D^{\Ref}/H^1_f(\ad(D))$, it suffices to control $\dim_L \mathfrak t_D^{\Ref,\HT}/H^1_f(\ad(D))$ where the Hodge--Tate condition is added as well. Here is the the crucial estimate for that dimension, the left-hand term in \eqref{eqn:key-sequence}.

\begin{prop}\label{prop:constant-weight-bounds}
$\dim_L \mathfrak t_D^{\Ref,\HT}/H^1_f(\ad(D)) \leq \# \{\sigma \in \Sigma_K \mid c_\sigma \neq 1\}$.
\end{prop}
\begin{proof}
In ths proof we use the notation $H^1_{/f}(E)$ for $H^1(E)/H^1_f(E)$. With this notation, we note that there is a natural inclusion  $\mathfrak t_D^{\Ref,\HT}/H^1_f(\ad D) \subset H^1_{/f}(D\left(\delta_2^{-1}\right))$. To see this, consider the diagram
\begin{equation}\label{eqn:second-diagram}
\xymatrix{
 & & \mathfrak t_D^{\Ref,\HT}/H^1_f(\ad(D)) \ar[d]\\
0 \ar[r] & H^1_{/f}(D\left(\delta_2^{-1}\right)) \ar[r] & \ar[r] H^1_{/f}(\ad(D)) \ar[r] & H^1_{/f}(D\left(\delta_1^{-1}\right)) \ar[r] & 0.
}
\end{equation}
The bottom row is exact by part (3) of Lemma \ref{lem:appendix-cohomology}, and then the proof of Lemma \ref{lem:ramification-lemma} implies that the composition from the top to the lower right is trivial. In fact, we make a stronger claim than just noted:

\begin{claim}
We have
$$
\mathfrak t_D^{\Ref,\HT}/H^1_f(\ad(D)) \subseteq \ker\left(H^1_{/f}(D\left(\delta_2^{-1}\right)) \rightarrow H^1_{/f}(\mathcal R_{K,L}(\delta_2\delta_2^{-1}))\right).
$$
Moreover, the natural map
$$
H^1_{/f}(D\left(\delta_2^{-1}\right)) \rightarrow H^1_{/f}(\mathcal R_{K,L}(\delta_2\delta_2^{-1}))
$$
is surjective. Thus,
$$
\dim_L \mathfrak t_D^{\Ref,\HT}/H^1_f(\ad(D)) \leq \dim H^1_{/f}(D\left(\delta_2^{-1}\right)) - \dim_L H^1_{/f}(\mathcal R_{K,L}(\delta_2\delta_2^{-1})).
$$
(We write $\delta_2\delta_2^{-1}$ to emphasize that the maps on cohomology is induced from the natural quotient $D\rightarrow \mathcal R_{K,L}(\delta_2)$.)
\end{claim}
The proof of the first portion of the claim follows from the technique in \cite[Section 3.3]{Bergdall-Smoothness}. Let $\widetilde D \in \mathfrak t_{D}^{\Ref,\HT}$. By the first paragraph of this proof, after changing $\widetilde D$ by an element in $H^1_f(\ad(D))$, we may suppose that $\widetilde D$ lies in the image of $H^1(D\left(\delta_2^{-1}\right)) \rightarrow H^1(\ad(D))$. By \cite[Lemma 3.8(a)]{Bergdall-Smoothness} there is a constant deformation $\mathcal R_{K,L}(\delta_1)[\varepsilon] \hookrightarrow \widetilde D$ with saturated image. Let $\mathcal R_{K,L[\varepsilon]}(\widetilde \delta_2)$ be the cokernel. By \cite[Lemma 3.8(b)]{Bergdall-Smoothness} the image of $\widetilde D$ in $H^1(\mathcal R_{K,L}(\delta_2\delta_2^{-1}))$ is the deformation $\widetilde \delta_2$ of $\delta_2$. But $\widetilde \delta_2$ is Hodge--Tate because $\widetilde D$ is, and a Hodge--Tate deformation of a crystalline character is a crystalline character, so $\widetilde D$ has trivial image in $H^1_{/f}(\mathcal R_{K,L}(\delta_2\delta_2^{-1}))$.

We now prove the second portion of the claim. By local Tate duality and \cite[Corollary 1.4.10]{Benois-GreenbergL} (which describes how the Selmer groups $H^1_f$ behave under duality), it is equivalent to check that $H^1_f(\mathcal R_{K,L}(\delta_2\delta_2^{-1}\chi_{\cycl})) \rightarrow H^1_f(D\left(\delta_2^{-1}\chi_{\cycl}\right))$ is injective. If $\iota \in H^0(\mathcal R_{K,L}(\delta_1\delta_2^{-1}\chi_{\cycl})) = \Hom(\mathcal R_{K,L}, \mathcal R_{K,L}(\delta_1\delta_2^{-1}\chi_{\cycl}))$ is a non-zero morphism, then its image in $H^1(\mathcal R_{K,L}(\delta_2\delta_2^{-1}\chi_{\cycl}))$ is the pullback $D_{\iota}$ that into a diagram
\begin{equation}\label{eqn:D-iota}
\xymatrix{
0 \ar[r] & \mathcal R_{K,L}(\delta_2\delta_2^{-1}\chi_{\cycl}) \ar@{=}[d]\ar[r] & D_{\iota} \ar[d] \ar[r] & \mathcal R_{K,L} \ar[d]^-{\iota} \ar[r] & 0\\
0 \ar[r] & \mathcal R_{K,L}(\delta_2\delta_2^{-1}\chi_{\cycl}) \ar[r] & D^{\vee}\left(\delta_2\chi_{\cycl}\right) \ar[r] & \mathcal R_{K,L}(\delta_2\delta_1^{-1}\chi_{\cycl}) \ar[r] & 0.
}
\end{equation}
The vertical arrows in \eqref{eqn:D-iota} are all injections by construction. Since $D$ is semi-stable, non-crystalline, of rank two, the same is true for $D^{\vee}\left(\delta_2\chi_{\cycl}\right)$. By part (1) of Lemma \ref{lem:appendix-iso-lemma} the functor $D_{\crys}$ induces an isomorphism 
$$
D_{\crys}(\mathcal R_{K,L}(\delta_2\delta_2^{-1}\chi_{\cycl})) \cong D_{\crys}(D^{\vee}\left(\delta_2\chi_{\cycl}\right)),
$$
which must factor through an isomorphism
$$
D_{\crys}(\mathcal R_{K,L}(\delta_2\delta_2^{-1}\chi_{\cycl})) \cong D_{\crys}(D_\iota).
$$
So, $D_{\crys}(D_\iota) \rightarrow D_{\crys}(\mathcal R_{K,L})$ is the zero map, which proves that $D_\iota \not\in H^1_f(D\left(\delta_2\delta_2^{-1}\chi_{\cycl}\right))$.

With the claim proven, we can finish the proof of the proposition. Recall that part (1) of Proposition \ref{prop:app-benois-dimensions} determines the $L$-dimension of the $H^1_f(E)$, when $E$ is a semi-stable $(\varphi,\Gamma_K)$-module. In particular, by the Euler--Poincar\'e formula (see \cite[Theorem 1.2(a)]{Liu-CohomologyDuality}) we have that if $E$ has rank $d$, then
$$
\dim_L H^1_{/f}(E) = d\cdot (K:\mathbf Q_p) - h^-(E) - \dim_L H^2(E),
$$
where we recall that $h^-(E)$ is the number of negative Hodge--Tate weights. Taking $E = \mathcal R_{K,L} \cong \mathcal R_{K,L}(\delta_2\delta_2^{-1})$ we see that
$$
\dim_L H^1_{/f}(\mathcal R_{K,L}(\delta_2\delta_2^{-1})) = (K:\mathbf Q_p).
$$
Taking $E = D\left(\delta_2^{-1}\right)$ we see, using part (1) of Lemma \ref{lem:appendix-cohomology} and the fact that $c_\sigma$ is non-trivial if and only if $\HT_\sigma(\delta_1) > \HT_\sigma(\delta_2)$, that
\begin{align*}
\dim_L H^1_{/f}(D\left(\delta_2^{-1}\right)) &= 2(K:\mathbf Q_p) - \# \{\sigma \mid \HT_\sigma(\delta_1) < \HT_\sigma(\delta_2)\} \\
&= (K:\mathbf Q_p) + \# \{\sigma \in \Sigma_K \mid c_\sigma \neq 1\}.
\end{align*}
Putting these calculations together, we see that
\begin{align*}
\dim \mathfrak t_D^{\Ref,\HT} &\leq \dim_L H^1_{/f}(D\left(\delta_2^{-1}\right)) - \dim_L H^1_{/f}(\mathcal R_{K,L}(\delta_2\delta_2^{-1}))\\
&= (K:\mathbf Q_p) + \# \{\sigma \in \Sigma_K \mid c_\sigma \neq 1\} - (K:\mathbf Q_p)\\
&= \#\{\sigma \in \Sigma_K \mid c_\sigma \neq 1\}.
\end{align*}
The proposition has been proven.
\end{proof}

Finally, we have the following corollary on the dimension of $\mathfrak t_D^{\Ref}/H^1_f(\ad(D))$.

\begin{cor}\label{corollary:semi-stable-bound}
$\dim_L \mathfrak t_D^{\Ref}/H^1_f(\ad(D)) \leq 2(K:\mathbf Q_p)$.
\end{cor}
\begin{proof}
This follows from Lemma \ref{lem:ramification-lemma} and Proposition \ref{prop:constant-weight-bounds}.
\end{proof}

%% file: decent-appendix.tex
\numberwithin{equation}{section}

The main goal of this appendix is the following theorem.

\begin{thm}\label{thm:decentvanishinggalois} Assume that $p \geq 5$. Let $x \in \mathscr{E}(\mathfrak{n})(\overline{\mathbf{Q}_p})$ be any point lying over a given weight $\lambda$, and let $\rho_x:G_F \to \mathrm{GL}_2(\overline{\mathbf{Q}_p})$ be the associated Galois representation constructed in \cite{JohanssonNewton}. If the image of the (semisimplified) mod-$p$ reduction $\overline{\rho_x}:G_F \to \mathrm{GL}_2(\overline{\mathbf{F}_p})$ contains $\mathrm{SL}_2(\mathbf{F}_p)$, then $H^{\ast}(\mathfrak{n},\mathscr{D}_{\lambda})_{\mathfrak{m}_x}$ is concentrated in the middle degree $d$, and in particular $x$ lies in $\mathscr{E}(\mathfrak{n})_{\mathrm{mid}}$.
\end{thm}

To prove this, pick a sufficiently large "radius" $\mathbf{s}$ for $\lambda$, and let $\mathbf{D}_{\lambda}^{\mathbf{s},\circ}$ be the associated unit ball in $\mathbf{D}_{\lambda}^{\mathbf{s}}$. Let $\mathfrak{m} \subset \mathbf{T}_{\mathbf{Z}_p}(\mathfrak{n})$ be the maximal ideal in the integral Hecke algebra associated with $\overline{\rho_x}$. Then $H^{\ast}(\mathfrak{n},\mathscr{D}_{\lambda})_{\mathfrak{m}_x}$ is obtained from $H^{\ast}(\mathfrak{n},\mathbf{D}_{\lambda}^{\mathbf{s},\circ})_{\mathfrak{m}}$ by inverting $p$ and taking a further localization, so it is clearly enough to prove the following result.

\begin{thm}Notation as above, the cohomology group $H^{\ast}(\mathfrak{n},\mathbf{D}_{\lambda}^{\mathbf{s},\circ})_{\mathfrak{m}}$ is concentrated in the middle degree.
\end{thm}

\begin{proof} We deduce this from work in progress of Caraiani-Tamiozzo \cite{CaraianiTamiozzo}. In the notation of the present paper, their work implies that if $p\geq 5$ and $K_p K^p \subset \mathrm{GL}_2(\mathbf{A}_{F}^{f})$ is any level subgroup, then $H^{\ast}_c(Y_{K_p K^p},\mathbf{Z}/p^n)_{\mathfrak{m}}$ is concentrated in  degree $d$ for any maximal ideal $\mathfrak{m}$ whose residual Galois representation contains $\mathrm{SL}_2(\mathbf{F}_p)$ in its image. By the Hochschild-Serre spectral sequence, we immediately deduce that if $M$ is any set-theoretically finite $\mathbf{Z}_p$-module with a continuous $K_p$-action, the cohomology groups $H^{\ast}_c(Y_{K_p K^p},M)_{\mathfrak{m}}$ are concentrated in degrees $\geq d$.  Rerunning this argument with $M$ replaced by its Pontryagin dual and with $\mathfrak{m}$ replaced by the ``dual" maximal idea, Poincar\'e duality then shows that $H^{\ast}(Y_{K_p K^p},M)_{\mathfrak{m}}$ is concentrated in degrees $\leq d$. But $H^{\ast}_c(Y_{K_p K^p},M)_{\mathfrak{m}}=H^{\ast}(Y_{K_p K^p},M)_{\mathfrak{m}}$ by \cite[Theorem 4.2]{NewtonThorne}, so we deduce that $H^{\ast}_c(Y_{K_p K^p},M)_{\mathfrak{m}}$ is concentrated in degree $d$.

We now apply these observations with $K^p = K_1(\mathfrak{n})$, $K_p = I$, and $M=\mathbf{D}_{\lambda}^{\mathbf{s},\circ} / \mathrm{Fil}^j$, and with $\mathfrak{m}$ as above, where $\mathrm{Fil}^j \subset \mathbf{D}_{\lambda}^{\mathbf{s},\circ}$ is the filtration constructed in \cite{JohanssonNewton}. This shows that $H^{\ast}(\mathfrak{n},\mathbf{D}_{\lambda}^{\mathbf{s},\circ}/\mathrm{Fil}^j)_{\mathfrak{m}}$ is concentrated in degree $d$, so now we conclude upon taking the limit over $j$, using \cite[Lemma 5.1.1]{JohanssonNewton}.

\end{proof}

\begin{proof}[Proof of Theorem \ref{thm:decency-ubiquitous}] As in the theorem, assume $p$ is a sufficiently large split prime. Then 2(c) follows from the assumption that $p$ splits completely, 2(b) follows from \cite{NewtonThorneSelmer}, and 2(a) follows from Theorem \ref{thm:decentvanishinggalois} by Serre's open image theorem \cite{SerreOpen}.
\end{proof}